\documentclass[a4paper]{rn-smai-jcm}
\usepackage{amssymb,amsmath,amsthm,epsfig,verbatim,tikz,pstricks,url}
\usepackage{ifpdf}  
\newcommand{\Rplus}{{\mathbb R}_{>0}}
\newcommand{\Rgeq}{{\mathbb R}_{\geq 0}}

\newcommand{\R}{{\mathbb R}}

\newcommand{\bZ}{\mathbb{Z}}
\newcommand{\II}{{\rm I\!I}}

\newcommand{\tr}{\operatorname{tr}}
\newcommand{\spa}{\operatorname{span}}

\newcommand{\Gauss}{{\mathcal{K}}}
\newcommand{\GT}{{\Gamma_T}}
\newcommand{\GTi}{{\Gamma_{i,T}}}
\newcommand{\GhT}{{\Gamma^h_T}}
\newcommand{\GhTi}{{\Gamma^h_{i,T}}}
\newcommand{\dH}[1]{\;{\rm d}{\mathcal{H}}^{#1}} 

\newcommand{\spont}{{\overline\varkappa}}
\newcommand{\kappastar}{\vec\varkappa\rule{0pt}{0pt}^\star}

\newcommand{\Vh}{\underline{V}^h(\Gamma^m)}
\newcommand{\Vhi}{\underline{V}^h(\Gamma^m_i)}
\newcommand{\Vhone}{\underline{V}^h(\Gamma^m_1)}
\newcommand{\Vhtwo}{\underline{V}^h(\Gamma^m_2)}
\newcommand{\Vhiz}{\underline{V}^h(\Gamma^0_i)}

\newcommand{\Vhpartialt}{\underline{V}^h(\gamma^h(t))}
\newcommand{\Whpartialt}{W^h(\gamma^h(t))}
\newcommand{\Vhpartial}{\underline{V}^h(\gamma^m)}
\newcommand{\Vhipartial}{\underline{V}^h(\partial\Gamma^m_i)}
\newcommand{\Vhpartialz}{\underline{V}^h(\gamma^0)}

\newcommand{\Wh}{W^h(\Gamma^m)}
\newcommand{\Whi}{W^h(\Gamma^m_i)}

\newcommand{\Vt}{[H^1(\Gamma(t))]^d}

\newcommand{\Vti}{[H^1(\Gamma_i(t))]^d}
\newcommand{\VLtwoti}{[L^2(\Gamma_i(t))]^d}

\newcommand{\VLtwotone}{[L^2(\Gamma_1(t))]^d}
\newcommand{\Vpartial}{[H^1(\gamma(t))]^d}
\newcommand{\VLtwopartial}{[L^2(\gamma(t))]^d}
\newcommand{\xspacec}{[H^1_\gamma(\Gamma(t))]^d}

\newcommand{\Vht}{\underline{V}^h(\Gamma^h(t))}
\newcommand{\Vhti}{\underline{V}^h(\Gamma_i^h(t))}
\newcommand{\Vhtz}{\underline{V}^h(\Gamma^h(0))}
\newcommand{\Vhzerot}{\underline{V}_0^h(\Gamma^h(t))}
\newcommand{\Vhzeroti}{\underline{V}^h_{0}(\Gamma^h_i(t))}
\newcommand{\Wht}{W^h(\Gamma^h(t))}
\newcommand{\Whti}{W^h(\Gamma_i^h(t))}
\newcommand{\hatVh}{\underline{\hat{V}}\rule{0pt}{0pt}^h(\Gamma^m)}

\newcommand{\nabs}{\nabla_{\!s}}
\newcommand{\Id}{{\rm Id}}
\newcommand{\id}{{\rm id}}
\newcommand{\deldel}[1]{\frac{\delta}{{\delta}#1}}
\newcommand{\dd}[1]{\frac{\rm d}{{\rm d}#1}}
\newcommand{\ddt}{\dd{t}}

\newcommand{\matpartx}{\partial_t^\circ}
\newcommand{\matpartxh}{\partial_t^{\circ,h}}

\newcommand{\unitt}{\vec{\mathfrak t}}

\newcommand{\ttau}{\Delta t}

\def\epsilon{\varepsilon} 
\newcommand{\mat}[1]{\underline{\underline{#1}}\rule{0pt}{0pt}}

\def\hat{\widehat}

\allowdisplaybreaks

\begin{document}

\title[Gradient flow dynamics of two-phase biomembranes]{
Gradient flow dynamics of two-phase biomembranes: Sharp interface variational
formulation and finite element approximation}

\author[J. W. Barrett]{\firstname{John} \middlename{W.} \lastname{Barrett}}
\address{Department of Mathematics, 
Imperial College London, London, SW7 2AZ, UK}
\email{j.barrett@imperial.ac.uk}

\author[H. Garcke]{\firstname{Harald} \lastname{Garcke}}
\address{Fakult{\"a}t f{\"u}r Mathematik, Universit{\"a}t Regensburg, 
93040 Regensburg, Germany}
\email{harald.garcke@ur.de}

\author[R. N\"urnberg]{\firstname{Robert} \lastname{N\"urnberg}}
\address{Department of Mathematics, 
Imperial College London, London, SW7 2AZ, UK}
\email{robert.nurnberg@imperial.ac.uk}

\date{}



\begin{abstract}
A finite element method for the evolution of a two-phase membrane in a
sharp interface formulation is introduced. The evolution equations are
given as an $L^2$--gradient flow of an energy involving an elastic
bending energy and a line energy. In the two phases Helfrich-type
evolution equations are prescribed, and on the interface, an evolving
curve on an evolving surface, highly nonlinear boundary conditions have
to hold. Here we consider both $C^0$-- and $C^1$--matching conditions
for the surface at the interface.
A new weak formulation is introduced, allowing for a
stable semidiscrete parametric finite element approximation of the
governing equations. In addition, we show existence and uniqueness for a
fully discrete version of the scheme. Numerical simulations demonstrate 
that the approach can deal with a multitude of geometries. In
particular, the paper shows the first computations based on a sharp
interface description, which are not restricted to the axisymmetric case.  
\end{abstract} 

\maketitle
\renewcommand{\shortauthors}{J. W. Barrett, H. Garcke, R. N\"urnberg}
%

\setcounter{equation}{0}
\section{Introduction} \label{sec:1}

Two-phase elastic membranes, consisting of coexisting fluid domains,
have received a lot of attention in the last 20 years. The interest in
two-phase membranes in particular was triggered by the multitude of
different shapes observed in experiments with inhomogeneous
biomembranes and vesicles. Biomembranes are typically formed as a
lipid bilayer, and often multiple lipid components are involved, which
laterally can separate into coexisting phases with different
properties. Among the complex morphologies that appear are micro-domains, which
resemble lipid rafts, and these are of huge interest in biology and
medicine. As the thickness of the membrane is much smaller than its
lateral length scale, typically the membrane is modelled as a
two-dimensional hypersurface in three dimensional Euclidean
space. The equilibrium shape of the membrane is obtained by minimizing an 
energy which --besides other contributions-- contains bending energies
involving the mean curvature and the Gaussian curvature of the
membrane. If different phases occur, parameters in the curvature energy
are inhomogeneous, leading to an interesting free boundary problem as
well as to a plethora of different shapes. We refer to
\cite{BaumgartHW03}, where multi-component giant unilamellar
vesicles (GUVs) separating into different phases were studied. 
These authors were 
able to optically resolve interactions between the different phases, its
curvature elasticity and the line tension of its interface.

There have been several studies on theoretical and numerical aspects
of two-phase membranes taking curvature elasticity and line energy
into account, see e.g.\
\cite{JulicherL93,JulicherL96,TuO-Y04,BaumgartDWJ05,WangD08,DasJB09,%
LowengrubRV09,ElliottS10,ElliottS10a,ElliottS13,Helmers11,Helmers13,%
ChoksiMV13,MerckerM-CRH13,CoxL15,nsns2phase},
which we discuss in the following. 

The by now classical model for a one-phase membrane rests on the
Canham--Helfrich--Evans elastic bending energy 
\begin{equation*}
\tfrac12\,\alpha\,
\int_\Gamma(\varkappa-\spont)^2\dH{2}+
\alpha^G\,\int_\Gamma\Gauss \dH{2} \,,
\end{equation*}
where $\Gamma$ is a closed two-dimensional hypersurface and $\mathcal{H}^{2}$
denotes the two-dimensional Hausdorff measure. 
The mean curvature of $\Gamma$ is denoted by $\varkappa$, and
$\Gauss$ is its Gaussian curvature. The constants $\alpha$ and $\alpha^G$
are bending rigidities, while $\spont$ is the spontaneous curvature
reflecting asymmetry in the membrane introduced, for instance, by
different environments on both sides of the membrane.

In a fundamental work, J{\"u}licher and Lipowsky 
(\cite{JulicherL93,JulicherL96}) generalized the
Canham--Helfrich--Evans model to two-phase membranes. The geometry is
now given by two smooth surfaces $\Gamma_1$ and $\Gamma_2$, with a
common boundary $\gamma$. In general, the constants
$\alpha$, $\alpha^G$ and $\spont$ take different values in the two phases
$\Gamma_1$ and $\Gamma_2$, which we will denote with an index $i$. On
the curve $\gamma$ line tension effects play an important role, and the
total energy introduced in \cite{JulicherL93,JulicherL96} is given as 
\begin{equation} \label{eq:Eneu}
E((\Gamma_i)^2_{i=1}) = \sum_{i=1}^2 \left[
\tfrac12\,\alpha_i\,\int_{\Gamma_i} (\varkappa_i - \spont_i)^2 \dH{2}
+ \alpha^G_i\,\int_{\Gamma_i} \Gauss_i \dH{2}
\right] 
+ \varsigma\,\mathcal{H}^1(\gamma)\,,
\end{equation}
where the constant $\varsigma \in \Rgeq$ denotes a possible line tension, and
where an index $i\in\{1,2\}$ states that quantities such as the
curvatures and physical constants are evaluated with respect to
$\Gamma_i$. Of course, $\mathcal{H}^1$ denotes the one-dimensional 
Hausdorff measure. 

In \cite{JulicherL96} it is assumed that the surface
$\Gamma=\Gamma_1\cup\gamma\cup\Gamma_2$ is a $C^1$--surface, meaning in
particular that the normal to $\Gamma$ is continuous across the phase
boundary $\gamma$. The works \cite{Helmers11,Helmers13,Helmers15},
on the other hand, 
also allow for discontinuities of the normal at $\gamma$. The first
variation of the energy $E$ in (\ref{eq:Eneu}) has been derived in
\cite{ElliottS10} for the $C^1$--case and in \cite{Wutz10} for the
$C^1$-- and the $C^0$--case. It is the goal of this paper to develop a
numerical method for a gradient flow evolution of the energy $E$. 
To be more precise, we will consider an evolution of the form 
\begin{equation} \label{eq:VrV}
\left\langle \vec{\mathcal{V}},\vec{\chi}\right\rangle_\Gamma +
\varrho\left\langle\vec{\mathcal{V}}, \vec{\chi}\right\rangle_\gamma =
\left[\deldel{\Gamma}\,E((\Gamma_i)^2_{i=1})\right](\vec{\chi})\,.
\end{equation}
Here $\vec{\mathcal{V}}$ is the velocity of the surface,
$\deldel{\Gamma}\,E$ is the first variation of the energy,
$\vec{\chi}$ 
is a test vector field on the surface related to directions in which
one perturbs the given surface $\Gamma$, and $\varrho\geq 0$ is a given
constant. In addition, $\langle \cdot,\cdot\rangle_\Gamma$
and $\langle \cdot,\cdot\rangle_\gamma$ 
denote the $L^2$--inner products on the surface
$\Gamma$ and on the curve $\gamma$, respectively. 
The evolution of the surface is
hence given as a steepest descent dynamics with respect to a weighted
$L^2$--inner product that combines contributions from the surface and the
curve. It will turn out that the governing equations in the case where
the surface is restricted to be $C^1$ are 
\begin{equation}\label{eq:intro1}
\vec{\mathcal{V}} = [-\alpha_i\,\Delta_s\,\varkappa_i
+\tfrac{1}{2}\,\alpha_i\,(\varkappa_i-\spont_i)^2\,\varkappa_i
  -\alpha_i\,(\varkappa_i-\spont_i)\,|\nabs\,\vec\nu_i|^2]\, \vec\nu_i 
\quad\mbox{ on } \Gamma_i\,,
\end{equation}
together with the boundary conditions on $\gamma=\partial\Gamma_i$: 
\begin{subequations}
\begin{align}\label{eq:intro2}
& \alpha_1\,(\varkappa-\spont_1)+\alpha^G_1\,\vec\varkappa_\gamma\,.\,
\vec\nu=\alpha_2\,(\varkappa_2-\spont_2) 
+ \alpha^G_2\,\vec\varkappa_\gamma\,.\,\vec\nu\,,\\
& [\alpha_i\,(\nabs\,\varkappa_i)]_1^2\,.\,\vec\mu 
- [\alpha^G_i]_1^2\,\tau_s +
\varsigma\,\vec\varkappa_\gamma\,.\,\vec\nu
 = \varrho\,\vec{\mathcal{V}}\,.\,\vec\nu \,, \label{eq:intro3}\\
& -\tfrac12\, [ \alpha_i\,(\varkappa_i - \spont_i)^2]_1^2
+ [\alpha_i\,(\varkappa_i - \spont_i)\,(\varkappa_i -
\vec\varkappa_\gamma\,.\,\vec\nu)]_1^2
+ [\alpha^G_i]_1^2\,\tau^2 
 + \varsigma\,\vec\varkappa_\gamma\,.\,\vec\mu = 
\varrho\,\vec{\mathcal{V}}\,.\,\vec\mu\,.
\label{eq:intro4}
\end{align}
\end{subequations}
Equation (\ref{eq:intro1}),
with $\Delta_s$ and $\nabs$ denoting the surface Laplacian and the surface
gradient on $\Gamma_i$, respectively, is Willmore flow taking spontaneous
curvature effects into account. The boundary condition
(\ref{eq:intro2}), with $\vec\varkappa_\gamma$ denoting the curvature vector on
$\gamma(t)$, generalizes the equation for the mean curvature in Navier boundary
conditions, appearing for example in \cite[(6)]{DeckelnickGR17}. 
The equations (\ref{eq:intro3},c), with $\tau$ being the geodesic torsion of 
the curve $\gamma(t)$ on $\Gamma(t)$ and with 
$[a_i]_1^2 = a_2 - a_1$
denoting the jump of $a$ across $\gamma(t)$, appear in the case
$\varrho=0$ in \cite[(3.17), (3.18)]{ElliottS10a},
where additional terms to fix the surface areas and the enclosed volume appear. 
In the axisymmetric
case, the equations (\ref{eq:intro2}--c) reduce to the
equations studied in \cite{JulicherL96}. Similar conditions have been
derived in \cite{TuO-Y04}, and it has already been discussed in
\cite[Appendix B]{ElliottS10a} that these authors miss one term. For
positive $\varrho$ the equations (\ref{eq:intro3},c)
give rise to dynamic boundary conditions taking into account an
additional dissipation mechanism at the boundary. A similar condition
for semi-free boundary conditions has been analyzed in
\cite[(1.3)]{AbelsGM15}. 
For evolutions where the surface areas of $\Gamma_1$ and $\Gamma_2$, as well as
the volume enclosed by $\Gamma$, are conserved, additional terms appear
in (\ref{eq:intro1}) and (\ref{eq:intro4}),
see (\ref{eq:gradflowlambda}) and (\ref{eq:C1bc3}), below.
Moreover, in the case that the surface $\Gamma$ is just
continuous, the boundary conditions (\ref{eq:intro2}--c) have to be replaced,
and we refer to (\ref{eq:C0bc1},b), below, for the relevant equations.

Numerically mainly the $C^1$--case has been studied, with the exception
of \cite{Helmers13}, where $C^0$--surfaces with kinks in the 
axisymmetric case were studied numerically with the help of a phase field 
method. In the $C^1$--case already in \cite{JulicherL96} several 
two-phase equilibrium shapes in the axisymmetric case were computed by 
solving a governing
boundary value problem for a system of ordinary differential
equations. Based on research on model membranes, see
\cite{BaumgartHW03}, it has now become possible to perform a
systematic analysis of the influence of parameters also in the case of
two-phase coexistence. We refer to \cite{BaumgartDWJ05}, where
experimental vesicle shapes were compared with shapes obtained by solving
numerically the axisymmetric shape equations derived in
\cite{JulicherL96}. In this context, we also refer to \cite{CoxL15},
where, in contrast to the above works, also the effect of
spontaneous curvature is taken into account in the axisymmetric case. These
authors were able to show that spontaneous curvatures already in an
axisymmetric setup give rise to a multitude of morphologies not seen
in the case without spontaneous curvature.

Almost all numerical results mentioned so far were for a sharp
interface setup. Another successful approach uses a phase field to
describe the two phases on the membrane. Line energy in this context
is replaced by a Ginzburg--Landau energy like in the classical
Cahn--Hilliard theory. 
We refer to
\cite{WangD08,LowengrubRV09,ElliottS10,ElliottS10a,ElliottS13,Helmers11,Helmers13,MerckerM-CRH13,MerckerM-C15}
for numerical results based on the phase field approach. The above
papers use a gradient flow approach to obtain equilibrium shapes in
the large time limit. An evolution law using a Cahn--Hilliard equation
on the membrane coupled to surface and bulk (Navier--)Stokes equations
has been studied by the present authors in \cite{nsns2phase}.

Rigorous analytical results for two-phase elastic membranes are very
limited. So far only results for the axisymmetric case are known.
We refer to the work \cite{ChoksiMV13}, where the existence of
global minimizers for axisymmetric multi-phase membranes was shown, 
and the works \cite{Helmers11,Helmers13,Helmers15}, where the sharp interface
limit of the phase field approach in an axisymmetric situation was studied.
Existence results for the evolution problem are not available in the literature
so far and should be addressed in the future.

It is the goal of this paper to introduce a finite element approximation
for a gradient flow dynamics of the membrane energy $E$, which is based
on a sharp interface approach. Instead of using a phase field on the
membrane, we will directly discretize the curve $\gamma$ separating the
two phases $\Gamma_1$ and $\Gamma_2$. In three dimensions the total
surface $\Gamma$ will be discretized with the help of polyhedral
surfaces consisting of a union of triangles. The curve $\gamma$ is
discretized as a polygonal curve in 
${\R}^3$ fitted to the discretization of $\Gamma$ in the sense that 
the polygonal curve is the boundary of the open polyhedral sets
$\Gamma_1$ and $\Gamma_2$. The boundary conditions
(\ref{eq:intro2}--c) are highly nonlinear and involve
derivatives of an order up to three when formulated with the help of a 
parameterization. It is hence highly non-trivial to discretize them in a
piecewise linear setup. In this work, a splitting method is used, which
basically uses the position vectors of the nodes and an approximation of
the mean curvature vector as unknowns. The approach in this paper relies
on a discretization of mean curvature leading to good mesh properties.
This discretization was introduced by the present authors in
\cite{triplej,gflows3d} and has been 
previously used for closed and open membranes, see
\cite{pwfade, pwfopen} and for elastic curvature flow of curves
with junctions, see \cite{pwftj}.

We will use the variational structure of the problem to derive a
discretization which will turn out to be stable in a spatially discrete
and continuous-in-time semidiscrete formulation. In order to do so, we
will make use of an appropriate Lagrangian and will use ideas of PDE
constrained optimization.

The outline of this paper is as follows. In the subsequent section we
will formulate the governing equations with all the details. 
In Section~\ref{sec:2}
a weak formulation is introduced using the calculus of PDE constrained
optimization. A semidiscrete discretization is formulated in
Section~\ref{sec:3}.
For this scheme also energy decay properties and conservation properties
are shown. In Section~\ref{sec:5} a fully discrete version of the scheme is
introduced, leading to a linear system at each time level, 
which is shown to be uniquely solvable.
In Section~\ref{sec:6} we discuss ideas on how to solve the resulting 
linear algebra problems numerically. In Section~\ref{sec:7}
we present several numerical results
showing that the new approach allows to approximate solutions to the
governing equations also in highly nontrivial geometries. In an appendix
we show that the weak formulation derived in this work yields in fact
the strong formulation for sufficiently smooth evolutions.
 
\setcounter{equation}{0}
\section{The governing equations} \label{sec:1a}
In this section we precisely formulate the governing equations
both for the $C^0$-- and the $C^1$--case.
We always assume that $(\Gamma(t))_{t\in [0,T]}$ 
is an evolving hypersurface without boundary in $\R^d$, $d=2,3$,
that is parameterized by $\vec x(\cdot,t):\Upsilon\to\R^d$,
where $\Upsilon\subset \R^d$ is a given reference manifold, i.e.\
$\Gamma(t) = \vec x(\Upsilon,t)$. Then
\begin{equation} \label{eq:V}
\vec{\mathcal{V}}(\vec q, t) := \vec x_t(\vec z, t)
\qquad \forall\ \vec q = \vec x(\vec z,t) \in \Gamma(t)
\end{equation}
defines the velocity of $\Gamma(t)$.
In order to introduce the two-phase aspect, we consider
the decomposition
 $\Gamma(t) = \Gamma_1(t) \cup \gamma(t) \cup
\Gamma_2(t)$, where
the interiors of $\Gamma_1(t) $ and 
$\Gamma_2(t) $ are disjoint and  $\gamma(t) = \partial\Gamma_1(t) = \partial\Gamma_2(t)$.
We assume that each $\Gamma_i(t)$ is smooth, with outer unit normal
$\vec\nu_i(t)$. 
See Figure~\ref{fig:sketch} for a sketch of the setup in the case $d=3$.
In particular, we
parameterize the two parts of the surface over fixed oriented, compact, smooth
reference manifolds $\Upsilon_i\subset\Upsilon$, i.e.\ we let
$\Gamma_i(t) = \vec x(\Upsilon_i,t)$, $i=1,2$.
\begin{figure}
\center
\begin{tikzpicture}
\draw[dashed,color=gray] (0,0) arc (-90:90:0.5 and 1.5);
\draw[semithick] (0,0) arc (270:90:0.5 and 1.5);
\draw[semithick] 
(0,3) to [out=0, in=180] (0.2,3)
to [out=0, in=180] (2.5,3.5)
to [out=0, in=20] (2.9,-0.5)
to [out=200, in=0] (0,0)
to [out=180, in=0] (-0.2,0)
to [out=180, in=0] (-2,-0.5)
to [out=180, in=200] (-2.5,3.5)
to [out=20, in=180] (0,3)
;
\draw[-latex,ultra thick] (-0.5,1.5) -- (-0.5,0.25) node[left] {$\vec\id_s$};
\draw[-latex,ultra thick] (-0.5,1.5) -- (0.75,1.5) node[below] {$\vec\mu_2$};
\draw[-latex,ultra thick] (-0.4,2.4) -- (-1.65,2.4) node[below] {$\vec\mu_1$};
\draw[-latex,ultra thick] (-2.5,3.5) -- (-3.0,4.6) node[left] {$\vec\nu_2$};
\draw[-latex,ultra thick] (3.3,3.1) -- (4.1,4.0) node[right] {$\vec\nu_1$};
\node at (-2.8,1.0) {$\Gamma_2$};
\node at (2.5,2.9) {$\Gamma_1$};
\node at (0.5,2.6) {$\gamma$};
\end{tikzpicture}
\caption{Sketch of $\Gamma = \Gamma_1\cup\gamma\cup\Gamma_2$ 
with outer unit normals $\vec\nu_i$, conormals $\vec\mu_i$ and tangent vector 
$\vec{\rm id}_s$ on $\gamma$ for the case $d=3$.
}
\label{fig:sketch}
\end{figure}
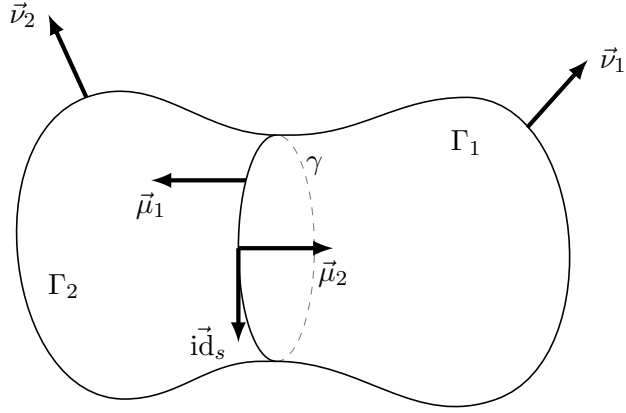
Throughout this paper we will investigate two different types of junction
conditions on $\gamma(t)$:
\begin{subequations}
\begin{align}
\text{$C^0$--case :} & 
\quad \gamma(t) = \partial\Gamma_1(t) = \partial\Gamma_2(t)\,, \label{eq:C0} \\ 
\text{$C^1$--case :} & 
\quad \gamma(t) = \partial\Gamma_1(t) = \partial\Gamma_2(t) \ \text{ and }\
\vec\nu_1 = \vec\nu_2 \quad\text{on } \gamma(t)\,.
\label{eq:C1}
\end{align}
\end{subequations}
Of course, in the case (\ref{eq:C1}) it also holds that $\vec\mu_1 =
-\vec\mu_2$, where $\vec\mu_i$ denotes the outer conormal to
$\Gamma_i(t)$ on $\gamma(t)$.

In order to formulate the governing problems in more detail, we 
denote by $\nabs =
(\partial_{s_1}, \ldots, \partial_{s_d})$ the surface gradient on
$\Gamma_i$, and then define $\nabs\,\vec\chi = \left( \partial_{s_j}\, \chi_k
\right)_{k,j=1}^d$, as well as the Laplace--Beltrami operator
$\Delta_s = \nabs\,.\,\nabs = \sum_{j=1}^d \partial_{s_j}^2$.
We then introduce the mean curvature vector as
\begin{equation} \label{eq:LBop}
\vec\varkappa_i = \varkappa_i\, \vec\nu_i = \Delta_s\, \vec\id 
\qquad \mbox{on $\Gamma_i$}\,,
\end{equation}
where $\vec\id$ is the identity function on $\R^d$, and $\varkappa_i$ is 
the mean curvature of $\Gamma_i$, 
i.e.\ the sum of the principal curvatures of $\Gamma_i$.
In particular, the principal curvatures $\varkappa_{i,j}$, $j=1,\ldots,d-1$, 
together with the eigenvalue zero for the eigenvector $\vec\nu_i$, are the $d$
eigenvalues of the symmetric linear map
$-\nabs\,\vec\nu_i : \R^d \to \R^d$; see e.g.\ \cite[p.~152]{DeckelnickDE05},
where a different sign convention is used. 
The map $- \nabs\,\vec\nu_i$ is also called the Weingarten map or 
shape operator. The mean curvature $\varkappa_i$ and the Gaussian
curvature $\Gauss_i$ of $\Gamma_i$ can now be stated as
\begin{equation}
\varkappa_i = \sum_{j=1}^{d-1} \varkappa_{i,j} = -\tr (\nabs\,\vec\nu_i) = -
\nabs\,.\,\vec\nu_i \qquad \mbox{and} \qquad
\Gauss_i = \prod_{j=1}^{d-1} \varkappa_{i,j}\,. \label{eq:secondform}
\end{equation}

Throughout the paper the main case we are interested in is $d=3$, but it is
often convenient to also discuss the case $d=2$ at the same time. To this end,
we generalize the free energy (\ref{eq:Eneu}) to
\begin{equation} \label{eq:E}
E((\Gamma_i(t))_{i=1}^2) = \sum_{i=1}^2 \left[
\tfrac12\,\alpha_i\,\int_{\Gamma_i(t)} (\varkappa_i - \spont_i)^2 \dH{d-1}
+ \alpha^G_i\,\int_{\Gamma_i(t)} \Gauss_i \dH{d-1}
\right]
+ \varsigma\,\mathcal{H}^{d-2}(\gamma(t))\,,
\end{equation}
where $\varkappa_i$ and $\Gauss_i$ are the mean and Gaussian curvatures of
$\Gamma_i(t)$, $i=1,2$, $\varsigma \in \Rgeq$ denotes a possible line tension,
and $\alpha_i\in\Rplus$ and $\alpha^G_i \in \R$ denote the bending 
and Gaussian bending rigidities of $\Gamma_i(t)$, $i=1,2$,
respectively.
Here and throughout $\mathcal{H}^{k}$, $k=0, 1, 2$, denotes the $k$-dimensional 
Hausdorff measure in $\R^d$.

In the case $d=2$, we always assume that 
$\varsigma = \alpha^G_1 = \alpha^G_2 = 0$.
For the case $d=3$, on the other hand,
we mention that the contributions
\begin{equation} \label{eq:Econ}
\sum_{i=1}^2\left[\tfrac12\,\alpha_i\, 
\int_{\Gamma_i(t)} \varkappa^2_i \dH{2} + 
\alpha^G_i\,\int_{\Gamma_i(t)} \Gauss_i \dH{2}\right]
\end{equation}
to the energy (\ref{eq:E}) 
are positive semidefinite with respect to the principal curvatures if 
\begin{equation} \label{eq:C0alphaGbound}
\alpha^G_i \in [-2\,\alpha_i,0]\,,\quad i=1,2\,. 
\end{equation}
In the $C^1$--case, recall (\ref{eq:C1}), 
adding multiples of $\sum_{i=1}^2 \Gauss_i \dH{2}$ to the energy only 
changes the energy by a constant which follows from the Gauss--Bonnet theorem,
see (\ref{eq:GB}) below. Hence we obtain that
the energy (\ref{eq:E}) can be bounded from below if
$\alpha^G_i \geq \max\{\alpha^G_1,\alpha^G_2\}-2\,\alpha_i$ for
$i=1,2$, which will hold whenever
\begin{equation} \label{eq:alphaGbound}
\min\{\alpha_1,\alpha_2\} \geq \tfrac12\,|\alpha^G_1 - \alpha^G_2|\,.
\end{equation}
Variational problems for integrals including the energy 
(\ref{eq:Econ}) require that the energy is definite, see e.g.\
\cite[p.\ 364]{Nitsche93}, in order to be able to show a priori estimates.
As discussed in \cite{Nitsche93}, the condition of definiteness leads to the
constraints (\ref{eq:C0alphaGbound}) and (\ref{eq:alphaGbound}),
and it is likely that these conditions also have
implications for the existence and regularity theory of 
gradient flows for (\ref{eq:E}) in the $C^0$-- and $C^1$--case, respectively.

In the case $d=3$, 
similarly to (\ref{eq:LBop}), fundamental to many approaches, 
which numerically approximate evolving curves in a parametric way, 
is the identity
\begin{equation}\label{eq:ident}
\vec\id_{ss} = \vec\varkappa_\gamma 
\quad \text{on }\ \gamma(t)\,,
\end{equation}
where $\vec\varkappa_\gamma$ is the curvature vector on
$\gamma(t)$. Here we choose the arclength $s$ of the curve
$\gamma(t)$ such that 
\begin{equation} \label{eq:mu}
\vec\mu_i= (-1)^{i}\,\vec\nu_i \times \vec\id_s \quad \text{on }\ \gamma(t),
\end{equation}
for $i=1,2$, denote the outer conormals to $\Gamma_i(t)$ on $\gamma(t)$.
Note that $\vec\mu_i$ is a vector that is perpendicular to the unit 
tangent $\vec\id_s$ on $\partial\Gamma_i(t)$ and lies in the tangent 
space of $\Gamma_i(t)$.
Now (\ref{eq:ident}) can be rewritten as
\begin{equation} \label{eq:idss}
\vec\id_{ss} = \vec\varkappa_\gamma =
(\vec\varkappa_\gamma\,.\,\vec\mu_i)\,\vec\mu_i 
+ (\vec\varkappa_\gamma\,.\,\vec\nu_i)\,\vec\nu_i
\quad \text{on }\ \gamma(t)\,,
\end{equation}
where $\vec\varkappa_\gamma\,.\,\vec\mu_i$ is the geodesic curvature
and $\vec\varkappa_\gamma\,.\,\vec\nu_i$ is the normal curvature
of $\gamma(t)$ on $\Gamma_i(t)$, $i=1,2$.
It then follows from the Gauss--Bonnet theorem,
\begin{equation} \label{eq:GB}
\int_{\Gamma_i(t)} \Gauss_i \dH{2} = 2\,\pi\,m(\Gamma_i(t)) 
+ \int_{\gamma(t)} \vec\varkappa_\gamma\,.\,\vec\mu_i \dH{1} ,
\end{equation}
where $m(\Gamma_i(t)) \in \bZ$ denotes the Euler characteristic of 
$\Gamma_i(t)$,
that the energy (\ref{eq:E}), is equivalent to
\begin{align}\label{eq:E2}
& E((\Gamma_i(t))_{i=1}^2) = \nonumber \\ & \quad
\sum_{i=1}^2 \left[
\tfrac12\,\alpha_i\,\int_{\Gamma_i(t)} (\varkappa_i - \spont_i)^2 \dH{2}
 + \alpha^G_i \left[\int_{\gamma(t)} 
\vec\varkappa_\gamma\,.\,\vec\mu_i \dH{1} + 2\,\pi\,m(\Gamma_i(t)) \right]
\right]
+ \varsigma\, \mathcal{H}^1(\gamma(t)) \,.
\end{align}
We note that we use a sign for the conormal that is different from many 
authors in differential geometry, and hence we obtain a different sign in 
the Gauss--Bonnet formula. 

In some cases, in particular in applications for biomembranes, cf.\
\cite{Tu13}, the surface areas of $\Gamma_1(t)$ and $\Gamma_2(t)$ need 
to stay constant during the evolution, as well as the volume enclosed by
$\Gamma(t)$. Here and throughout we use the terminology ``surface area'' and
``enclosed volume'' also for the case $d=2$, when the former is really 
curve length, and the latter means enclosed area.
In this case one can consider
\begin{equation}
E_\lambda ((\Gamma_i(t))_{i=1}^2) = 
E ((\Gamma_i(t))_{i=1}^2) + \lambda^V(t)\,\mathcal{L}^d(\Omega(t))
+ \sum_{i=1}^2 \lambda^A_i(t)\,\mathcal{H}^{d-1} (\Gamma_i(t)) \,,
\label{eq:areaE}
\end{equation}
where $\Omega(t)$ denotes the interior of $\Gamma(t)$ and
$\mathcal{L}^d$ denotes the Lebesgue measure in $\mathbb{R}^d$.
Here, $\lambda^A_i(t)$ are Lagrange multipliers for the
area constraints, which can be interpreted as a surface tension,
and $\lambda^V(t)$ is a Lagrange multiplier for the volume constraint
which might be interpreted as a pressure difference.

For the convenience of the reader, 
we end this section by stating the strong formulations of the 
$L^2$--gradient flows for (\ref{eq:E}) in the presence of the matching
conditions (\ref{eq:C0}) and (\ref{eq:C1}), respectively.
These strong formulations directly follow from 
the weak formulation introduced in Section~\ref{sec:2}, as we show rigorously
in the appendix.

The weighted $L^2$--gradient flow, (\ref{eq:VrV}), 
of (\ref{eq:E2}), for $d=2$ or $d=3$,
then leads to the evolution law 
\begin{equation} 
\vec{\mathcal{V}}\,.\,\vec\nu_i = - \alpha_i\,\Delta_s\,\varkappa_i +
\tfrac12\,\alpha_i\,(\varkappa_i - \spont_i)^2 \,\varkappa_i -
\alpha_i\,(\varkappa_i - \spont_i)\,|\nabs\,\vec\nu_i|^2 
\quad\text{on } \Gamma_i(t)\,,\ i=1,2\,.
\label{eq:gradflow}
\end{equation}
See (\ref{eq:strongopen}) in the appendix for a derivation of 
(\ref{eq:gradflow}).
We remark that if the more general energy (\ref{eq:areaE}) is considered, then
(\ref{eq:gradflow}) is replaced by
\begin{equation}
\vec{\mathcal{V}}\,.\,\vec\nu_i = - \alpha_i\,\Delta_s\,\varkappa_i +
\tfrac12\,\alpha_i\,(\varkappa_i - \spont_i)^2 \,\varkappa_i -
\alpha_i\,(\varkappa_i - \spont_i)\,|\nabs\,\vec\nu_i|^2 +
\lambda^A_i\,\varkappa_i - \lambda^V
\quad\text{on } \Gamma_i(t)\,,
\label{eq:gradflowlambda}
\end{equation}
for $i=1,2$, see (\ref{eq:stronggradflowlambda}) in the appendix.

In the case $d=3$ we introduce
the second fundamental form $\II_i$ of $\Gamma_i(t)$, which is given as
\begin{equation} \label{eq:II}
\II_i (\unitt_1,\unitt_2) = -[\partial_{\unitt_1}\, \vec\nu_i]
\,.\, \unitt_2 
= - [ (\nabs\,\vec\nu_i)\,\unitt_1] \,.\, \unitt_2 
\qquad\text{on }\ \Gamma_i(t)\,,
\end{equation}
for all tangential vectors $\unitt_j$, $j=1,2$.
We note that $\II_i(\cdot,\cdot)$ is a symmetric bilinear form, as
$\nabs\,\vec\nu_i$ is symmetric.
In addition, we define
\begin{equation} \label{eq:torsion2}
\tau_i = \II_i(\vec\id_s,\vec\mu_i) \qquad \text{on } \gamma(t)\,,
\end{equation}
i.e.\ $\tau_i = - (\vec\nu_i)_s\,.\,\vec\mu_i$ on $\gamma(t)$.

Still considering the case $d=3$,
in the $C^0$--junction case, the boundary conditions on $\gamma(t)$ are given 
by
\begin{subequations}
\begin{align}
&
\alpha_i\,(\varkappa_i - \spont_i) + \alpha^G_i\,
\vec\varkappa_\gamma\,.\,\vec\nu_i = 0 
\quad \text{on }\ \gamma(t)\,,\quad i = 1,2\,, \label{eq:C0bc1} \\
&
\sum_{i=1}^2 \left[
 ( (\alpha_i\,(\nabs\,\varkappa_i)\,.\,\vec\mu_i - 
 \alpha^G_i\,(\tau_i)_s)\,\vec\nu_i
-(\tfrac12\,\alpha_i\,(\varkappa_i - \spont_i)^2 
+ \alpha^G_i\,\Gauss_i + \lambda^A_i)\,\vec\mu_i \right] +
 \varsigma\, \vec\varkappa_\gamma
 = \varrho\,\vec{\mathcal{V}} 
\quad \text{on }\ \gamma(t)\,, \label{eq:C0bc2} 
\end{align}
\end{subequations}
see (\ref{eq:strongC0bc1}), (\ref{eq:strong2C0bc2}) in the appendix.
We note that (\ref{eq:C0bc1}) are two scalar conditions, while
(\ref{eq:C0bc2}) gives rise to two conditions as $\vec\mu_i$,
$\vec\nu_i$ and $\vec\varkappa_\gamma$ are all perpendicular to the tangent
space to $\gamma(t)$. Expressing $\Gamma_1$ and $\Gamma_2$ locally as
two graphs, we also obtain one condition for the height functions stemming
from the $C^0$--condition. Altogether we have five conditions, as is to be
expected for a free boundary problem involving fourth order operators
on both sides of the free boundary. In this context we also refer to
Remark~2.1 in \cite{pwftj}.

In the $C^1$--junction case, when $\vec\nu = \vec\nu_1 = \vec\nu_2$ and 
$\vec\mu = \vec\mu_2 = - \vec\mu_1$ on $\gamma(t)$, 
the boundary conditions on $\gamma(t)$ 
for the dissipation dynamics (\ref{eq:VrV}), with $E$ replaced by $E_\lambda$,
are given by
\begin{subequations}
\begin{align}
& [\alpha_i\,(\varkappa_i - \spont_i)]_1^2 + [\alpha^G_i]_1^2\,
\vec\varkappa_\gamma\,.\,\vec\nu = 0 
\quad \text{on }\ \gamma(t)\,, \label{eq:C1bc1} \\
& [\alpha_i\,(\nabs\,\varkappa_i)]_1^2\,.\,\vec\mu +\varsigma\,\vec\varkappa_\gamma
\,.\,\vec\nu - [\alpha^G_i]_1^2\,\tau_s = 
\varrho\,\vec{\mathcal{V}}\,.\,\vec\nu \quad \text{on }\ \gamma(t)\,,
\label{eq:C1bc2} \\
& [-\tfrac12\,\alpha_i\,(\varkappa_i - \spont_i)^2
+\alpha_i\,(\varkappa_i - \spont_i)\,
(\varkappa_i - \vec\varkappa_\gamma\,.\,\vec\nu)
- \lambda^A_i]_1^2
+ [\alpha^G_i]_1^2\,\tau^2
 +\varsigma\,\vec\varkappa_\gamma\,.\,\vec\mu  = 
\varrho\,\vec{\mathcal{V}}\,.\,\vec\mu 
\quad \text{on }\ \gamma(t)\,, \label{eq:C1bc3}
\end{align}
\end{subequations}
where $\tau = \tau_2 = -\tau_1$ is the geodesic torsion of the curve
$\gamma(t)$ on $\Gamma(t)$. We note that (\ref{eq:C1bc1}--c),
in the case $\varrho=0$, agree with 
(3.16)--(3.18) in \cite{ElliottS10a}, see also \cite[(2.7b,a,c)]{ElliottS13}.
In terms of counting the number of equations, we see that (\ref{eq:C1bc1}--c) 
are three conditions, together with one condition coming from
$\vec\nu_1 = \vec\nu_2$ and one condition from the requirement that the two
phases match up continuously, leading to five conditions in total.
We refer to (\ref{eq:akkakn}), (\ref{eq:C1norm},b) in the appendix 
for a derivation of (\ref{eq:C1bc1}--c). 

\begin{remark} \label{rem:2d}
We note that although the conditions {\rm (\ref{eq:C0bc1},b)} and
{\rm (\ref{eq:C1bc1}--c)} were derived for the case $d=3$, they are also valid
in the case $d=2$ on recalling that in this case we set $\varsigma =
\alpha^G_1 = \alpha^G_2 = 0$.
In particular, {\rm (\ref{eq:C0bc1})} then 
simplifies to $\varkappa_i=\spont_i$ on $\gamma(t)$, $i=1,2$, which is the same
as the condition \cite[(2.13c)]{pwftj} that was derived by the authors
for a $C^0$--junction between two curves meeting in 2d.
In addition, {\rm (\ref{eq:C0bc2})} for $d=2$ and $\varrho=0$ collapses to 
\cite[(2.13b)]{pwftj}, modulo the different sign convention employed there.

Similarly, {\rm (\ref{eq:C1bc1})} for $d=2$ simplifies to
$\alpha_1\,(\varkappa_1-\spont_1) 
= \alpha_2\,(\varkappa_2-\spont_2)$ on $\gamma(t)$, which is the same
as the condition \cite[(2.18e)]{pwftj},
modulo the different sign convention employed there, 
that was derived by the authors
for a $C^1$--junction between two curves meeting in 2d.
In addition, {\rm (\ref{eq:C1bc2},c)} for $d=2$ and $\varrho=0$, collapse to 
\cite[(2.18b,c)]{pwftj}.
\end{remark}

\setcounter{equation}{0}
\section{Weak formulation} \label{sec:2}

In this section we derive a weak formulation of a generalized $L^2$--gradient 
flow of $E((\Gamma_i(t))_{i=1}^2)$.
The weak formulation of the standard $L^2$--gradient flow is given by 
(\ref{eq:weakGD3a}), below, 
with $\varrho = 0$, where $f_\Gamma$ represents 
the first variation of the energy 
$E((\Gamma_i(t))_{i=1}^2)$ formulated in a suitable weak form. 
In what follows we will define a Lagrangian
involving the energy and suitable constraints, which for example relate the
curvatures to the parametrizations of the surfaces. 
This Lagrangian will allow us to derive (\ref{eq:LM4a}--f), below, 
which defines $f_\Gamma$ in a weak formulation involving only 
first order derivatives.
This formulation will be suitable for a numerical approximation based on 
continuous, piecewise linear finite elements, and such an approximation will be
considered in Section~\ref{sec:3}.

On recalling (\ref{eq:V}), we define the following time derivative that
follows the parameterization $\vec x(\cdot, t)$ of $\Gamma(t)$. Let
\begin{equation} \label{eq:matpartx}
(\matpartx\, \phi)\!\mid_{\Gamma_i(t)} = 
(\phi_t + \vec{\mathcal{V}} \,.\,\nabla\,\phi)\!\mid_{\Gamma_i(t)}
\qquad \forall\ \phi \in H^1(\GTi)\,,
\end{equation}
where we have defined the space-time surfaces
\begin{equation*} 
\GTi := \bigcup_{t \in [0,T]} \Gamma_i(t) \times \{t\}\,,\ i=1,2\,,
\quad\text{and}\quad
\GT := \bigcup_{t \in [0,T]} \Gamma(t) \times \{t\}\,.
\end{equation*}
Here we stress that (\ref{eq:matpartx}) is well-defined, 
even though $\phi_t$ and $\nabla\,\phi$ do not make sense separately for a
function $\phi \in H^1(\GTi)$.
We note that
\begin{equation} \label{eq:DElem5.2}
\ddt \left\langle \psi_i, \phi_i \right\rangle_{\Gamma_i(t)}
 = \left\langle \matpartx\,\psi_i, \phi \right\rangle_{\Gamma_i(t)}
 + \left\langle \psi_i, \matpartx\,\phi_i \right\rangle_{\Gamma_i(t)}
+ \left\langle \psi_i\,\phi_i, \nabs\,.\,\vec{\mathcal{V}} 
 \right\rangle_{\Gamma_i(t)}
\qquad \forall\ \psi_i,\phi_i \in H^1(\GTi)\,,
\end{equation}
see Lemma~5.2 in \cite{DziukE13}.
Here $\langle \cdot, \cdot \rangle_{\Gamma_i(t)}$ denotes the 
$L^2$--inner product on $\Gamma_i(t)$,
and $\langle \cdot, \cdot \rangle_{\Gamma(t)}
= \sum_{i=1}^2$ $\langle \cdot, \cdot \rangle_{\Gamma_i(t)}$. 
It immediately follows from (\ref{eq:DElem5.2}) that
\begin{equation} \label{eq:dtarea}
\ddt\, \mathcal{H}^{d-1}(\Gamma_i(t)) 
= \left\langle \nabs\,.\,\vec{\mathcal{V}} , 1 \right\rangle_{\Gamma_i(t)}
= \left\langle \nabs\,\vec\id ,\nabs\,\vec{\mathcal{V}}
\right\rangle_{\Gamma_i(t)} .
\end{equation}
Moreover, on recalling Lemma~2.1 from \cite{DeckelnickDE05}, it holds that
\begin{equation} \label{eq:dtvol}
\ddt\, \mathcal{L}^{d}(\Omega(t)) = 
\sum_{i=1}^2
\left\langle \vec{\mathcal{V}} , \vec\nu_i\right\rangle_{\Gamma_i(t)}.
\end{equation}

In this section we would like to derive a weak formulation for the 
$L^2$--gradient flow of $E((\Gamma_i(t))_{i=1}^2)$. To this end,
we need to consider variations of the energy with respect to 
$\Gamma(t) = \vec x(\Upsilon,t)$.
Let
\begin{align*} 
H^1_\gamma(\Gamma(t)) := \{
\eta \in L^2(\Gamma(t)) : &\  \eta\!\mid_{\Gamma_i(t)} \in H^1(\Gamma_i(t)), 
i=1,2\,,\ \nonumber \\ & 
(\eta\!\mid_{\Gamma_1(t)})\!\mid_{\gamma(t)} = 
(\eta\!\mid_{\Gamma_2(t)})\!\mid_{\gamma(t)}
=:\eta\!\mid_{\gamma(t)} \in H^1(\gamma(t))\}\,.
\end{align*}
In addition, for any given $\vec\chi \in \xspacec$ 
and for any $\epsilon \in (0,\epsilon_0)$ for some $\epsilon_0
\in {\mathbb R}_{>0}$, let
\begin{equation}
\Gamma_{\epsilon}(t) := \{ \vec\Psi(\vec z,\epsilon) : 
\vec z \in \Gamma(t) \} \,, \ \text{~where~}\
\vec\Psi(\vec z, 0) = \vec z \text{~and~}
\tfrac{\partial\vec\Psi}{\partial\epsilon}(\vec z, 0) = \vec\chi(\vec z)
\quad \forall\ \vec z \in \Gamma(t)\,.
\label{eq:Gammadelta}
\end{equation}
Of course, we have that
$\Gamma_\epsilon(t) = \Gamma_{1,\epsilon}(t) \cup \gamma_\epsilon(t) \cup
\Gamma_{2,\epsilon}(t)$, where
\begin{equation*} 
\Gamma_{i,\epsilon}(t) := \{ \vec\Psi(\vec z,\epsilon) : 
\vec z \in \Gamma_i(t) \}\,,\ i=1,2\,,\quad\text{and}\quad
\gamma_{\epsilon}(t) = \partial\Gamma_{1,\epsilon}(t)
= \partial\Gamma_{2,\epsilon}(t) \,.
\end{equation*}
Similarly to (\ref{eq:dtarea}), 
the first variation of $\mathcal{H}^{d-1}(\Gamma_i(t))$
with respect to $\Gamma(t)$ in the direction $\vec\chi \in \xspacec$ 
is given by
\begin{align}
\left[ \deldel{\Gamma} \, \mathcal{H}^{d-1}(\Gamma_i(t))\right]
(\vec\chi) & = 
\dd\epsilon\,\mathcal{H}^{d-1}(\Gamma_{i,\epsilon}(t)) \mid_{\epsilon=0} 
\nonumber \\ & 
= \lim_{\epsilon\to0} 
\tfrac1\epsilon\left[\mathcal{H}^{d-1}(\Gamma_{i,\epsilon}(t))
- \mathcal{H}^{d-1}(\Gamma_i(t)) \right] = 
\left\langle \nabs\,\vec\id , \nabs \,\vec\chi \right\rangle_{\Gamma_i(t)} ,
\label{eq:vardet}
\end{align}
see e.g.\ the proof of Lemma~1 in \cite{Dziuk08}. 

In order to derive a suitable weak formulation, we formally consider
the first variation of (\ref{eq:E}) subject to the following 
side constraint, which is inspired by the weak formulation of (\ref{eq:LBop}),
\begin{equation}
\left\langle \mat Q_{i,\theta}\,\kappastar_i, \vec\eta 
\right\rangle_{\Gamma_i(t)}
+ \left\langle \nabs \,\vec\id,\nabs\, \vec\eta\right\rangle_{\Gamma_i(t)}= 
\left\langle \vec{\rm m}_i, \vec\eta \right\rangle_{\gamma(t)}
\qquad \forall\ \vec\eta \in \Vti\,,\ i = 1,2\,,
\label{eq:side3}
\end{equation}
where $\theta \in [0,1]$ is a fixed parameter, and 
where $\mat Q_{i,\theta}$ are defined by
\begin{equation} \label{eq:Q}
\mat Q_{i,\theta} = \theta\,\mat\Id +
(1-\theta)\,\vec\nu_i\otimes\vec\nu_i
\quad\text{on }\ \Gamma_i(t)\,.
\end{equation}
Of course, (\ref{eq:side3}) holds trivially on the continuous level for 
$\kappastar_i = \vec\varkappa_i$
and for $\vec{\rm m}_i$ being the conormal $\vec\mu_i$,
independently of the choice of $\theta \in [0,1]$.
Here we remark that the natural weak formulation of (\ref{eq:LBop})
would correspond to (\ref{eq:side3}) with $\theta =1$.
However, under discretization that formulation would lead to 
undesirable mesh effects. Hence, in line with the authors previous work
in \cite{pwfopen}, we also allow $\theta \in [0,1)$, which under
discretization
leads to an induced tangential motion and good meshes for $\theta=0$, 
in general.
In rare cases we may need to dampen the tangential motion that occurs
in the case $\theta=0$. To this end, we allow for the full range
of values $\theta \in [0,1]$.

Similarly to (\ref{eq:side3}), we introduce the following side constraint,
inspired by the weak formulation of (\ref{eq:ident}): 
\begin{equation} 
 \left\langle \kappastar_\gamma, 
\vec\eta \right\rangle_{\gamma(t)} +
\left\langle \vec\id_s , \vec\eta_s \right\rangle_{\gamma(t)} = 
0 \quad \forall\ \vec\eta \in \Vpartial\,. \label{eq:varkappapartial} 
\end{equation}
Finally, in order to model a $C^0$-- or $C^1$--contact we require 
\begin{equation} \label{eq:mC1}
C_1\,(\vec{\rm m}_1 + \vec{\rm m}_2) = \vec 0 \qquad \text{on}\quad \gamma(t)\,,
\end{equation}
where $C_1 = 0$ for $C^0$ and $C_1 = 1$ for $C^1$.

We now define the Lagrangian 
\begin{align*} &
L((\Gamma_i(t), \kappastar_i, \vec{\rm m}_i, \vec y_i)_{i=1}^2, 
\kappastar_\gamma, \vec z, \vec\phi) =
\sum_{i=1}^2 \left[
\tfrac12 \alpha_i \left\langle \kappastar_i - \spont_i\,\vec\nu_i, 
\kappastar_i - \spont_i\,\vec\nu_i\right\rangle_{\Gamma_i(t)} 
+ \alpha^G_i \left\langle \kappastar_\gamma, \vec{\rm m}_i 
\right\rangle_{\gamma(t)} \right] 
 \nonumber \\ & \quad
+ \varsigma\,\mathcal{H}^{d-2}(\gamma(t))
- \left\langle \kappastar_\gamma, \vec z \right\rangle_{\gamma(t)} -
\left\langle \vec\id_s , \vec z_s \right\rangle_{\gamma(t)}
+ C_1 
\left\langle \vec{\rm m}_1 + \vec{\rm m}_2, \vec\phi\right\rangle_{\gamma(t)} 
\nonumber \\ & \quad
- \sum_{i=1}^2 \left[
\left\langle \mat Q_{i,\theta}\,\kappastar_i, \vec y_i 
\right\rangle_{\Gamma_i(t)} +
\left\langle \nabs \,\vec\id, \nabs\,\vec y_i \right\rangle_{\Gamma_i(t)} 
- \left\langle \vec{\rm m}_i, \vec y_i \right\rangle_{\gamma(t)} \right] , 
\end{align*}
where $\vec y_i\in \Vti$ and
$\vec z \in \Vpartial$ are Lagrange multipliers for 
(\ref{eq:side3}) and (\ref{eq:varkappapartial}), respectively. 
Similarly, $\vec\phi \in [L^2(\gamma(t))]^d$ is a Lagrange multiplier for
(\ref{eq:mC1}).
We now want to compute the first variation 
$f_\Gamma$ of $E((\Gamma_i(t))_{i=1}^2)$, 
subject to the side constraints (\ref{eq:side3}),
(\ref{eq:varkappapartial}) and (\ref{eq:mC1}). 
This means that $f_\Gamma$ needs to fulfill
\begin{equation}
f_\Gamma(\vec\chi) = 
 - \left[ \deldel\Gamma\, E(t) \right] (\vec\chi)
\qquad \forall \ \vec\chi \in \xspacec\,.
\label{eq:varEkb}
\end{equation}
In particular, on using ideas from 
the formal calculus of PDE constrained optimization, see
e.g.\ \cite{Troltzsch10},
we can formally compute $f_\Gamma$ by requiring that
\begin{subequations}
\begin{align}
\left[\deldel{\Gamma}\,L\right](\vec\chi) 
& = \lim_{\epsilon\to0} \tfrac1\epsilon \left[ 
(L(\Gamma_{i,\epsilon}(t),
\kappastar_i, \vec{\rm m}_i,\vec y_i)_{i=1}^2,\kappastar_\gamma,\vec z,
\vec\phi) \right. \nonumber \\ & \qquad \qquad \left. - 
L((\Gamma_i(t), \kappastar_i, \vec{\rm m}_i, \vec y_i)_{i=1}^2, 
\kappastar_\gamma, \vec z, \vec\phi)
\right] = - f_\Gamma(\vec\chi)
\label{eq:deltaLa} \\
\left[\deldel{\kappastar_1}\, L\right](\vec\xi_1) & =
\lim_{\epsilon\to0} \tfrac1\epsilon\left[ 
L(\Gamma_1(t), \kappastar_1 + \epsilon\,\vec\xi_1,
\vec{\rm m}_1,\vec y_1,
\Gamma_2(t),  \kappastar_2,
\vec{\rm m}_2,\vec y_2, \kappastar_\gamma,\vec z,
\vec\phi) \right. \nonumber \\ & \qquad \qquad \left. - 
L((\Gamma_i(t), \kappastar_i, \vec{\rm m}_i, \vec y_i)_{i=1}^2, 
\kappastar_\gamma, \vec z, \vec\phi)
 \right] = 0\,, \label{eq:deltaLb} \\
\left[\deldel{\vec{\rm m}_1}\, L\right](\vec\zeta_1) & =
\lim_{\epsilon\to0} \tfrac1\epsilon\left[ 
L(\Gamma_1(t), \kappastar_1,
\vec{\rm m}_1 + \epsilon\,\vec\zeta_1,\vec y_1,
\Gamma_2(t),  \kappastar_2,
\vec{\rm m}_2,\vec y_2, \kappastar_\gamma,\vec z,
\vec\phi)
 \right. \nonumber \\ & \qquad \qquad \left. - 
L((\Gamma_i(t), \kappastar_i, \vec{\rm m}_i, \vec y_i)_{i=1}^2, 
\kappastar_\gamma, \vec z, \vec\phi)
 \right] = 0\,, \label{eq:deltaLc} \\
\left[\deldel{\vec y_1}\, L\right](\vec\eta_1) & =
\lim_{\epsilon\to0} \tfrac1\epsilon\left[ 
L(\Gamma_1(t), \kappastar_1,
\vec{\rm m}_1 ,\vec y_1+ \epsilon\,\vec\eta_1,
\Gamma_2(t),  \kappastar_2,
\vec{\rm m}_2,\vec y_2, \kappastar_\gamma,\vec z,
\vec\phi)
 \right. \nonumber \\ & \qquad \qquad \left. - 
L((\Gamma_i(t), \kappastar_i, \vec{\rm m}_i, \vec y_i)_{i=1}^2, 
\kappastar_\gamma, \vec z, \vec\phi)
 \right] = 0\,, \label{eq:deltaLd}
\end{align}
\end{subequations}
for variations $\vec\chi \in \xspacec$, 
$\vec\xi_1 \in \VLtwotone$, $\vec\zeta_1 \in \VLtwopartial$ and 
$\vec\eta_1 \in \VLtwotone$; 
and similarly for the variations for 
$\kappastar_2$, $\vec{\rm m}_2$, $\vec y_2$, $\kappastar_\gamma$, 
$\vec z$ and $\vec\phi$.

In order to calculate (\ref{eq:deltaLa}--d), we note that
generalized variants of (\ref{eq:vardet}) also hold. Namely, we have that
\begin{equation}
\left[\deldel{\Gamma} \left\langle w_i, 1 \right\rangle_{\Gamma_i(t)}
\right] (\vec\chi) = 
\dd\epsilon \left\langle w_{i,\epsilon}, 1 
\right\rangle_{\Gamma_{i,\epsilon}(t)} 
\mid_{\epsilon=0} = 
\left\langle w_i\,\nabs\,\vec\id , \nabs \,\vec\chi 
\right\rangle_{\Gamma_i(t)}
 \quad \forall\ w_i \in L^\infty(\Gamma_i(t))\,,
\label{eq:vardetw}
\end{equation}
where $w_{i,\epsilon} \in L^\infty(\Gamma_{i,\epsilon}(t))$, for any
$w_i \in L^\infty(\Gamma_i(t))$, is defined by 
\begin{equation*} 
w_{i,\epsilon}(\vec\Psi(\vec z, \epsilon)) = w_i(\vec z) \qquad \forall\
\vec z \in \Gamma_i(t)\,,
\end{equation*}
and similarly for $\vec w \in [L^\infty(\Gamma_i(t))]^d$.
This definition of $w_{i,\epsilon}$ yields that
$\partial^0_\epsilon\, w_i=0$, where
\begin{equation}
\partial^0_\epsilon\,w_i(\vec{z})=
\dd\epsilon w_{i,\epsilon}(\vec\Psi(\vec{z},\epsilon)) \mid_{\epsilon=0}
\qquad \forall\ \vec z \in \Gamma_i(t).
\label{eq:deltaderiv}
\end{equation}
Of course, (\ref{eq:vardetw}) is the first variation analogue of
(\ref{eq:DElem5.2}) with $w_i = \psi_i\,\phi_i$ and 
$\partial^0_\epsilon\,\psi_i = \partial^0_\epsilon\,\phi_i = 0$.
Similarly, it holds that
\begin{align}
& \left[ \deldel{\Gamma} \left\langle \vec w_i, \vec\nu_i 
\right\rangle_{\Gamma_i(t)} \right] (\vec\chi) = 
\dd\epsilon \left\langle \vec w_{i,\epsilon}, \vec\nu_{i,\epsilon}
\right\rangle_{\Gamma_{i,\epsilon}(t)} \mid_{\epsilon=0}
\nonumber \\ & \qquad
= \left\langle (\vec w_i\,.\,\vec\nu_i)\,\nabs\,\vec\id , \nabs \,\vec\chi 
\right\rangle_{\Gamma_i(t)}
+ \left\langle \vec w_i, 
 \partial^0_\epsilon\,\vec\nu_i \right\rangle_{\Gamma_i(t)}
 \quad \forall\ \vec w_i \in [L^\infty(\Gamma_i(t))]^d\,,
\label{eq:vardetwnu}
\end{align}
where $\partial^0_\epsilon\, \vec w_i =\vec 0$ and  
$\vec\nu_{i,\epsilon}(t)$ denotes the unit normal on $\Gamma_{i,\epsilon}(t)$.
Moreover, we will make use of the following result
concerning the variation of $\vec\nu_i$, with respect to $\Gamma(t)$, 
in the direction $\vec\chi \in \xspacec$:
\begin{equation}
\partial^0_\epsilon \,\vec\nu_i
= - [\nabs\, \vec\chi]^T \,\vec\nu_i \quad \text{on}\quad \Gamma_i(t)
\quad \Rightarrow \quad
\matpartx\,\vec\nu_i = - [\nabs\, \vec{\mathcal{V}}]^T \,\vec\nu_i
\quad \text{on}\quad \Gamma_i(t)\,, 
\label{eq:normvar}
\end{equation}
see \cite[Lemma~9]{SchmidtS10}.
We also note that for $\vec\eta_i \in \Vti$ it holds that
\begin{align}
&
\left[ \deldel{\Gamma} \left\langle \nabs\,\vec\id,
\nabs\,\vec\eta_i\right\rangle_{\Gamma_i(t)} \right](\vec\chi)
= \dd\epsilon \left\langle \nabs\,\vec\id,
\nabs\,\vec\eta_{i,\epsilon}\right\rangle_{\Gamma_{i,\epsilon}(t)} 
\mid_{\epsilon=0} 
= \left\langle \nabs\,.\,\vec\eta_i, \nabs \,.\,\vec\chi
\right\rangle_{\Gamma_i(t)}
\nonumber \\ & \qquad \qquad
+ \sum_{l,\,m=1}^d \left[
\left\langle (\vec\nu_i)_l\,(\vec\nu_i)_m\,
 \nabs\,(\vec\eta_i)_m, \nabs \,(\vec\chi)_l \right\rangle_{\Gamma_i(t)}
- \left\langle (\nabs)_m\,(\vec\eta_i)_l, (\nabs)_l \,(\vec\chi)_m 
\right\rangle_{\Gamma_i(t)} \right]  \nonumber \\ & \qquad
 = 
 \left\langle \nabs\,\vec\eta_i, \nabs \,\vec\chi\right\rangle_{\Gamma_i(t)}
+ \left\langle \nabs\,.\,\vec\eta_i, \nabs \,.\,\vec\chi
 \right\rangle_{\Gamma_i(t)}
- \left\langle (\nabs \,\vec\eta_i)^T, \mat D(\vec\chi)\,(\nabs\,\vec\id)^T 
 \right\rangle_{\Gamma_i(t)},
\label{eq:secvar}
\end{align}
where $\partial^0_\epsilon\, \vec\eta_i =\vec 0$, 
see Lemma~2 and the proof of Lemma~3 in \cite{Dziuk08}. Here
\begin{equation*}
\mat D(\vec\chi) := \nabs\,\vec\chi + (\nabs\,\vec\chi)^T\,, 
\end{equation*}
and we note that our notation is such that 
$\nabs\,\vec\chi = (\nabla_{\!\Gamma}\,\vec\chi)^T$, with 
$\nabla_{\!\Gamma}\,\vec\chi=\left( \partial_{s_l}\, \chi_m \right)_{l,m=1}^d$ 
defined as in \cite{Dziuk08}.
It follows from (\ref{eq:secvar}) that
\begin{align}
& \ddt \left\langle \nabs\,\vec\id,
\nabs\,\vec\eta\right\rangle_{\Gamma_i(t)} = 
 \left\langle \nabs\,\vec\eta, \nabs \,\vec{\mathcal V} 
 \right\rangle_{\Gamma_i(t)}
+ \left\langle \nabs\,.\,\vec\eta, \nabs \,.\,\vec{\mathcal V}
 \right\rangle_{\Gamma_i(t)}
\nonumber \\ 
& \qquad \qquad - \left\langle (\nabs \,\vec\eta)^T, \mat D(\vec{\mathcal V})
 \,(\nabs\,\vec\id)^T \right\rangle_{\Gamma_i(t)} 
\qquad \forall\ \vec\eta \in \{ \vec\xi \in H^1(\GTi) : \matpartx\,\vec\xi = \vec0
\}\,.
\label{eq:secvar2}
\end{align}
Similarly to (\ref{eq:vardetw}) it holds that
\begin{equation}
\left[\deldel{\Gamma} \left\langle w, 1 \right\rangle_{\gamma(t)}
\right] (\vec\chi) = 
\dd\epsilon \left\langle w_\epsilon, 1 
\right\rangle_{\gamma_\epsilon(t)} \mid_{\epsilon=0} = 
\left\langle w\,\vec\id_s , \vec\chi_s \right\rangle_{\gamma(t)}
 \quad \forall\ w \in L^\infty(\gamma(t)),\
 \vec\chi \in \xspacec\,, 
\label{eq:vardetwpartial} 
\end{equation}
where $\partial^0_\epsilon\, w = 0$.
Moreover, similarly to (\ref{eq:secvar}), we note that for 
$\vec\eta \in \xspacec$ it holds that
\begin{equation}
\left[ \deldel{\Gamma} \left\langle \vec\id_s,
\vec\eta_s\right\rangle_{\gamma(t)} \right](\vec\chi)
= \left\langle \mat{\mathcal{P}}_\gamma\,\vec\eta_s, \vec\chi_s
\right\rangle_{\gamma(t)}\,,
\label{eq:secvarpartial}
\end{equation}
where $\partial^0_\epsilon\, \vec\eta =\vec 0$, and where
\begin{equation} \label{eq:mathcalP}
\mat{\mathcal{P}}_\gamma = \mat\Id - \vec\id_s \otimes \vec\id_s
\qquad\text{on }\ \gamma(t)\,.
\end{equation}

Now combining (\ref{eq:deltaLa}--d), on noting
(\ref{eq:vardetw})--(\ref{eq:mathcalP}), yields that
\begin{subequations}
\begin{align}
& 
f_\Gamma(\vec\chi) = 
\sum_{i=1}^2 \left[
\left\langle \nabs\,\vec y_i, \nabs\,\vec\chi \right\rangle_{\Gamma_i(t)} +
\left\langle \nabs\,.\,\vec y_i, \nabs\,.\,\vec\chi \right\rangle_{\Gamma_i(t)}
-\left\langle (\nabs\,\vec y_i)^T , \mat D(\vec\chi)\,(\nabs\,\vec\id)^T
\right\rangle_{\Gamma_i(t)} \right.
\nonumber \\ & \qquad \left.
-\tfrac12\left\langle [ \alpha_i\,|\kappastar_i - \spont_i\,\vec\nu_i|^2
- 2\,(\kappastar_i\,.\,\mat Q_{i,\theta}\,\vec y_i)]
\,\nabs\,\vec\id,\nabs\,\vec\chi \right\rangle_{\Gamma_i(t)}
+ \alpha_i\,\spont_i 
\left\langle \kappastar_i, \partial^0_\epsilon\,\vec\nu_i
\right\rangle_{\Gamma_i(t)} \right. \nonumber \\ & \qquad \left.
+ \left\langle 
\partial^0_\epsilon\,[\mat Q_{i,\theta}\,\kappastar_i]
, \vec y_i \right\rangle_{\Gamma_i(t)}
\right] - \varsigma \left\langle \vec\id_s, \vec\chi_s \right\rangle_{\gamma(t)}
\nonumber\\ &\qquad 
+ \left\langle \kappastar_\gamma\,.\,\vec z 
- C_1\,(\vec{\rm m}_1 + \vec{\rm m}_2)\,.\,\vec\phi
- \sum_{i=1}^2 (\alpha^G_i\,
 \kappastar_\gamma + \vec y_i)\,.\,\vec{\rm m}_i , 
\vec\id_s\,.\,\vec\chi_s \right\rangle_{\gamma(t)}
+ \left\langle \mat{\mathcal{P}}_\gamma\,\vec z_s, 
\vec\chi_s \right\rangle_{\gamma(t)} 
\nonumber\\ &\hspace{11cm}
\qquad \forall\ \vec\chi\in \xspacec\,, \label{eq:LM3a} \\
& \alpha_i\,(\kappastar_i - \spont_i\,\vec\nu_i) - 
\mat Q_{i,\theta}\,\vec y_i = \vec 0
\qquad \text{on }\ \Gamma_i(t)\,,\ i = 1,2\,,\label{eq:LM3b} \\
& \alpha^G_i \,\kappastar_\gamma
 + \vec y_i + C_1\,\vec\phi = \vec 0 
\qquad \text{on }\ \gamma(t)\,,\ i = 1,2\,,
\label{eq:LM3c}  \\
&  \sum_{i=1}^2 \alpha^G_i \,\vec{\rm m}_i - \vec z = \vec 0 
\qquad \text{on }\ \gamma(t)\,,\ i = 1,2\,,
 \label{eq:LM3d}  
\end{align} 
\end{subequations}
with (\ref{eq:side3}), (\ref{eq:mC1})  and (\ref{eq:varkappapartial}).
As $\partial^0_\epsilon\, \kappastar_i = \vec 0$, we have that
\begin{equation}  \label{eq:Qdiv}
\partial^0_\epsilon\,[\mat Q_{i,\theta}\,\kappastar_i] =
(1-\theta)\left[ (\kappastar_i\,.\,\partial^0_\epsilon\,\vec\nu_i)\,\vec\nu_i
+ (\kappastar_i\,.\,\vec\nu_i)\,\partial^0_\epsilon\,\vec\nu_i \right].
\end{equation}
We observe that (\ref{eq:LM3b},c) imply that 
\begin{equation} \label{eq:yboundstar}
\mat Q_{i,\theta}\,\vec y_i  =  \alpha_i\,\kappastar_i -\alpha_i\,\spont_i
\,\vec\nu_i
\quad \text{on }\ \Gamma_i(t) \qquad\text{and}\qquad
\vec y_i + C_1\,\vec\phi =  - \alpha^G_i\,\kappastar_\gamma
\quad \text{on }\ \gamma(t)\,.
\end{equation}

Let us now recover $\kappastar_i$ and $\kappastar_\gamma$ in terms of
the geometry again. To this end, we first recall the identity
\begin{equation} \label{eq:DEthm2.10}
\int_{\Gamma_i(t)} \nabs\,g \dH{d-1} =
- \int_{\Gamma_i(t)} g\,\varkappa_i\,\vec\nu_i \dH{d-1} 
+ \int_{\gamma(t)}  g\, \vec\mu_i \dH{d-2}
\qquad\forall\ g\in H^1(\Gamma_i(t))\,,
\end{equation}
see e.g.\ Theorem~2.10 in \cite{DziukE13} and Proposition~4.5 in
\cite[p.~334]{Taylor11I}.
It immediately follows from (\ref{eq:side3}),
(\ref{eq:LBop}) and (\ref{eq:DEthm2.10}) that $\vec{\rm m}_i = \vec\mu_i$ and 
$\mat Q_{i,\theta}\,\kappastar_i = \vec\varkappa_i = \varkappa_i\,\vec\nu_i$, 
with the latter implying that 
\begin{equation} \label{eq:kappastarnu}
\kappastar_i\,.\,\vec\nu_i = \varkappa_i\,.
\end{equation}
Hence we immediately get $\kappastar_i = \vec\varkappa_i$ for
$\theta\in(0,1]$. For $\theta=0$, on the other hand, it follows from
(\ref{eq:yboundstar}) and (\ref{eq:kappastarnu}) that
$\alpha_i\,\kappastar_i = 
[\vec y_i\,.\,\vec\nu_i + \alpha_i\,\spont_i]
\,\vec\nu_i$, and so
$\kappastar_i = \varkappa_i\,\vec\nu_i = \vec\varkappa_i$. 
Moreover, combining (\ref{eq:varkappapartial}) 
and (\ref{eq:ident}) yields that $\kappastar_\gamma =
\vec\varkappa_\gamma$. Overall, we obtain from (\ref{eq:yboundstar}) 
that
\begin{equation} \label{eq:ybound}
\mat Q_{i,\theta}\,\vec y_i = \alpha_i\,(\varkappa_i
-\spont_i)\,\vec\nu_i \quad \text{on }\ \Gamma_i(t) \qquad\text{and}\qquad
\vec y_i + C_1\,\vec\phi =  - \alpha^G_i\,\vec\varkappa_\gamma
\quad \text{on }\ \gamma(t)\,.
\end{equation}
However, if $\theta\in(0,1]$, then the two conditions in
(\ref{eq:ybound}) are incompatible in general if $\alpha^G_i \not=0$, 
since the first condition in (\ref{eq:ybound}) yields that
$\vec y_i = \alpha_i(\varkappa_i - \spont_i)\,\vec\nu_i$. If $C_1=1$ then
the two conditions are in general incompatible even if $\alpha^G_i =0$.
Hence for general boundaries $\gamma(t)$ and $\alpha^G_i \not=0$
we need to take $\theta = 0$, at least locally at the boundary. 
Therefore it may be desirable to consider a variable 
$\theta \in L^\infty(\Gamma(t))$. 
The calculation 
(\ref{eq:LM3a}--d) remains valid provided 
that $\partial^0_\epsilon\,\theta = 0$.
We will make this more rigorous on the discrete
level, see (\ref{eq:thetah}) below.

Using (\ref{eq:normvar}), (\ref{eq:Qdiv}) and 
(\ref{eq:LM3c},d) in (\ref{eq:LM3a}) yields the 
condensed version
\begin{subequations}
\begin{align}
& 
f_\Gamma(\vec\chi) = 
\sum_{i=1}^2 \left[
\left\langle \nabs\,\vec y_i, \nabs\,\vec\chi \right\rangle_{\Gamma_i(t)} +
\left\langle \nabs\,.\,\vec y_i, \nabs\,.\,\vec\chi \right\rangle_{\Gamma_i(t)}
-\left\langle (\nabs\,\vec y_i)^T , \mat D(\vec\chi)\,(\nabs\,\vec\id)^T
\right\rangle_{\Gamma_i(t)} \right.
\nonumber \\ & \qquad \left.
-\tfrac12\left\langle [ \alpha_i\, |\vec\varkappa_i - \spont_i\,\vec\nu_i|^2
- 2\,(\vec\varkappa_i\,.\,\mat Q_{i,\theta}\,\vec y_i)]
\,\nabs\,\vec\id,\nabs\,\vec\chi \right\rangle_{\Gamma_i(t)}
- \alpha_i\,\spont_i \left\langle \vec\varkappa_i, 
[\nabs\,\vec\chi]^T\,\vec\nu_i \right\rangle_{\Gamma_i(t)} \right.
\nonumber \\ & \qquad \left.
- (1-\theta) \left\langle 
 \left[ (\vec\varkappa_i\,.\,[\nabs\,\vec\chi]^T\,\vec\nu_i)\,\vec\nu_i
+ (\vec\varkappa_i\,.\,\vec\nu_i)\,[\nabs\,\vec\chi]^T\,\vec\nu_i \right] 
, \vec y_i \right\rangle_{\Gamma_i(t)}
\right] - \varsigma \left\langle \vec\id_s, \vec\chi_s \right\rangle_{\gamma(t)}
\nonumber\\ &\qquad 
+\sum_{i=1}^2 \alpha^G_i 
\left[ \left\langle \vec\varkappa_\gamma\,.\,\vec{\rm m}_i, 
\vec\id_s\,.\,\vec\chi_s \right\rangle_{\gamma(t)}
+ \left\langle \mat{\mathcal{P}}_\gamma\,(\vec{\rm m}_{i})_s, 
\vec\chi_s \right\rangle_{\gamma(t)} \right]
\qquad \forall\ \vec\chi\in \xspacec\,, \label{eq:LM4a} \\
& \alpha_i\,(\vec\varkappa_i - \spont_i\,\vec\nu_i) - 
\mat Q_{i,\theta}\,\vec y_i= \vec 0 
\qquad \text{on }\ \Gamma_i(t)\,,\ i = 1,2\,, \label{eq:LM4b} \\
& \alpha^G_i \,\vec\varkappa_\gamma
 + \vec y_i + C_1\,\vec\phi = \vec0 
\qquad \text{on }\ \gamma(t)\,,\ i = 1,2\,, \label{eq:LM4c} \\ 
& \left\langle \mat Q_{i,\theta}\,\vec\varkappa_i, \vec\eta 
\right\rangle_{\Gamma_i(t)}
+ \left\langle \nabs \,\vec\id,\nabs\, \vec\eta\right\rangle_{\Gamma_i(t)}= 
\left\langle \vec{\rm m}_i, \vec\eta \right\rangle_{\gamma(t)}
\qquad \forall\ \vec\eta \in \Vti\,,\ i = 1,2\,, \label{eq:LM4d} \\ 
& C_1\,(\vec{\rm m}_1 + \vec{\rm m}_2) = \vec 0 
\qquad \text{on}\quad \gamma(t)\,, \label{eq:LM4e} \\ 
& \left\langle \vec\varkappa_\gamma, 
\vec\eta \right\rangle_{\gamma(t)} +
\left\langle \vec\id_s , \vec\eta_s \right\rangle_{\gamma(t)} = 
0 \quad \forall\ \vec\eta \in \Vpartial\,. \label{eq:LM4f}
\end{align} 
\end{subequations}

\begin{remark} \label{rem:C1}
We recall from {\rm (\ref{eq:ybound})} and the discussion below that in
general we require $\theta=0$.
If $C_1=0$ then it follows from {\rm (\ref{eq:LM4c})} that
$\vec y_i = - \alpha^G_i\,\vec\varkappa_\gamma$ on $\gamma(t)$, for $i=1,2$.
Combining this with {\rm (\ref{eq:LM4b})} for $\theta=0$
then yields that {\rm (\ref{eq:C0bc1})} holds.

On the other hand, in the case of a $C^1$--junction, when $C_1=1$, then
{\rm (\ref{eq:LM4e})} implies that $\vec\mu_1 + \vec\mu_2=\vec0$ and hence that
$\vec\nu_1 = \vec\nu_2 = \vec\nu$ on $\gamma(t)$, and so it follows from
{\rm (\ref{eq:LM4b},c)} with $\theta=0$ that
\begin{equation*} 
\alpha_i\,(\varkappa_i-\spont_i) +
\alpha^G_i\,\vec\varkappa_\gamma\,.\,\vec\nu + \vec\phi\,.\,\vec\nu = 0
\qquad\text{on }\ \gamma(t)\,,\ i=1,2\,,
\end{equation*}
which means that {\rm (\ref{eq:C1bc1})} holds.
\end{remark}

The weak formulation of a generalized $L^2$--gradient flow of 
$E((\Gamma_i(t))_{i=1}^2)$ can then be formulated as follows.
Given $\Gamma_i(0)$, $i=1,2$, for all $t\in(0,T]$ find
$\Gamma_i(t) = \vec x_i(\Upsilon_i,t)$, $i=1,2$, 
with $\vec{\mathcal{V}}(t) \in \Vt$, and
$\vec\varkappa_i(t) \in \VLtwoti$, $\vec y_i(t) \in \Vti$,
$\vec{\rm m}_i(t) \in \Vpartial$, $i=1,2$, as well as
$\vec\varkappa_\gamma \in \VLtwopartial$,
$\vec z \in \VLtwopartial$, $\vec\phi \in \VLtwopartial$ such that
\begin{equation}
\left\langle \vec{\mathcal{V}} , 
\vec\chi \right\rangle_{\Gamma(t)} 
+ \varrho\left\langle\vec{\mathcal{V}} ,\vec\chi 
\right\rangle_{\gamma(t)} \black
 = 
f_\Gamma(\vec\chi) 
 \quad \forall\ \vec\chi \in \xspacec 
\label{eq:weakGD3a}
\end{equation}
and (\ref{eq:LM4a}--f) hold. Here we note that $\varrho=0$ recovers 
a weak formulation for the standard $L^2$--gradient flow. 
As stated in (\ref{eq:VrV}), we allow for $\varrho\geq0$ in general, to
allow for a damping of the movement of the contact line $\gamma(t)$. 
In numerical simulations such a damping often proves beneficial, as it 
suppresses possible oscillations at the contact line. On the other hand,
such a dissipation mechanism at the boundary is probably also relevant 
in applications.

\setcounter{equation}{0}
\section{Semidiscrete finite element approximation} \label{sec:3}
It is the aim of this section to introduce a semidiscrete 
continuous-in-time finite element approximation of the weak formulation
(\ref{eq:weakGD3a}), (\ref{eq:LM4a}--f) derived in the previous section.
Our finite element discretization will be given by
(\ref{eq:weakGFa}--f) below, and the main result of this section is the
stability proof in Theorem~\ref{thm:sd3stab} below.

Similarly to \cite{gflows3d}, we introduce the following discrete
spaces. Let
$\Gamma^h(t)\subset\R^d$ be $(d-1)$-dimensional {\em polyhedral surfaces}, 
i.e.\ unions of non-degenerate $(d-1)$-simplices with no
hanging vertices (see \cite[p.~164]{DeckelnickDE05} for $d=3$),
approximating the surfaces $\Gamma(t)$. 
In
particular, let $\Gamma^h(t)=\bigcup_{j=1}^{J}
\overline{\sigma^h_{j}(t)}$, where $\{\sigma^h_{j}(t)\}_{j=1}^{J}$
is a family mutually disjoint open $(d-1)$-simplices with vertices
$\{\vec{q}^h_{k}(t)\}_{k=1}^{K}$.
In analogy to the continuous setting, we write 
$\Gamma^h(t) = \Gamma_1^h(t) \cup \gamma^h(t) \cup \Gamma_2^h(t)$, where
$\gamma^h(t) = \partial \Gamma_1^h(t) = \partial \Gamma_2^h(t)$. 
Here we let $\Gamma^h_i(t)=\bigcup_{j=1}^{J_i}
\overline{\sigma^h_{i,j}(t)}$, with vertices
$\{\vec{q}^h_{i,k}(t)\}_{k=1}^{K_i}$, $i=1,2$.
We also assume that $\gamma^h(t)$ has the vertices
$\{\vec q^h_{\gamma,k}(t)\}_{k=1}^{K_\gamma}$.
Clearly, it holds that $J = J_1 + J_2$ and
$K = K_1 + K_2 - K_\gamma$.
Then let
\[
\Vhti = \{\vec\chi \in [C(\Gamma_i^h(t))]^d:\vec\chi\!\mid_{\sigma^h_{i,j}}
\mbox{ is linear}\ \forall\ j=1,\ldots, J_i\} 
= [\Whti]^d \,,\ i=1,2\,,
\]
where $\Whti$ is the space of scalar continuous
piecewise linear functions on $\Gamma_i^h(t)$, with \linebreak
$\{\chi^h_{i,k}(\cdot,t)\}_{k=1}^{K_i}$ 
denoting the standard basis of $\Whti$, i.e.\
\begin{equation} \label{eq:bf}
\chi^h_{i,k}(\vec q^h_{i,l}(t),t) = \delta_{kl}\qquad
\forall\ k,l \in \{1,\ldots,K_i\}\,, t \in [0,T]\,.
\end{equation}
In addition, let
\begin{equation*} 
\Vht = \{ \vec\chi \in  [C(\Gamma^h(t))]^d:\vec\chi\!\mid_{\Gamma^h_i(t)}
\in \Vhti,\ i=1,2 \} = [\Wht]^d\,.
\end{equation*}
We denote the basis functions of $\Wht$ by $\{\chi^h_{k}(\cdot,t)\}_{k=1}^{K}$.
Moreover, let
\begin{subequations}
\begin{align} \label{eq:Vhpartialt}
\Vhpartialt & := \{\vec\psi \in [C(\gamma^h(t))]^d: 
\exists\ \vec\chi\in\Vht\ \vec\chi\!\mid_{\gamma^h(t)} = \vec\psi \} =:
 [\Whpartialt]^d \,, \\
\Vhzerot & := \{\vec\chi \in \Vht : \vec\chi\!\mid_{\gamma^h(t)} = \vec 0\}\,,
\label{eq:Vhzerot} \\
\Vhzeroti & := \{\vec\chi \in \Vhti : \vec\chi\!\mid_{\gamma^h(t)} = \vec 0\}
\,.
\label{eq:Vhizerot}
\end{align}
\end{subequations}
We denote the basis functions of $\Whpartialt$ by 
$\{\phi^h_{k}(\cdot,t)\}_{k=1}^{K_\gamma}$.
We require that $\Gamma^h_i(t) = \vec X^h(\Gamma^h_i(0), t)$ with
$\vec X^h \in \Vhtz$, and that $\vec q^h_k \in [H^1(0,T)]^d$,
$k = 1,\ldots,K$.

We denote the $L^2$--inner products on $\Gamma^h(t)$, $\Gamma^h_i(t)$ and
and $\gamma^h(t)$ by
$\langle\cdot,\cdot\rangle_{\Gamma^h(t)}$,
$\langle\cdot,\cdot\rangle_{\Gamma^h_i(t)}$
and $\langle\cdot,\cdot\rangle_{\gamma^h(t)}$, respectively. 
In addition, for piecewise continuous functions, with possible jumps
across the edges of $\{\sigma_{i,j}^h\}_{j=1}^{J_i}$,
we also introduce the mass lumped inner product
\begin{equation*} 
\left\langle \eta, \phi \right\rangle^h_{\Gamma^h_i(t)} 
= \sum_{j=1}^{J_i} \left\langle \eta, \phi \right\rangle^h_{\sigma^h_{i,j}(t)} 
 :=
\sum_{j=1}^{J_i}\tfrac1d\,\mathcal{H}^{d-1}(\sigma^h_{i,j}(t))\,\sum_{k=1}^{d} 
(\eta\,\phi)((\vec{q}^h_{i,j_k}(t))^-)\,,
\end{equation*}
where $\{\vec{q}^h_{i,j_k}(t)\}_{k=1}^{d}$ are the vertices of 
$\sigma^h_{i,j}(t)$, and where
we define $\eta((\vec{q}^h_{i,j_k}(t))^-):= 
\underset{\sigma^h_j(t)\ni \vec{p}\to \vec{q}^h_{i,j_k}(t)}{\lim}\, 
\eta(\vec{p})$.
We naturally extend this definition to vector and tensor functions. 
We also define the mass lumped inner products 
$\langle\cdot,\cdot\rangle_{\Gamma^h(t)}^h$ and
$\langle\cdot,\cdot\rangle_{\gamma^h(t)}^h$ in the obvious way.

Let $\vec\nu^h_i$ denote the the outward unit normal to $\Gamma^h_i(t)$,
$i=1,2$, and similarly let $\vec\nu^h$ denote the the outward unit normal to 
$\Gamma^h(t)$.
Then we introduce the vertex normal functions
$\vec\omega^h_i(\cdot, t) \in\Vhti$ with
\begin{equation} \label{eq:omegahi}
\vec\omega^h_i(\vec{q}^h_{i,k}(t),t ) := 
\frac{1}{\mathcal{H}^{d-1}(\Lambda^h_{i,k}(t))}
\sum_{j\in\Theta_{i,k}^h} \mathcal{H}^{d-1}(\sigma^h_{i,j}(t))\,
\vec\nu^h_i\!\mid_{\sigma^h_{i,j}(t)}\,,
\end{equation}
where for $k= 1 ,\ldots, K_i$ we define
$\Theta_{i,k}^h:= \{j : \vec{q}^h_{i,k}(t) \in \overline{\sigma^h_{i,j}(t)}\}$
and set
$\Lambda_{i,k}^h(t) := \cup_{j \in \Theta_{i,k}^h} 
\overline{\sigma^h_{i,j}(t)}$. 
Here we note that 
\begin{equation} 
\left\langle \vec z, w\,\vec\nu^h_i\right\rangle_{\Gamma^h_i(t)}^h =
\left\langle \vec z, w\,\vec\omega^h_i\right\rangle_{\Gamma^h_i(t)}^h 
\qquad \forall\ \vec z \in \Vhti\,,\ w \in \Whti \,.
\label{eq:NIhi}
\end{equation}
In the analogous fashion, we introduce the vertex normal function
$\vec\omega^h(\cdot, t) \in \Vht$, i.e.\ we set
\begin{equation} \label{eq:omegah}
\vec\omega^h(\vec{q}^h_k(t),t ) := 
\frac{1}{\mathcal{H}^{d-1}(\Lambda^h_k(t))}
\sum_{j\in\Theta_k^h} \mathcal{H}^{d-1}(\sigma^h_j(t))\,
\vec\nu^h\!\mid_{\sigma^h_j(t)}\,,
\end{equation}
where for $k= 1 ,\ldots, K$ we define
$\Theta_k^h:= \{j : \vec{q}^h_k(t) \in \overline{\sigma^h_j(t)}\}$
and set $\Lambda_k^h(t) := \cup_{j \in \Theta_k^h} \overline{\sigma^h_j(t)}$. 
Of course, it holds that
\begin{equation} 
\left\langle \vec z, w\,\vec\nu^h\right\rangle_{\Gamma^h(t)}^h =
\left\langle \vec z, w\,\vec\omega^h\right\rangle_{\Gamma^h(t)}^h 
\qquad \forall\ \vec z \in \Vht\,,\ w \in \Wht \,.
\label{eq:NIh}
\end{equation}
It clearly follows from (\ref{eq:NIhi}) and (\ref{eq:NIh}) that
\begin{equation} \label{eq:zo}
\left\langle \vec z, \vec\omega^h\right\rangle_{\Gamma^h(t)}^h 
= \sum_{i=1}^2
\left\langle \vec z, \vec\omega^h_i\right\rangle_{\Gamma^h_i(t)}^h 
\qquad \forall\ \vec z \in \Vht \,.
\end{equation}

In addition, for a given parameter $\theta \in [0,1]$ we introduce 
$\theta^h \in \Wht$ and $\theta^h_\star \in \Wht$ such that 
\begin{equation} \label{eq:thetah}
\theta^h(\vec q^h_k(t), t) = \begin{cases}
0 & \vec q^h_k(t) \in \gamma^h(t)\,,\\
\theta & \vec q^h_k(t) \not\in \gamma^h(t)\,,
\end{cases}
\quad\text{and}\quad
\theta^h_\star(\vec q^h_k(t), t) = \begin{cases}
1 & \vec q^h_k(t) \in \gamma^h(t)\,,\\
\theta & \vec q^h_k(t) \not\in \gamma^h(t)\,.
\end{cases}
\end{equation}
Then, similarly to (\ref{eq:Q}), 
we introduce $\mat Q^h_{i,\theta^h} \in [\Whti]^{d\times d}$ and
$\mat Q^h_{i,\theta^h_\star} \in [\Whti]^{d\times d}$ 
by setting, for $k \in \{1,\ldots,K_i\}$,
\begin{equation} \label{eq:Qalphah}
\mat Q^h_{i,\theta^h}(\vec q^h_{i,k}(t), t) = 
\theta^h(\vec q^h_{i,k}(t), t)\,\mat\Id + (1-\theta^h(\vec q^h_{i,k}(t), t))
\,\frac{\vec\omega^h_i(\vec{q}^h_{i,k}(t), t)
\otimes \vec\omega^h_i(\vec{q}^h_{i,k}(t), t)}{
|\vec\omega^h_i(\vec{q}^h_{i,k}(t), t)|^2}
\,,
\end{equation}
and similarly for $\mat Q^h_{i,\theta^h_\star}$,
where here and
throughout we assume that $\vec\omega^h_i(\vec{q}^h_{i,k}(t), t) \not= \vec0$
for $k= 1 ,\ldots, K_i$ and $t\in[0,T]$. Only in pathological cases could
this assumption be violated, and in practice this never occurred.
We note that
\begin{equation} \label{eq:aQb}
\left\langle \mat Q^h_{i,\theta^h}\, \vec z, 
\vec v\right\rangle_{\Gamma^h_i(t)}^h =
\left\langle \vec z, \mat Q^h_{i,\theta^h}\, 
\vec v\right\rangle_{\Gamma^h_i(t)}^h 
\quad\text{and}\quad
\left\langle \mat Q^h_{i,\theta^h} \vec z, \vec\omega^h_i
\right\rangle_{\Gamma^h_i(t)}^h =
\left\langle \vec z, \vec\omega^h_i \right\rangle_{\Gamma^h_i(t)}^h 
\end{equation}
for all $\vec z, \vec v \in \Vhti$, and analogously for
$\theta^h$ in (\ref{eq:aQb}) replaced by $\theta^h_\star$.
In addition, similarly to (\ref{eq:zo}), it holds that
\begin{equation} \label{eq:sumzo}
 \sum_{i=1}^2
\left\langle \vec z, \mat Q^h_{i,\theta^h_\star}\, 
\vec\omega^h_i\right\rangle_{\Gamma^h_i(t)}^h 
=
 \sum_{i=1}^2
\left\langle \vec z, \mat Q^h_{i,\theta^h_\star} \,
\vec\omega^h\right\rangle_{\Gamma^h_i(t)}^h 
\qquad \forall\ \vec z \in \Vht \,.
\end{equation}

Following the approach in the continuous setting, recall (\ref{eq:E2}), 
(\ref{eq:side3}), (\ref{eq:varkappapartial}), we consider the
first variation of the discrete energy
\begin{align} \label{eq:Eh}
E^h((\Gamma^h_i(t))_{i=1}^2) & 
:= \sum_{i=1}^2 
\left[\tfrac12\,\alpha_i
\left\langle |\vec\kappa^h_i - \spont_i\,\vec\nu^h_i|^2, 1
\right\rangle_{\Gamma^h_i(t)}^h
 + \alpha^G_i \left[ \left\langle \vec\kappa_\gamma^h,
 \vec{\rm m}^h_i \right\rangle_{\gamma^h(t)}^h
 + 2\,\pi\,m(\Gamma^h_i(t))\right] \right] \nonumber \\ & \qquad
+ \varsigma\, \mathcal{H}^{d-2}(\gamma^h(t)) \,,
\end{align}
where $\vec\kappa^h_i \in \Vhti$, $\vec{\rm m}_i^h \in \Vhpartialt$, $i=1,2$,
and $\vec\kappa^h_\gamma \in \Vhpartialt$, 
subject to the side constraints
\begin{subequations}
\begin{align}
\left\langle \mat Q^h_{i,\theta^h}\,\vec\kappa^h_i, 
\vec\eta \right\rangle_{\Gamma^h_i(t)}^h
+ \left\langle
\nabs \,\vec\id,\nabs\, \vec\eta\right\rangle_{\Gamma^h_i(t)} & = 
\left\langle \vec{\rm m}^h_i, \vec\eta \right\rangle_{\gamma^h(t)}^h
 \quad \forall\ \vec\eta \in \Vhti\,,\ i=1,2\,,\label{eq:side3h} \\
\left\langle \vec\kappa_\gamma^h, \vec\chi 
\right\rangle_{\gamma^h(t)}^h
+ \left\langle \vec\id_s, \vec\chi_s \right\rangle_{\gamma^h(t)} & = 0
\qquad \forall\ \vec\chi \in \Vhpartialt\,, \label{eq:sidezh} \\
C_1\,(\vec{\rm m}_1^h + \vec{\rm m}_2^h) & = \vec 0 \qquad \text{on}\quad 
\gamma^h(t)\,. \label{eq:mC1h}
\end{align}
\end{subequations}
In particular, we define the Lagrangian 
\begin{align}
L^h(t) & =
\sum_{i=1}^2 \left[
\tfrac12\,\alpha_i \left\langle |\vec\kappa^h_i - \spont_i\,\vec\nu^h_i|^2, 1
\right\rangle_{\Gamma^h_i(t)}^h
 + \alpha^G_i\left\langle
   \vec\kappa_\gamma^h,\vec{\rm m}^h_i
 \right\rangle_{\gamma^h(t)}^h 
\right] 
+ \varsigma\, \mathcal{H}^{d-2}(\gamma^h(t))
- \left\langle \vec\kappa_\gamma^h, \vec Z^h
\right\rangle_{\gamma^h(t)}^h
 \nonumber \\ & \qquad
- \left\langle \vec\id_s, \vec Z^h_s \right\rangle_{\gamma^h(t)}
- \sum_{i=1}^2 \left[
 \left\langle (\mat Q^h_{i,\theta^h}\,\vec\kappa^h_i, 
\vec Y^h_i \right\rangle_{\Gamma^h_i(t)}^h 
+ \left\langle \nabs \,\vec\id, \nabs\,\vec Y^h_i \right\rangle_{\Gamma^h_i(t)} 
- \left\langle \vec{\rm m}^h_i, \vec Y^h_i 
\right\rangle_{\gamma^h(t)}^h \right]
 \nonumber \\ & \qquad
+ C_1 \left\langle \vec{\rm m}_1^h + \vec{\rm m}_2^h, 
\vec\Phi^h\right\rangle_{\gamma^h(t)}^h, \label{eq:Lagh}
\end{align}
where $\vec\kappa^h_i\in \Vhti$ and
$\vec\kappa^h_\gamma \in \Vhpartialt$
satisfy (\ref{eq:side3h}) and (\ref{eq:sidezh}), respectively, 
with $\vec Y^h_i\in \Vhti$ and 
$\vec Z^h \in \Vhpartialt$ being the corresponding
Lagrange multipliers. Similarly, 
$\vec\Phi^h\!\in \Vhpartialt$ is a Lagrange multiplier for (\ref{eq:mC1h}).
It turns out that when mimicking the continuous approach from
Section~\ref{sec:2} on the discrete level, we need several technical
definitions to make the arguments rigorous. We present the majority of the
necessary definitions and properties next, before proceeding with taking
suitable variations of (\ref{eq:Lagh}).

Following \cite[(5.23)]{DziukE13}, we define the discrete material velocity
for $\vec z \in \Gamma^h(t)$ by
\begin{equation*} 
\vec{\mathcal{V}}^h(\vec z, t) := \sum_{k=1}^{K}
\left[\ddt\,\vec q^h_k(t)\right] \chi^h_k(\vec z, t) \,.
\end{equation*}
We also introduce the finite element spaces
\begin{align*}
W^h_{T}(\GhTi) := \{ \phi \in C(\GhTi) & : 
\phi(\cdot, t) \in \Whti \quad \forall\ t \in [0,T]\,,\nonumber \\ & \quad
\phi(\vec q^h_{i,k}(\cdot), \cdot) \in H^1(0,T) \quad \forall\ k \in\{1,\ldots,K_i\}
 \}\,,
\end{align*}
where $\GhTi := \bigcup_{t \in [0,T]} \Gamma^h_i(t) \times \{t\}$,
as well as the vector valued analogue $\underline{V}^h_{T}(\GhTi)$.
In a similar fashion, we introduce $W^h_{T}(\sigma^h_{j,T})$ and
$\underline {V}^h_{T}(\sigma^h_{j,T})$ via e.g.\
\[
W^h_T(\sigma^h_{j,T}) := \{ \phi \in C(\overline{\sigma^h_{j,T}}) : 
\phi(\cdot, t) \text{~ is linear~} \forall\ t \in [0,T]\,,\
\phi(\vec q^h_{j_k}(\cdot), \cdot) \in H^1(0,T) \quad k = 1,\ldots,d
 \}\,,
\]
where 
$\{\vec{q}^h_{j_k}(t)\}_{k=1}^{d}$ are the vertices of $\sigma^h_j(t)$,
and where
$\sigma^h_{j,T} := \bigcup_{t \in [0,T]} \sigma^h_j(t) \times \{t\}$,
for $j\in\{1,\ldots,J\}$.
Moreover, we define the analogue variants $W^h_{T}(\GhT)$ and
$\underline {V}^h_{T}(\GhT)$ on $\GhT = 
\bigcup_{t \in [0,T]} \Gamma^h(t) \times \{t\}$, as well as
$W^h_{T}(\gamma^h_T)$ and $\underline {V}^h_{T}(\gamma^h_T)$ on 
$\gamma^h_T := \bigcup_{t \in [0,T]} \gamma^h(t) \times \{t\}$, with the scalar
space for the latter e.g.\ being given by
\[
W^h_T(\gamma^h_T) := 
\{\psi \in C(\gamma^h_T): 
\exists\ \chi\in W^h_T(\GhT)\ \chi(\cdot,t)\!\mid_{\gamma^h(t)} = 
\psi(\cdot,t)\quad \forall\ t\in[0,T] \}\,.
\]

Then, similarly to (\ref{eq:matpartx}), we define the discrete material
derivatives on $\Gamma^h(t)$ element-by-element via the equations
\begin{equation*} 
(\matpartxh\, \phi)\!\mid_{\sigma^h_j(t)} = 
(\phi_t + \vec{\mathcal{V}}^h\,.\,\nabla\,\phi) \!\mid_{\sigma^h_j(t)}
\qquad\forall\ \phi\in W^h_T(\sigma^h_{j,T})\,,\ j \in \{1,\ldots,J\}\,.
\end{equation*}
On differentiating (\ref{eq:bf}) with respect to $t$, it immediately follows
that
\begin{equation} \label{eq:mpbf}
\matpartxh\, \chi^h_k = 0
\quad\forall\ k \in \{1,\ldots,K\}\,,
\end{equation}
see also \cite[Lem.\ 5.5]{DziukE13}. 
It follows directly from (\ref{eq:mpbf}) that
\begin{equation} \label{eq:p1}
\matpartxh\,\phi(\cdot,t) = \sum_{k=1}^{K} \chi^h_k(\cdot,t)\,
\ddt\,\phi_k(t) \quad \text{on}\ \Gamma^h(t)
\end{equation}
for $\phi(\cdot,t) = \sum_{k=1}^{K} \phi_k(t)\,\chi^h_k(\cdot,t)
\in \Wht$.

We recall from \cite[Lem.~5.6]{DziukE13} that
\begin{equation} \label{eq:DEeq5.28}
\ddt\, \int_{\sigma^h_j(t)} \phi \dH{d-1} 
= \int_{\sigma^h_j(t)} \matpartxh\,\phi + \phi\,\nabs\,.\,\vec{\mathcal{V}}^h 
\dH{d-1} \quad \forall\ \phi \in W^h_T(\sigma^h_{j,T})\,, j \in
\{1,\ldots,J\}\,.
\end{equation}
Similarly, we recall from \cite[Lem.~3.1]{tpfs} that
\begin{equation} \label{eq:DElem5.6NI}
\ddt \langle \eta, \phi \rangle_{\sigma^h_j(t)}^h
 = \langle \matpartxh\,\eta, \phi \rangle_{\sigma^h_j(t)}^h
 + \langle \eta, \matpartxh\,\phi \rangle_{\sigma^h_j(t)}^h
+ \langle \eta\,\phi, \nabs\,.\,\vec{\mathcal{V}}^h \rangle_{\sigma^h_j(t)}^h
\quad \forall\ \eta,\phi \in W^h_T(\sigma^h_{j,T})\,,
\end{equation}
for all $j\in\{1,\ldots,J\}$.
Moreover, it holds that
\begin{equation} \label{eq:DElem5.6NIpartial}
\ddt \langle \eta, \phi \rangle_{\gamma^h(t)}^h
 = \langle \matpartxh\,\eta, \phi \rangle_{\gamma^h(t)}^h
 + \langle \eta, \matpartxh\,\phi \rangle_{\gamma^h(t)}^h
+ \langle \eta\,\phi, \vec\id_s\,.\,\vec{\mathcal{V}}^h_s 
\rangle_{\gamma^h(t)}^h
\ \forall\ \eta,\phi \in W^h_T(\gamma^h_T)\,.
\end{equation}
We also note the discrete version of the time derivative variant of 
(\ref{eq:secvar}),
\begin{align}
& \ddt \left\langle \nabs\,\vec\id,
\nabs\,\vec\eta\right\rangle_{\Gamma^h_i(t)} = 
 \left\langle \nabs\,\vec\eta, \nabs \,\vec{\mathcal V}^h 
 \right\rangle_{\Gamma^h_i(t)}
+ \left\langle \nabs\,.\,\vec\eta, \nabs \,.\,\vec{\mathcal V}^h
 \right\rangle_{\Gamma^h_i(t)}
\nonumber \\ & \hspace{3cm}
- \left\langle (\nabs \,\vec\eta)^T, \mat D(\vec{\mathcal V}^h)
 \,(\nabs\,\vec\id)^T \right\rangle_{\Gamma^h_i(t)} 
\quad \forall\ \vec\eta \in \{ \vec\xi \in \underline{V}^h_T(\GhTi)
: \matpartxh\,\vec\xi = \vec0
\}\,,
\label{eq:secvar2h}
\end{align}
as well as the corresponding version for $\gamma^h(t)$,
\begin{equation}
 \ddt \left\langle \vec\id_s,
\vec\eta_s\right\rangle_{\gamma^h(t)} = 
\left\langle \mat{\mathcal{P}}_\gamma^h\,\vec\eta_s, 
\vec{\mathcal V}^h_s \right\rangle_{\gamma^h(t)}
\quad \forall\ \vec\eta \in \{ 
\vec\xi \in \underline{V}^h_T(\gamma^h_T) :
\matpartxh\,\vec\xi = \vec0 \}\,,
\label{eq:secvarpartial2h}
\end{equation}
which follows similarly to (\ref{eq:secvarpartial}). 
Here, similarly to (\ref{eq:mathcalP}), we have defined
\begin{equation} \label{eq:Ph}
\mat{\mathcal{P}}^h_\gamma = \mat\Id - \vec\id_s \otimes \vec\id_s
\quad \text{on} \quad \gamma^h(t)\,.
\end{equation}

Finally, 
when taking variations of (\ref{eq:Lagh}), we need to compute variations of
the discrete vertex normals $\vec\omega^h_i$. To this end, 
for any given $\vec\chi \in \Vht$ we introduce  
$\Gamma^h_\epsilon(t)$ as in (\ref{eq:Gammadelta})
and $\partial^{0,h}_\epsilon$ defined by
(\ref{eq:deltaderiv}), both
with $\Gamma(t)$ replaced by $\Gamma^h(t)$.
We then observe that it
follows from (\ref{eq:NIhi}) with $w=1$ and the discrete analogue of 
(\ref{eq:vardetwnu}) that
\begin{align}
& \left\langle \vec z, 
 \partial^{0,h}_\epsilon \,\vec\omega^h_i 
 \right\rangle_{\Gamma^h_i(t)}^h
= \left\langle \vec z, 
\partial^{0,h}_\epsilon \,\vec\nu^h_i 
 \right\rangle_{\Gamma^h_i(t)}^h
+ \left\langle (\vec z\,.\,(\vec\nu^h_i - \vec\omega^h_i))\,\nabs\,\vec\id , 
\nabs \,\vec\chi \right\rangle_{\Gamma^h_i(t)}^h
\nonumber \\ & \hspace{8cm}
 \quad \forall\ \vec z\in\Vhti,\, \vec\chi \in \Vht\,,
\label{eq:ddomegah}
\end{align}
where $\partial^{0,h}_\epsilon \,\vec z = \vec 0$.
In addition, we note that for all $\vec\xi$, $\vec\eta \in \Vhti$ with
$\partial_\epsilon^{0,h} \, \vec\xi = \partial_\epsilon^{0,h} \, \vec\eta =
\vec0$ it holds that
\begin{equation}
\partial^{0,h}_\epsilon \,\pi^h_i \left[
\left( \vec\xi\,.\,\frac{\vec\omega^h_i}{|\vec\omega^h_i|}\right)\left(
\vec\eta\,.\,\frac{\vec\omega^h_i}{|\vec\omega^h_i|} \right)\right] 
= \pi^h_i \left[
\vec G^h_i(\vec\xi,\vec\eta)\,.\,
\partial^{0,h}_\epsilon\,\vec\omega^h_i
\right] \quad\text{on}\ \Gamma^h_i(t)\,,
\label{eq:ddxieta}
\end{equation}
where
\begin{equation}
\vec G^h_i(\vec\xi,\vec\eta) =
\frac{1}{|\vec\omega_i^h|^2}
 \left( (\vec\xi\,.\,\vec\omega_i^h)\,\vec\eta +
 (\vec\eta\,.\,\vec\omega^h_i)\,\vec\xi
- 2\, \frac{(\vec\eta\,.\,\vec\omega^h_i)\,(\vec\xi\,.\,\vec\omega^h_i)}
{|\vec\omega^h_i|^2}\,\vec\omega^h_i
\right) 
, \label{eq:Gxieta}
\end{equation}
and where
$\pi^h_i(t): C(\Gamma_i^h(t))\to \Whti$ is the standard interpolation operator
at the nodes \linebreak $\{\vec{q}_{i,k}^h(t)\}_{k=1}^{K_i}$.
It follows that
\begin{equation} \label{eq:Gomega}
\vec G^h_i(\vec\xi,\vec\eta)\,.\,\vec\omega^h_i = 0 \quad \forall\ \vec\xi,\ 
\vec\eta \in \Vhti\,.
\end{equation}

We are now in a position to formally derive the 
$L^2$--gradient flow of $E^h(t)$ subject 
to the side constraints (\ref{eq:side3h}--c). In particular,
on recalling the formal calculus of PDE
constrained optimization, we set
$[\deldel{\Gamma^h}\,L^h](\vec\chi)=
-\sum_{i=1}^2 \left\langle \mat Q^h_{i,\theta^h_\star}\,\vec{\mathcal{V}}^h, 
\vec\chi \right\rangle_{\Gamma^h_i(t)}^h$
for $\vec\chi\in\Vht$,
$[\deldel{\vec\kappa^h_i}\,L^h](\vec\xi) = 0$ for $\vec\xi\in\Vhti$,
$[\deldel{\vec Y^h_i}\,L^h](\vec\eta) = 0$ for $\vec\eta\in\Vhti$,
$[\deldel{\vec{\rm m}^h_i}\,L^h](\vec\varphi) = 0$ for 
$\vec\varphi\in\Vhpartialt$,
$[\deldel{\vec\kappa_\gamma^h}\,L^h](\vec\phi) = 0$ for 
$\vec\phi\in\Vhpartialt$, leading to $\vec Z^h =
\sum_{i=1}^2\alpha^G_i\,\vec{\rm m}^h_i$, 
$[\deldel{\vec Z^h}\,L^h](\vec\phi) = 0$ for 
$\vec\phi\in\Vhpartialt$ and
$[\deldel{\vec\Phi^h}\,L^h](\vec\eta) = 0$ for 
$\vec\eta\in\Vhpartialt$. 
Here we recall the definition of $\theta^h_\star$ in 
(\ref{eq:thetah}). We employ this doctored version of $\theta^h$ in order to
obtain existence and uniqueness for the fully discrete approximation introduced
in the next section. See also Remark~3.1 in \cite{pwftj}, where the
analogue to our situation here corresponds to two curves meeting in the plane,
i.e.\ $N=d=2$ in their notation.

Overall this gives rise to the following semidiscrete finite element
approximation of the gradient flow (\ref{eq:weakGD3a}), where we have noted 
the discrete version of (\ref{eq:normvar}), 
(\ref{eq:ddomegah}), (\ref{eq:ddxieta}), (\ref{eq:Gomega}), 
variational versions of 
(\ref{eq:DElem5.6NI})--(\ref{eq:secvarpartial2h}) and  
that $\partial_\epsilon^{0,h}\,\theta^h = 0$. 
Given $\Gamma^h(0)$, find $(\Gamma^h(t))_{t\in(0,T]}$ such that
$\vec\id\!\mid_{\Gamma^h(\cdot)} \in \underline{V}^h_T(\GhT)$.
In addition, for all $t\in(0,T]$ find 
$(\vec\kappa^{h}_i,\vec Y^{h}_i) \in [\Vhti]^2$, $i=1,2$,
$\vec\kappa_\gamma^{h}\in \Vhpartialt$, 
$\vec{\rm m}^{h}_i \in \Vhpartialt$, $i=1,2$,
and $C_1\,\Phi^{h} \in \Vhpartialt$ such that
\begin{subequations}
\begin{align}
& 
\sum_{i=1}^2 
\left\langle \mat Q^h_{i,\theta^h_\star}\,\vec{\mathcal{V}}^h,
\vec\chi \right\rangle_{\Gamma^h_i(t)}^h
+ \varrho\left\langle\vec{\mathcal{V}}^h ,\vec\chi 
\right\rangle_{\gamma^h(t)}^h \black
\nonumber \\ & \quad
= 
\sum_{i=1}^2 \left[
  \left\langle \nabs\,\vec Y^{h}_i, \nabs\,\vec\chi 
\right\rangle_{\Gamma^h_i(t)} 
+ \left\langle \nabs\,.\,\vec Y^h_i, \nabs\,.\,\vec\chi 
\right\rangle_{\Gamma^h_i(t)} 
- \left\langle (\nabs\,\vec Y^h_i)^T , 
\mat D(\vec\chi)\,(\nabs\,\vec\id)^T \right\rangle_{\Gamma^h_i(t)} 
\right. \nonumber \\ & \qquad \left.
-\tfrac12\left\langle [ \alpha_i\,|\vec\kappa^h_i - \spont_i\,\vec\nu^h_i|^2
- 2\,(\vec Y^h_i\,.\,\mat Q^h_{i,\theta^h}\,\vec\kappa^h_i)]
\,\nabs\,\vec\id,\nabs\,\vec\chi \right\rangle_{\Gamma^h_i(t)}^h
\right. \nonumber \\ & \qquad \left.
-\alpha_i\,\spont_i 
\left\langle \vec\kappa^h_i, [\nabs\,\vec\chi]^T\,\vec\nu^h_i
\right\rangle_{\Gamma^h_i(t)}^h 
+\left\langle(1-\theta^h)\, (\vec G^h_i(\vec Y^h_i, \vec\kappa^h_i)\,.\,
\vec\nu^h_i)\,
\nabs\,\vec\id, \nabs\,\vec\chi \right\rangle_{\Gamma^h_i(t)}^h
\right. \nonumber \\ & \qquad \left.
- \left\langle(1-\theta^h)\,\vec G^h_i(\vec Y^h_i,\vec\kappa^h_i), 
[\nabs\,\vec\chi]^T\,\vec\nu^h_i \right\rangle_{\Gamma^h_i(t)}^h 
 \right]
\nonumber \\ & \qquad 
+ \sum_{i=1}^2 \alpha^G_i 
\left[\left\langle \vec\kappa_\gamma^h \,.\,\vec{\rm m}^h_i,
\vec\id_s\,.\,\vec\chi_s \right\rangle_{\gamma^h(t)}^h
+ \left\langle\mat{\mathcal{P}}^h_\gamma\,(\vec{\rm m}^h_i)_s, 
\vec\chi_s \right\rangle_{\gamma^h(t)} \right]
- \varsigma \left\langle \vec\id_s, \vec\chi_s \right\rangle_{\gamma^h(t)}
\nonumber \\ & \hspace{10cm}
\quad \forall\ \vec\chi \in \Vht\,, \label{eq:weakGFa} \\
& \left\langle \mat Q^h_{i,\theta^h}\,\vec\kappa^{h}_i, 
\vec\eta\right\rangle_{\Gamma_i^h(t)}^h
+ \left\langle
\nabs \,\vec\id,\nabs\, \vec\eta\right\rangle_{\Gamma^h_i(t)} = 
\left\langle \vec{\rm m}^{h}_i, \vec\eta \right\rangle_{\gamma^h(t)}^h
 \quad\forall\ \vec\eta \in \Vhti\,,\ i=1,2\,, \label{eq:weakGFb} \\
& \left\langle \vec\kappa_\gamma^{h}, \vec\chi 
\right\rangle_{\gamma^h(t)}^h
+ \left\langle \vec\id_s, \vec\chi_s \right\rangle_{\gamma^h(t)} = 0
\qquad \forall\ \vec\chi \in \Vhpartialt\,, \label{eq:weakGFc} \\
& C_1\,(\vec{\rm m}_1^{h} + \vec{\rm m}_2^{h}) = 
\vec 0 \qquad \text{on}\quad \gamma^h(t)\,, \label{eq:weakGFd} \\
& \alpha^G_i \,\vec\kappa_\gamma^{h}
 + \vec Y^{h}_i + C_1\,\vec\Phi^{h} = \vec0 
\qquad \text{on}\quad \gamma^h(t) \,, \ i = 1,2\,, \label{eq:weakGFe} \\
& \left\langle \alpha_i\,(\vec\kappa^{h}_i 
- \spont_i\,\vec\nu^h_i) - \mat Q^h_{i,\theta^h}\,\vec Y^h_i,
\vec\xi \right\rangle_{\Gamma^h_i(t)}^h = 0 
\qquad \forall\ \vec\xi \in \Vhti \,,\ i=1,2\,. \label{eq:weakGFf} 
\end{align}
\end{subequations}
We observe that choosing 
$\vec\xi = \alpha^{-1}_i\, \vec\pi^h_i[\mat Q^h_{i,\theta^h}\,\vec\eta]$ 
in (\ref{eq:weakGFf}) and 
combining with (\ref{eq:weakGFb}), on recalling (\ref{eq:aQb}) and 
(\ref{eq:NIhi}), yields that
\begin{equation} \label{eq:weakGFbnew}
 \alpha^{-1}_i\,\left\langle \mat Q^h_{i,\theta^h}\,\vec Y^h_i, 
\mat Q^h_{i,\theta^h}\,\vec\eta \right\rangle_{\Gamma^h_i(t)}^h
+ \left\langle \nabs\,\vec\id,\nabs\,\vec\eta\right\rangle_{\Gamma^h_i(t)} =  
 \left\langle \vec{\rm m}^h_i, \vec\eta \right\rangle_{\gamma^h(t)}^h
- \spont_i \left\langle \vec\omega^h_i,\vec\eta\right\rangle_{\Gamma^h_i(t)}^h
\quad \forall\ \vec\eta \in \Vhti\,.
\end{equation}
Here $\vec\pi^h_i(t): [C(\Gamma_i^h(t))]^d\to \Vhti$
is the standard interpolation operator
at the nodes $\{\vec{q}_{i,k}^h(t)\}_{k=1}^{K_i}$.

In order to be able to consider area and volume preserving variants of
(\ref{eq:weakGFa}--f), we introduce the Lagrange multipliers
$\lambda_i^{A,h}(t) \in \R$, $i=1,2$, and 
$\lambda^{V,h}(t) \in \R$ for the constraints
\begin{equation} \label{eq:constraintareah}
\ddt\, \mathcal{H}^{d-1}(\Gamma^h_i(t)) = 
\left\langle \nabs\,.\,\vec{\mathcal{V}}^h, 1\right\rangle_{\Gamma^h_i(t)} 
= \left\langle \nabs\,\vec\id, \nabs\,\vec{\mathcal{V}}^h
\right\rangle_{\Gamma^h_i(t)} =0\,,
\end{equation}
where we recall (\ref{eq:DEeq5.28}), and
\begin{equation} \label{eq:constraintvolh}
\ddt\, \mathcal{L}^{d}(\Omega^h(t)) = 
\left\langle \vec{\mathcal{V}}^h, \vec\nu^h\right\rangle_{\Gamma^h(t)} =
\left\langle \vec{\mathcal{V}}^h, \vec\omega^h\right\rangle_{\Gamma^h(t)}^h 
= 0\,,
\end{equation}
where we note a discrete variant of (\ref{eq:dtvol})
and (\ref{eq:NIh}).
Here $\Omega^h(t)$ denotes the interior of $\Gamma^h(t)$.
On recalling (\ref{eq:zo}), (\ref{eq:aQb}) and
(\ref{eq:sumzo}), we can rewrite the constraint (\ref{eq:constraintvolh}) as
\begin{equation} \label{eq:hfd}
0 = \left\langle \vec{\mathcal{V}}^h, \vec\omega^h 
\right\rangle_{\Gamma^h(t)}^h =
\sum_{i=1}^2 \left\langle \vec{\mathcal{V}}^h, \vec\omega^h_i 
\right\rangle_{\Gamma^h_i(t)}^h 
= \sum_{i=1}^2 \left\langle 
\mat Q^h_{i,\theta^h_\star}\,\vec{\mathcal{V}}^h, \vec\omega^h_i 
\right\rangle_{\Gamma^h_i(t)}^h =
\sum_{i=1}^2 \left\langle 
\mat Q^h_{i,\theta^h_\star}\,\vec{\mathcal{V}}^h, \vec\omega^h
\right\rangle_{\Gamma^h_i(t)}^h .
\end{equation}
Hence, on writing (\ref{eq:weakGFa}) as
\[
\sum_{i=1}^2
\left\langle \mat Q^h_{i,\theta^h_\star}\,
\vec{\mathcal{V}}^h, \vec\chi \right\rangle_{\Gamma^h_i(t)}^h
+ \varrho\left\langle\vec{\mathcal{V}}^h ,\vec\chi 
\right\rangle_{\gamma^h(t)}^h \black 
= \left\langle \vec r^h, \vec\chi \right\rangle_{\Gamma^h(t)}^h,
\]
we consider
\begin{equation} \label{eq:LMh3aa}
 \sum_{i=1}^2 \left\langle \mat Q^h_{i,\theta^h_\star}\,
\vec{\mathcal{V}}^h, \vec\chi \right\rangle_{\Gamma^h_i(t)}^h
+ \varrho\left\langle\vec{\mathcal{V}}^h ,\vec\chi 
\right\rangle_{\gamma^h(t)}^h \black
= \left\langle \vec r^h, \vec\chi \right\rangle_{\Gamma^h(t)}^h
- \lambda^{V,h} 
\left\langle \vec\omega^h, \vec\chi \right\rangle_{\Gamma^h(t)}^h
- \sum_{i=1}^2 \lambda^{A,h}_i 
\left\langle \nabs\,\vec\id, \nabs\,\vec\chi \right\rangle_{\Gamma^h_i(t)}
\end{equation}
for all $\vec\chi\in \Vht$,
where $\lambda^{V,h}(t) \in \R$ and $\lambda^{A,h}_i(t) \in \R$, $i=1,2$,
need to be determined. Of course, if we consider a volume preserving variant
only, then we let $\lambda^{A,h}_1(t) = \lambda^{A,h}_2(t) = 0$ and
\begin{equation} \label{eq:muh}
\lambda^{V,h}(t) = 
\left[\left\langle \vec r^h, \vec\omega^h \right\rangle_{\Gamma^h(t)}^h
- \varrho\left\langle\vec{\mathcal{V}}^h ,\vec\omega^h
\right\rangle_{\gamma^h(t)}^h \black \right]
/ \left\langle \vec\omega^h, \vec\omega^h \right\rangle_{\Gamma^h(t)}^h ,
\end{equation}
which we derived on choosing
$\vec\chi = \vec\omega^h$ in (\ref{eq:LMh3aa}), and noting 
(\ref{eq:hfd}).

For the general volume and area preserving flow, we introduce the
projection $\vec\Pi^h_0 : \Vht \to \Vhzerot$ onto
$\Vhzerot$, recall (\ref{eq:Vhzerot}), and similarly 
$\vec\Pi^h_{i,0} : \Vhti \to \Vhzeroti$. 
We introduce the symmetric 
bilinear forms $a^h_{i,\theta} : \Vhti \times \Vhti \to \R$ by setting
\begin{equation} \label{eq:ai}
a^h_{i,\theta}(\vec\zeta, \vec\eta) = 
\left\langle \mat Q^h_{i, \theta^h}\,\vec\zeta
, \vec\Pi^h_{i,0}\,\vec\eta \right\rangle_{\Gamma^h_i(t)}^h\qquad
\forall\ \vec\zeta,\,\vec\eta \in \Vhti\,,\ i=1,2\,,
\end{equation}
where we have noted (\ref{eq:aQb}). 
It holds that $a^h_{i,\theta}(\vec\zeta,\vec\zeta) \geq 0$ for all 
$\vec\zeta \in \Vhti$, with the inequality being strict if 
$\vec\Pi^h_{i,0}\,[\vec Q^h_{i,\theta^h}\,\vec\zeta]\not=\vec0$.
Hence the Cauchy--Schwarz inequality holds, i.e.\
\begin{equation} \label{eq:aiCS}
|a^h_{i,\theta}(\vec\zeta,\vec\eta)| 
\leq [a^h_{i,\theta}(\vec\zeta,\vec\zeta)]^\frac12 \,
[a^h_{i,\theta}(\vec\eta,\vec\eta)]^\frac12 \qquad 
\forall\ \vec\zeta,\,\vec\eta \in \Vhti\,,\ i=1,2\,,
\end{equation}
with strict inequality 
if $\vec\Pi^h_{i,0}\,[\vec Q^h_{i,\theta^h}\,\vec\zeta]$ 
and $\vec\Pi^h_{i,0}\,[\vec Q^h_{i,\theta^h}\,\vec\eta]$ 
are linearly independent.
Then we note, on recalling (\ref{eq:zo}), (\ref{eq:weakGFb}) and 
(\ref{eq:aQb}), that
\begin{subequations}
\begin{equation} \label{eq:aikappaomega}
- \left\langle \nabs\,\vec\id, \nabs\,\vec\Pi^h_{0}\,\vec\omega^h
\right\rangle_{\Gamma^h_i(t)} =
- \left\langle \nabs\,\vec\id, \nabs\,\vec\Pi^h_{i,0}\,\vec\omega^h_i
\right\rangle_{\Gamma^h_i(t)} = a^h_{i,\theta}(\vec\kappa^h_i, \vec\omega^h_i)
= \left\langle \vec\omega^h_i, \vec\Pi^h_{i,0}\,\vec\kappa^h_i 
\right\rangle_{\Gamma^h_i(t)}^h
\end{equation}
and
\begin{equation}
- \left\langle \nabs\,\vec\id, \nabs\,\vec\Pi^h_{i,0}\,\vec\kappa^h_i
\right\rangle_{\Gamma^h_i(t)} = 
a^h_{i,\theta}(\vec\kappa^h_i,\vec\kappa^h_i)\,. 
\label{eq:aikappakappa}
\end{equation}
\end{subequations}
In addition, it follows from (\ref{eq:zo}), (\ref{eq:aQb}) and
(\ref{eq:ai}) that
\begin{equation}
\left\langle \vec\omega^h , \vec\Pi^h_{0}\,\vec\omega^h
\right\rangle_{\Gamma^h(t)}^h
= \sum_{i=1}^2 \left\langle \vec\omega^h_i , \vec\Pi^h_{0}\,\vec\omega^h
\right\rangle_{\Gamma^h_i(t)}^h
= \sum_{i=1}^2 \left\langle \vec\omega^h_i , \vec\Pi^h_{i,0}\,\vec\omega^h_i
\right\rangle_{\Gamma^h_i(t)}^h
= \sum_{i=1}^2 a^h_{i,\theta}(\vec\omega^h_i,\vec\omega^h_i)\,.
\label{eq:aoo}
\end{equation}

Then (\ref{eq:LMh3aa}), (\ref{eq:aoo}) and (\ref{eq:aikappaomega},b) yield that
$(\lambda^{V,h}, \lambda^{A,h}_1, \lambda^{A,h}_2)(t)$ are such that
\begin{subequations}
\begin{equation}
\begin{pmatrix}
\sum_{i=1}^2 a^h_{i,\theta}(\vec\omega^h_i, \vec\omega^h_i) &
a^h_{1,\theta}(\vec\kappa^h_1, \vec\omega^h_1) &
a^h_{2,\theta}(\vec\kappa^h_2, \vec\omega^h_2) \\
a^h_{1,\theta}(\vec\kappa^h_1, \vec\omega^h_1) &
a^h_{1,\theta}(\vec\kappa^h_1, \vec\kappa^h_1) & 0 \\
a^h_{2,\theta}(\vec\kappa^h_2, \vec\omega^h_2) & 0 &
a^h_{2,\theta}(\vec\kappa^h_2, \vec\kappa^h_2) 
\end{pmatrix}
\begin{pmatrix}
- \lambda^{V,h}(t) \\ \lambda^{A,h}_1(t)\\ \lambda^{A,h}_2(t)
\end{pmatrix}
= \begin{pmatrix}
b_0(t) \\ b_1(t)\\ b_2(t)
\end{pmatrix} \,,
\label{eq:mll}
\end{equation}
where
\begin{align}
b_0(t) & =  \sum_{i=1}^2 
\left\langle \vec\Pi^h_0\,\vec{\mathcal{V}}^h - \vec{\mathcal{V}}^h, 
\vec\omega^h \right\rangle_{\Gamma^h_i(t)}^h
- \left\langle \vec r^h, \vec\Pi^h_0\,\vec\omega^h\right\rangle_{\Gamma^h(t)}^h
,\label{eq:bjb} \\ 
b_{i}(t) &
= \left\langle \vec\Pi^h_{i,0}\,\vec{\mathcal{V}}^h -\vec{\mathcal{V}}^h,
\mat Q^h_{i,\theta^h}\,\vec\kappa^h_i \right\rangle_{\Gamma^h_i(t)}^h 
+ \left\langle \vec{\rm m}^h_i, \vec{\mathcal{V}}^h 
\right\rangle_{\gamma^h(t)}^h
- \left\langle \vec r^h, \vec\Pi^h_{i,0}\,\vec\kappa^h_i 
\right\rangle_{\Gamma^h(t)}^h,\ i=1,2\,.
\label{eq:bjc}
\end{align}
\end{subequations}
On recalling (\ref{eq:aiCS}), we observe that the matrix in 
(\ref{eq:mll}) is symmetric and positive definite as long as 
$\vec\Pi^h_{i,0}\,\vec\omega^h_i$ and $\vec\Pi^h_{i,0}\,[
\vec Q^h_{i,\theta^h}\,\vec\kappa^h_i]$ are linearly
independent, for $i=1,2$.
The right hand sides (\ref{eq:bjb},c) 
are obtained by recalling (\ref{eq:LMh3aa}),
and on noting that (\ref{eq:aQb}) and (\ref{eq:hfd}) imply that
\begin{align}
\sum_{i=1}^2 
\left\langle \mat Q^h_{i,\theta^h_\star}\,\vec{\mathcal{V}}^h, 
\vec\Pi^h_0\,\vec\omega^h \right\rangle_{\Gamma^h_i(t)}^h & =
\sum_{i=1}^2 \left[
\left\langle \mat Q^h_{i,\theta^h_\star}\,\vec{\mathcal{V}}^h, 
\vec\Pi^h_0\,\vec\omega^h - \vec\omega^h \right\rangle_{\Gamma^h_i(t)}^h 
+ \left\langle \mat Q^h_{i,\theta^h_\star}\,\vec{\mathcal{V}}^h, 
 \vec\omega^h \right\rangle_{\Gamma^h_i(t)}^h \right] \nonumber \\ &
= \sum_{i=1}^2 
\left\langle \vec{\mathcal{V}}^h, 
\vec\Pi^h_0\,\vec\omega^h - \vec\omega^h \right\rangle_{\Gamma^h_i(t)}^h
= \sum_{i=1}^2 
\left\langle \vec\Pi^h_0\,\vec{\mathcal{V}}^h - \vec{\mathcal{V}}^h, 
\vec\omega^h \right\rangle_{\Gamma^h_i(t)}^h ,
\label{eq:newhfca}
\end{align}
while (\ref{eq:thetah}), (\ref{eq:aQb}), (\ref{eq:weakGFb}) 
and (\ref{eq:constraintareah}) yield that
\begin{align}
 \left\langle \mat Q^h_{i,\theta^h_\star}\,\vec{\mathcal{V}}^h,
\vec\Pi^h_{i,0}\,\vec\kappa^h_i \right\rangle_{\Gamma^h_i(t)}^h & =
\left\langle \vec\Pi^h_{i,0}\,\vec{\mathcal{V}}^h,
\mat Q^h_{i,\theta^h}\,\vec\kappa^h_i \right\rangle_{\Gamma^h_i(t)}^h 
\nonumber \\ & 
=
\left\langle \vec\Pi^h_{i,0}\,\vec{\mathcal{V}}^h -\vec{\mathcal{V}}^h,
\mat Q^h_{i,\theta^h}\,\vec\kappa^h_i \right\rangle_{\Gamma^h_i(t)}^h 
+ \left\langle \vec{\rm m}^h_i, \vec{\mathcal{V}}^h 
\right\rangle_{\gamma^h(t)}^h
- \left\langle \nabs\,\vec\id, \nabs\,\vec{\mathcal{V}}^h
\right\rangle_{\Gamma^h_i(t)} 
\nonumber \\ & 
=
\left\langle \vec\Pi^h_{i,0}\,\vec{\mathcal{V}}^h -\vec{\mathcal{V}}^h,
\mat Q^h_{i,\theta^h}\,\vec\kappa^h_i \right\rangle_{\Gamma^h_i(t)}^h 
+ \left\langle \vec{\rm m}^h_i, \vec{\mathcal{V}}^h 
\right\rangle_{\gamma^h(t)}^h .
\label{eq:newhfcb}
\end{align}
We see that on removing the last two rows and columns in (\ref{eq:mll}),
we obtain an expression similar to (\ref{eq:muh}) for $\lambda^{V,h}(t)$,
but here we test with $\vec\Pi^h_0\,\vec\omega^h$ as opposed to
$\vec\omega^h$.
Analogously, if we want to consider phase area preservations only, then 
removing the first row and column in (\ref{eq:mll}) 
yields a reduced system for the two
Lagrange multipliers $\lambda^{A,h}_i(t)$, $i=1,2$.

The following theorem establishes that (\ref{eq:weakGFa}--f) is indeed a weak
formulation for a generalized $L^2$--gradient flow of
$E^h(t)$ subject to the side constraints (\ref{eq:side3h}--c). 
We will also show that for $\theta=0$ the
scheme produces {\em conformal polyhedral surfaces} 
$\Gamma_1(t)$ and $\Gamma_2(t)$. 
Here we recall from \cite{pwfopen}, see also \cite[\S4.1]{gflows3d}, 
that the open surfaces $\Gamma^h_i(t)$, $i=1,2$, are conformal 
polyhedral surfaces if
\begin{equation} \label{eq:conf}
 \left\langle \nabs\,\vec\id, \nabs\,\vec\eta
\right\rangle_{\Gamma^h_i(t)} = 0
\quad \forall\ \vec\eta \in \left\{ \vec\xi \in \Vhzeroti :
\vec\xi(\vec{q}^h_{i,k}(t)) \,.\,\vec\omega^h_i(\vec{q}^h_{i,k}(t), t) = 0,\
k = 1,\ldots,K_i \right\},\ i=1,2\,.
\end{equation}
We recall from \cite{gflows3d,pwfopen} 
that conformal polyhedral surfaces exhibit good meshes. 
Moreover, we
recall that in the case $d=2$, conformal polyhedral surfaces are 
equidistributed polygonal curves, see \cite{triplej,fdfi}.

\begin{theorem} \label{thm:sd3stab}
Let $\theta\in[0,1]$, $\varrho\geq0$ and let 
$\{(\Gamma^h,\vec\kappa^h_1,\vec\kappa^h_2,\vec Y^h_1,\vec Y^h_2,
\vec\kappa^h_\gamma,\vec{\rm m}^h_1,\vec{\rm m}^h_2,\vec\Phi^h)(t)
\}_{t\in[0,T]}$ be a solution to {\rm (\ref{eq:weakGFa}--f)}. 
In addition, we assume that
$\vec\kappa^h_\gamma \in \underline{V}^h_T(\gamma^h_T)$,
$\vec\kappa^h_i 
,\vec\pi^h_i[\mat Q^h_{i,\theta^h}\,\vec\kappa^h_i]\in\underline{V}^h_T(\GhTi)$,
$\vec{\rm m}^h_i\in \underline{V}^h_T(\gamma^h_T)$, $i=1,2$.
Then 
\begin{equation}
\ddt \, E^h((\Gamma^h_i(t))_{i=1}^2)
= - \sum_{i=1}^2 \left\langle \mat Q^h_{i,\theta^h_\star}\,\vec{\mathcal{V}}^h,
\vec{\mathcal{V}}^h \right\rangle_{\Gamma^h_i(t)}^h 
- \varrho\left\langle\vec{\mathcal{V}}^h ,\vec{\mathcal{V}}^h
\right\rangle_{\gamma^h(t)}^h \black
.
\label{eq:sd3stab}
\end{equation}
Moreover, if $\theta=0$ then $\Gamma^h_1(t)$ and $\Gamma^h_2(t)$
are open conformal polyhedral surfaces for all $t\in(0,T]$.
\end{theorem}
\begin{proof}
Taking the time derivative of (\ref{eq:side3h}), where we choose
discrete test functions $\vec\eta$ such that
$\matpartxh\,\vec\eta = \vec0$, yields for $i=1,2$ that
\begin{align} \label{eq:s31h}
& \left\langle \matpartxh\,(\mat Q^h_{i,\theta^h}\,\vec\kappa^h_i), 
\vec\eta\right\rangle_{\Gamma^h_i(t)}^h
+ \left\langle [(\mat Q^h_{i,\theta^h}\,\vec\kappa^h_i)\,.\,\vec\eta]\,
 \nabs\,\vec\id, \nabs\,\vec{\mathcal{V}}^h \right\rangle_{\Gamma^h_i(t)}^h
+ \left\langle \nabs\,\vec{\mathcal{V}}^h, 
\nabs\,\vec\eta\right\rangle_{\Gamma^h_i(t)} 
\nonumber \\ & \
+ \left\langle \nabs\,.\, \vec{\mathcal{V}}^h, 
\nabs\,.\,\vec\eta\right\rangle_{\Gamma^h_i(t)} 
- \left\langle(\nabs\,\vec\eta)^T, 
 \mat D(\vec{\mathcal{V}}^h)\, (\nabs\,\vec\id)^T \right\rangle_{\Gamma^h_i(t)}
 = 
\left\langle\matpartxh\,\vec{\rm m}^h_i,\vec\eta \right\rangle_{\gamma^h(t)}^h
+ \left\langle\vec{\rm m}^h_i\,.\,\vec\eta,
\vec\id_s\,.\,\vec{\mathcal{V}}^h_s \right\rangle_{\gamma^h(t)}^h ,
\end{align}
where we have noted (\ref{eq:DElem5.6NI}),
(\ref{eq:DElem5.6NIpartial}), (\ref{eq:secvar2h}) and that
$\vec\pi^h_i[\mat Q^h_{i,\theta^h}\,\vec\kappa^h_i]\in\underline{V}^h_T(\GhTi)$,
$\vec{\rm m}^h_i\in \underline{V}^h_T(\gamma^h_T)$, $i=1,2$.
Similarly, taking the time derivative of (\ref{eq:sidezh})
with $\matpartxh\,\vec\chi = \vec0$ yields, on noting
(\ref{eq:DElem5.6NIpartial}), (\ref{eq:secvarpartial2h}) and
$\vec\kappa^h_\gamma \in \underline{V}^h_T(\gamma^h_T)$, that
\begin{equation} \label{eq:s31hpartial}
\left\langle \matpartxh\,\vec\kappa_\gamma^h, \vec\chi 
\right\rangle_{\gamma^h(t)}^h
+\left\langle \vec\kappa_\gamma^h\,.\,\vec\chi, 
\vec\id_s\,.\,\vec{\mathcal{V}}^h_s \right\rangle_{\gamma^h(t)}^h
+ \left\langle \mat{\mathcal{P}}^h_\gamma\,\vec\chi_s,
\vec{\mathcal{V}}^h_s \right\rangle_{\gamma^h(t)} = 0 \,.
\end{equation}
Choosing $\vec\chi = \vec{\mathcal{V}}^h$ in (\ref{eq:weakGFa}),
$\vec\eta = \vec Y^h_i$ in (\ref{eq:s31h}), $i=1,2$, and combining yields,
on noting the discrete variant of (\ref{eq:normvar}), that
\begin{align} \label{eq:s32h}
& \sum_{i=1}^2 \left\langle \mat Q^h_{i,\theta^h_\star}\,\vec{\mathcal{V}}^h, 
\vec{\mathcal{V}}^h \right\rangle_{\Gamma^h(t)}^h 
+ \varrho\left\langle\vec{\mathcal{V}}^h ,\vec{\mathcal{V}}^h
\right\rangle_{\gamma^h(t)}^h \black
\nonumber \\ &\quad
+ \sum_{i=1}^2 \left[
\tfrac12\left\langle [ \alpha_i\,|\vec\kappa^h_i - \spont_i\,\vec\nu^h_i|^2
- 2\,\vec Y^h_i\,.\,\mat Q^h_{i,\theta^h}\,\vec\kappa^h_i ]\,
\nabs\,\vec\id,\nabs\,\vec{\mathcal{V}}^h\right\rangle_{\Gamma^h_i(t)}^h
\right. \nonumber \\ &\quad \left.
 - \alpha_i\,\spont_i \left\langle \vec\kappa^h_i,\matpartxh\,\vec\nu^h_i
\right\rangle_{\Gamma^h_i(t)}^h 
+ \left\langle \matpartxh\,(\mat Q^h_{i,\theta^h}\,\vec\kappa^h_i), 
\vec Y^h_i\right\rangle_{\Gamma^h_i(t)}^h
+ \left\langle (\mat Q^h_{i,\theta^h}\,\vec\kappa^h_i\,.\,\vec Y^h_i)\,
 \nabs\,\vec\id, \nabs\,\vec{\mathcal{V}}^h \right\rangle_{\Gamma^h_i(t)}^h
\right. \nonumber \\ &\quad \left.
-\left\langle(1-\theta^h)\, (\vec G^h_i(\vec Y^h_i, \vec\kappa^h_i)\,.\,
\vec\nu^h_i)\,
\nabs\,\vec\id, \nabs\,\vec{\mathcal{V}}^h \right\rangle_{\Gamma^h_i(t)}^h
- \left\langle(1-\theta^h)\,\vec G^h_i(\vec Y^h_i,\vec\kappa^h_i), 
\matpartxh\,\vec\nu^h_i \right\rangle_{\Gamma^h_i(t)}^h 
\right]
\nonumber \\ & \quad
+ \varsigma \left\langle \vec\id_s, \vec{\mathcal{V}}^h_s 
\right\rangle_{\gamma^h(t)}
- \sum_{i=1}^2
 \alpha^G_i \left[\left\langle \vec\kappa_\gamma^h \,.\,\vec{\rm m}^h_i,
\vec\id_s\,.\,\vec{\mathcal{V}}^h_s \right\rangle_{\gamma^h(t)}^h
+ \left\langle\mat{\mathcal{P}}^h_\gamma\,(\vec{\rm m}^h_i)_s, 
\vec{\mathcal{V}}^h_s \right\rangle_{\gamma^h(t)} \right]
\nonumber \\ & \
= \sum_{i=1}^2
\left[
\left\langle \matpartxh\,\vec{\rm m}^h_i , \vec Y^h_i
\right\rangle_{\gamma^h(t)}
+ \left\langle \vec{\rm m}^h_i\,.\,\vec Y^h_i , 
\vec\id_s, \vec{\mathcal{V}}^h_s \right\rangle_{\gamma^h(t)}
\right] .
\end{align}
Choosing $\vec\chi = \sum_{i=1}^2 \alpha^G_i\,\vec{\rm m}^h_i$ in 
(\ref{eq:s31hpartial}) and recalling 
(\ref{eq:weakGFd},e) and (\ref{eq:DElem5.6NIpartial}), 
it follows from (\ref{eq:s32h}) that
\begin{align} \label{eq:s33h}
& \sum_{i=1}^2 \left\langle\mat Q^h_{i,\theta^h_\star}\, \vec{\mathcal{V}}^h ,
\vec{\mathcal{V}}^h \right\rangle_{\Gamma^h(t)}^h 
+ \varrho\left\langle\vec{\mathcal{V}}^h ,\vec{\mathcal{V}}^h
\right\rangle_{\gamma^h(t)}^h \black 
+ \sum_{i=1}^2 \left[
 \tfrac12\, \alpha_i\left\langle |\vec\kappa^h_i - \spont_i\,\vec\nu^h_i|^2\,
\nabs\,\vec\id,\nabs\,\vec{\mathcal{V}}^h\right\rangle_{\Gamma^h_i(t)}^h
\right. \nonumber \\ &\quad \left.
- \alpha_i\,\spont_i
\left\langle \vec\kappa^h_i, \matpartxh\,\vec\nu^h_i 
\right\rangle_{\Gamma^h_i(t)}^h 
+ \left\langle \matpartxh\,(\mat Q^h_{i,\theta^h}\,\vec\kappa^h_i),
 \vec Y^h_i \right\rangle_{\Gamma^h_i(t)}^h 
\right. \nonumber \\ &\quad \left.
-\left\langle(1-\theta^h) \,(\vec G^h_i(\vec Y^h_i, \vec\kappa^h_i)
\,.\,\vec\nu^h_i)\,\nabs\,\vec\id, \nabs\,\vec{\mathcal{V}}^h 
\right\rangle_{\Gamma^h_i(t)}^h 
-\left\langle(1-\theta^h)\,  \vec G^h_i(\vec Y^h_i, \vec\kappa^h_i), 
\matpartxh\,\vec\nu^h_i \right\rangle_{\Gamma^h_i(t)}^h \right]
\nonumber \\ &\quad
+ \varsigma \left\langle \vec\id_s, \vec{\mathcal{V}}^h_s 
\right\rangle_{\partial\Gamma^h(t)}
=-\sum_{i=1}^2 \alpha^G_i \left[
\left\langle \matpartxh\,\vec{\rm m}^h_i, \vec\kappa^h_\gamma
\right\rangle_{\gamma^h(t)}^h
+ \left\langle \vec{\rm m}^h_i,\matpartxh\,\vec\kappa^h_\gamma
\right\rangle_{\gamma^h(t)}^h
+ \left\langle \vec\kappa_\gamma^h \,.\,\vec{\rm m}^h_i,
\vec\id_s\,.\,\vec{\mathcal{V}}^h_s \right\rangle_{\gamma^h(t)}^h \right]
\nonumber \\ & \
= - \ddt \sum_{i=1}^2 \alpha^G_i 
\left\langle \vec\kappa_\gamma^h, \vec{\rm m}^h_i
\right\rangle_{\gamma^h(t)}^h .
\end{align}
We have from (\ref{eq:aQb}), (\ref{eq:weakGFf}) and (\ref{eq:NIhi}) that
\begin{align}
& \sum_{i=1}^2 \left[
\left\langle \matpartxh\,(\mat Q^h_{i,\theta^h}\,\vec\kappa^h_i),
 \vec Y^h_i \right\rangle_{\Gamma^h_i(t)}^h - \alpha_i\,\spont_i
\left\langle\vec\kappa^h_i, \matpartxh\,\vec\nu^h_i 
\right\rangle_{\Gamma^h_i(t)}^h \right]
\nonumber \\ & \quad
= \sum_{i=1}^2 \left[
\left\langle \matpartxh\,\vec\kappa^h_i,
\mat Q^h_{i,\theta^h}\, \vec Y^h_i \right\rangle_{\Gamma^h_i(t)}^h
- \alpha_i\,\spont_i \left\langle\vec\kappa^h_i - \spont_i\,\vec\nu^h_i, 
\matpartxh\,\vec\nu^h_i \right\rangle_{\Gamma^h_i(t)}^h
\right. \nonumber \\ &\qquad\qquad \left.
+ \left\langle \matpartxh\,(\mat Q^h_{i,\theta^h}\,\vec\kappa^h_i)
- \mat Q^h_{i,\theta^h}\,\matpartxh\,\vec\kappa^h_i ,
 \vec Y^h_i \right\rangle_{\Gamma^h_i(t)}^h  \right]
\nonumber \\ & \quad
= \sum_{i=1}^2 \left[
\tfrac12\,\alpha_i \left\langle \matpartxh\,
|\vec\kappa^h_i - \spont_i\,\vec\nu^h_i|^2, 1
\right\rangle_{\Gamma^h_i(t)}^h
+ \left\langle \matpartxh\,(\mat Q^h_{i,\theta^h}\,\vec\kappa^h_i)
- \mat Q^h_{i,\theta^h}\,\matpartxh\,\vec\kappa^h_i ,
 \vec Y^h_i \right\rangle_{\Gamma^h_i(t)}^h \right].
\label{eq:dagdag}
\end{align}
Combining (\ref{eq:s33h}) and (\ref{eq:dagdag}),
on noting (\ref{eq:DElem5.6NI}), (\ref{eq:DElem5.6NIpartial}), (\ref{eq:Eh}),
$\matpartxh\,\theta^h = 0$ (which follows from (\ref{eq:p1}) and
(\ref{eq:thetah})),
$\vec\kappa^h_i\in\underline{V}^h_T(\GhTi)$, $i=1,2$,
$\vec\nu^h_i\!\mid_{\sigma^h_{i,j}(\cdot)} 
\in\underline{V}^h_T(\sigma^h_{i,j,T})$, $j=1,\ldots,J_i$, $i=1,2$,
(which follows from the discrete analogue of (\ref{eq:normvar}) and
as $\vec\id\!\mid_{\Gamma^h(\cdot)} \in \underline{V}^h_T(\GhT)$)
and the invariance of $m(\Gamma^h_i(t))$ under continuous
deformations, yields that
\begin{equation*}
\sum_{i=1}^2 \left\langle\mat Q^h_{i,\theta^h_\star}\, \vec{\mathcal{V}}^h ,
\vec{\mathcal{V}}^h \right\rangle_{\Gamma^h_i(t)}^h 
+ \varrho\left\langle\vec{\mathcal{V}}^h ,\vec{\mathcal{V}}^h
\right\rangle_{\gamma^h(t)}^h \black 
+ \ddt \, E^h((\Gamma_i^h(t))_{i=1}^2) + \sum_{i=1}^2 P_i = 0\,,
\end{equation*}
where, on noting (\ref{eq:Qalphah}), 
\begin{align}
& P_i := 
\left\langle (1-\theta^h)\,\vec\kappa^h_i\,.\,\matpartxh\,\vec\omega^h_i, 
\frac{\vec Y^h_i\,.\,\vec\omega^h_i}{|\vec\omega^h_i|^2} 
\right\rangle_{\Gamma^h_i(t)}^h
+ \left\langle (1-\theta^h)\,\vec Y^h_i\,.\,\matpartxh\,\vec\omega^h_i, 
\frac{\vec\kappa^h_i\,.\,\vec\omega^h_i}{|\vec\omega^h_i|^2} 
\right\rangle_{\Gamma^h_i(t)}^h
\nonumber \\ &\quad
- 2 \left\langle (1-\theta^h)\,(\vec\kappa^h_i\,.\,\vec\omega^h_i)
\,(\vec Y^h_i\,.\,\vec\omega^h_i)
, \frac{\vec\omega^h_i\,.\,\matpartxh\,\vec\omega^h_i}{|\vec\omega^h_i|^4}
\right\rangle_{\Gamma^h_i(t)}^h
-\left\langle (1-\theta^h)\,(\vec G^h_i(\vec Y^h_i, \vec\kappa^h_i)
\,.\,\vec\nu^h_i)\,\nabs\,\vec\id, \nabs\,\vec{\mathcal{V}}^h 
\right\rangle_{\Gamma^h_i(t)}^h
\nonumber \\ &\quad
- \left\langle (1-\theta^h)\,\vec G^h_i(\vec Y^h_i, \vec\kappa^h_i), 
\matpartxh\,\vec\nu^h_i \right\rangle_{\Gamma^h_i(t)}^h
,\ i=1,2\,.
\label{eq:s38h}
\end{align}
It remains to show that $P_1$ and $P_2$ as defined in (\ref{eq:s38h}) vanish. 
To see this,
we observe that it follows from (\ref{eq:Gomega}), (\ref{eq:Gxieta})
and the time derivative version of (\ref{eq:ddomegah}) that
\begin{align*}
P_i & = 
\left\langle(1-\theta^h)\, \vec G^h_i(\vec Y^h_i, \vec\kappa^h_i) , 
\matpartxh\,\vec\omega^h_i\right\rangle_{\Gamma^h_i(t)}^h
+\left\langle(1-\theta^h)\, \vec G^h_i(\vec Y^h_i, \vec\kappa^h_i)
\,.\,(\vec\omega^h_i-\vec\nu^h_i)\,\nabs\,\vec\id, \nabs\,\vec{\mathcal{V}}^h 
\right\rangle_{\Gamma^h_i(t)}^h 
\nonumber \\ &\qquad
-\left\langle(1-\theta^h)\, \vec G^h_i(\vec Y^h_i, \vec\kappa^h_i) , 
\matpartxh\,\vec\nu^h_i \right\rangle_{\Gamma^h_i(t)}^h = 0\,,\ i=1,2\,.
\end{align*}
This proves the desired result (\ref{eq:sd3stab}).

Finally, 
if $\theta=0$ then it immediately follows from (\ref{eq:weakGFb}) that 
(\ref{eq:conf}) holds. Hence $\Gamma^h_1(t)$ and $\Gamma^h_2(t)$ are
open conformal polyhedral surfaces.
\qquad\end{proof}

\begin{theorem} \label{thm:sd3HFstab}
Let $\theta\in[0,1]$, $\varrho\geq0$ and let 
$\{(\Gamma^h, \vec\kappa^h_1,\vec\kappa^h_2,\vec Y^h_1,\vec Y^h_2,
\vec\kappa^h_\gamma,\vec{\rm m}^h_1,\vec{\rm m}^h_2,\vec\Phi^h,
\lambda^{V,h},\lambda_1^{A,h},\lambda_2^{A,h})(t)$ $\}_{t \in [0,T]}$ 
be a solution to
{\rm (\ref{eq:LMh3aa}}), {\rm (\ref{eq:weakGFb}--f)} and 
{\rm (\ref{eq:mll})}. 
In addition, we assume that
$\vec\kappa^h_\gamma \in \underline{V}^h_T(\gamma^h_T)$,
$\vec\kappa^h_i 
,\vec\pi^h_i[\mat Q^h_{i,\theta^h}\,\vec\kappa^h_i]\in\underline{V}^h_T(\GhTi)$,
$\vec{\rm m}^h_i\in \underline{V}^h_T(\gamma^h_T)$, $i=1,2$.
Then it holds that
\begin{equation}
\ddt \, E^h((\Gamma^h_i(t))_{i=1}^2)
= - \sum_{i=1}^2 \left\langle \mat Q^h_{i,\theta^h_\star}\,\vec{\mathcal{V}}^h,
\vec{\mathcal{V}}^h \right\rangle_{\Gamma^h_i(t)}^h 
- \varrho\left\langle\vec{\mathcal{V}}^h ,\vec{\mathcal{V}}^h
\right\rangle_{\gamma^h(t)}^h \black
,
\label{eq:sd3HFstab}
\end{equation}
as well as
\begin{equation} \label{eq:dtthm}
\ddt\, \mathcal{H}^{d-1}(\Gamma^h_i(t)) = 0\,,\ i=1,2\,,\qquad
\ddt\, \mathcal{L}^d(\Omega^h(t)) = 0 \,,
\end{equation}
where $\Omega^h(t)$ denotes the region bounded by $\Gamma^h(t)$.
Moreover, if $\theta=0$ then $\Gamma^h_1(t)$ and $\Gamma^h_2(t)$ are open 
conformal polyhedral surfaces for all $t\in(0,T]$.
\end{theorem}
\begin{proof}
We recall that on choosing $(\lambda^{V,h}, \lambda^{A,h}_1, \lambda^{A,h}_2)$
solving the system (\ref{eq:mll}) yields that (\ref{eq:constraintareah}) 
and (\ref{eq:constraintvolh}) hold, and hence the desired results 
(\ref{eq:dtthm}) hold.
The stability result (\ref{eq:sd3HFstab}) directly follows from the proof of
Theorem~\ref{thm:sd3stab}. In
particular, choosing $\vec\chi = \vec{\mathcal{V}}^h$ in (\ref{eq:LMh3aa}),
on noting (\ref{eq:constraintareah}) and (\ref{eq:constraintvolh}), 
yields that
\begin{equation*} 
\left\langle \mat Q^h_{i,\theta^h_\star}\,
\vec{\mathcal{V}}^h, \vec{\mathcal{V}}^h \right\rangle_{\Gamma^h_i(t)}^h
+ \varrho\left\langle\vec{\mathcal{V}}^h ,\vec{\mathcal{V}}^h
\right\rangle_{\gamma^h(t)}^h \black
= \left\langle \vec r^h, \vec{\mathcal{V}}^h
 \right\rangle_{\Gamma^h(t)}^h .
\end{equation*}
Combining this with (\ref{eq:s31h}) yields that (\ref{eq:s32h}) holds, 
and the rest of the proof proceeds as that of Theorem~\ref{thm:sd3stab}.
Finally, as in the proof of Theorem~\ref{thm:sd3stab},
for $\theta=0$ it follows from (\ref{eq:weakGFb}) that $\Gamma^h_1(t)$
and $\Gamma^h_2(t)$ are conformal polyhedral surfaces.
\end{proof}

\setcounter{equation}{0}
\section{Fully discrete finite element approximation} \label{sec:5}
In this section we consider a fully discrete variant of the scheme
{\rm (\ref{eq:weakGFa}--f)} from Section~\ref{sec:3}. To this end, let
$0= t_0 < t_1 < \ldots < t_{M-1} < t_M = T$ be a
partitioning of $[0,T]$ into possibly variable time steps $\ttau_m := t_{m+1} -
t_{m}$, $m=0,\ldots, M-1$. 
Let $\Gamma^{m}$ be a $(d-1)$-dimensional polyhedral surface,
approximating $\Gamma^h(t_m)$, $m=0 ,\ldots, M$, with 
the two parts $\Gamma^m_i$, $i=1,2$ and their common boundary 
$\gamma^m$.
Following \cite{Dziuk91}, we now
parameterize the new surface $\Gamma^{m+1}$ over $\Gamma^m$. 
Hence, we introduce the following
finite element spaces. Let $\Gamma^m=\bigcup_{j=1}^J \overline{\sigma^m_j}$,
where $\{\sigma^m_j\}_{j=1}^J$ is a family of mutually disjoint open triangles
with vertices $\{\vec{q}^m_k\}_{k=1}^K$.
Then for $m =0 ,\ldots, M-1$, let
\begin{equation} \label{eq:Vh}
\Vh := \{\vec\chi \in [C(\Gamma^m)]^d:\vec\chi\!\mid_{\sigma^m_j}
\mbox{ is linear}\ \forall\ j=1,\ldots, J\} 
=: [\Wh]^d \,.
\end{equation}
We denote the standard basis of $\Wh$ by $\{\chi^m_k\}_{k=1}^K$. 
In addition, similarly to the semidiscrete setting in Section~\ref{sec:3},
we introduce the spaces $\Whi$ and $\Vhi$,
denoting the standard basis of $\Whi$ by $\{\chi^m_{i,k}\}_{k=1}^{K_i}$,
as well as $\Vhpartial$, 
and the interpolation operators
$\pi^m_i: C(\Gamma^m_i)\to \Whi$ and similarly 
$\vec\pi^m_i: [C(\Gamma^m_i)]^d\to \Vhi$.

We also introduce the $L^2$--inner 
products $\langle\cdot,\cdot\rangle_{\Gamma^m}$,
$\langle\cdot,\cdot\rangle_{\Gamma^m_i}$ and
$\langle\cdot,\cdot\rangle_{\gamma^m}$, as well as their 
mass lumped inner variants
$\langle\cdot,\cdot\rangle_{\Gamma^m}^h$,
$\langle\cdot,\cdot\rangle_{\Gamma^m_i}^h$ and
$\langle\cdot,\cdot\rangle_{\gamma^m}^h$.
Similarly to (\ref{eq:omegahi}) and (\ref{eq:omegah}) 
we introduce the discrete vertex normals
$\vec\omega^m_i := \sum_{k=1}^{K_i} \chi^m_{i,k}\,\vec\omega^m_{i,k} \in \Vhi$
and $\vec\omega^m := \sum_{k=1}^{K} \chi^m_k\,\vec\omega^m_k \in \Vh$.

We make the following mild assumption.

\begin{itemize}
\item[$(\mathcal{A})$]
We assume for $m=0,\ldots, M-1$ that $\mathcal{H}^{d-1}(\sigma^m_j) > 0$ 
for $j=1,\ldots, J$, and that
$\vec 0 \not \in \{ \vec\omega^m_{i,k} : k = 1,\ldots, K_i\,,i=1,2 \}$. 
Moreover, in the case $C_1 = 1$ and $\theta = 0$ we assume that
$\dim \spa\{ \vec\omega^m_{i,k} : k = 1,\ldots, K_i\,,i=1,2 \} = d$, 
for $m=0,\ldots, M-1$.
\end{itemize}

In addition, and similarly to (\ref{eq:thetah}) and (\ref{eq:Qalphah}),
we introduce $\theta^m$ and $\theta^m_\star \in \Wh$, and then 
$\mat Q^m_{i,\theta^m},\,\mat Q^m_{i,\theta^m_\star} \in [\Whi]^{d \times d}$, 
by setting
$\mat Q^m_{i,\theta^m}(\vec q^m_{i,k}) = 
\theta^m(\vec q^m_{i,k})\,\vec\Id + (1-\theta^m(\vec q^m_{i,k}))
\,|\vec\omega^m_{i,k}|^{-2}\,\vec\omega^m_{i,k} \otimes \vec\omega^m_{i,k}$ 
and
$\mat Q^m_{i,\theta^m_\star}(\vec q^m_{i,k}) = 
\theta^m_\star(\vec q^m_{i,k})\,\vec\Id + (1-\theta^m_\star(\vec q^m_{i,k}))
\,|\vec\omega^m_{i,k}|^{-2}\,\vec\omega^m_{i,k} \otimes \vec\omega^m_{i,k}$ 
for $k= 1,\ldots, K_i$, $i=1,2$.
Similarly to (\ref{eq:Gxieta}) and (\ref{eq:Ph}), we let
\[
\vec G^m_i(\vec\xi,\vec\eta) = \frac{1}{|\vec\omega^m_i|^2}
 \left( (\vec\xi\,.\,\vec\omega^m_i)\,\vec\eta +
 (\vec\eta\,.\,\vec\omega^m_i)\,\vec\xi
- 2\, \frac{(\vec\eta\,.\,\vec\omega^m_i)\,(\vec\xi\,.\,\vec\omega^m_i)}
{|\vec\omega^m_i|^2}\,\vec\omega^m_i
\right) 
\]
and
\begin{equation*} 
\mat{\mathcal{P}}^m_\gamma = \mat\Id - \vec\id_s \otimes \vec\id_s
\quad \text{on} \quad \gamma^m\,.
\end{equation*}

On recalling (\ref{eq:weakGFbnew}), we consider the following fully 
discrete approximation of \mbox{(\ref{eq:weakGFa}--f)}.
For $m=0,\ldots,M-1$, find $\vec X^{m+1} \in \Vh$, 
$(\vec Y^{m+1}_i,\vec{\rm m}^{m+1}_i)_{i=1}^2 
\in \Vhone \times \Vhpartial \times \Vhtwo \times \Vhpartial$,
$\vec\kappa_\gamma^{m+1}\in \Vhpartial$
and $C_1\,\Phi^{m+1} \in \Vhpartial$ such that
\begin{subequations}
\begin{align}
& 
\sum_{i=1}^2 \left[
\left\langle \mat Q^m_{i,\theta^m_\star}\,
\frac{\vec X^{m+1} - \vec\id}{\ttau_m} ,
\vec\chi \right\rangle_{\Gamma^m_i}^h
-  \left\langle \nabs\,\vec Y^{m+1}_i, \nabs\,\vec\chi 
\right\rangle_{\Gamma^m_i} 
+ \alpha^G_i \left\langle (\vec{\rm m}^{m+1}_i)_s, 
\vec\chi_s \right\rangle_{\gamma^m}\right]
\nonumber \\ & \qquad \qquad
+ \varsigma \left\langle \vec X^{m+1}_s, \vec\chi_s \right\rangle_{\gamma^m}
+ \varrho\left\langle \frac{\vec X^{m+1} - \vec\id}{\ttau_m},\vec\chi 
\right\rangle_{\gamma^m}^h 
\black
\nonumber \\ & \quad
= 
\sum_{i=1}^2 \left[
\left\langle \nabs\,.\,\vec Y^m_i, \nabs\,.\,\vec\chi 
\right\rangle_{\Gamma^m_i} 
- \left\langle (\nabs\,\vec Y^m_i)^T , 
\mat D(\vec\chi)\,(\nabs\,\vec\id)^T \right\rangle_{\Gamma^m_i} 
\right. \nonumber \\ & \qquad \left.
-\tfrac12\left\langle [ \alpha_i\,|\vec\kappa^m_i - \spont_i\,\vec\nu^m_i|^2
- 2\,(\vec Y^m_i\,.\,\mat Q^m_{i,\theta^m}\,\vec\kappa^m_i)]
\,\nabs\,\vec\id,\nabs\,\vec\chi \right\rangle_{\Gamma^m_i}^h
-\alpha_i\,\spont_i 
\left\langle \vec\kappa^m_i, [\nabs\,\vec\chi]^T\,\vec\nu^m_i
\right\rangle_{\Gamma^m_i}^h 
\right. \nonumber \\ & \qquad \left.
+\left\langle(1-\theta^m)\, (\vec G^m_i(\vec Y^m_i, \vec\kappa^m_i)\,.\,
\vec\nu^m_i)\,
\nabs\,\vec\id, \nabs\,\vec\chi \right\rangle_{\Gamma^m_i}^h
- \left\langle(1-\theta^m)\,\vec G^m_i(\vec Y^m_i,\vec\kappa^m_i), 
[\nabs\,\vec\chi]^T\,\vec\nu^m_i \right\rangle_{\Gamma^m_i}^h 
 \right]
\nonumber \\ & \qquad 
+ \sum_{i=1}^2 \alpha^G_i 
\left[\left\langle \vec\kappa_\gamma^m \,.\,\vec{\rm m}^m_i,
\vec\id_s\,.\,\vec\chi_s \right\rangle_{\gamma^m}^h
+ \left\langle(\mat\Id + \mat{\mathcal{P}}^m_\gamma)\,(\vec{\rm m}^m_i)_s, 
\vec\chi_s \right\rangle_{\gamma^m} \right]
\nonumber \\ & \qquad
- \lambda^{V,m} 
\left\langle \vec\omega^m, \vec\chi \right\rangle_{\Gamma^m}^h
- \sum_{i=1}^2 \lambda^{A,m}_i 
\left\langle \nabs\,\vec\id, \nabs\,\vec\chi \right\rangle_{\Gamma^m_i}
\quad \forall\ \vec\chi \in \Vh \,,\label{eq:GF2a} \\
& \alpha_i^{-1}
\left\langle \mat Q^m_{i,\theta^m}\,\vec Y^{m+1}_i, \mat Q^m_{i,\theta^m}\,
\vec\eta \right\rangle_{\Gamma^m_i}^h +
\left\langle \nabs\,\vec X^{m+1} , \nabs\,\vec\eta \right\rangle_{\Gamma^m_i} = 
\left\langle \vec{\rm m}^{m+1}_i, \vec\eta \right\rangle_{\gamma^m}^h
- \spont_i \left\langle \vec\omega^m_i, \vec\eta \right\rangle_{\Gamma^m_i}^h
\nonumber \\ & \hspace{9cm}
 \quad\forall\ \vec\eta \in \Vhi\,,\ i=1,2\,, \label{eq:GF2b} \\
& \left\langle \vec\kappa_\gamma^{m+1}, \vec\chi 
\right\rangle_{\gamma^m}^h
+ \left\langle \vec X^{m+1}_s, \vec\chi_s \right\rangle_{\gamma^m} = 0
\qquad \forall\ \vec\chi \in \Vhpartial\,, \label{eq:GF2c} \\
& C_1\,(\vec{\rm m}_1^{m+1} + \vec{\rm m}_2^{m+1}) = 
\vec 0 \qquad \text{on}\quad \gamma^m\,, \label{eq:GF2d} \\
& \alpha^G_i \,\vec\kappa_\gamma^{m+1}
 + \vec Y^{m+1}_i + C_1\,\vec\Phi^{m+1} = \vec0 
\qquad \text{on}\quad \gamma^m \,, \ i = 1,2\,, \label{eq:GF2e}
\end{align}
\end{subequations}
and set $\vec\kappa^{m+1}_i = 
\alpha^{-1}_i\,\vec\pi^m_i\,[\mat Q^m_{i,\theta^m}\,\vec Y^{m+1}_i]
+ \spont_i\,\vec\omega^m_i$ and $\Gamma^{m+1}_i = \vec X^{m+1}(\Gamma^m_i)$,
$i=1,2$.
For $m \geq 1$ we note that here and throughout, as no confusion 
can arise, we denote by $\vec\kappa^m_i$ the function $\vec z \in \Vhi$, 
defined by $\vec z(\vec q^m_{i,k}) = \vec\kappa^m_i(\vec q^{m-1}_{i,k})$, 
$k=1\to K_i$, 
where $\vec\kappa^m_i \in \underline{V}(\Gamma^{m-1}_i)$ is given, 
and similarly for e.g.\ $\vec Y^m_i$, $\vec{\rm m}^m_i$ and 
$\vec\kappa_{\gamma}^m$.

We note that if $C_1= \alpha^G_1 = \alpha^G_2 = 0$ then the weak
conormals $\vec {\rm m}^{m+1}_i$ play no role in the evolution. However, for
surface area conservation they do play a role also in that case, 
see (\ref{eq:bjm2}) below. 
We also remark that the parameter $\varrho\geq0$ has a
stabilizing effect on the evolution of $\gamma^m$. In practice, this was
particularly useful for simulations involving surface area preservation,
and for $C^0$ experiments with Gaussian curvature energy contributions.

Of course, (\ref{eq:GF2a}--e) with $\lambda^{V,m} = \lambda^{A,m}_1 =
\lambda^{A,m}_2 = 0$ corresponds to
a fully discrete approximation of (\ref{eq:weakGFa}--f), 
on recalling (\ref{eq:weakGFbnew}). 
For a fully discrete approximation of the volume and/or surface area 
preserving flow, on recalling (\ref{eq:mll}--c), we let 
$(\lambda^{V,m}, \lambda^{A,m}_1, \lambda^{A,m}_2)$ be the solution of
\begin{subequations}
\begin{equation} 
\begin{pmatrix}
\sum_{i=1}^2 a^m_{i,\theta}(\vec\omega^m_i, \vec\omega^m_i) &
a^m_{1,\theta}(\vec\kappa^m_1, \vec\omega^m_1) &
a^m_{2,\theta}(\vec\kappa^m_2, \vec\omega^m_2) \\
a^m_{1,\theta}(\vec\kappa^m_1, \vec\omega^m_1) &
a^m_{1,\theta}(\vec\kappa^m_1, \vec\kappa^m_1) & 0 \\
a^m_{2,\theta}(\vec\kappa^m_2, \vec\omega^m_2) & 0 &
a^m_{2,\theta}(\vec\kappa^m_2, \vec\kappa^m_2) 
\end{pmatrix}
\begin{pmatrix}
- \lambda^{V,m} \\ \lambda^{A,m}_1\\ \lambda^{A,m}_2
\end{pmatrix}
= \begin{pmatrix}
b_0^m \\ b_1^m\\ b_2^m
\end{pmatrix} \,,
\label{eq:mllm}
\end{equation}
where, on noting the fully discrete variant of (\ref{eq:zo}), 
\begin{align}
b_0^m & =   \sum_{i=1}^2 \left[
\left\langle (\vec\Pi^m_{i,0} - \mat\Id) 
\frac{\vec\id - \vec X^{m-1}}{\ttau_{m-1}},
\vec\omega^m \right\rangle_{\Gamma^m_i}^h
- \left\langle \nabs\,\vec Y^m_i, \nabs\, (\vec\Pi^m_0\,\vec\omega^m)
\right\rangle_{\Gamma^m_i} \right] 
- \left\langle \vec f^m, \vec\Pi^m_0\,\vec\omega^m\right\rangle_{\Gamma^m}^h
 \nonumber \\ & =
\sum_{i=1}^2 \left[ 
\left\langle (\vec\Pi^m_{i,0} - \mat\Id) 
\frac{\vec\id - \vec X^{m-1}}{\ttau_{m-1}},
\vec\omega^m_i \right\rangle_{\Gamma^m_i}^h
- \left\langle \nabs\,\vec Y^m_i, \nabs\, (\vec\Pi^m_{i,0}\,\vec\omega^m_i)
\right\rangle_{\Gamma^m_i} 
- \left\langle \vec f^m, \vec\Pi^m_{i,0}\,\vec\omega^m_i
\right\rangle_{\Gamma^m_i}^h \right]
, \label{eq:bjm1} \\ 
b_{i}^m & =
\left\langle (\vec\Pi^m_{i,0} - \mat\Id)\,
\frac{\vec\id - \vec X^{m-1}}{\ttau_{m-1}},
\mat Q^m_{i,\theta^m}\,\vec\kappa^m_i \right\rangle_{\Gamma^m_i}^h 
+ \left\langle \vec{\rm m}^m_i, \frac{\vec\id - \vec X^{m-1}}{\ttau_{m-1}}
\right\rangle_{\gamma^m}^h \nonumber \\ & \qquad
- \left\langle \nabs\,\vec Y^m_i, \nabs\,(\vec\Pi^m_{i,0}\,\vec\kappa^m_i)
\right\rangle_{\Gamma^m_i} 
- \left\langle \vec f^m, \vec\Pi^m_{i,0}\,\vec\kappa^m_i 
\right\rangle_{\Gamma^m}^h \nonumber \\ & =
\left\langle (\vec\Pi^m_{i,0} - \mat\Id)\,
\frac{\vec\id - \vec X^{m-1}}{\ttau_{m-1}},
\mat Q^m_{i,\theta^m}\,\vec\kappa^m_i \right\rangle_{\Gamma^m_i}^h 
+ \left\langle \vec{\rm m}^m_i, \frac{\vec\id - \vec X^{m-1}}{\ttau_{m-1}}
\right\rangle_{\gamma^m}^h \nonumber \\ & \qquad 
- \left\langle \nabs\,\vec Y^m_i, \nabs\,(\vec\Pi^m_{i,0}\,\vec\kappa^m_i)
\right\rangle_{\Gamma^m_i}
- \left\langle \vec f^m, \vec\Pi^m_{i,0}\,\vec\kappa^m_i 
\right\rangle_{\Gamma^m_i}^h,\ i=1,2\,.
\label{eq:bjm2}
\end{align}
\end{subequations}
Here, for convenience, we have re-written (\ref{eq:GF2a}) as
\begin{align} \label{eq:GF2aa}
&\sum_{i=1}^2 \left[
\left\langle \mat Q^m_{i,\theta^m_\star}\,
\frac{\vec X^{m+1} - \vec\id}{\ttau_m} ,
\vec\chi \right\rangle_{\Gamma^m_i}^h
-  \left\langle \nabs\,\vec Y^{m+1}_i, \nabs\,\vec\chi 
\right\rangle_{\Gamma^m_i}
+ \alpha^G_i \left\langle (\vec{\rm m}^{m+1}_i)_s, 
\vec\chi_s \right\rangle_{\gamma^m} \right]
\nonumber \\ & \qquad\qquad
+ \varsigma \left\langle \vec X^{m+1}_s, \vec\chi_s \right\rangle_{\gamma^m}
+ \varrho\left\langle \frac{\vec X^{m+1} - \vec\id}{\ttau_m},\vec\chi 
\right\rangle_{\gamma^m}^h \black
\nonumber \\ & \quad
= \left\langle \vec f^m , \vec\chi \right\rangle_{\Gamma^m}^h 
- \lambda^{V,m} 
\left\langle \vec\omega^m, \vec\chi \right\rangle_{\Gamma^m}^h
- \sum_{i=1}^2 \lambda^{A,m}_i 
\left\langle \nabs\,\vec\id, \nabs\,\vec\chi \right\rangle_{\Gamma^m_i}
\quad \forall\ \vec\chi \in \Vh \,,
\end{align}
and, analogously to (\ref{eq:ai}), we have defined 
$a^m_{i,\theta} : \Vhi \times \Vhi \to \R$ by setting
\begin{equation*} 
a^m_{i,\theta}(\vec\zeta, \vec\eta) = 
\left\langle \mat Q^m_{i, \theta^m}\,\vec\zeta
, \vec\Pi^m_{i,0}\,\vec\eta \right\rangle_{\Gamma^m_i}^h\qquad
\forall\ \vec\zeta,\,\vec\eta \in \Vhi\,,\ i=1,2\,.
\end{equation*}
As before, we note that the matrix in 
(\ref{eq:mllm}) is symmetric positive definite as long as 
$\vec\Pi^m_{i,0}\,\vec\omega^m_i$ and $\vec\Pi^m_{i,0}\,[
\vec Q^m_{i,\theta^m}\,\vec\kappa^m_i]$ are linearly
independent, for $i=1,2$.

\begin{theorem} \label{thm:ex2}
Let $\theta\in[0,1]$, $\varrho\geq0$ and $\alpha_1,\alpha_2>0$.
Let the assumptions $(\mathcal{A})$ hold. 
Then there exists a unique solution
$\vec X^{m+1} \in \Vh$, 
$(\vec Y^{m+1}_i,\vec{\rm m}^{m+1}_i)_{i=1}^2 
\in \Vhone \times \Vhpartial \times \Vhtwo\times \Vhpartial$,
$\vec\kappa_\gamma^{m+1}\in \Vhpartial$ and $C_1\,\Phi^{m+1} \in \Vhpartial$
to \mbox{\rm (\ref{eq:GF2a}--e)}.
\end{theorem}
\begin{proof}
As (\ref{eq:GF2a}--e) is linear, existence follows from uniqueness.
To investigate the latter, we consider the system:
Find $\vec X \in \Vh$, $(\vec Y_i,\vec{\rm m}_i)_{i=1}^2 
\in \Vhone \times \Vhpartial \times \Vhtwo\times\Vhpartial$,
$\vec\kappa_\gamma\in \Vhpartial$ and $C_1\,\Phi \in \Vhpartial$
such that
\begin{subequations}
\begin{align}
& \sum_{i=1}^2 \left[ \frac1{\ttau_m}
\left\langle \mat Q^m_{i,\theta^m_\star}\,\vec X, 
\vec\chi \right\rangle_{\Gamma^m_i}^h
-  \left\langle \nabs\,\vec Y_i, \nabs\,\vec\chi 
\right\rangle_{\Gamma^m_i} 
+ \alpha^G_i \left\langle (\vec{\rm m}_i)_s, 
\vec\chi_s \right\rangle_{\gamma^m}\right]
+ \varsigma \left\langle \vec X_s, \vec\chi_s \right\rangle_{\gamma^m}
+ \varrho \left\langle \vec X, \vec\chi \right\rangle_{\gamma^m}^h\black 
 = 0 
\nonumber \\ & \hspace{11cm}
\quad \forall\ \vec\chi \in \Vh\,, \label{eq:proofGF2a} \\
& \alpha_i^{-1} \left\langle \mat Q^m_{i,\theta^m}\,\vec Y_i, 
\mat Q^m_{i,\theta^m}\,\vec\eta \right\rangle_{\Gamma^m_i}^h +
\left\langle \nabs\,\vec X, \nabs\,\vec\eta \right\rangle_{\Gamma^m_i} 
= \left\langle \vec{\rm m}_i, \vec\eta \right\rangle_{\gamma^m}^h 
 \quad\forall\ \vec\eta \in \Vhi\,,\ i=1,2\,, \label{eq:proofGF2b} \\
& \left\langle \vec\kappa_\gamma, \vec\chi 
\right\rangle_{\gamma^m}^h
+ \left\langle \vec X_s, \vec\chi_s \right\rangle_{\gamma^m} = 0
\qquad \forall\ \vec\chi \in \Vhpartial\,, \label{eq:proofGF2c} \\
& C_1\,(\vec{\rm m}_1 + \vec{\rm m}_2) = 
\vec 0 \qquad \text{on}\quad \gamma^m\,, \label{eq:proofGF2d} \\
& \alpha^G_i \,\vec\kappa_\gamma
 + \vec Y_i + C_1\,\vec\Phi = \vec0 
\qquad \text{on}\quad \gamma^m \,, \ i = 1,2\,. \label{eq:proofGF2e}
\end{align}
\end{subequations}
Choosing $\vec\chi = \vec X$ in (\ref{eq:proofGF2a}), 
$\vec\eta = \vec Y_i$ in (\ref{eq:proofGF2b}), 
$\vec\chi = \vec{\rm m}_i$ in (\ref{eq:proofGF2c}) 
and noting (\ref{eq:proofGF2d},e), leads to
\begin{align}
& \sum_{i=1}^2 \frac1{\ttau_m} \left\langle \mat Q^m_{i,\theta^m_\star}\,\vec X,
\vec X \right\rangle_{\Gamma^m_i}^h
+ \varsigma \left\langle \vec X_s, \vec X_s \right\rangle_{\gamma^m} 
+ \varrho \left\langle \vec X, \vec X \right\rangle_{\gamma^m}^h\black 
\nonumber \\ & \qquad
= \sum_{i=1}^2 \left[
 \left\langle \vec{\rm m}_i, \vec Y_i \right\rangle_{\gamma^m}^h
- \alpha_i^{-1} \left\langle\mat Q^m_{i,\theta^m}\, \vec Y_i, 
 \mat Q^m_{i,\theta^m}\,\vec Y_i \right\rangle_{\Gamma^m_i}^h
- \alpha^G_i \left\langle (\vec{\rm m}_i)_s, \vec X_s 
 \right\rangle_{\gamma^m} \right]
\nonumber \\ & \qquad
= \sum_{i=1}^2 \left[
 \left\langle \vec Y_i ,\vec{\rm m}_i\right\rangle_{\gamma^m}^h
+ \alpha^G_i \left\langle \vec{\rm m}_i, \vec\kappa_\gamma
 \right\rangle_{\gamma^m} 
- \alpha_i^{-1} \left\langle\mat Q^m_{i,\theta^m}\, \vec Y_i, 
 \mat Q^m_{i,\theta^m}\,\vec Y_i \right\rangle_{\Gamma^m_i}^h
\right] \nonumber \\ & \qquad
= - \sum_{i=1}^2 \alpha_i^{-1} \left\langle\mat Q^m_{i,\theta^m}\, \vec Y_i, 
 \mat Q^m_{i,\theta^m}\,\vec Y_i \right\rangle_{\Gamma^m_i}^h
 \,.
\label{eq:proof3GF2}
\end{align}
It follows from (\ref{eq:proof3GF2}) and the definition of
$\mat Q^{m}_{i,\theta^m_\star}$ that $\vec X = \vec 0$ on 
$\gamma^m$, and so (\ref{eq:proofGF2c}) implies that $\vec\kappa_\gamma=\vec0$.
In addition, (\ref{eq:proof3GF2}) yields that
$\vec\pi^m_i\,[\mat Q^m_{i,\theta^m}\,\vec Y_i] = \vec 0$, $i=1,2$.
Hence, on adding the two equations in (\ref{eq:proofGF2b})
with $\vec\eta=\vec X$, we obtain that 
$\left\langle\nabs \,\vec X,\nabs\, \vec X\right\rangle_{\Gamma^m} = 0$,
and so $\vec X = \vec 0$.
Then (\ref{eq:proofGF2b}) implies that $\vec {\rm m}_1 = \vec{\rm m}_2 = \vec0$.
Next, we have from (\ref{eq:proofGF2e}), on recalling that 
$\vec\kappa_\gamma=\vec0$, that there exists a $\vec Y \in \Vh$
such that $\vec Y_i = \vec Y\!\mid_{\Gamma^m_i}$, $i=1,2$. Choosing
$\vec\eta=\vec Y$ in (\ref{eq:proofGF2a})  yields that
$\left\langle\nabs \,\vec Y,\nabs\, \vec Y\right\rangle_{\Gamma^m} = 0$,
and hence $\vec Y$ is constant. If $C_1=0$ we immediately obtain from
(\ref{eq:proofGF2e}) that $\vec Y=\vec 0$.
If $C_1 = 1$, on the other hand, it follows that
$\mat Q^m_{i,\theta^m}(\vec q^m_{i,k})\,\vec Y = \vec 0$
for $k=1,\ldots K_i$, $i=1,2$, and hence 
\begin{equation} \label{eq:Yoi}
\vec Y \,.\, \vec\omega^m_i(\vec q^m_{i,k}) = 0\qquad
k=1,\ldots K_i,\ i=1,2. 
\end{equation}
The definition of $\mat Q^m_{i,\theta^m}$, recall the fully discrete version of
(\ref{eq:Qalphah}), and (\ref{eq:Yoi}) then yield 
that $\theta\, \vec Y = \vec 0$. 
Hence for $\theta \in (0,1]$ we immediately obtain that $\vec Y=\vec0$, while
in the case $\theta =0$ it follows from assumption $(\mathcal{A})$ 
and (\ref{eq:Yoi}) that $\vec Y = \vec 0$.
Finally, 
we obtain that $\vec\Phi=\vec0$ from (\ref{eq:proofGF2e}). 
\end{proof}

\subsection{Implicit treatment of volume and area conservation} 
\label{rem:implfree}
In practice it can be advantageous to consider implicit Lagrange multipliers
$(\lambda^{V,m+1}, \lambda^{A,m+1}_1, \lambda^{A,m+1}_2)$ 
in order to obtain better discrete volume and surface area conservations.
In particular, we replace {\rm (\ref{eq:GF2aa})} with
\begin{align} \label{eq:GF2aaa}
&\sum_{i=1}^2 \left[
\left\langle \mat Q^m_{i,\theta^m_\star}\,
\frac{\vec X^{m+1} - \vec\id}{\ttau_m} ,
\vec\chi \right\rangle_{\Gamma^m_i}^h
-  \left\langle \nabs\,\vec Y^{m+1}_i, \nabs\,\vec\chi 
\right\rangle_{\Gamma^m_i} 
+ \alpha^G_i \left\langle (\vec{\rm m}^{m+1}_i)_s, 
\vec\chi_s \right\rangle_{\gamma^m}
\right]
\nonumber \\ & \qquad\qquad
+ \varsigma \left\langle \vec X^{m+1}_s, \vec\chi_s \right\rangle_{\gamma^m}
+ \varrho\left\langle \frac{\vec X^{m+1} - \vec\id}{\ttau_m},\vec\chi 
\right\rangle_{\gamma^m}^h \black
\nonumber \\ & \quad
= \left\langle \vec f^m , \vec\chi \right\rangle_{\Gamma^m}^h 
- \lambda^{V,m+1}
\left\langle \vec\omega^{m}, \vec\chi \right\rangle_{\Gamma^m}^h
- \sum_{i=1}^2 \lambda^{A,m+1}_i 
\left\langle \nabs\,\vec X^{m+1}, \nabs\,\vec\chi \right\rangle_{\Gamma^m_i}
\quad \forall\ \vec\chi \in \Vh \,,
\end{align}
and require the coupled solutions 
$\vec X^{m+1} \in \Vh$, 
$(\vec Y^{m+1}_i,\vec{\rm m}^{m+1}_i)_{i=1}^2 
\in \Vhone \times \Vhpartial \times \Vhtwo\times\Vhpartial$,
$\vec\kappa_\gamma^{m+1}\in \Vhpartial$, $C_1\,\Phi^{m+1} \in \Vhpartial$ 
and $(\lambda^{V,m+1},\lambda^{A,m+1}_1,\lambda^{A,m+1}_2) \in \R^3$ 
to satisfy the nonlinear system
{\rm (\ref{eq:GF2aaa})}, {\rm (\ref{eq:GF2b}--e)} 
as well as an adapted variant of {\rm (\ref{eq:mllm}--c)}, where 
the superscript $m$ is replaced by $m+1$ in all occurrences of 
$\vec{\rm m}^{m}_i$, $\vec\kappa^m_i$,
$\vec Y^m_i$, $\lambda^{V,m}$ and $\lambda^{A,m}_i$. In addition,
$\frac{\vec\id - \vec X^{m-1}}{\ttau_{m-1}}$ in {\rm (\ref{eq:bjm1},c)} 
is replaced by $\frac{\vec X^{m+1} - \vec\id}{\ttau_m}$.
In practice this nonlinear system can be solved with a fixed point iteration as
follows. Let $(\lambda^{V,m+1,0},\lambda^{A,m+1,0}_1,\lambda^{A,m+1,0}_2)
= (\lambda^{V,m},\lambda^{A,m}_1,\lambda^{A,m}_2)$
and $\vec X^{m+1,0} = \vec\id\!\mid_{\Gamma^m}$.
Then, for $j\geq 0$, find a solution 
$(\vec X^{m+1,j+1},$ $\vec Y^{m+1,j+1},
\vec\kappa_{\partial\Gamma}^{m+1,j+1}, \vec{\rm m}^{m+1,j+1})$
to the linear system {\rm (\ref{eq:GF2aaa})}, {\rm (\ref{eq:GF2b}--e)},
where any superscript $m+1$ on left hand sides is replaced by $m+1,j+1$,
and by $m+1,j$ on the right hand side of {\rm (\ref{eq:GF2aaa})}. Then 
let $\vec\kappa^{m+1,j+1}_i = \alpha^{-1}_i\,\vec\pi^m_i\,
[\mat Q^m_{i,\theta^m}\,\vec Y^{m+1,j+1}_i] + \spont_i\,\vec\omega^m_i$ 
be defined as usual, and compute
$(\lambda^{V,m+1,j+1},\lambda^{A,m+1,j+1}_1,\lambda^{A,m+1,j+1}_2)$
as the unique solution to
\begin{align*} &
\begin{pmatrix}
\sum_{i=1}^2 a^m_{i,\theta}(\vec\omega^m_i, \vec\omega^m_i) &
a^m_{1,\theta}(\vec\kappa^{m+1,j+1}_1, \vec\omega^m_1) &
a^m_{2,\theta}(\vec\kappa^{m+1,j+1}_2, \vec\omega^m_2) \\
a^m_{1,\theta}(\vec\kappa^{m+1,j+1}_1, \vec\omega^m_1) &
a^m_{1,\theta}(\vec\kappa^{m+1,j+1}_1, \vec\kappa^{m+1,j+1}_1) & 0 \\
a^m_{2,\theta}(\vec\kappa^{m+1,j+1}_2, \vec\omega^m_2) & 0 &
a^m_{2,\theta}(\vec\kappa^{m+1,j+1}_2, \vec\kappa^{m+1,j+1}_2) 
\end{pmatrix}
\begin{pmatrix}
- \lambda^{V,m+1,j+1} \\ \lambda^{A,m+1,j+1}_1\\ \lambda^{A,m+1,j+1}_2
\end{pmatrix} \nonumber \\ & \qquad
= \begin{pmatrix}
b_0^{m+1,j+1} \\ b_1^{m+1,j+1}\\ b_2^{m+1,j+1}
\end{pmatrix} \,,
\end{align*}
where
\begin{align*}
b_0^{m+1,j+1} & = \sum_{i=1}^2 \left[ 
\left\langle (\vec\Pi^m_{i,0} - \mat\Id)\,
\frac{\vec X^{m+1,j+1} - \vec\id }{\ttau_m},
\vec\omega^m_i \right\rangle_{\Gamma^m_i}^h
- \left\langle \nabs\,\vec Y^{m+1,j+1}_i, 
\nabs\, (\vec\Pi^m_{i,0}\,\vec\omega^m_i)
\right\rangle_{\Gamma^m_i} \right. \nonumber \\ & \qquad \left.
- \left\langle \vec f^m, \vec\Pi^m_{i,0}\,\vec\omega^m_i
\right\rangle_{\Gamma^m_i}^h \right], \\ 
b_{i}^{m+1,j+1} & = \left\langle (\vec\Pi^m_{i,0} - \mat\Id)\,
\frac{\vec X^{m+1,j+1} - \vec\id}{\ttau_m},
\mat Q^m_{i,\theta^m}\,\vec\kappa^{m+1,j+1}_i \right\rangle_{\Gamma^m_i}^h 
+ \left\langle \vec{\rm m}^{m+1}_i, 
\frac{\vec X^{m+1,j+1} - \vec\id}{\ttau_m}
\right\rangle_{\gamma^m}^h \nonumber \\ & \qquad 
- \left\langle \nabs\,\vec Y^{m+1,j+1}_i, 
\nabs\,(\vec\Pi^m_{i,0}\,\vec\kappa^{m+1,j+1}_i)
\right\rangle_{\Gamma^m_i}
- \left\langle \vec f^m, \vec\Pi^m_{i,0}\,\vec\kappa^{m+1,j+1}_i 
\right\rangle_{\Gamma^m_i}^h,\ i=1,2\,; 
\end{align*}
and continue the iteration until 
\[|\lambda^{V,m+1,j+1} - \lambda^{V,m+1,j}| 
+|\lambda^{A,m+1,j+1}_1 - \lambda^{A,m+1,j}_1| 
+|\lambda^{A,m+1,j+1}_2 - \lambda^{A,m+1,j}_2| 
< 10^{-8}\,.\]

We remark that the implicit scheme is chosen such that no new system matrices
need to be assembled during the fixed point iteration. In particular, all
integrals are evaluated on the old interfaces $\Gamma^m_i$. But all the
quantities that are calculated during the linear solves are treated
implicitly, i.e.\ $\vec X^{m+1}$, 
$(\vec Y^{m+1}_i,\vec\kappa^{m+1}_i, \vec {\rm m}^{m+1}_i)_{i=1}^2$,
$\vec\kappa_\gamma^{m+1}$, $C_1\,\Phi^{m+1}$,
as well as the Lagrange multipliers.

\setcounter{equation}{0}
\section{Solution methods} \label{sec:6}

Let us briefly outline how we solve the linear system (\ref{eq:GF2a}--e) in
practice. First of all, similarly to our approach in \cite{clust3d} for the
numerical approximation of surface clusters with triple junction lines,
we reformulate (\ref{eq:GF2a}) as follows.

On introducing the following equivalent 
characterization of $\Vh$, recall (\ref{eq:Vh}), 
\[
\hatVh = \{ (\vec\eta_1,\vec\eta_2) \in \mathop{\times}_{i=1}^2\Vhi :
\vec\eta_1 \!\mid_{\gamma^m} = \vec\eta_2 \!\mid_{\gamma^m}\} \,,
\]
we can rewrite (\ref{eq:GF2a}--e) equivalently as:
Find $(\vec X^{m+1}_1,\vec X^{m+1}_2) \in \hatVh$, 
$(\vec Y^{m+1}_i,\vec{\rm m}^{m+1}_i)_{i=1}^2 
\in \Vhone \times \Vhpartial \times \Vhtwo \times \Vhpartial$,
$\vec\kappa_\gamma^{m+1}\in \Vhpartial$
and $C_1\,\Phi^{m+1} \in \Vhpartial$ such that
\begin{align}
& \sum_{i=1}^2 \left[
\left\langle \mat Q^m_{i,\theta^m_\star}\,
\frac{\vec X^{m+1}_i - \vec\id}{\ttau_m} ,
\vec\chi_i \right\rangle_{\Gamma^m_i}^h
-  \left\langle \nabs\,\vec Y^{m+1}_i, \nabs\,\vec\chi_i 
\right\rangle_{\Gamma^m_i} 
+ \alpha^G_i \left\langle (\vec{\rm m}^{m+1}_i)_s, 
[\vec\chi_i]_s \right\rangle_{\gamma^m} \right.
\nonumber \\ & \qquad \qquad \left.
+ \tfrac12\,\varsigma 
\left\langle [\vec X^{m+1}_i]_s, [\vec\chi_i]_s \right\rangle_{\gamma^m}
+ \tfrac12\,\varrho
\left\langle \frac{\vec X^{m+1}_i - \vec\id}{\ttau_m},\vec\chi_i
\right\rangle_{\gamma^m}^h \black
\right]
\nonumber \\ & \quad
= 
\sum_{i=1}^2 \left[
\left\langle \nabs\,.\,\vec Y^m_i, \nabs\,.\,\vec\chi_i 
\right\rangle_{\Gamma^m_i} 
- \left\langle (\nabs\,\vec Y^m_i)^T , 
\mat D(\vec\chi_i)\,(\nabs\,\vec\id)^T \right\rangle_{\Gamma^m_i} 
\right. \nonumber \\ & \qquad \left.
-\tfrac12\left\langle [ \alpha_i\,|\vec\kappa^m_i - \spont_i\,\vec\nu^m_i|^2
- 2\,(\vec Y^m_i\,.\,\mat Q^m_{i,\theta^m}\,\vec\kappa^m_i)]
\,\nabs\,\vec\id,\nabs\,\vec\chi_i \right\rangle_{\Gamma^m_i}^h
-\alpha_i\,\spont_i 
\left\langle \vec\kappa^m_i, [\nabs\,\vec\chi_i]^T\,\vec\nu^m_i
\right\rangle_{\Gamma^m_i}^h 
\right. \nonumber \\ & \qquad \left.
+\left\langle(1-\theta^m)\, (\vec G^m_i(\vec Y^m_i, \vec\kappa^m_i)\,.\,
\vec\nu^m_i)\,
\nabs\,\vec\id, \nabs\,\vec\chi_i \right\rangle_{\Gamma^m_i}^h
- \left\langle(1-\theta^m)\,\vec G^m_i(\vec Y^m_i,\vec\kappa^m_i), 
[\nabs\,\vec\chi_i]^T\,\vec\nu^m_i \right\rangle_{\Gamma^m_i}^h 
 \right]
\nonumber \\ & \qquad 
+ \sum_{i=1}^2 \alpha^G_i 
\left[\left\langle \vec\kappa_\gamma^m \,.\,\vec{\rm m}^m_i,
\vec\id_s\,.\,[\vec\chi_i]_s \right\rangle_{\gamma^m}^h
+ \left\langle(\mat\Id + \mat{\mathcal{P}}^m_\gamma)\,(\vec{\rm m}^m_i)_s, 
[\vec\chi_i]_s \right\rangle_{\gamma^m} \right]
\nonumber \\ & \qquad 
- \lambda^{V,m} \sum_{i=1}^2 
\left\langle \vec\omega^m_i, \vec\chi_i \right\rangle_{\Gamma^m_i}^h
- \sum_{i=1}^2 \lambda^{A,m}_i 
\left\langle \nabs\,\vec\id, \nabs\,\vec\chi_i \right\rangle_{\Gamma^m_i}
\qquad \forall\ (\vec\chi_1,\vec\chi_2) \in \hatVh \label{eq:hatGFa}
\end{align}
and (\ref{eq:GF2b}--e) hold, 
where in (\ref{eq:hatGFa}) 
we have used the fully discrete version of (\ref{eq:zo}). 

The above reformulation is crucial for the construction 
of fully practical solution methods, as it avoids the use of the global
finite element space $\Vh$. With the help of (\ref{eq:hatGFa}), it is now
possible to work with the basis of the
simple product finite element space $\hatVh$, on employing suitable projections
in the formulation of the linear problem. This
construction is similar to e.g.\ the standard technique used for 
an ODE with periodic boundary conditions. 

We recall from \cite[(4.4a--d)]{pwfopen} the following finite element
approximation for Willmore flow of a single open surface $\Gamma_i(t)$
with free boundary conditions for $\partial\Gamma_i(t)$.
For $m=0,\ldots,M-1$, 
find $(\delta\vec X^{m+1}_i, \vec Y^{m+1}_i) \in \Vhi \times \Vhi$, with
$\vec X^{m+1}_i = \vec\id\!\mid_{\Gamma^m_i} + \delta\vec X^{m+1}_i$,
and $(\vec\kappa_{\partial\Gamma_i}^{m+1}, \vec{\rm m}^{m+1}_i) 
\in [\Vhipartial]^2$ such that
\begin{subequations}
\begin{align}
&
\left\langle \mat Q^m_{i,\theta^m_\star}\,\frac{\vec X^{m+1}_i - 
\vec\id}{\ttau_m} , \vec\chi \right\rangle_{\Gamma^m_i}^h
- \left\langle \nabs\,\vec Y^{m+1}_i , \nabs\,\vec\chi 
\right\rangle_{\Gamma^m_i}
+ \varsigma \left\langle [\vec X^{m+1}_i]_s, \vec\chi_s 
\right\rangle_{\partial\Gamma^m_i} 
+ \alpha^G_i \left\langle [\vec{\rm m}^{m+1}_i]_s, 
\vec\chi_s \right\rangle_{\partial\Gamma^m_i}
\nonumber \\ & \
= \left\langle \nabs\,.\,\vec Y^{m}_i , \nabs\,.\,\vec\chi 
 \right\rangle_{\Gamma^m_i}
 - \left\langle (\nabs\,\vec Y^{m}_i)^T, \mat D(\vec\chi)\,(\nabs\,\vec\id)^T
 \right\rangle_{\Gamma^m_i}
 - \alpha_i\,\spont_i \left\langle \vec\kappa^m_i ,
  [\nabs\,\vec\chi]^T\,\vec\nu^m_i \right\rangle_{\Gamma^m_i}^h
\nonumber \\ & \quad
-\tfrac12 \left\langle \left[ \alpha_i\,
 |\vec{\kappa}^m_i - \spont_i\,\vec\nu^m_i|^2 
 - 2\,\vec Y^m_i\,.\,\mat Q^m_{i,\theta^m}\,\vec\kappa^m_i \right] 
\nabs\,\vec\id, \nabs\,\vec\chi \right\rangle_{\Gamma^m_i}^h
\nonumber \\ & \quad
+\left\langle(1-\theta^m)\, (\vec G^m_i(\vec Y^m_i, \vec\kappa^m_i)\,.\,
\vec\nu^m_i)\,
\nabs\,\vec\id, \nabs\,\vec\chi \right\rangle_{\Gamma^m_i}^h
- \left\langle(1-\theta^m)\,\vec G^m_i(\vec Y^m_i,\vec\kappa^m_i), 
[\nabs\,\vec\chi]^T\,\vec\nu^m_i \right\rangle_{\Gamma^m_i}^h 
\nonumber \\ & \quad
+ \alpha^G_i \left\langle \vec\kappa_{\partial\Gamma_i}^m \,.\,
 \vec{\rm m}^{m}_i,
\vec\id_s\,.\,\vec\chi_s \right\rangle_{\partial\Gamma^m_i}^h
+ \alpha^G_i \left\langle (\mat\Id+\mat{\mathcal{P}}^m_{\partial\Gamma_i})\,
[\vec{\rm m}^{m}_i]_s, \vec\chi_s \right\rangle_{\partial\Gamma^m_i}
- \lambda^{A,m}_i \left\langle \nabs\,\vec\id,\nabs\,\vec\chi 
\right\rangle_{\Gamma^m_i}
\nonumber \\ & \hspace{11cm}
\qquad \forall\ \vec\chi \in \Vhi\,, \label{eq:ifreePWFa} \\
& \alpha^{-1}_i
\left\langle \mat Q^m_{i,\theta^m}\,\vec Y^{m+1}_i, \mat Q^m_{i,\theta^m}\,
\vec\eta \right\rangle_{\Gamma^m_i}^h +
\left\langle \nabs\,\vec X^{m+1}_i , \nabs\,\vec\eta \right\rangle_{\Gamma^m_i}
 = 
\left\langle \vec{\rm m}^{m+1}_i, \vec\eta \right\rangle_{\partial\Gamma^m_i}^h
- \spont \left\langle \vec\omega^m_i, 
\vec\eta \right\rangle_{\Gamma^m_i}^h
\quad \forall\ \vec\eta \in \Vhi\,, \label{eq:ifreePWFb} \\
& \left\langle \alpha^G_i \,\vec\kappa^{m+1}_{\partial\Gamma_i}
+ \vec Y^{m+1}_i, \vec\varphi \right\rangle_{\partial\Gamma^m_i}^h  = 0 \qquad
\forall\ \vec\varphi \in \Vhipartial \,, \label{eq:ifreePWFc}  \\
& \left\langle \vec\kappa_{\partial\Gamma_i}^{m+1}, 
\vec\eta \right\rangle_{\partial\Gamma^m_i}^h +
\left\langle [\vec X^{m+1}_i]_s , \vec\eta_s 
\right\rangle_{\partial\Gamma^m_i} = 
0 \quad \forall\ \vec\eta \in \Vhipartial\,. \label{eq:ifreePWFd} 
\end{align}
\end{subequations}
The corresponding linear system from \cite[(5.1)]{pwfopen} is then given by
\begin{align} &
\begin{pmatrix} \vec A & - \frac1{\ttau_m}\,\vec{\mathcal{M}}_{Q^\star} 
- \vec A_\varsigma
& 0 & -\alpha^G_i\,\vec A_{\partial\Gamma,\Gamma} \\ 
\vec{\mathcal{M}}_{Q^2} & \vec A 
& 0 & -\vec M_{\partial\Gamma,\Gamma} \\
(\vec M_{\partial\Gamma,\Gamma})^T & 0 & \alpha^G_i\,\vec M_{\partial\Gamma} & 0
\\
0 & (\vec A_{\partial\Gamma,\Gamma})^T & \vec M_{\partial\Gamma} & 0
\end{pmatrix}\begin{pmatrix} \vec Y^{m+1}_i \\ \delta\vec X^{m+1}_i \\
\vec\kappa^{m+1}_{\partial\Gamma_i} \\ \vec{\rm m}^{m+1}_i
\end{pmatrix}
\nonumber \\ & \quad
= \begin{pmatrix}  
[\vec{\mathcal{B}}^\star - \vec{\mathcal{B}} +\vec{\mathcal{R}}]\,\vec Y^m_i 
+ (\vec A_\theta + \vec A_\varsigma + \lambda^{A,m}_i\,\vec A)\,\vec X^m_i 
+ \vec b_{\theta} - \vec b_\alpha  \\ 
-\vec A\,\vec X^m_i 
 - \spont\,\vec M\,\vec\omega^m_i \\ \vec 0 \\ 
-(\vec A_{\partial\Gamma,\Gamma})^T\,\vec X^m_i
\end{pmatrix}.
\label{eq:pwfopen}
\end{align}
On replacing $\vec A_\varsigma$ with 
($\tfrac12\,\vec A_\varsigma + \tfrac1{2\ttau}\,\vec M_\varrho)$,
where the definition of $\vec M_\varrho$ is clear from (\ref{eq:hatGFa}),
and similarly adapting the first entry in the right hand side 
of (\ref{eq:pwfopen}) 
to account for the term involving $\lambda^{V,m}$, we write 
(\ref{eq:pwfopen}) as 
\[
B_i\,Z_i = g_i\,.
\]
Hence we can rewrite the linear system for 
(\ref{eq:hatGFa}), (\ref{eq:GF2b}--e) as
\begin{subequations}
\begin{equation}
\mathcal{P}_B\,\mathcal{B}\,
\,\mathcal{P}_Z\,
\begin{pmatrix}
Z_1 \\ Z_2 \\ \vec \Phi^{m+1}
\end{pmatrix}
= \mathcal{P}_B\,
\begin{pmatrix}
g_1 \\ g_2 \\ 0
\end{pmatrix}\,,
\label{eq:block}
\end{equation}
where 
\begin{equation} \label{eq:blockB}
\mathcal{B} = 
\begin{pmatrix}
B_1 & 0 & \begin{pmatrix} 0 \\ 0 \\ C_1\,\vec M^\gamma \\ 0 \end{pmatrix} \\
0 & B_2 & \begin{pmatrix} 0 \\ 0 \\ C_1\,\vec M^\gamma \\ 0 \end{pmatrix} \\
(0\ 0\ 0\ C_1\,\vec M^\gamma) & (0\ 0\ 0\ C_1\,\vec M^\gamma) & 0 
\end{pmatrix},
\end{equation}
\end{subequations}
and where $\vec M^\gamma$ is a mass matrix on $\gamma^m$.
Moveover,
$\mathcal{P}_B$ and $\mathcal{P}_Z$ are the orthogonal projections that 
encode the test and trial space $\hatVh$ in (\ref{eq:hatGFa}), i.e.\
they act on the first and fifth block row in 
(\ref{eq:blockB}), and on the second entries of $Z_1$ and $Z_2$, respectively.

The system (\ref{eq:block}) 
can be efficiently solved in practice with a preconditioned 
BiCGSTAB or GMRES iterative solver, where we employ the preconditioners
\[
\mathcal{P}_Z\,
\begin{pmatrix}
B_1^{-1} & 0 & 0 \\
0 & B_2^{-1} & 0 \\
0 & 0 & \Id 
\end{pmatrix}
\,\mathcal{P}_B
\quad\text{and}\quad
\mathcal{P}_Z\,\mathcal{B}^{-1}
\,\mathcal{P}_B
\]
for the cases $C_1=0$ and $C_1=1$, respectively.
Here we recall from \cite{pwfopen} that $B_1$ and $B_2$ are invertible. The
inverses $B_1^{-1}$ and $B_2^{-1}$
can be computed with the help of the sparse factorization
package UMFPACK, see \cite{Davis04}. Similarly, the inverse
$\mathcal{B}^{-1}$, which existed in all our numerical tests, can also be
computed with the help of UMFPACK.

In practice we note that the preconditioned Krylov subspace
solvers usually take fewer than ten iterations per time step to converge. 
We stress that the chosen preconditioners are crucial, as without appropriate
preconditioning the iterative solvers do not converge. This suggests that the
linear systems (\ref{eq:block}) are badly conditioned.

\setcounter{equation}{0}
\section{Numerical results}  \label{sec:7}

We implemented our fully discrete finite element approximations within the
finite element toolbox ALBERTA, see \cite{Alberta}. The arising systems of
linear equations were solved with the help of the sparse factorization 
package UMFPACK, see \cite{Davis04}.
For the computations involving surface area preserving Willmore flow, we 
always employ the implicit Lagrange multiplier formulation discussed
in \S\ref{rem:implfree}.

The fully discrete scheme (\ref{eq:GF2a}--e) needs initial
data $\vec\kappa^0_i$, $\vec Y^0_i$, $\vec{\rm m}^0_i$, $i=1,2$, and
$\vec\kappa_\gamma^0$.
Given the initial triangulation $\Gamma^0_i$, we let 
$\vec{\rm m}^0_i \in \Vhpartialz$ be such that
$$
\left\langle \vec{\rm m}^0_i, \vec\eta \right\rangle_{\gamma^0}^h
= \left\langle \vec\mu^0_i, \vec\eta \right\rangle_{\gamma^0}
\qquad\forall\ \vec\eta \in \Vhpartialz\,,
$$
with $\vec\mu^0_i$ denoting the conormal on $\partial\Gamma^0_i$, $i=1,2$. 
In addition, we let 
\begin{equation*} 
\vec\kappa^0_i = - \tfrac2R\, \vec\omega^0_i
\end{equation*}
for simulations where $\Gamma_i(0)$ is part of a sphere of radius $R$, i.e.\
$\Gamma_i(0) \subset \partial B_R(\vec 0)$, and otherwise define
$\vec\kappa^0_i \in \Vhiz$ to be the solution of
\begin{equation*} 
 \left\langle \vec\kappa^0_i , \vec\eta \right\rangle_{\Gamma^0_i}^{h}
+ \left\langle \nabs\,\vec\id, \nabs\,\vec\eta \right\rangle_{\Gamma^0_i}
 = 
\left\langle \vec{\rm m}^0_i, \vec\eta \right\rangle_{\gamma^0}^h
  \qquad\forall\ \vec\eta \in \Vhiz\,.
\end{equation*}
Then we define
\begin{equation*} 
\vec Y^0_i = \alpha_i\,[\vec\kappa^0_i - \spont_i\,\vec\omega^0_i]\,.
\end{equation*}
Moreover, we let $\vec\kappa_\gamma^0 \in \Vhpartialz$  be such that
$$
\left\langle \vec\kappa_\gamma^0, 
\vec\eta \right\rangle_{\gamma^0}^h +
\left\langle \vec\id_s , \vec\eta_s \right\rangle_{\gamma^0} = 
0 \quad \forall\ \vec\eta \in \Vhpartialz\,.
$$

Throughout this section we use uniform time steps
$\ttau_m=\ttau$, $m=0,\ldots, M-1$, and set $\ttau=10^{-3}$ unless stated
otherwise. In addition, unless stated otherwise, we fix 
$\alpha_i = 1$ and
$\spont_i = \alpha^G_i = 0$, $i=1,2$, as well as $\varsigma = 0$.
At times we will discuss the discrete energy of the numerical solutions,
which, similarly to (\ref{eq:Eh}), is defined by
\begin{align*}
E^{m+1}((\Gamma^m_i)_{i=1}^2) & 
:= \sum_{i=1}^2 
\left[\tfrac12\,\alpha_i
\left\langle |\vec\kappa^{m+1}_i - \spont_i\,\vec\nu^m_i|^2, 1
\right\rangle_{\Gamma^m_i}^h
 + \alpha^G_i \left[ \left\langle \vec\kappa_\gamma^{m+1},
 \vec{\rm m}^{m+1}_i \right\rangle_{\gamma^m}^h
 + 2\,\pi\,m(\Gamma^m_i)\right] \right] \nonumber \\ & \qquad
+ \varsigma\, \mathcal{H}^{d-2}(\gamma^m) \,.
\end{align*}
Finally, we fix $\theta = 0$ throughout, unless otherwise stated.

For the visualization of our numerical results we will use the colour red for
$\Gamma_1^m$, and the colour yellow for $\Gamma_2^m$.

\subsection{The $C^0$--case}

In Figure~\ref{fig:c0cat0} we show the evolution of the outer shell of a torus
joined with two spherical caps. Here the two caps make up phase $1$, with the
remainder representing phase $2$.
The initial surface $\Gamma^0$ 
satisfies $(J_1,J_2) = (2048,4096)$ and $(K_1,K_2) = (1090,2112)$ and
has maximal dimensions $6\times6\times6$, i.e.\ up to translations, the
smallest cuboid containing $\Gamma^0$ is $[0,6]^3$.
For the parameters $\spont_1 = \spont_2 = 0$ and $\varsigma = 0.1$,
the surface evolves towards a catenoid.
\begin{figure}
\center
\includegraphics[angle=-0,width=0.24\textwidth]{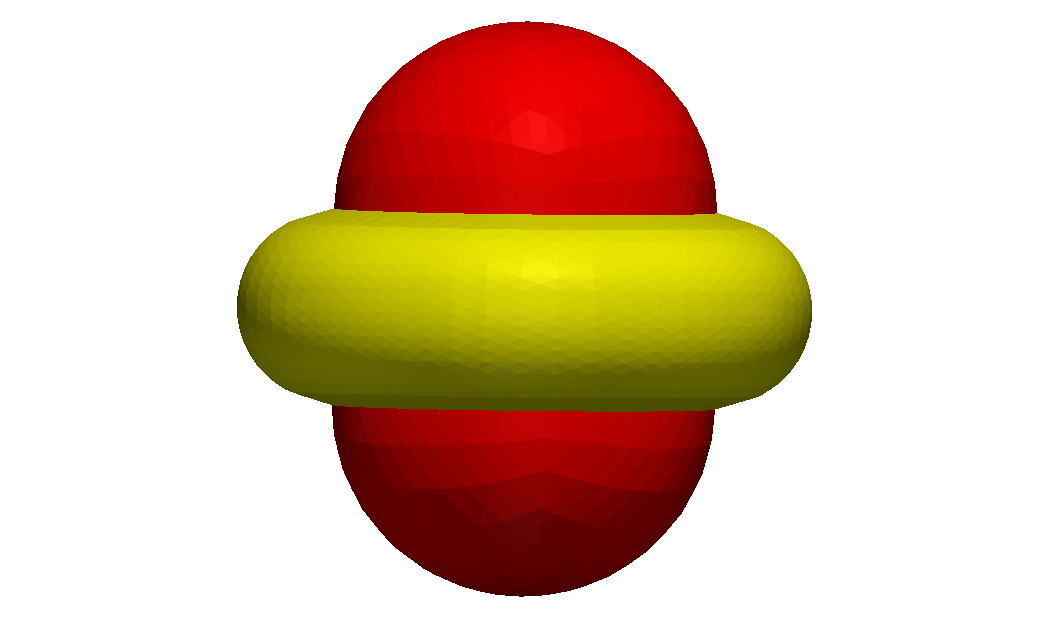}
\includegraphics[angle=-0,width=0.24\textwidth]{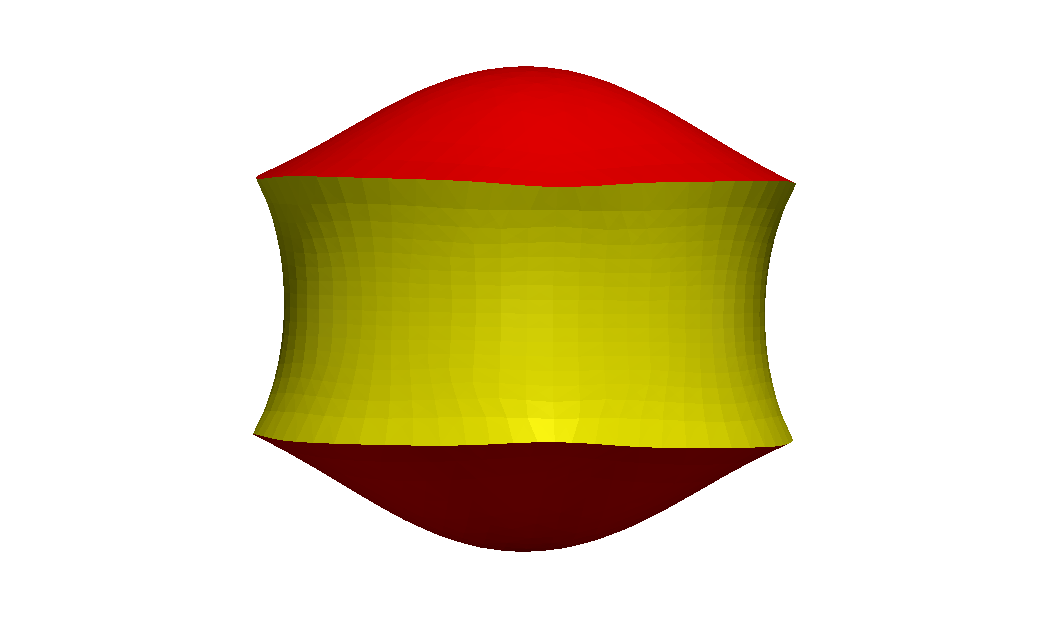}
\includegraphics[angle=-0,width=0.24\textwidth]{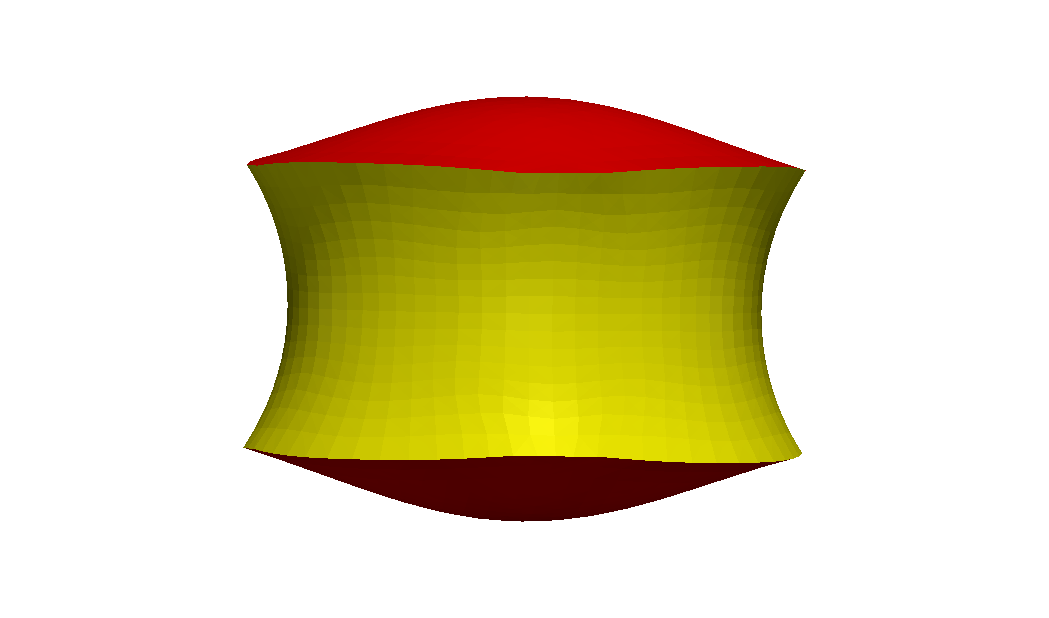} 
\includegraphics[angle=-0,width=0.24\textwidth]{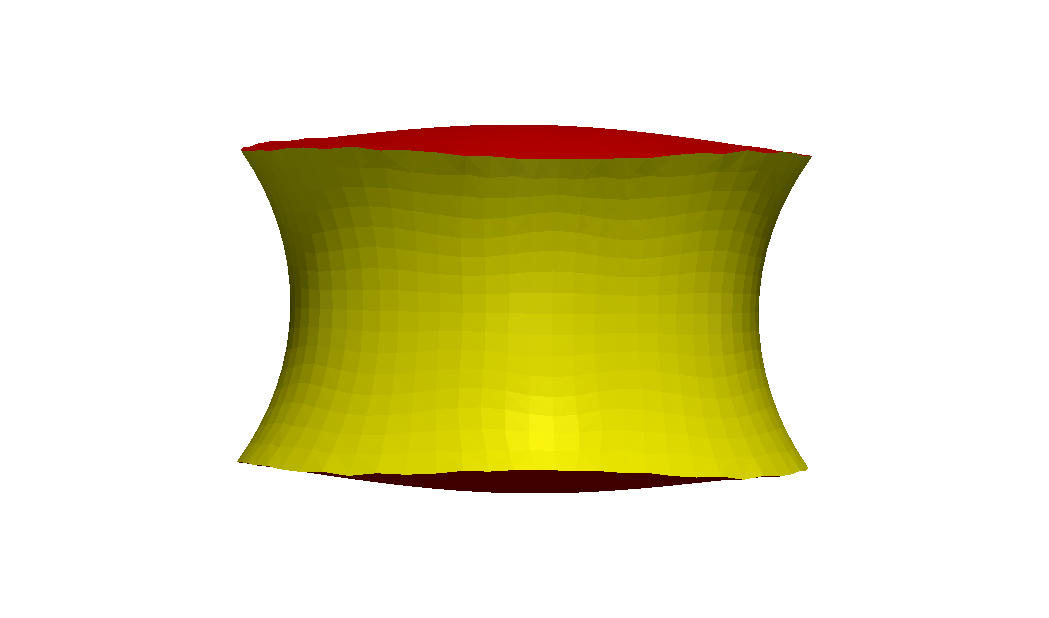} \\
\includegraphics[angle=-90,width=0.32\textwidth]{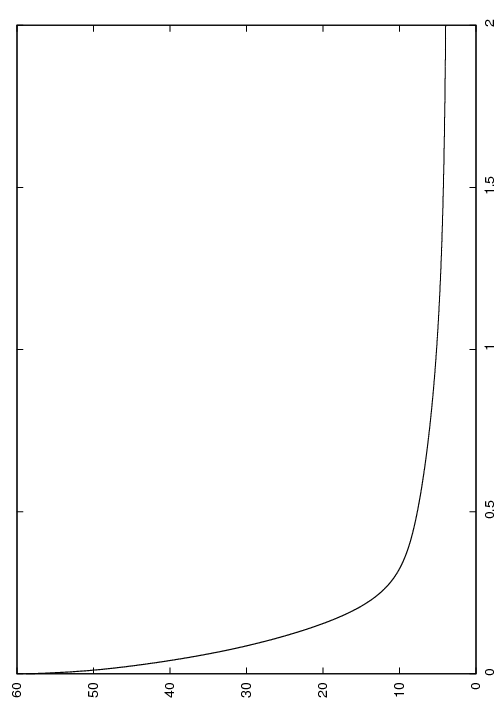}
\caption{($C^0$: $\spont_1 = \spont_2 = 0$, $\varsigma = 0.1$)
A plot of $(\Gamma^m_i)_{i=1}^2$ at times $t=0,\ 0.5,\ 1,\ 2$.
Below a plot of the discrete energy $E^{m+1}((\Gamma^m)_{i=1}^2)$.
}
\label{fig:c0cat0}
\end{figure}%
In Figure~\ref{fig:c0udb0} we show the same evolution for the values
$\spont_1 = -2$ and $\spont_2 = -0.5$, which is now markedly different.
\begin{figure}
\center
\includegraphics[angle=-0,width=0.24\textwidth]{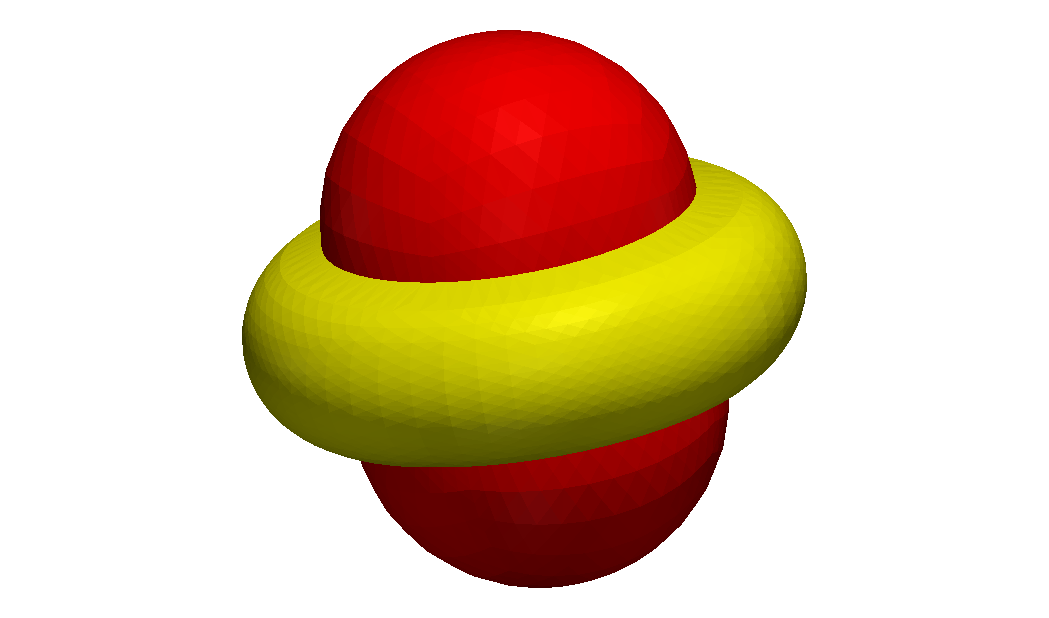}
\includegraphics[angle=-0,width=0.24\textwidth]{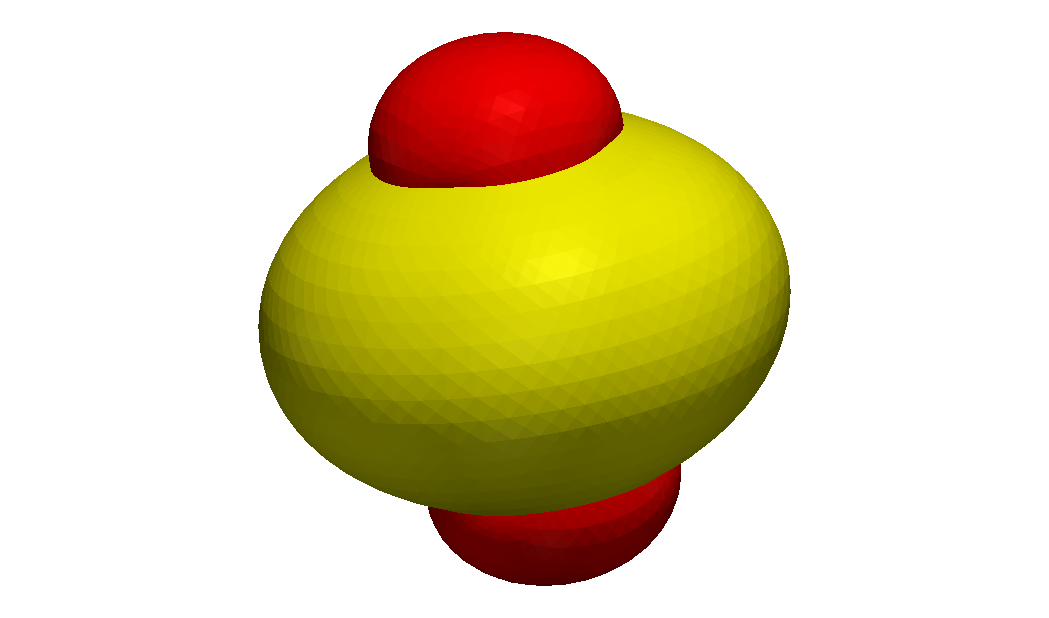}
\includegraphics[angle=-0,width=0.24\textwidth]{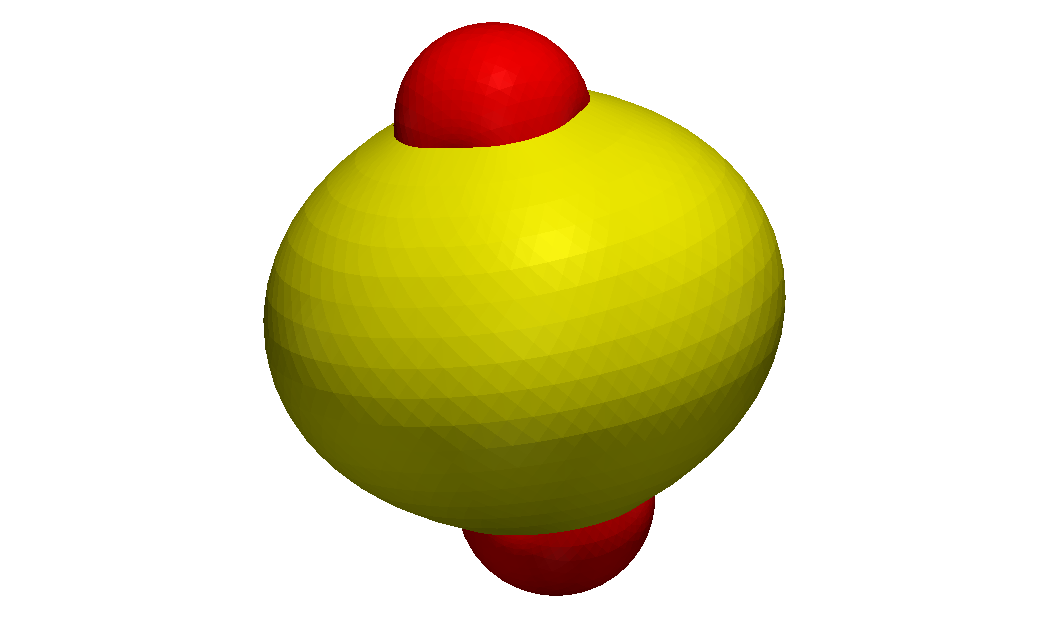} 
\includegraphics[angle=-0,width=0.24\textwidth]{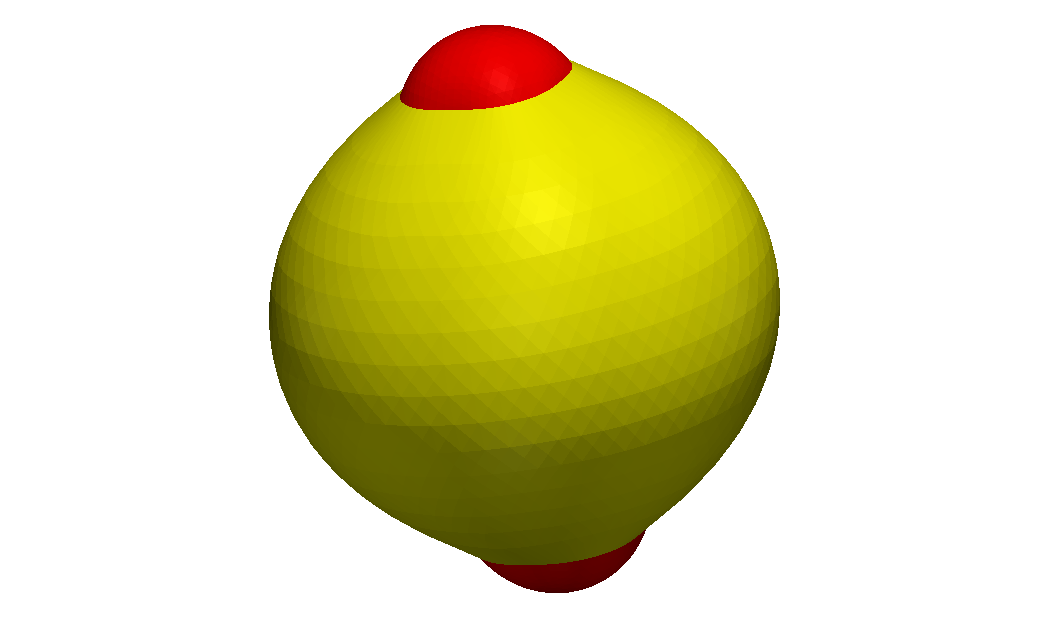} \\
\includegraphics[angle=-90,width=0.32\textwidth]{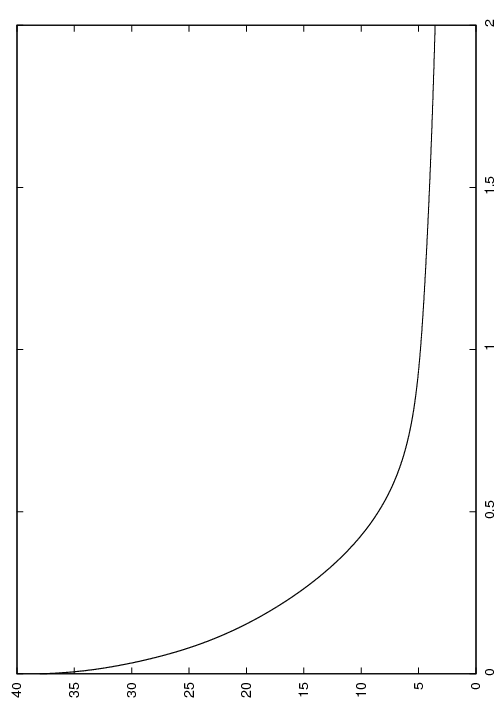}
\caption{($C^0$: $\spont_1 = -2$, $\spont_2 = -0.5$, $\varsigma = 0.1$)
A plot of $(\Gamma^m_i)_{i=1}^2$ at times $t=0,\ 0.5,\ 1,\ 2$.
Below a plot of the discrete energy $E^{m+1}((\Gamma^m)_{i=1}^2)$.
}
\label{fig:c0udb0}
\end{figure}%
The same evolution with $\varrho = 2$, which shows the slowing influence
of $\varrho > 0$, is shown in Figure~\ref{fig:c0udb0_varrho}. 
In both experiments the effect of the two different spontaneous curvature
values for the two phases can clearly be seen.
\begin{figure}
\center
\includegraphics[angle=-0,width=0.24\textwidth]{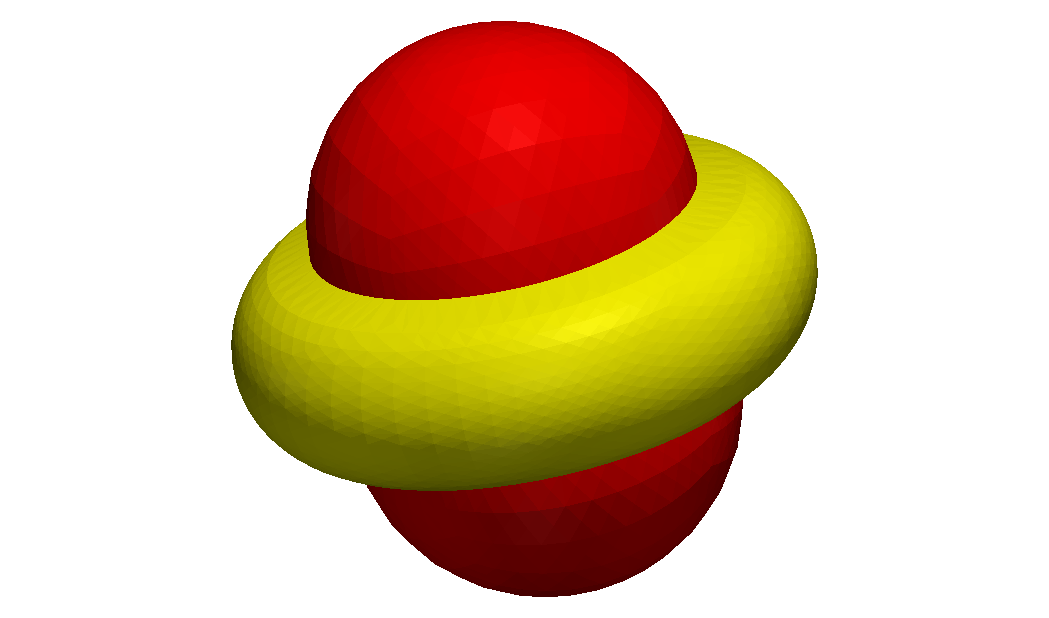}
\includegraphics[angle=-0,width=0.24\textwidth]{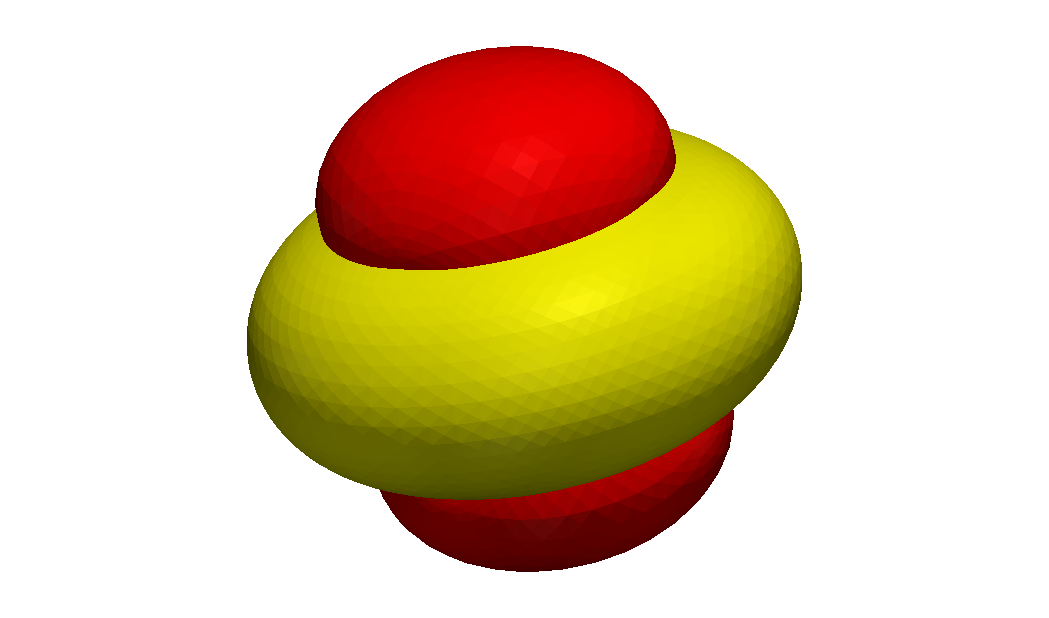}
\includegraphics[angle=-0,width=0.24\textwidth]{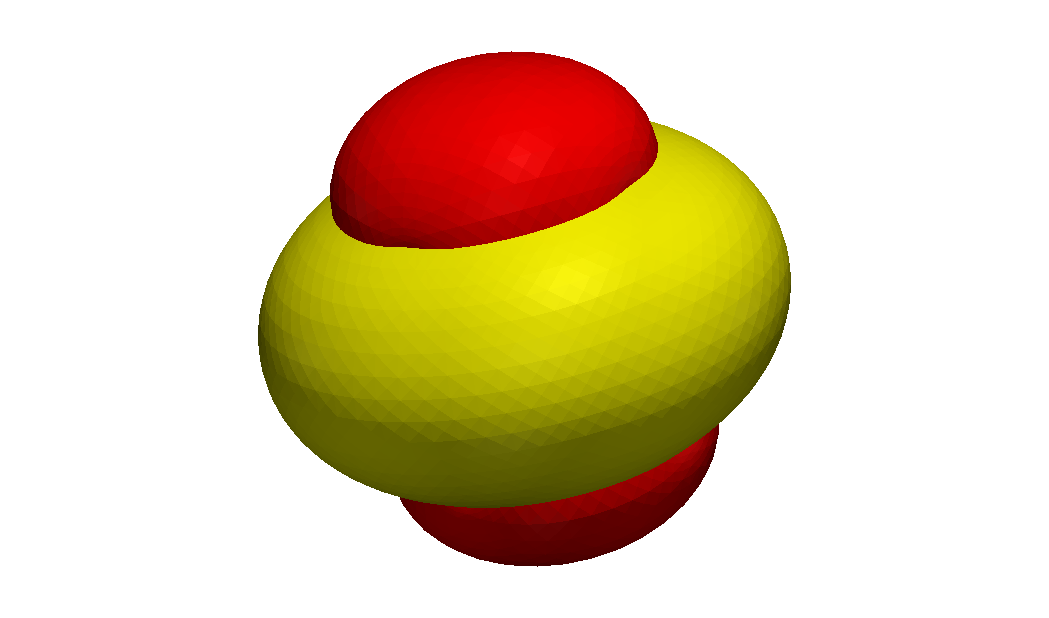} 
\includegraphics[angle=-0,width=0.24\textwidth]{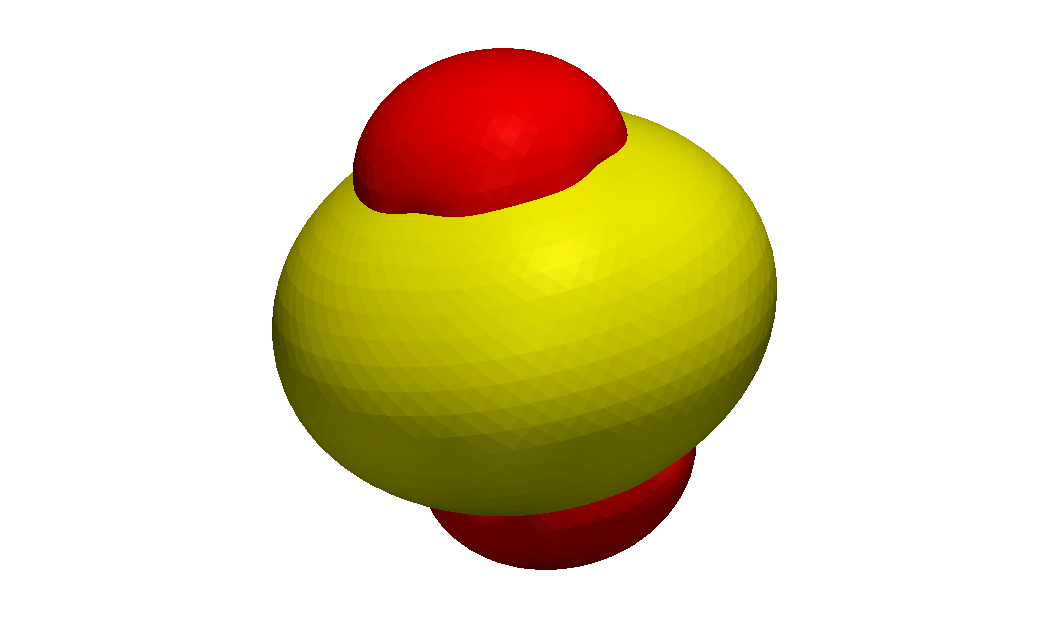} \\
\includegraphics[angle=-90,width=0.32\textwidth]{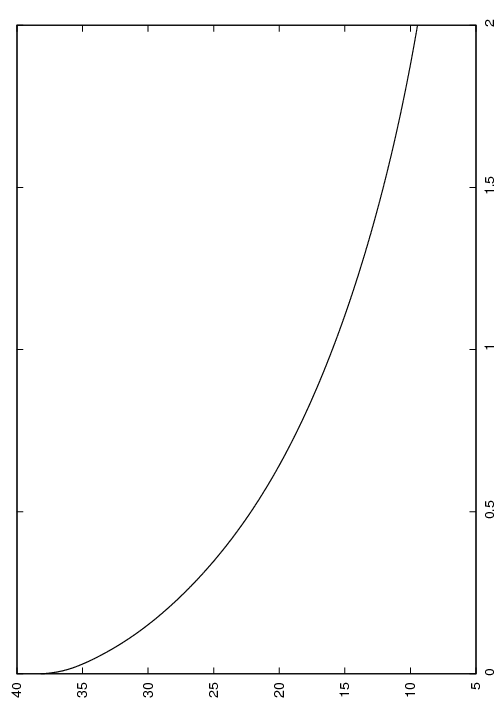}
\caption{($C^0$: $\spont_1 = -2$, $\spont_2 = -0.5$, $\varsigma = 0.1$,
$\varrho=2$)
A plot of $(\Gamma^m_i)_{i=1}^2$ at times $t=0,\ 0.5,\ 1,\ 2$.
Below a plot of the discrete energy $E^{m+1}((\Gamma^m)_{i=1}^2)$.
}
\label{fig:c0udb0_varrho}
\end{figure}%
The same evolution as in Figure~\ref{fig:c0udb0_varrho}, but now for 
surface area preserving flow, is shown in Figure~\ref{fig:c0udb1}. 
Here the observed relative surface area loss is $0.12\%$.
The interplay between the different values of $\spont_i$, the surface area
constraints, and the $C^0$--attachment condition lead to an interesting
evolution.
\begin{figure}
\center
\includegraphics[angle=-0,width=0.24\textwidth]{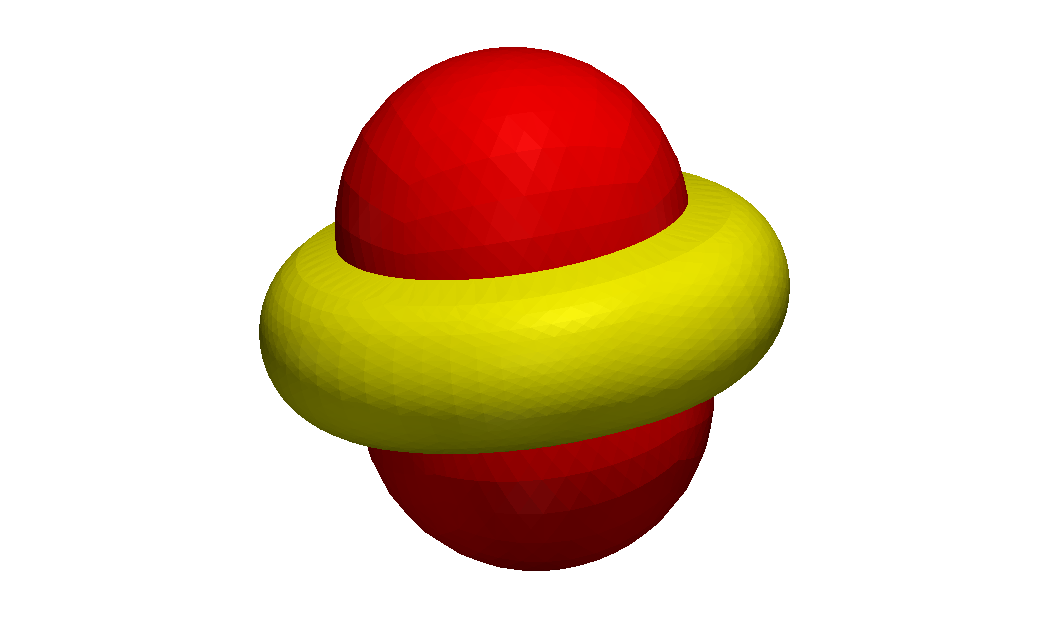}
\includegraphics[angle=-0,width=0.24\textwidth]{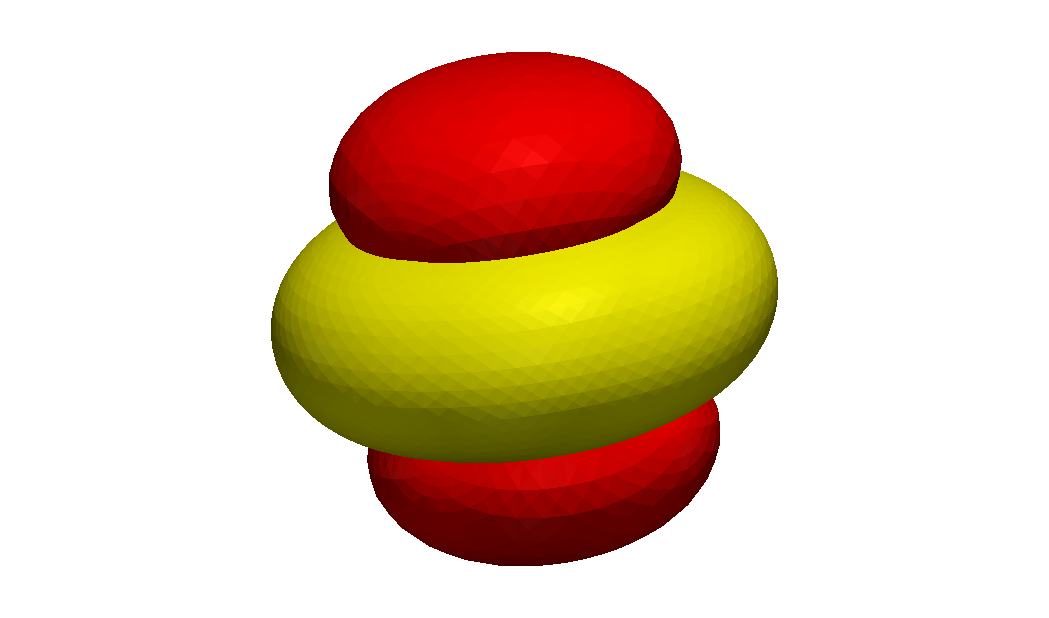}
\includegraphics[angle=-0,width=0.24\textwidth]{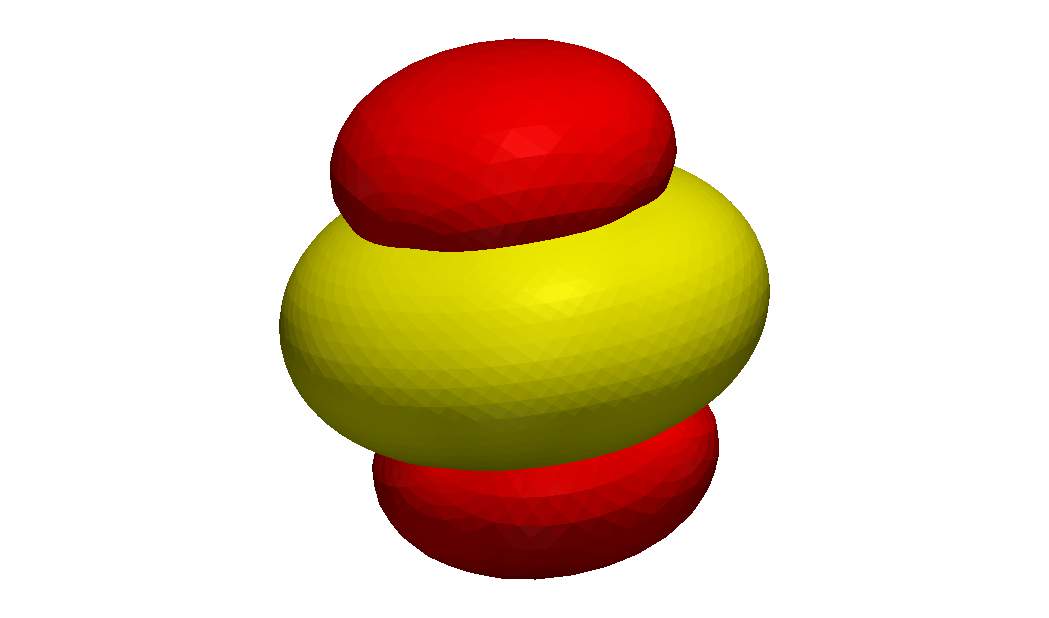} 
\includegraphics[angle=-0,width=0.24\textwidth]{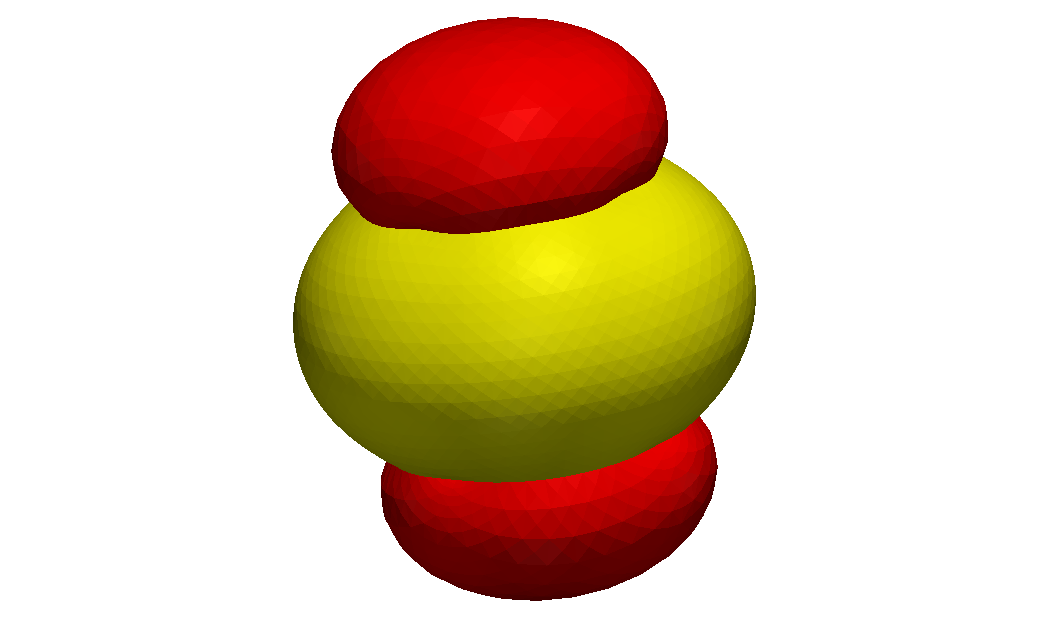} \\
\includegraphics[angle=-90,width=0.32\textwidth]{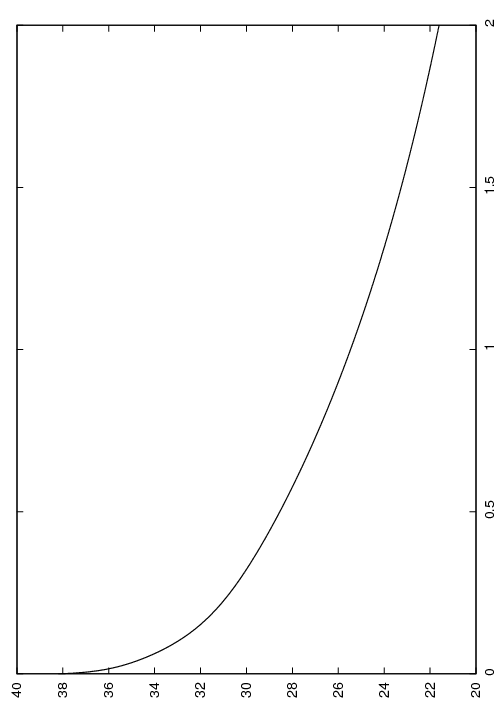}
\caption{($C^0$: $\spont_1 = -2$, $\spont_2 = -0.5$, $\varsigma = 0.1$,
$\varrho = 2$) Surface area preserving flow.
A plot of $(\Gamma^m_i)_{i=1}^2$ at times $t=0,\ 0.5,\ 1,\ 2$.
Below a plot of the discrete energy $E^{m+1}((\Gamma^m)_{i=1}^2)$.
}
\label{fig:c0udb1}
\end{figure}%
A completely different evolution is obtained when we replace surface area
conservation with volume conservation. This new simulation is visualized in 
Figure~\ref{fig:c0udb2}, where the observed relative volume loss is 
$0.00\%$.
\begin{figure}
\center
\includegraphics[angle=-0,width=0.24\textwidth]{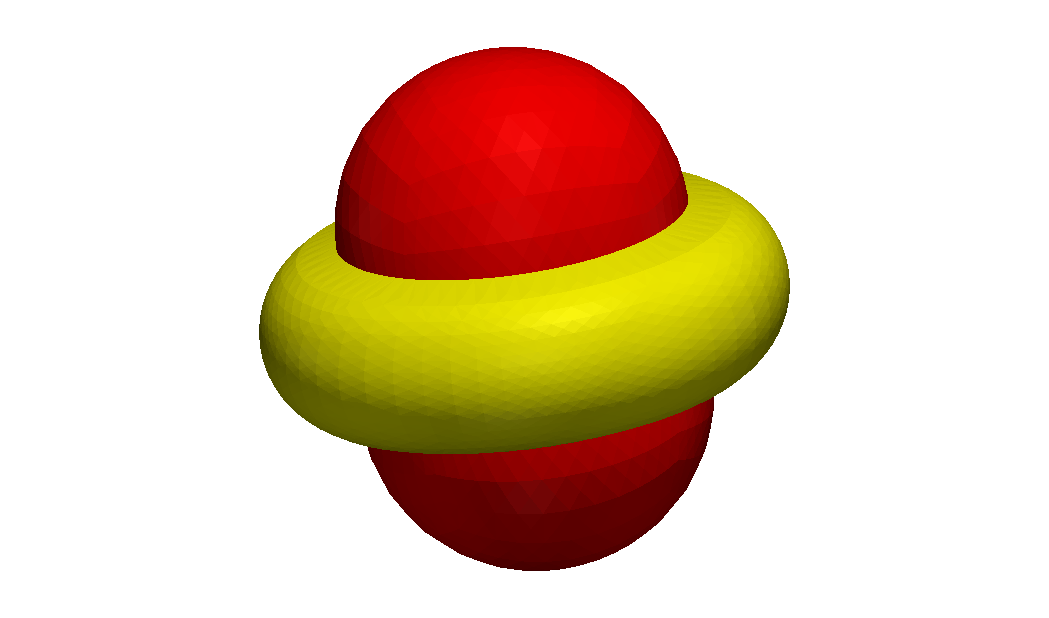}
\includegraphics[angle=-0,width=0.24\textwidth]{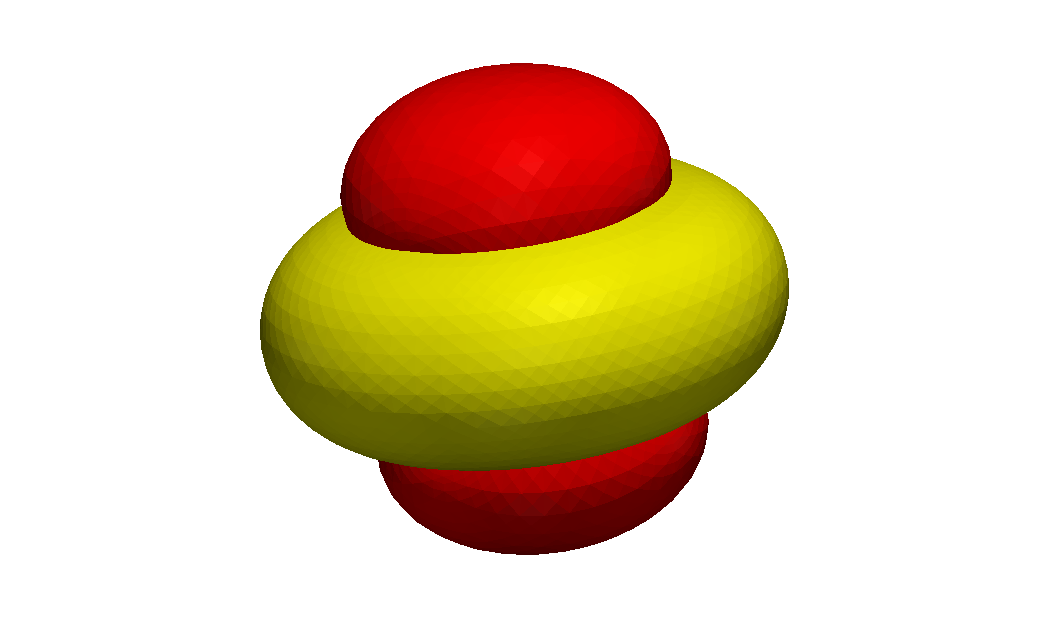}
\includegraphics[angle=-0,width=0.24\textwidth]{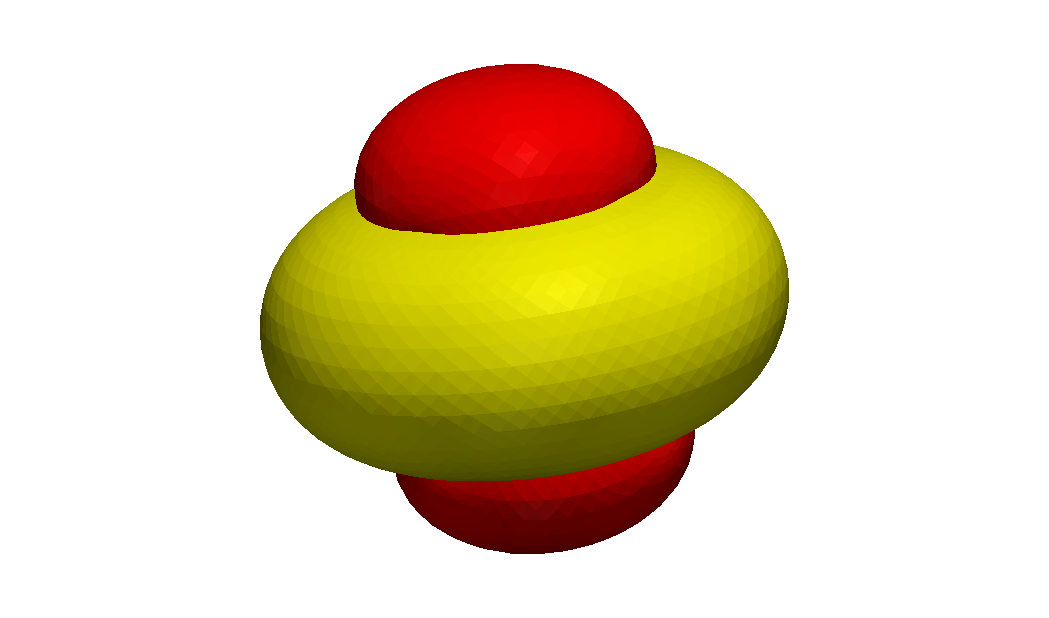} 
\includegraphics[angle=-0,width=0.24\textwidth]{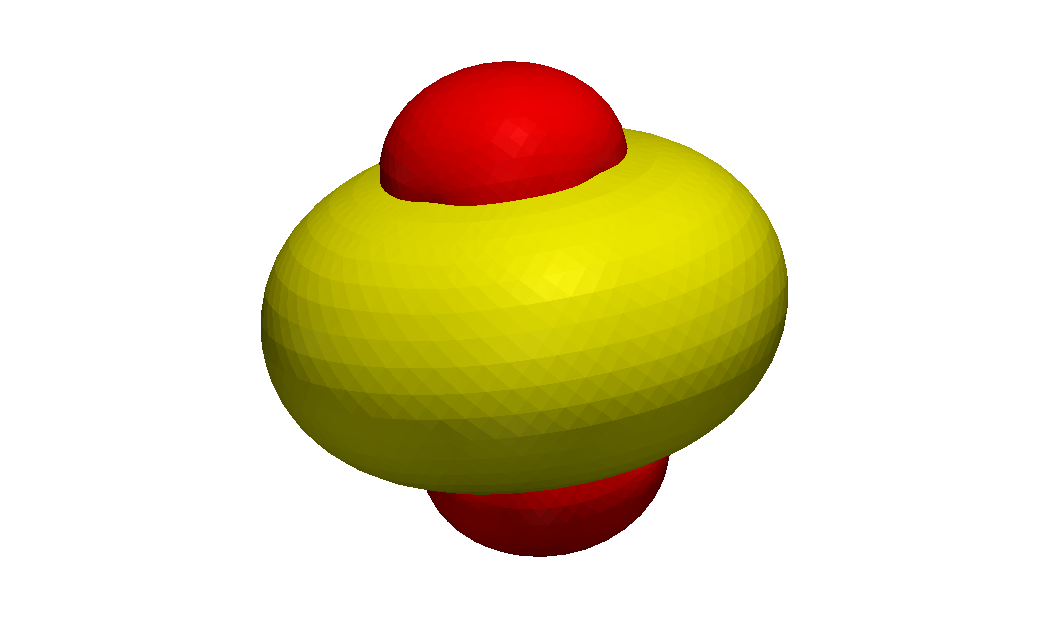} \\
\includegraphics[angle=-90,width=0.32\textwidth]{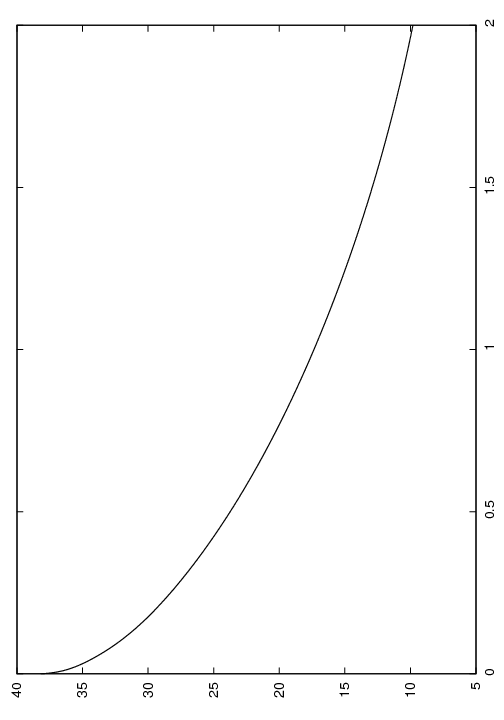}
\caption{($C^0$: $\spont_1 = -2$, $\spont_2 = -0.5$, $\varsigma = 0.1$,
$\varrho = 2$) Volume preserving flow.
A plot of $(\Gamma^m_i)_{i=1}^2$ at times $t=0,\ 0.5,\ 1,\ 2$.
Below a plot of the discrete energy $E^{m+1}((\Gamma^m)_{i=1}^2)$.
}
\label{fig:c0udb2}
\end{figure}%

A simulation with four disconnected components for phase $1$ is shown in
Figure~\ref{fig:c0budding0}. 
The initial surface $\Gamma^0$ 
satisfies $(J_1,J_2) = (1816,4328)$ and $(K_1,K_2) = (1000,2250)$ and
has maximal dimensions $4.2\times4.2\times1.1$.
The evolution for the parameters $\spont_1 = \spont_2 = 0$ and $\varsigma = 1$
goes towards a fournoid.
\begin{figure}
\center
\includegraphics[angle=-0,width=0.24\textwidth]{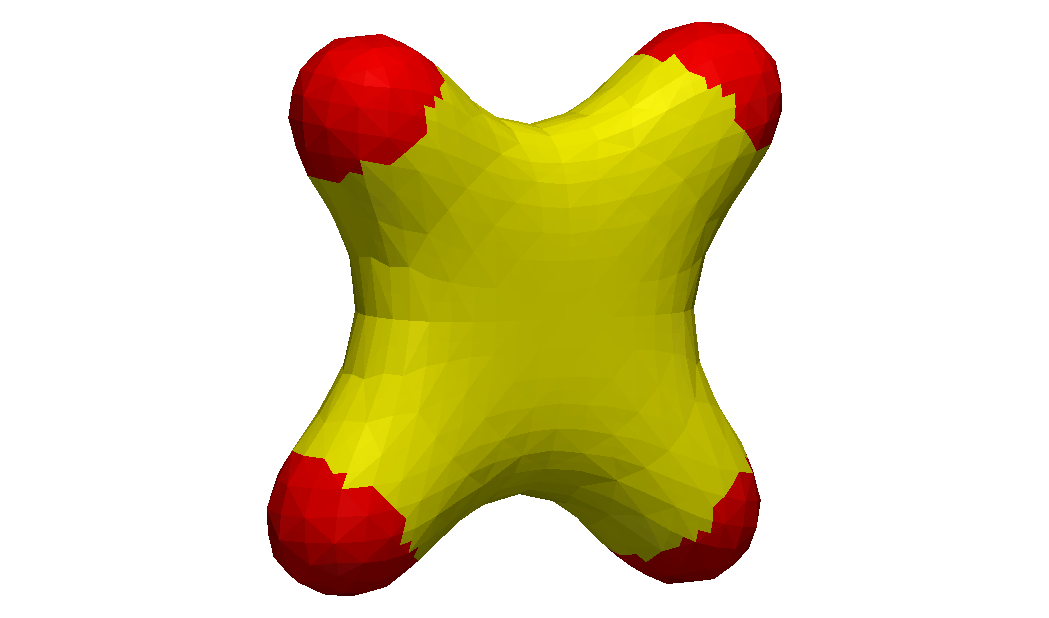}
\includegraphics[angle=-0,width=0.24\textwidth]{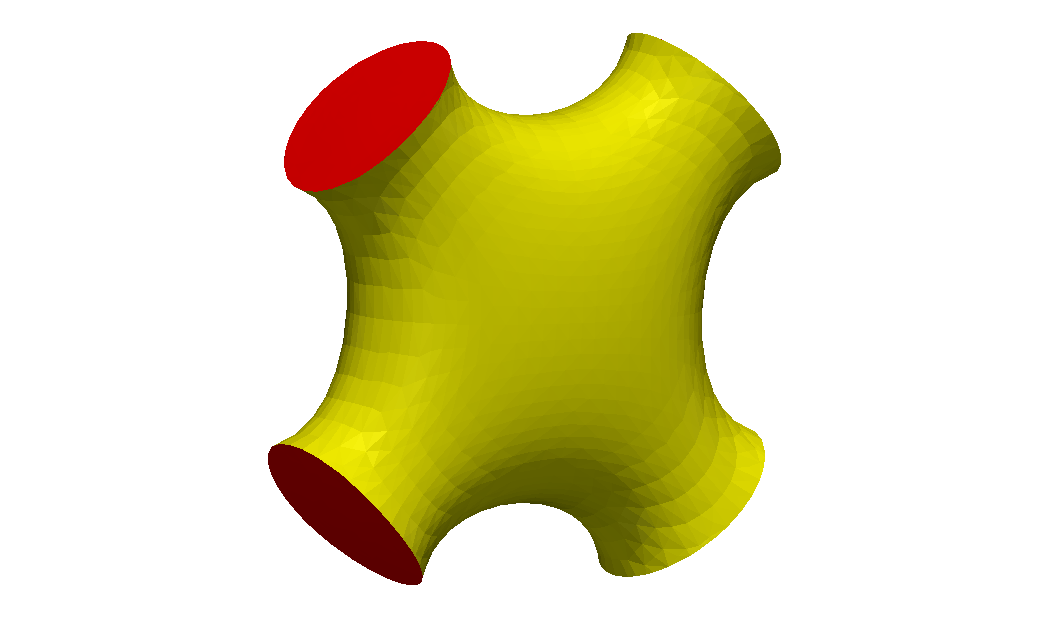}
\includegraphics[angle=-0,width=0.24\textwidth]{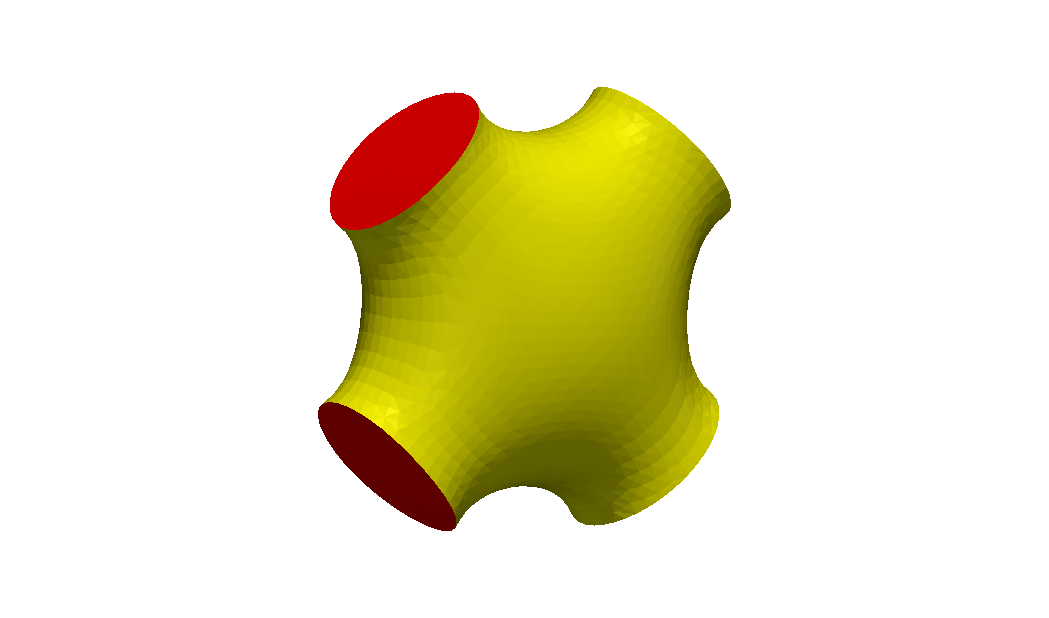}
\includegraphics[angle=-0,width=0.24\textwidth]{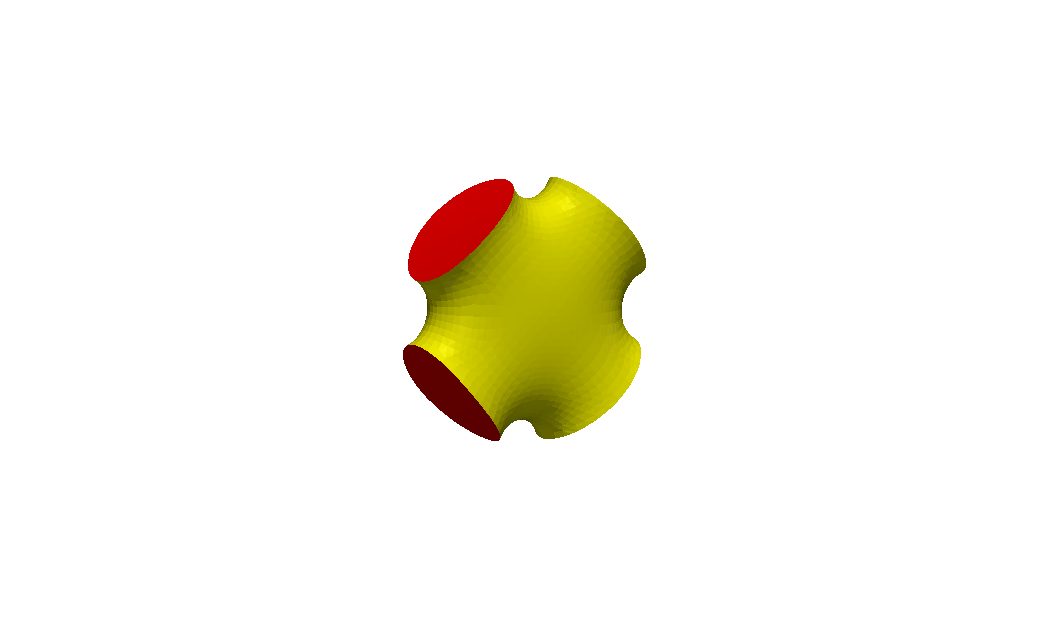}
\includegraphics[angle=-90,width=0.32\textwidth]{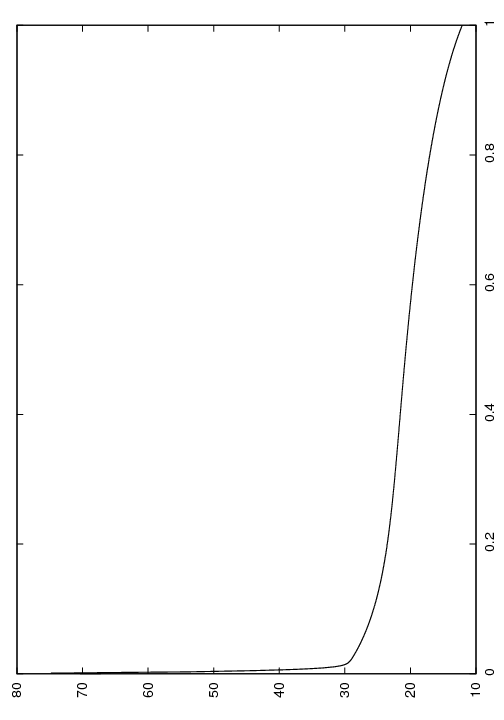}
\caption{($C^0$: $\spont_1 = \spont_2 = 0$, $\varsigma = 1$)
A plot of $(\Gamma^m_i)_{i=1}^2$ at times $t=0,\ 0.1,\ 0.5,\  1$.
Below a plot of the discrete energy $E^{m+1}((\Gamma^m)_{i=1}^2)$.
}
\label{fig:c0budding0}
\end{figure}%

We now consider surface area preserving experiments for setups where
phase $1$ is represented by six or eight disconnected components on the unit
sphere.
For these experiments we use the time step size $\ttau=10^{-4}$
and let $\spont_1 = -4$, $\spont_2 = -2$, $\varsigma = 1$ and $\varrho=2$.
The initial surface $\Gamma^0$ in Figure~\ref{fig:c0sixdot1}
satisfies $(J_1,J_2) = (1032,7160)$ and $(K_1,K_2) = (614,3668)$ and
is an approximation of the unit sphere.
Phase $1$ is made up of six disconnected components.
Here the observed relative surface area loss is $0.36\%$.
\begin{figure}
\center
\includegraphics[angle=-0,width=0.24\textwidth]{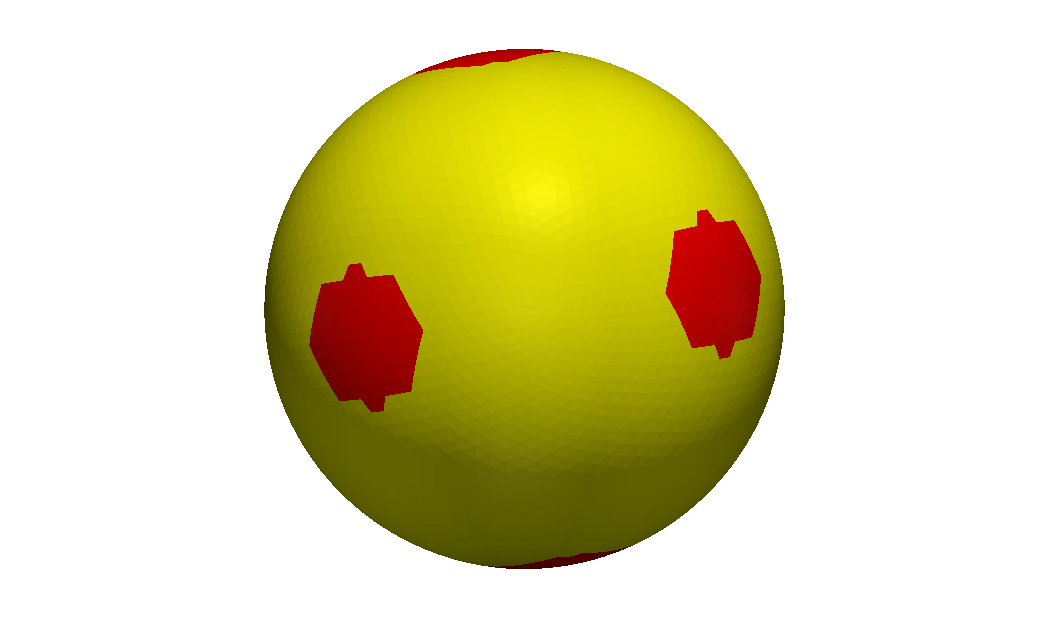}
\includegraphics[angle=-0,width=0.24\textwidth]{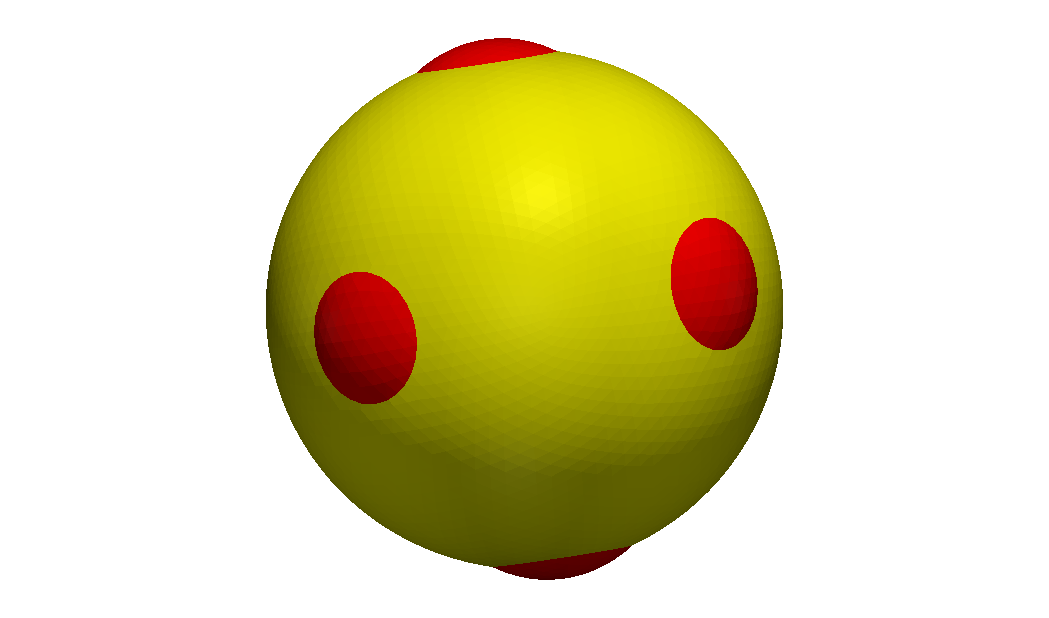}
\includegraphics[angle=-0,width=0.24\textwidth]{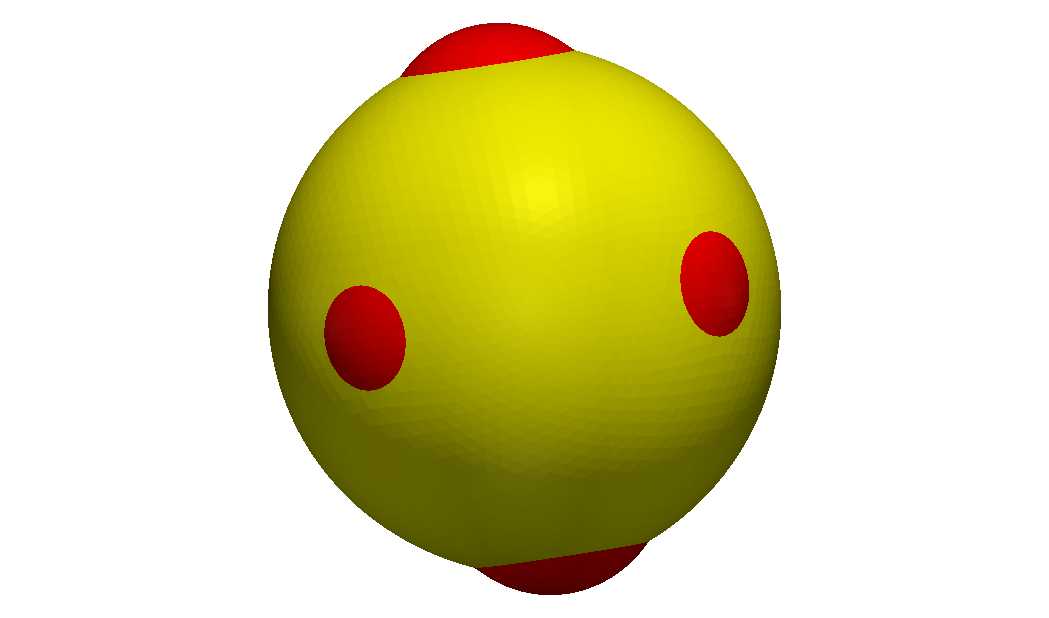}
\includegraphics[angle=-0,width=0.24\textwidth]{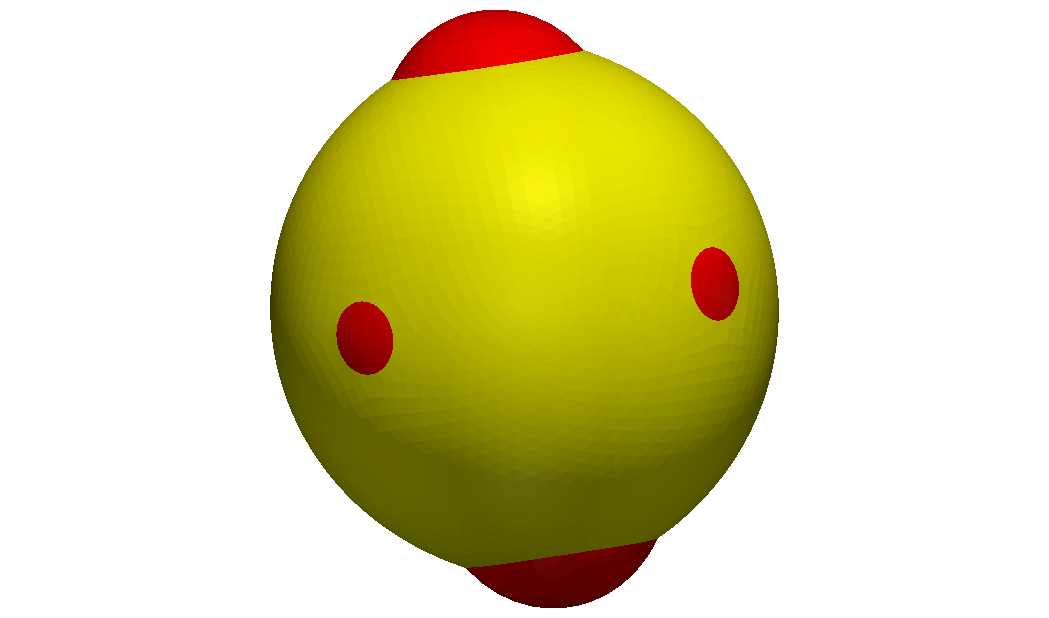}
\includegraphics[angle=-90,width=0.32\textwidth]{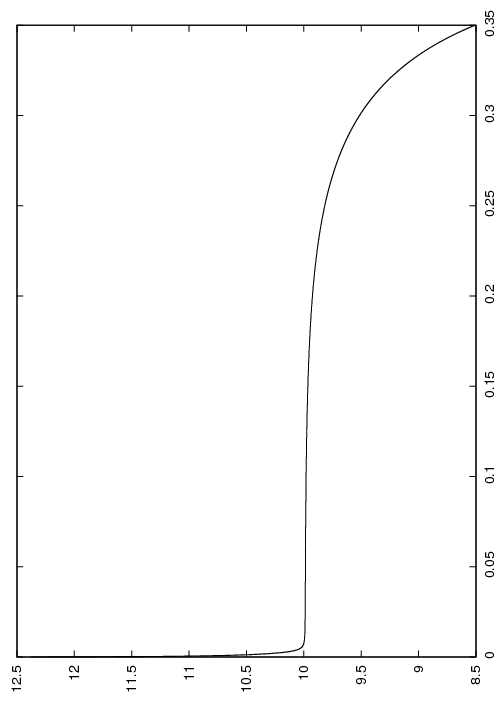}
\caption{($C^0$: $\spont_1 = -4$, $\spont_2 = -2$, $\varsigma = 1$,
$\varrho=2$)
Surface area preserving flow.
A plot of $(\Gamma^m_i)_{i=1}^2$ at times $t=0,\ 0.1,\ 0.3,\ 0.35$.
Below a plot of the discrete energy $E^{m+1}((\Gamma^m)_{i=1}^2)$.
}
\label{fig:c0sixdot1}
\end{figure}%
A simulation with eight disconnected components for phase $1$
is shown in Figure~\ref{fig:c0eightdot1}.
The initial surface $\Gamma^0$ 
satisfies $(J_1,J_2) = (2048,6144)$ and $(K_1,K_2) = (1184,3218)$. 
Here the observed relative surface area loss is $0.28\%$.
\begin{figure}
\center
\includegraphics[angle=-0,width=0.24\textwidth]{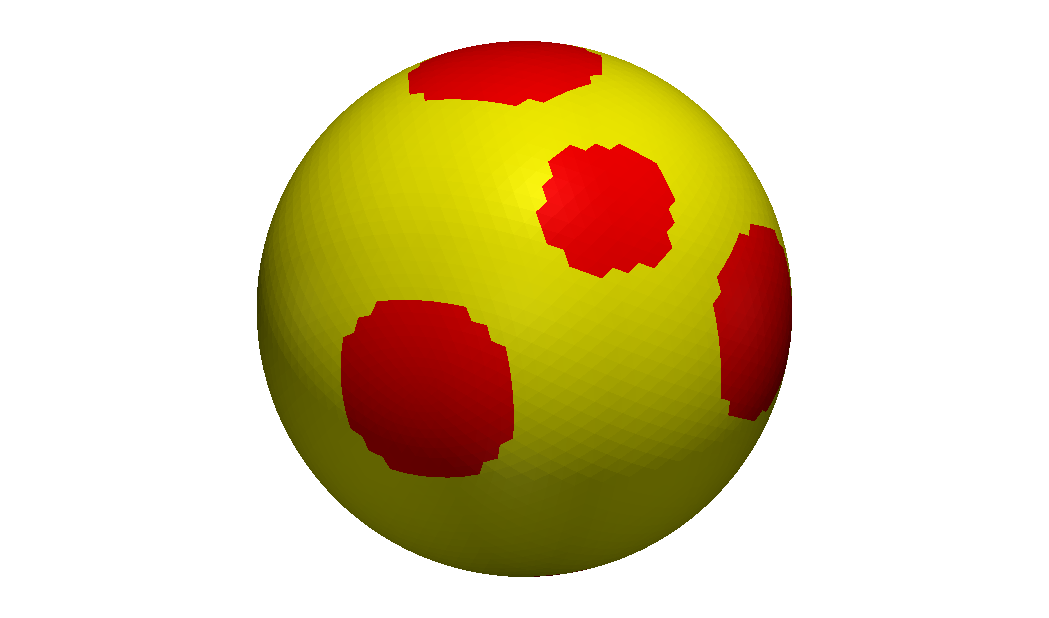}
\includegraphics[angle=-0,width=0.24\textwidth]{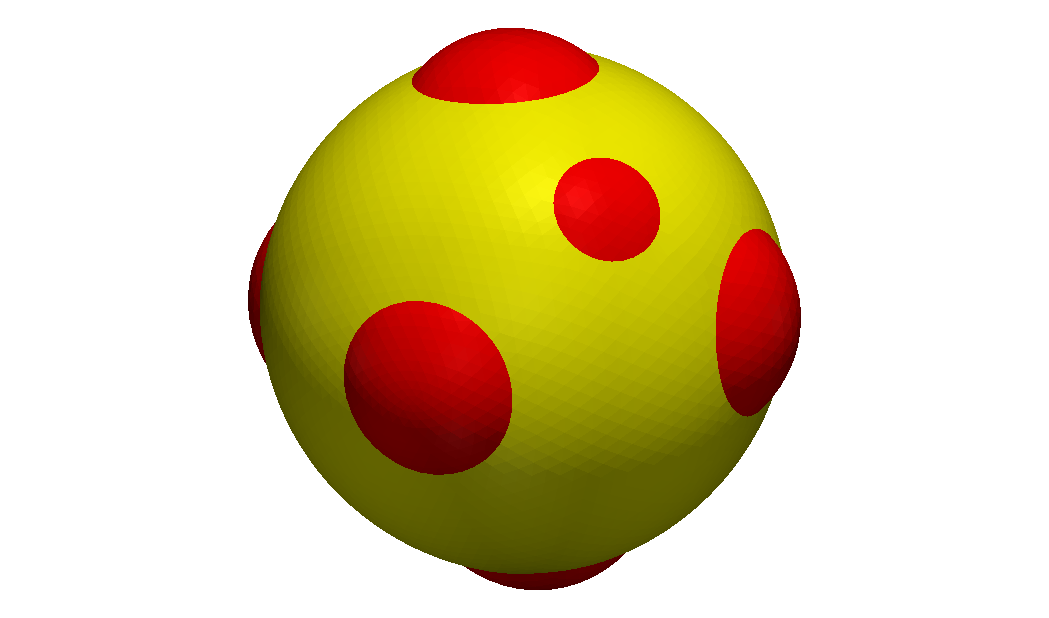}
\includegraphics[angle=-0,width=0.24\textwidth]{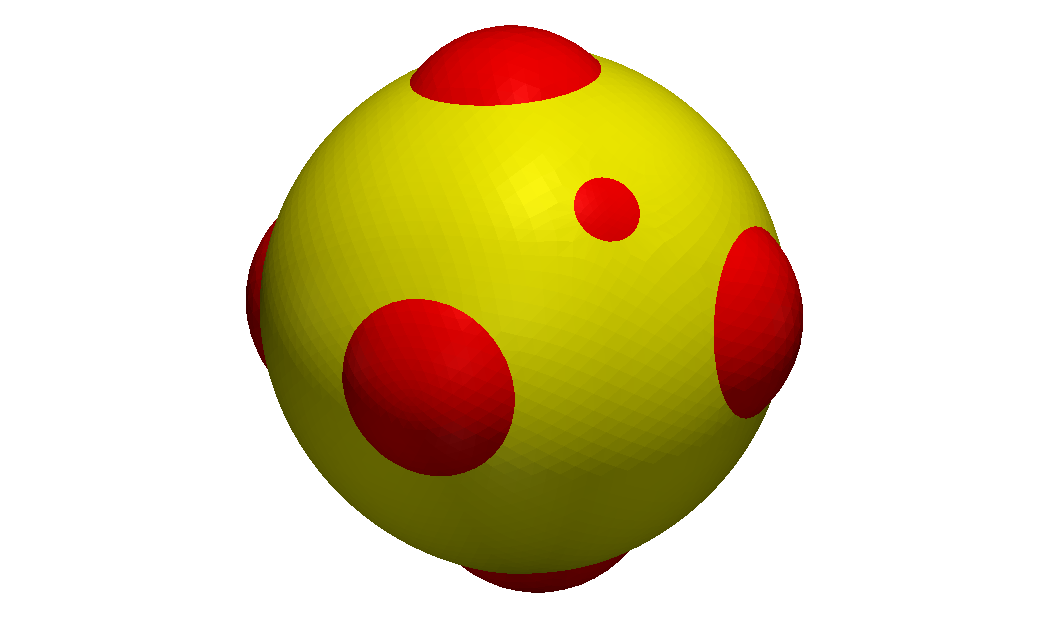}
\includegraphics[angle=-0,width=0.24\textwidth]{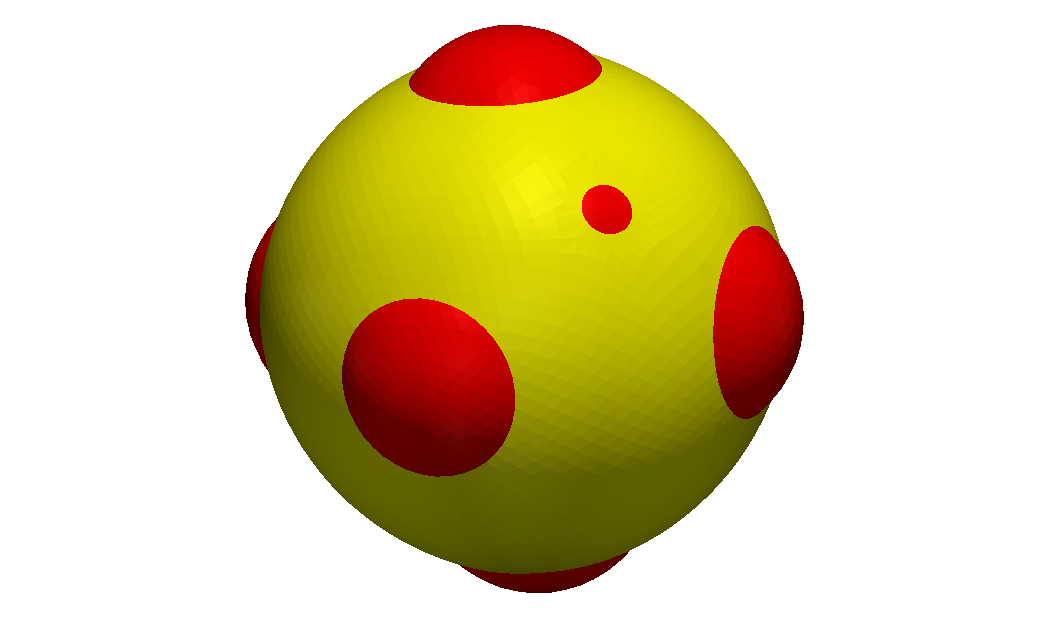}
\includegraphics[angle=-90,width=0.32\textwidth]{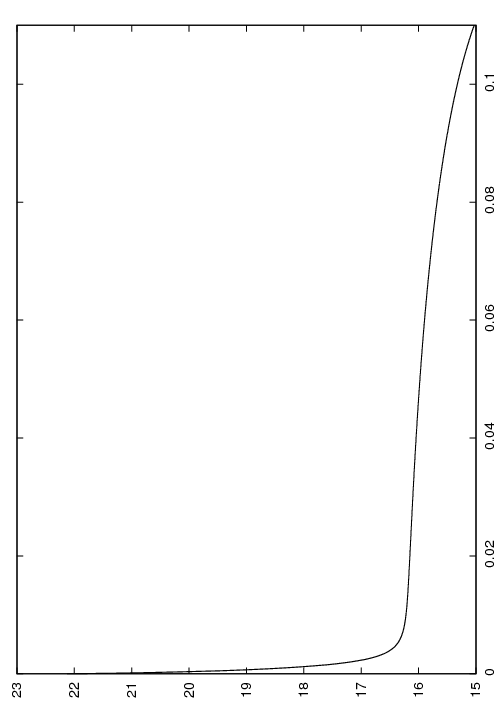}
\caption{($C^0$: $\spont_1 = -4$, $\spont_2 = -2$, $\varsigma = 1$,
$\varrho=2$)
Surface area preserving flow.
A plot of $(\Gamma^m_i)_{i=1}^2$ at times $t=0,\ 0.05,\ 0.1,\ 0.11$.
Below a plot of the discrete energy $E^{m+1}((\Gamma^m)_{i=1}^2)$.
}
\label{fig:c0eightdot1}
\end{figure}%

An example for volume and surface area preserving flow
is shown in Figure~\ref{fig:c0volarea}. 
The initial surface $\Gamma^0$ 
satisfies $(J_1,J_2) = (2274,2274)$ and $(K_1,K_2) = (1188,1188)$ and
has maximal dimensions $1.5\times1.5\times2.8$.
In this experiment we choose $\spont_1 = \spont_2 = -1$, $\varsigma = 1$ and
$\varrho = 2$.
The relative surface area loss for this experiment is $0.07\%$, while the
relative volume loss is $0.00\%$.
\begin{figure}
\center
\includegraphics[angle=-0,width=0.24\textwidth]{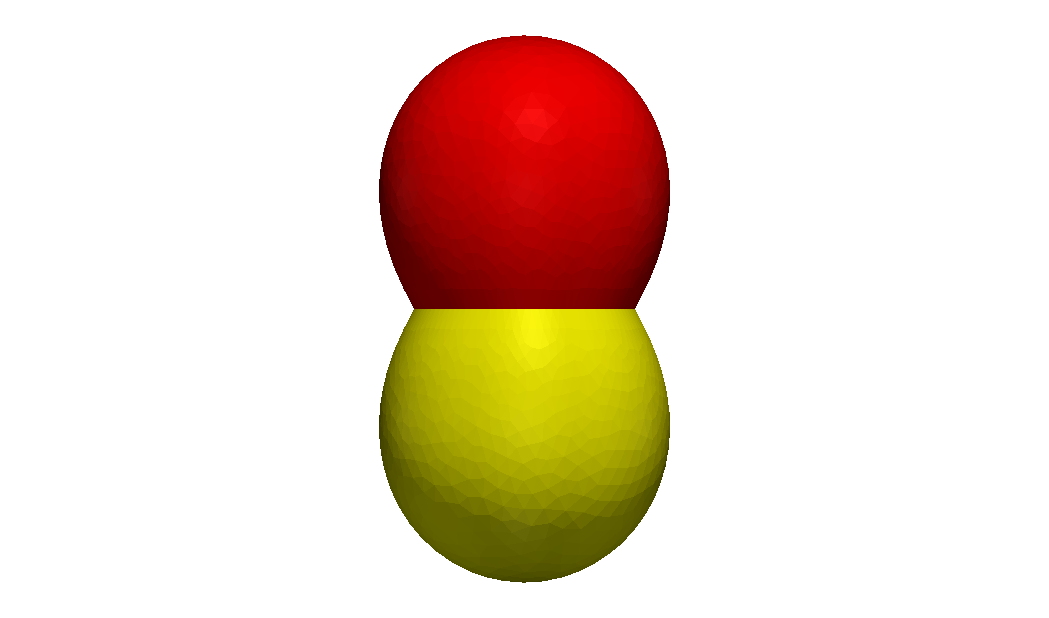}
\includegraphics[angle=-0,width=0.24\textwidth]{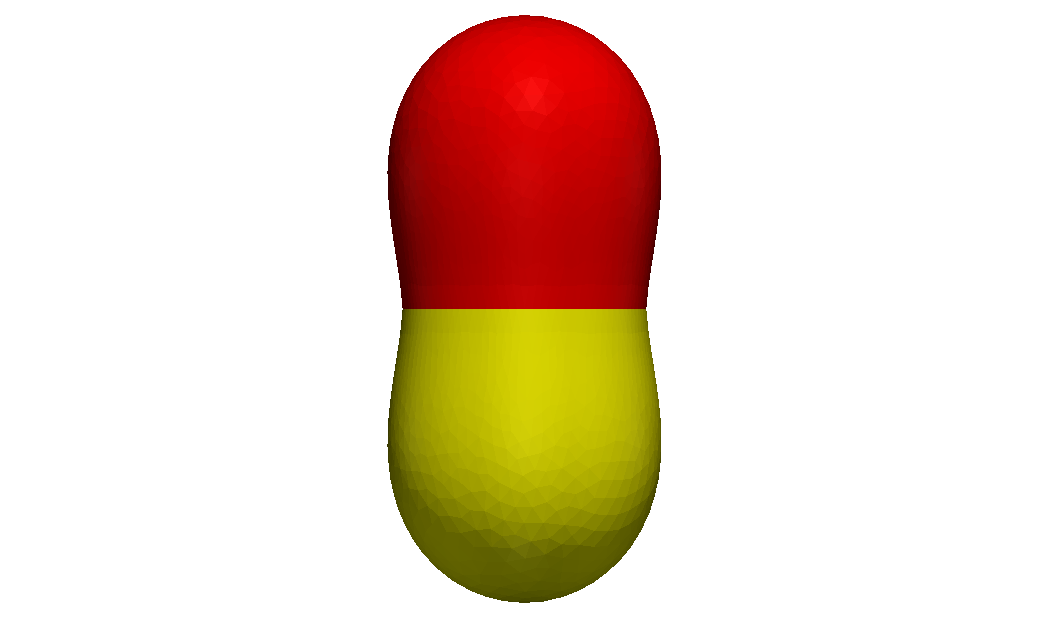}
\includegraphics[angle=-0,width=0.24\textwidth]{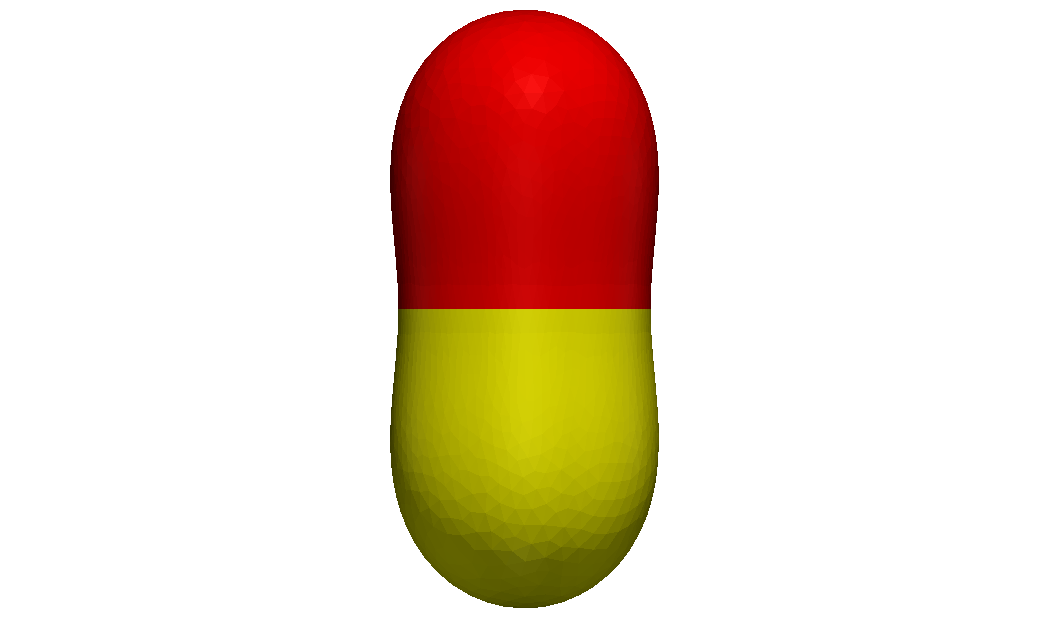} 
\includegraphics[angle=-0,width=0.24\textwidth]{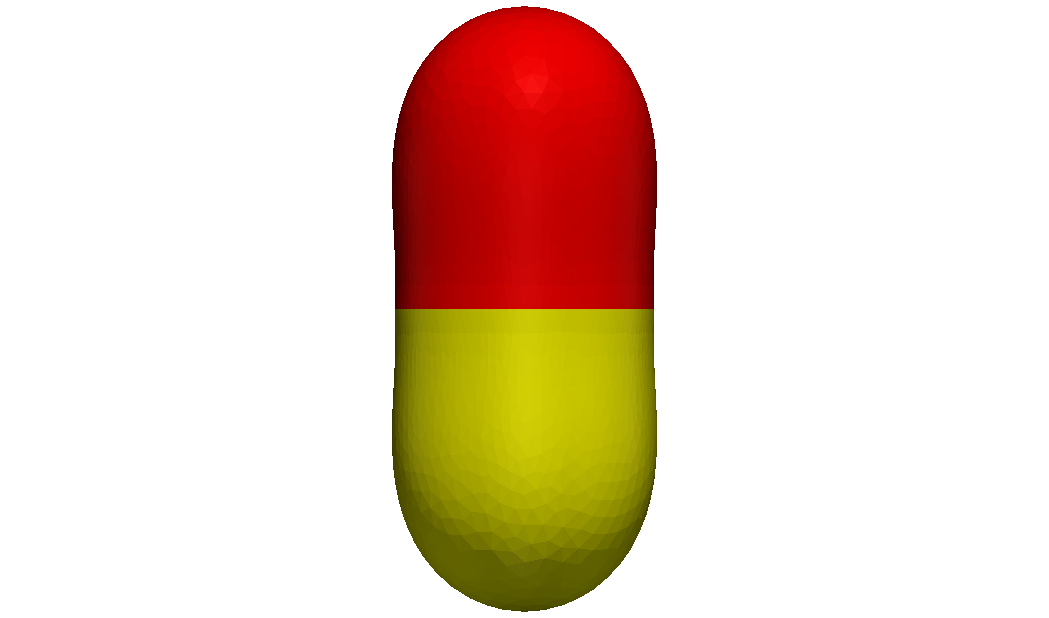} \\
\includegraphics[angle=-90,width=0.32\textwidth]{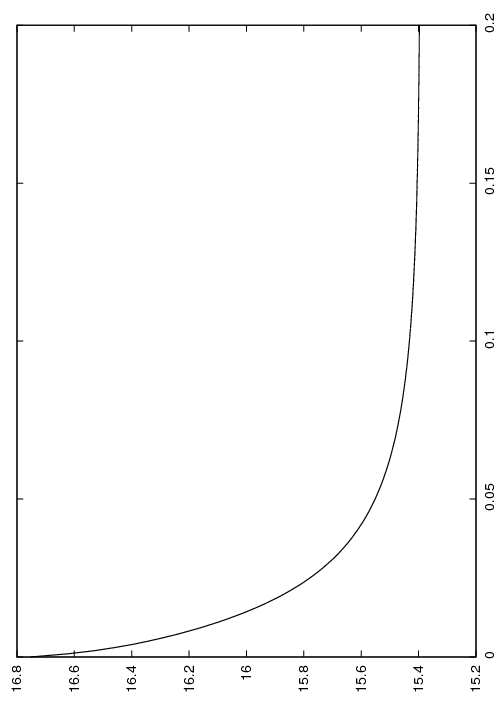}
\caption{($C^0$: $\spont_1 = \spont_2 = -1$, $\varsigma = 1$,
$\varrho = 2$)
Volume and surface area preserving flow.
A plot of $(\Gamma^m_i)_{i=1}^2$ at times $t=0,\ 0.05,\ 0.1,\ 0.2$.
Below a plot of the discrete energy $E^{m+1}((\Gamma^m)_{i=1}^2)$.
}
\label{fig:c0volarea}
\end{figure}%

The next set of experiments illustrates the impact of the Gaussian curvature
energy.
The initial surface $\Gamma^0$ is made up of two halves of 
an approximation of the unit sphere and
satisfies $(J_1,J_2) = (2274,2274)$ and $(K_1,K_2) = (1188,1188)$.
An experiment for $\spont_1 = \spont_2 = 0$, $\varsigma = 1$ and $\varrho=2$ 
is shown in Figure~\ref{fig:c0exp0_lvr}. 
The evolution eventually reaches a slowly shrinking disk.
\begin{figure}
\center
\includegraphics[angle=-0,width=0.24\textwidth]{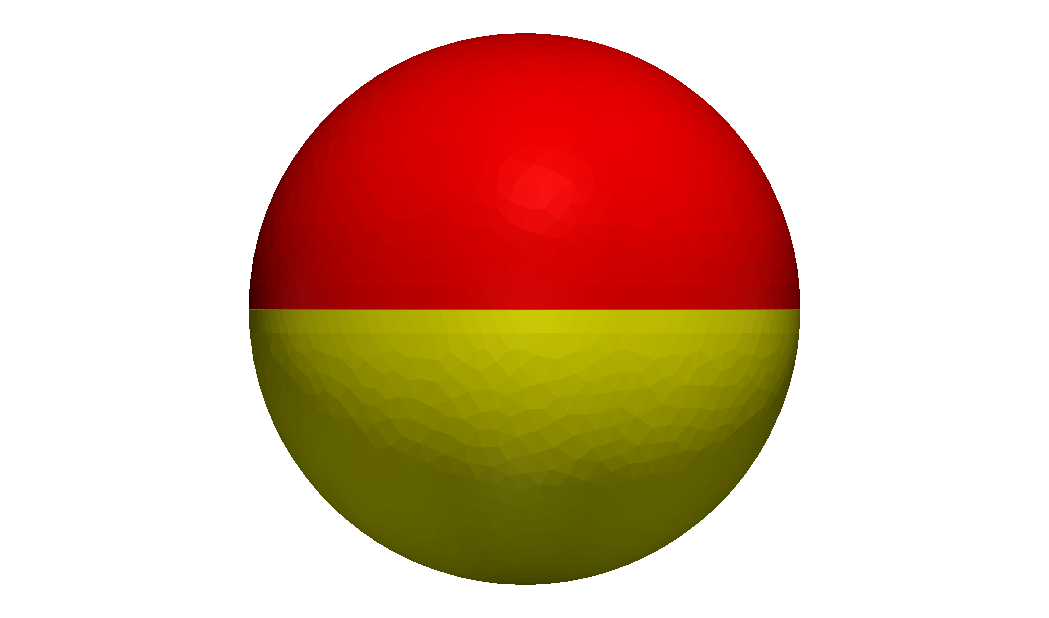}
\includegraphics[angle=-0,width=0.24\textwidth]{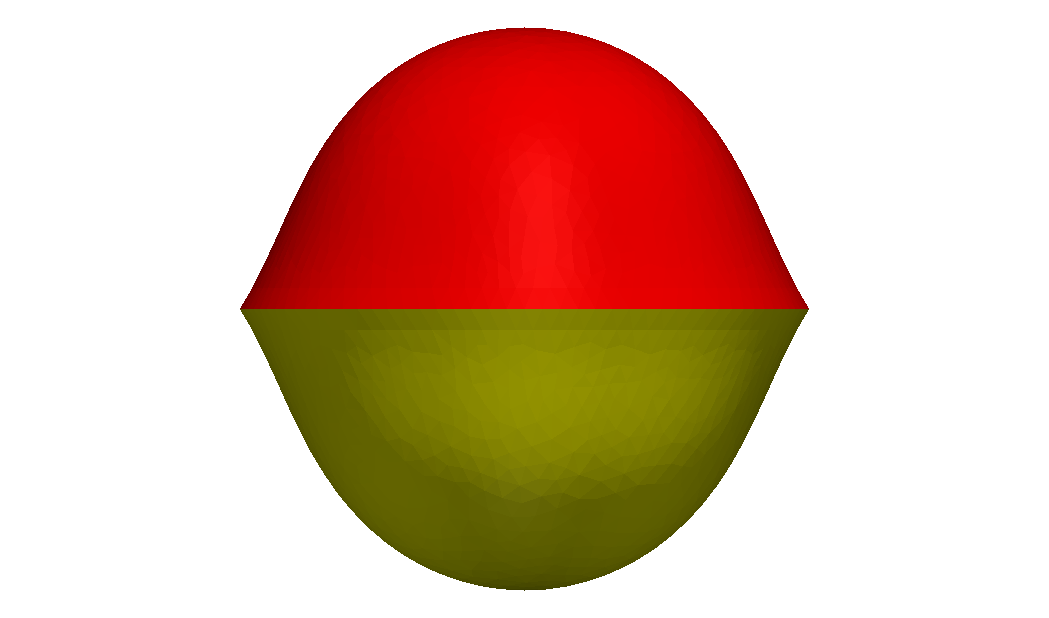}
\includegraphics[angle=-0,width=0.24\textwidth]{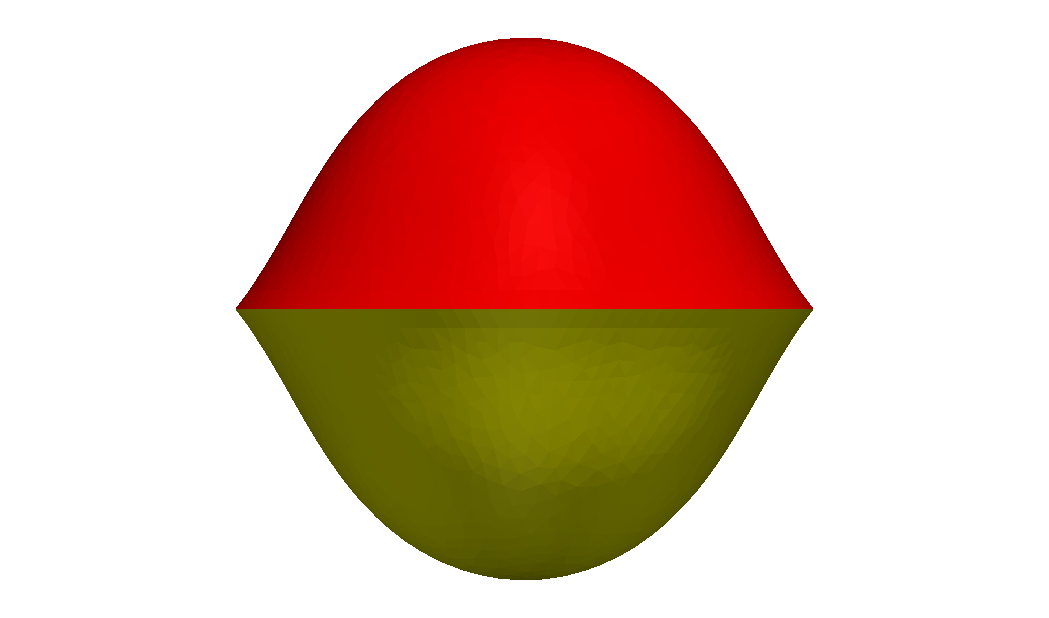} \\
\includegraphics[angle=-0,width=0.24\textwidth]{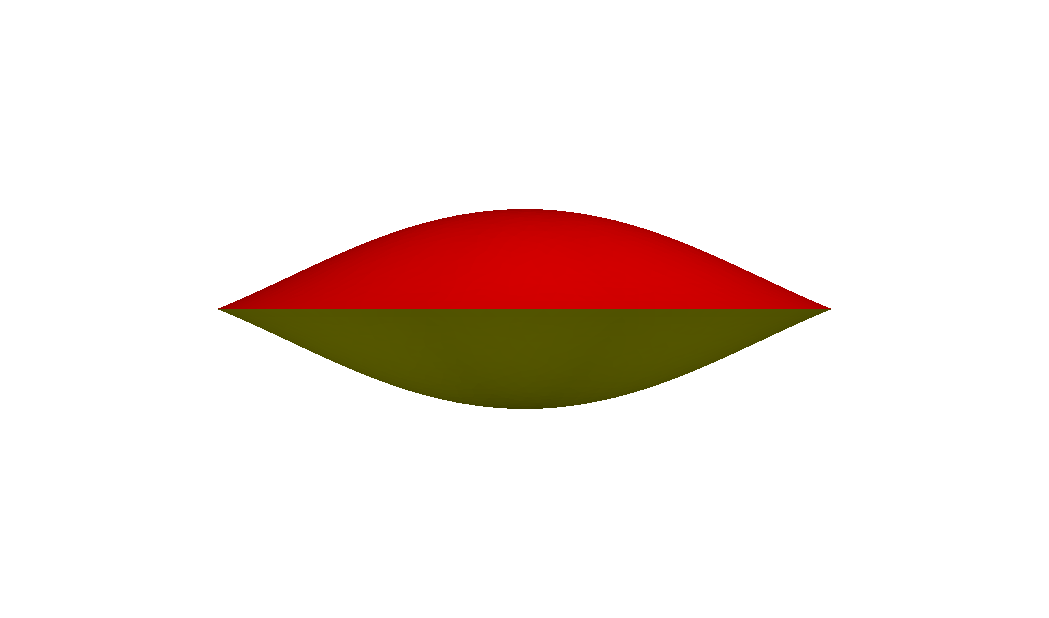}
\includegraphics[angle=-0,width=0.24\textwidth]{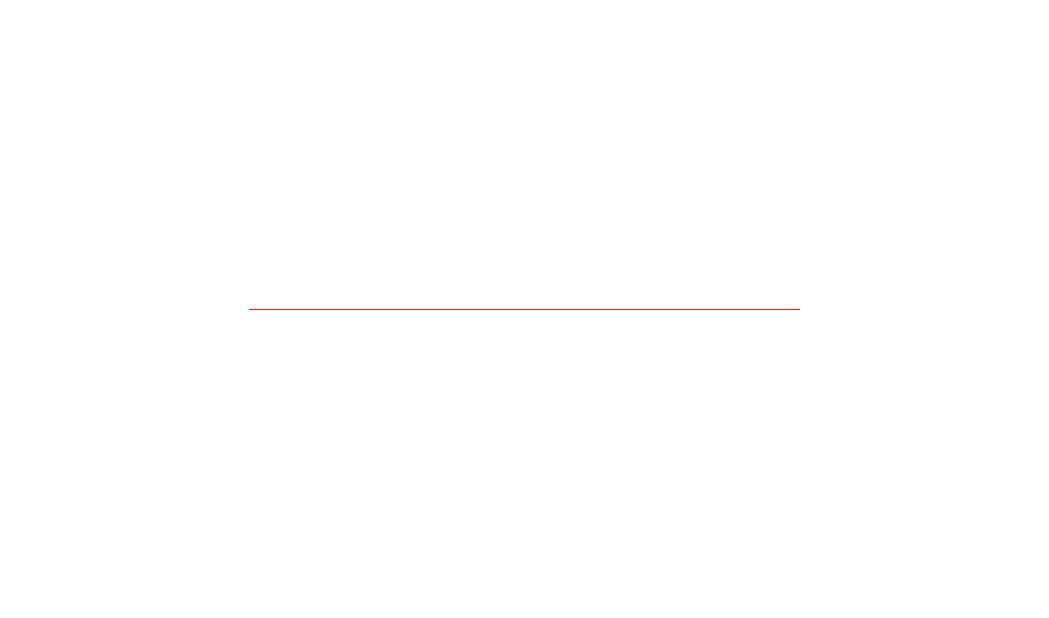} \\
\includegraphics[angle=-90,width=0.32\textwidth]{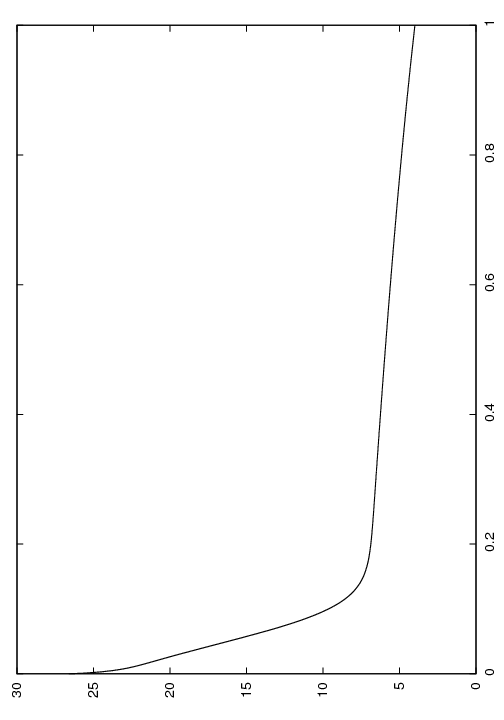}
\caption{($C^0$: $\spont_1 = \spont_2 = 0$, $\varsigma = 1$,
$\varrho=2$)
A plot of $(\Gamma^m_i)_{i=1}^2$ at times $t=0,\ 0.01,\ 0.02,\ 0.1,\ 1$.
At time $t=1$ the evolution has reached a disk.
Below a plot of the discrete energy $E^{m+1}((\Gamma^m)_{i=1}^2)$.
}
\label{fig:c0exp0_lvr}
\end{figure}%
Choosing the parameters $\alpha^G_1 = \alpha^G_2 = -1$, and using the time 
step size $\ttau=10^{-5}$, we obtain the simulation in 
Figure~\ref{fig:c0gauss1}.
We remark that the conditions (\ref{eq:C0alphaGbound}) trivially hold.
Moreover, and in contrast to the $C^1$--case, a nonzero Gaussian bending energy
coefficient has an influence on the evolution even if
$\alpha^G_1 = \alpha^G_2$. In this example we observe that for a negative
$\alpha^G_1 = \alpha^G_2$, the term
$\sum_{i=1}^2\alpha^G_i\,\int_{\Gamma_i} \Gauss_i \dH{2}$ for the initial
sphere is negative, and hence the evolution remains convex throughout, in
contrast to the evolution in Figure~\ref{fig:c0exp0_lvr}. Moreover,
the evolution in Figure~\ref{fig:c0gauss1} is generally slower, since
large values of the Gaussian curvatures make 
$\sum_{i=1}^2 \alpha^G_i\,\int_{\Gamma_i} \Gauss_i \dH{2}$ more negative.
\begin{figure}
\center
\includegraphics[angle=-0,width=0.24\textwidth]{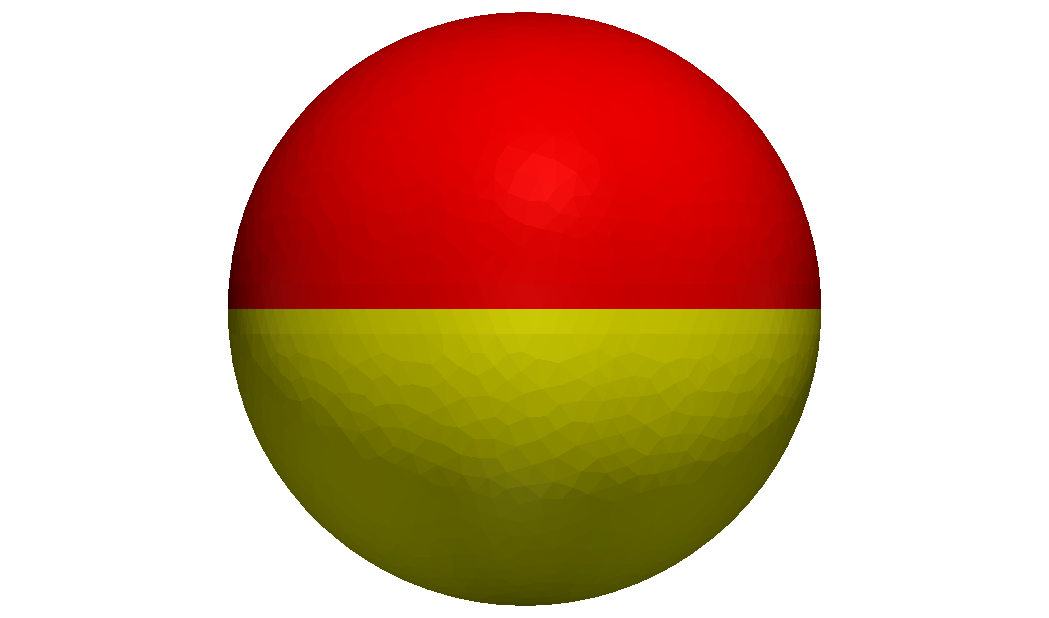}
\includegraphics[angle=-0,width=0.24\textwidth]{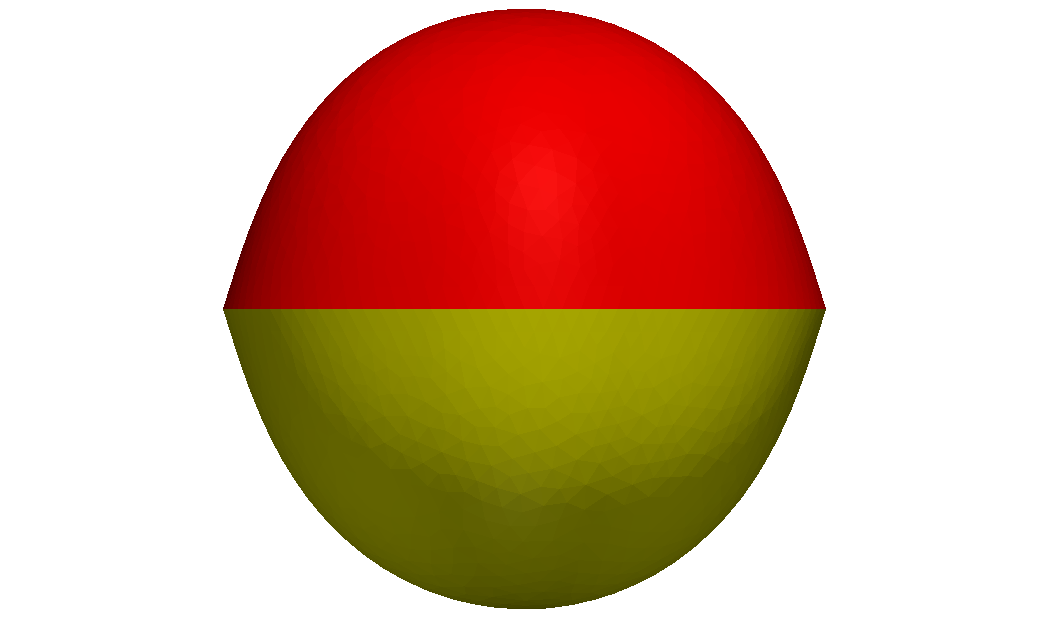}
\includegraphics[angle=-0,width=0.24\textwidth]{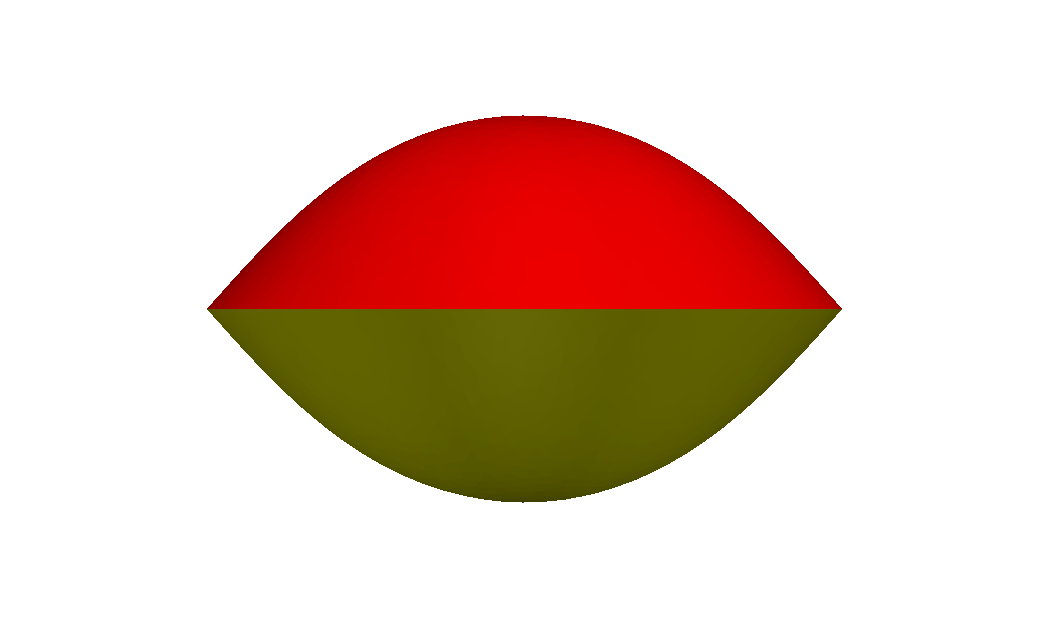}
\includegraphics[angle=-0,width=0.24\textwidth]{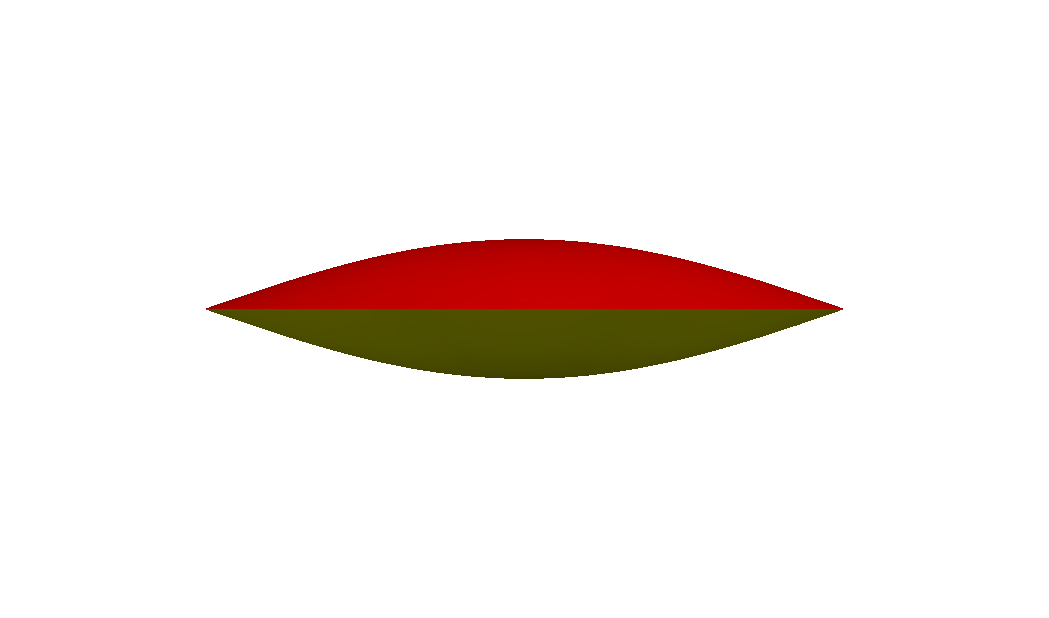}
\includegraphics[angle=-90,width=0.32\textwidth]{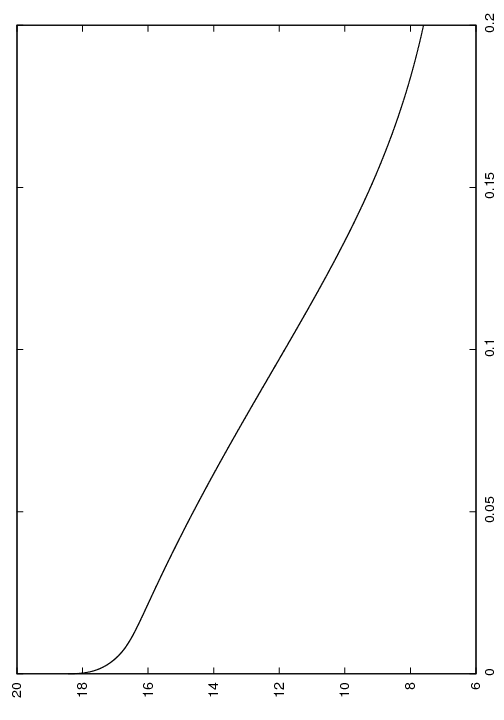}
\caption{($C^0$: $\spont_1 = \spont_2 = 0$, $\varsigma = 1$,
$\alpha^G_1=\alpha^G_2=-1$, $\varrho=2$)
A plot of $(\Gamma^m_i)_{i=1}^2$ at times $t=0,\ 0.01,\ 0.1,\ 0.2$.
Below a plot of the discrete energy $E^{m+1}((\Gamma^m)_{i=1}^2)$.
}
\label{fig:c0gauss1}
\end{figure}%
Repeating the computation for $\alpha^G_1=-1$ and $\alpha^G_2=-1.5$ yields the
results in Figure~\ref{fig:c0gauss2}. We note once again that the conditions
(\ref{eq:C0alphaGbound}) hold.
For the evolution in Figure~\ref{fig:c0gauss2} we observe that the curvature 
of phase 2 is decreasing slower due to the fact that
large values of $\int_{\Gamma_2} \Gauss_2 \dH{2}$ decrease the energy.
\begin{figure}
\center
\includegraphics[angle=-0,width=0.24\textwidth]{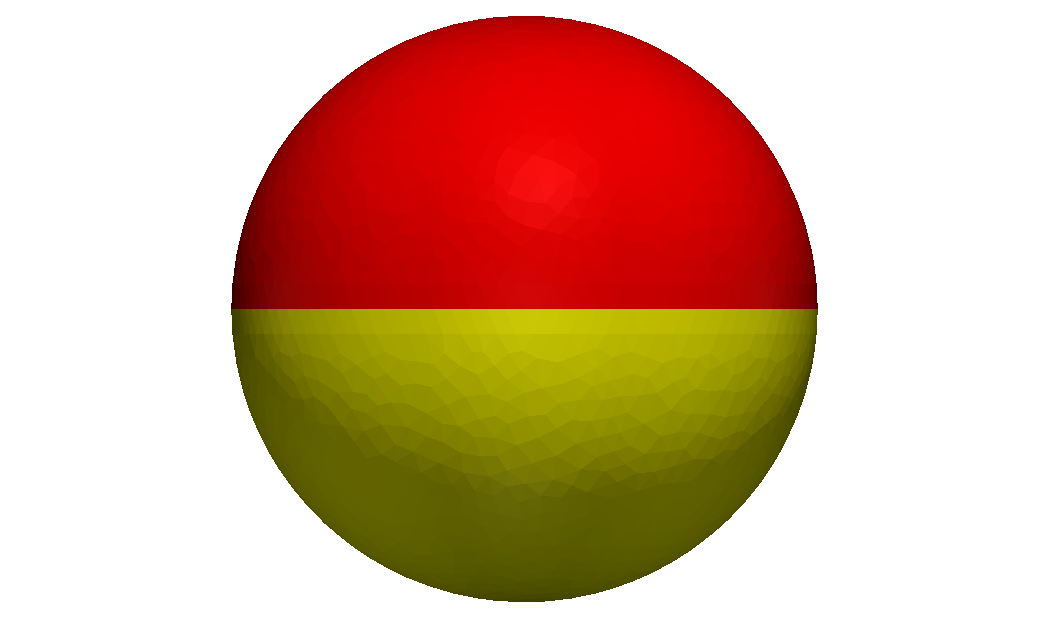}
\includegraphics[angle=-0,width=0.24\textwidth]{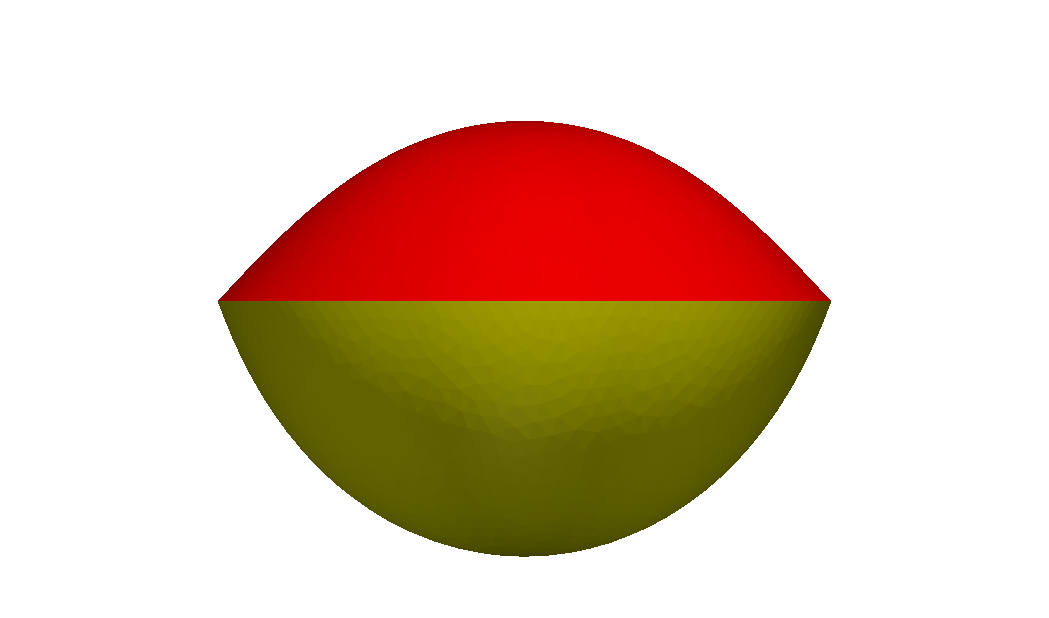}
\includegraphics[angle=-0,width=0.24\textwidth]{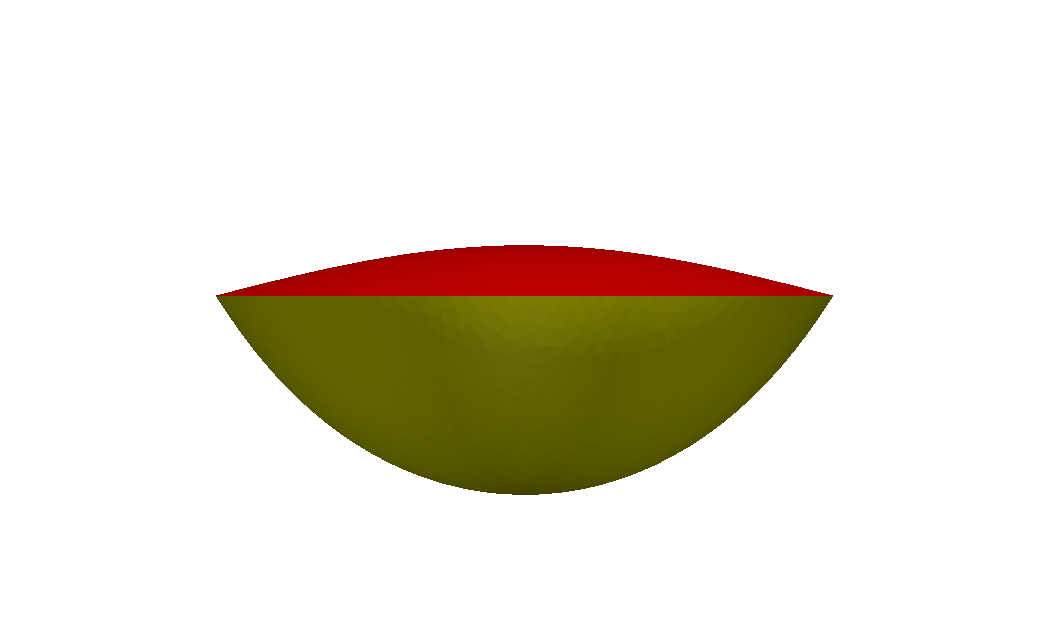} 
\includegraphics[angle=-0,width=0.24\textwidth]{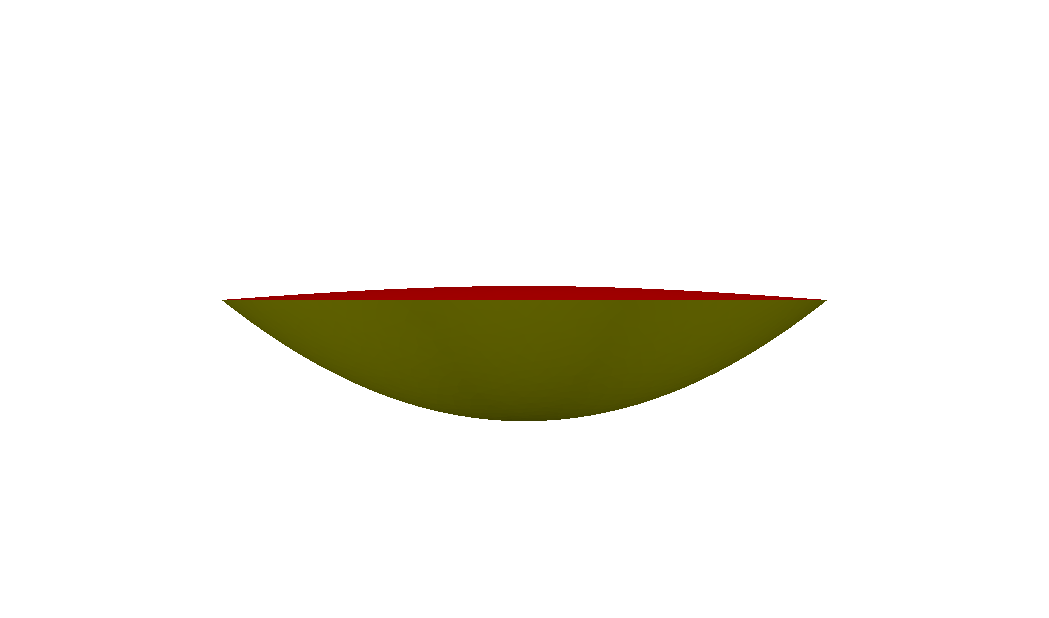} 
\includegraphics[angle=-90,width=0.32\textwidth]{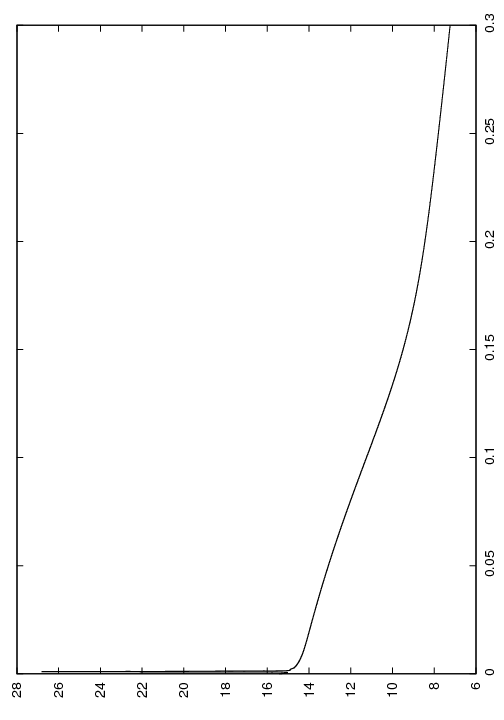}
\caption{($C^0$: $\spont_1 = \spont_2 = 0$, $\varsigma = 1$,
$\alpha^G_1=-1$, $\alpha^G_2=-1.5$, $\varrho=2$)
A plot of $(\Gamma^m_i)_{i=1}^2$ at times $t=0,\ 0.1,\ 0.2,\ 0.3$.
Below a plot of the discrete energy $E^{m+1}((\Gamma^m)_{i=1}^2)$.
}
\label{fig:c0gauss2}
\end{figure}%

\subsection{The $C^1$--case}

We remark that in the $C^1$--case, with uniform data $\alpha_1 = \alpha_2 =
\alpha$, $\spont_1 = \spont_2 = \spont$ and 
$\varsigma = \varrho = \alpha^G_1=\alpha^G_2 = 0$, 
our finite element approximation
collapses to the scheme from \cite{pwfade} for the Willmore flow of closed
surfaces. Indeed, as a numerical check we confirmed that Table~1 in
\cite{pwfade}, for the approximation of 
the nonlinear ODE \cite[(5.1)]{pwfade}, 
is reproduced exactly by our implementation of the scheme (\ref{eq:GF2a}--e).

A repeat of the simulation in Figure~\ref{fig:c0udb0_varrho}
in the context of a $C^1$--condition on $\gamma$ is shown in 
Figure~\ref{fig:c1udb0_varrho}, where for this experiment we use the time 
step size $\ttau=10^{-4}$.
The evolution goes towards a cylinder with two round caps, which is
dramatically different to the evolution in the $C^0$--case.
\begin{figure}
\center
\includegraphics[angle=-0,width=0.24\textwidth]{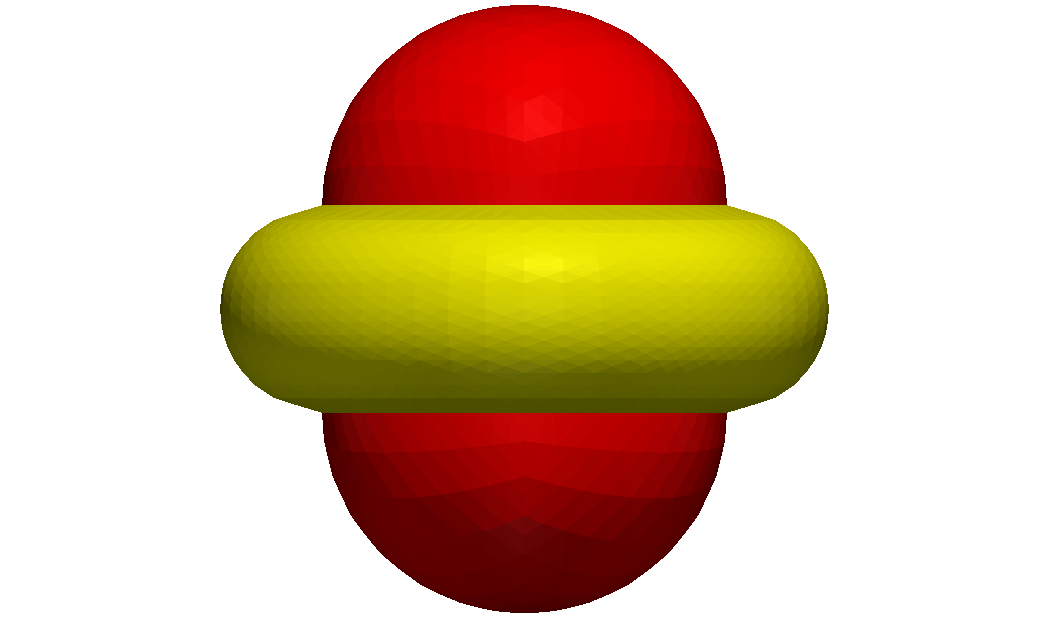}
\includegraphics[angle=-0,width=0.24\textwidth]{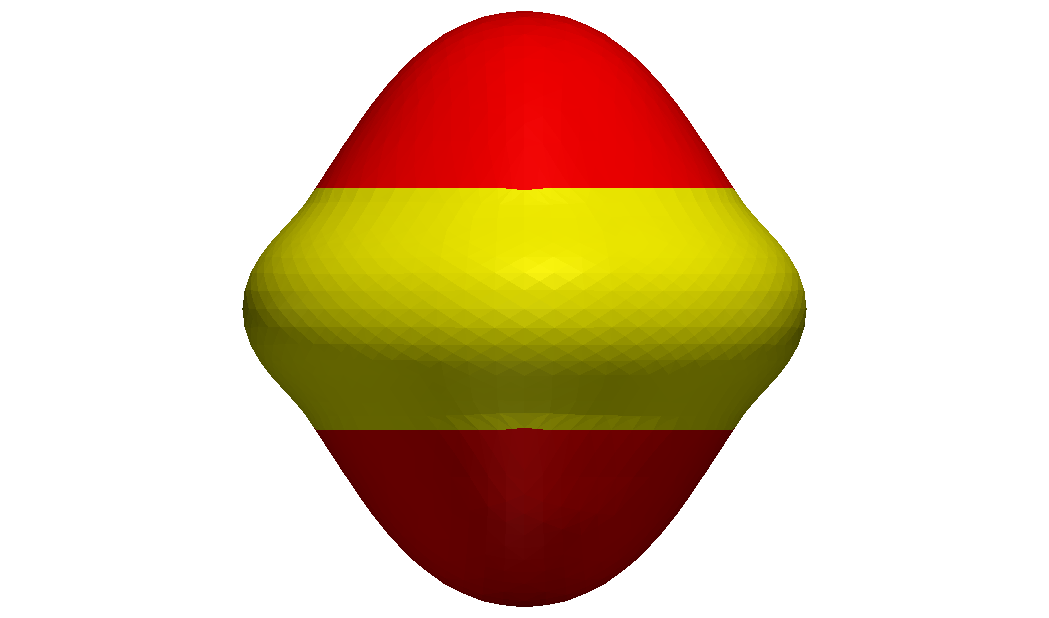}
\includegraphics[angle=-0,width=0.24\textwidth]{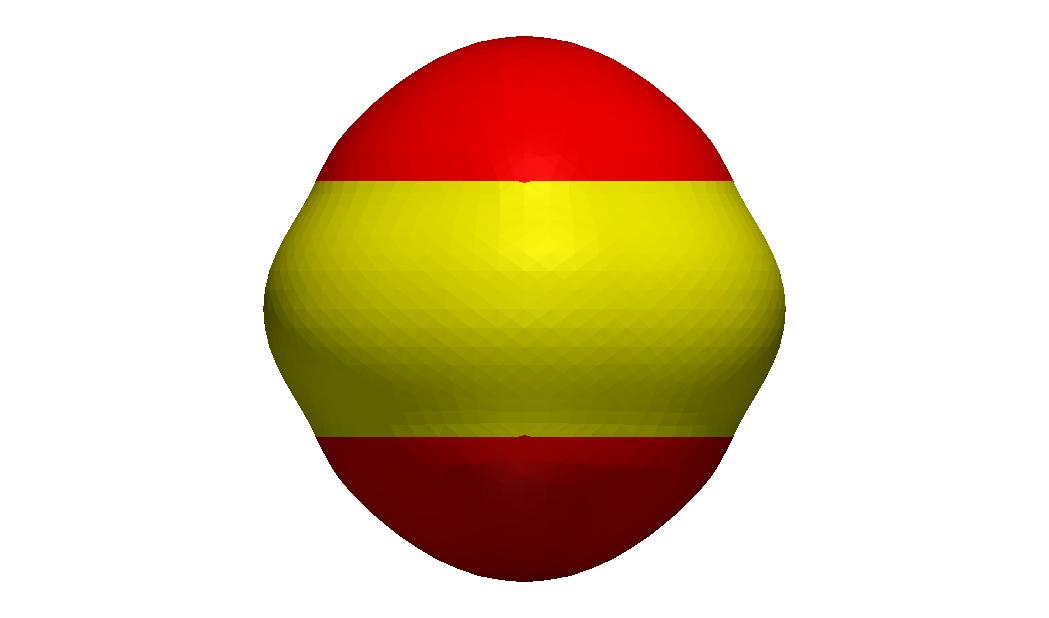} 
\includegraphics[angle=-0,width=0.24\textwidth]{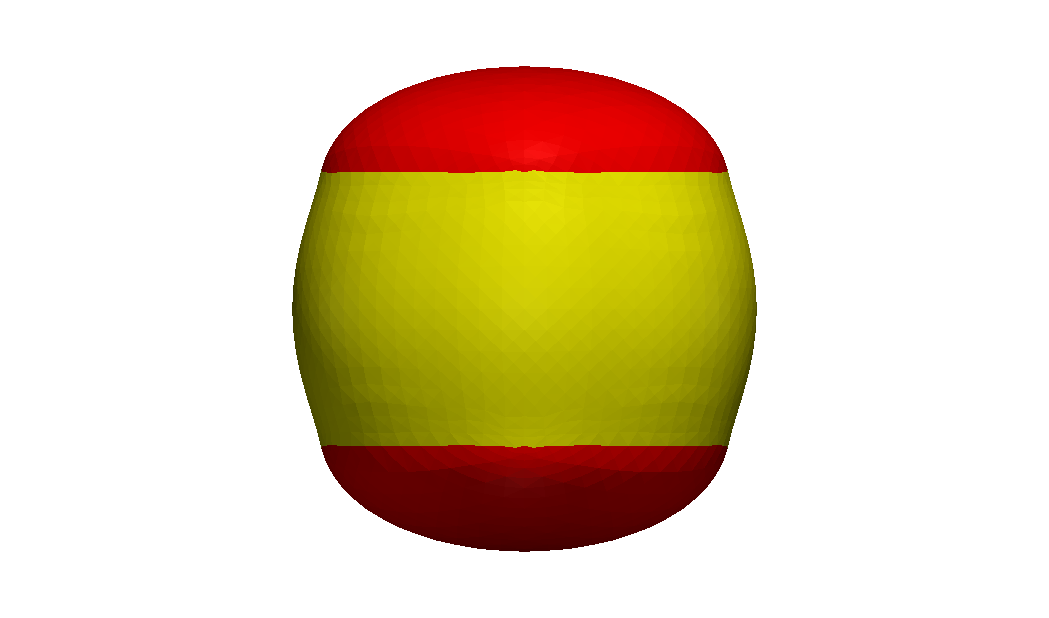} \\
\includegraphics[angle=-90,width=0.32\textwidth]{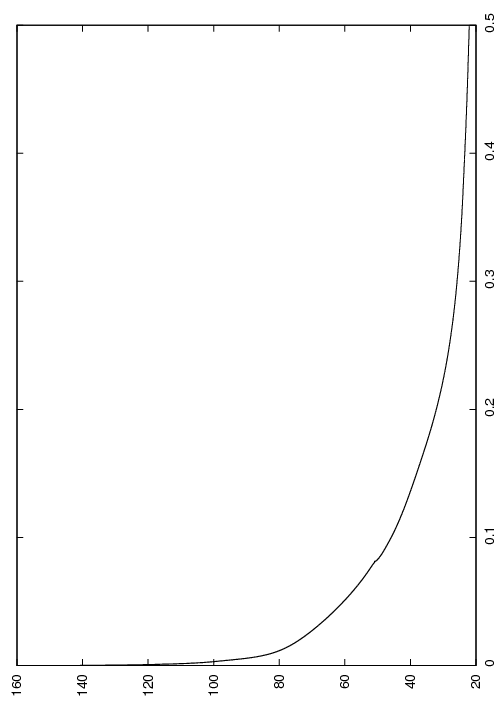}
\caption{($C^1$: $\spont_1 = -2$, $\spont_2 = -0.5$, $\varsigma = 0.1$,
$\varrho=2$)
A plot of $(\Gamma^m_i)_{i=1}^2$ at times $t=0,\ 0.1,\ 0.2,\ 0.5$.
Below a plot of the discrete energy $E^{m+1}((\Gamma^m)_{i=1}^2)$.
}
\label{fig:c1udb0_varrho}
\end{figure}%

If we project the initial surface from Figure~\ref{fig:c1udb0_varrho} 
to the unit sphere, we obtain the evolution shown in
Figure~\ref{fig:c1ball0}. 
The evolution goes towards a cylinder with two round caps.
\begin{figure}
\center
\includegraphics[angle=-0,width=0.24\textwidth]{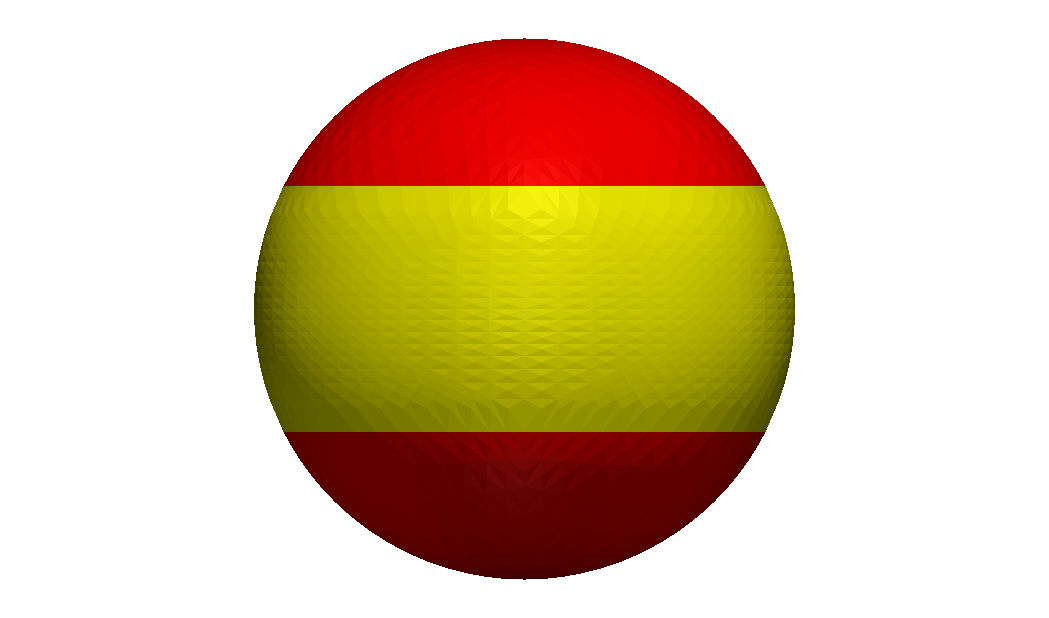}
\includegraphics[angle=-0,width=0.24\textwidth]{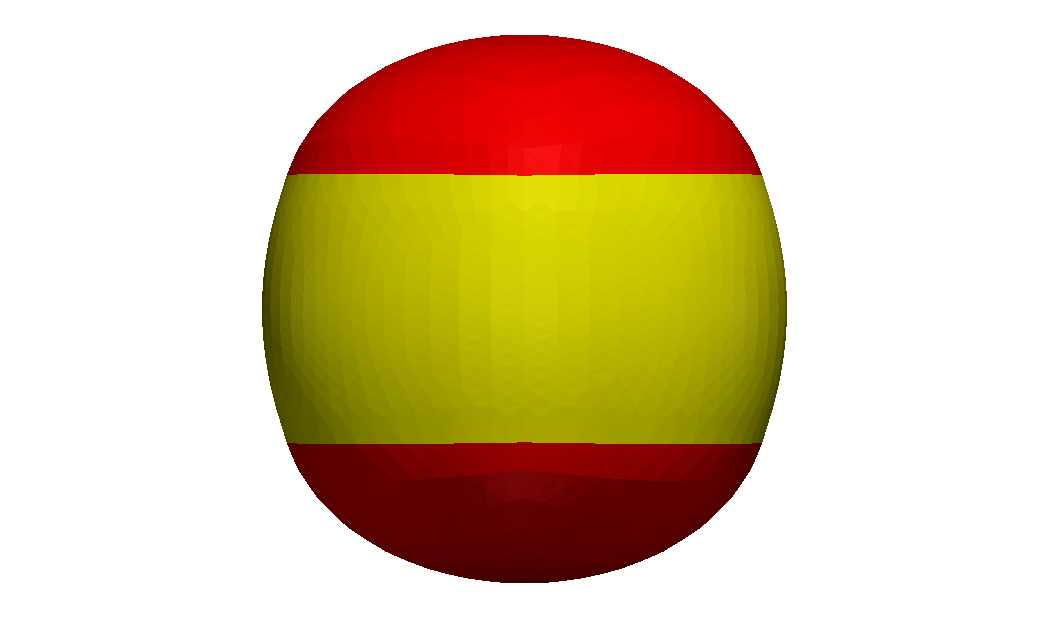}
\includegraphics[angle=-0,width=0.24\textwidth]{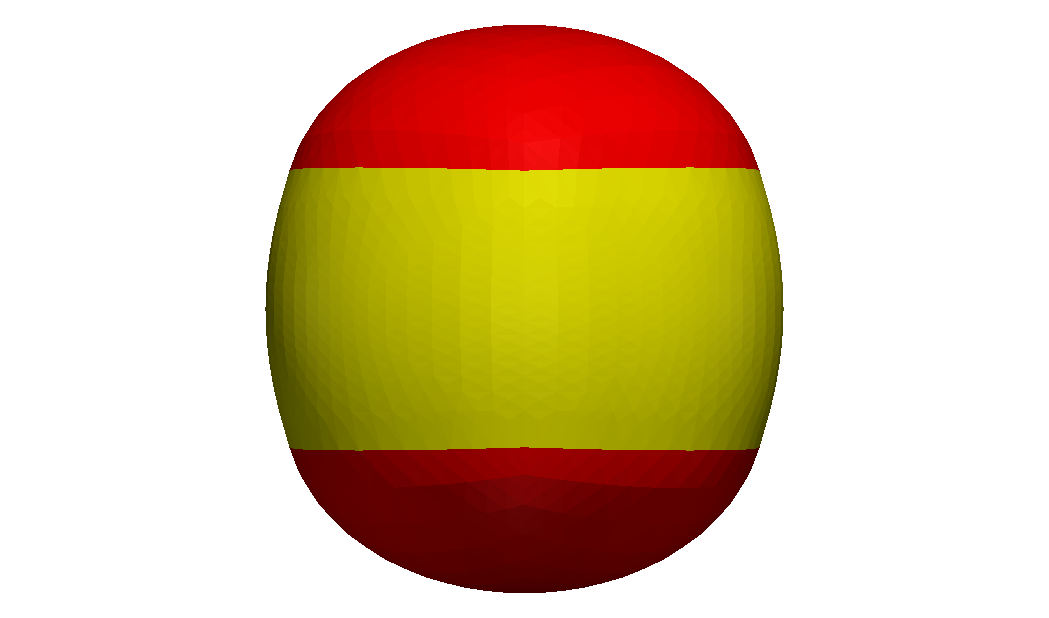} 
\includegraphics[angle=-0,width=0.24\textwidth]{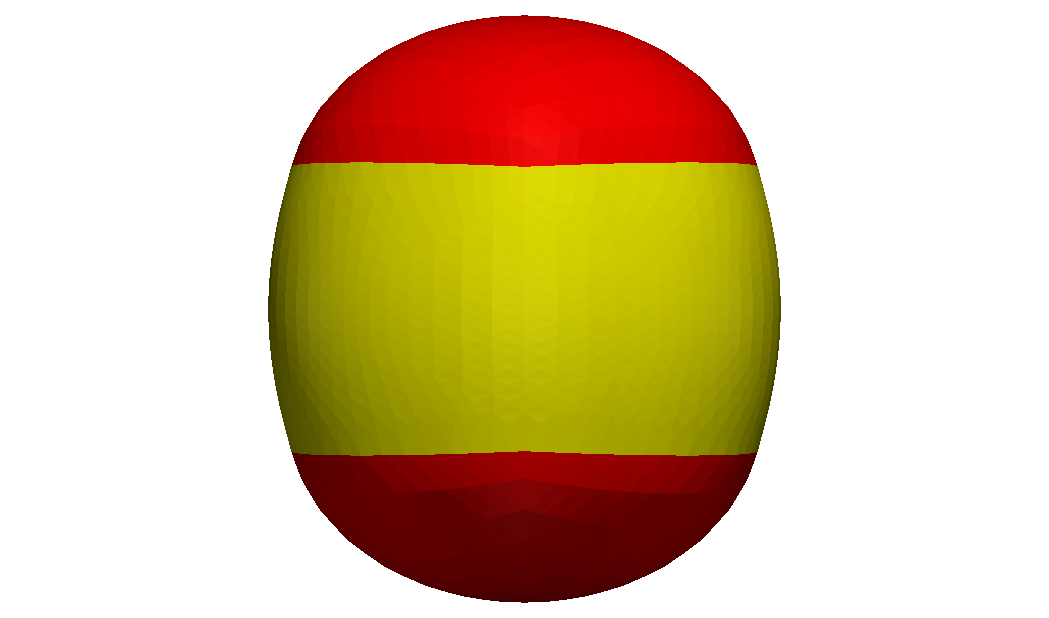} \\
\includegraphics[angle=-90,width=0.32\textwidth]{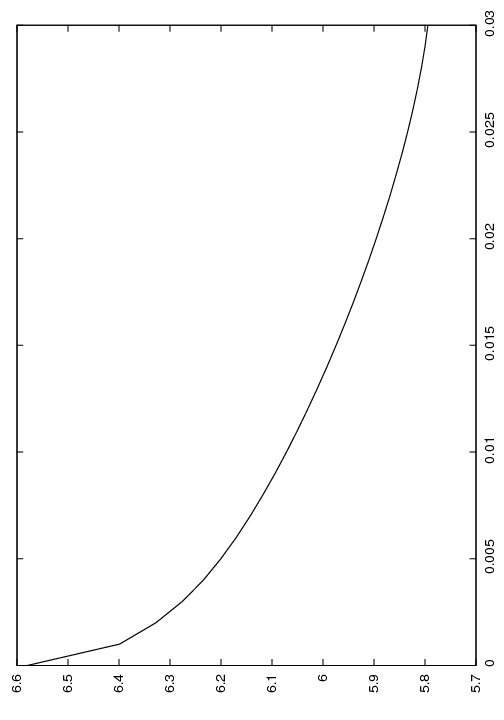}
\caption{($C^1$: $\spont_1 = -2$, $\spont_2 = -0.5$, $\varsigma = 0.1$)
A plot of $(\Gamma^m_i)_{i=1}^2$ at times $t=0,\ 0.01,\ 0.02,\ 0.03$.
Below a plot of the discrete energy $E^{m+1}((\Gamma^m)_{i=1}^2)$.
}
\label{fig:c1ball0}
\end{figure}%
Using the same parameters as in Figure~\ref{fig:c1ball0} to simulate
surface area preserving flow, 
we obtain the evolution shown in Figure~\ref{fig:c1ball1},
where here we have chosen $\varrho = 2$.
The evolution goes towards a more elongated cylinder with two round caps.
Here the observed relative surface area loss is $0.18\%$.
\begin{figure}
\center
\includegraphics[angle=-0,width=0.24\textwidth]{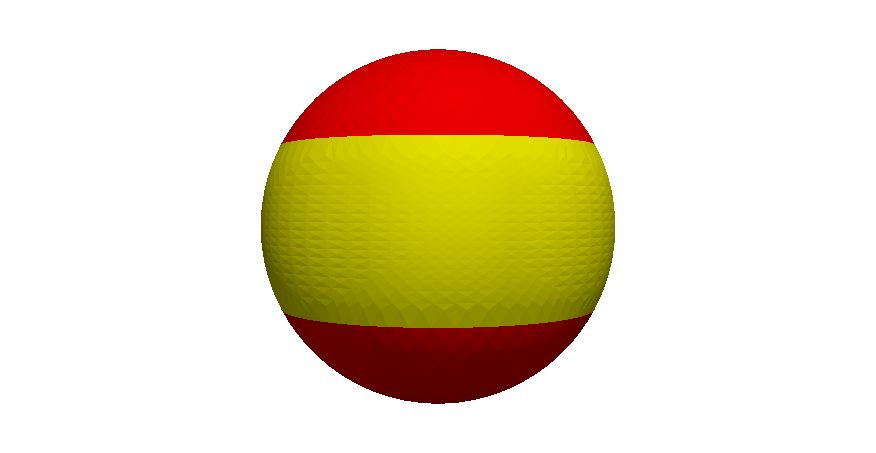}
\includegraphics[angle=-0,width=0.24\textwidth]{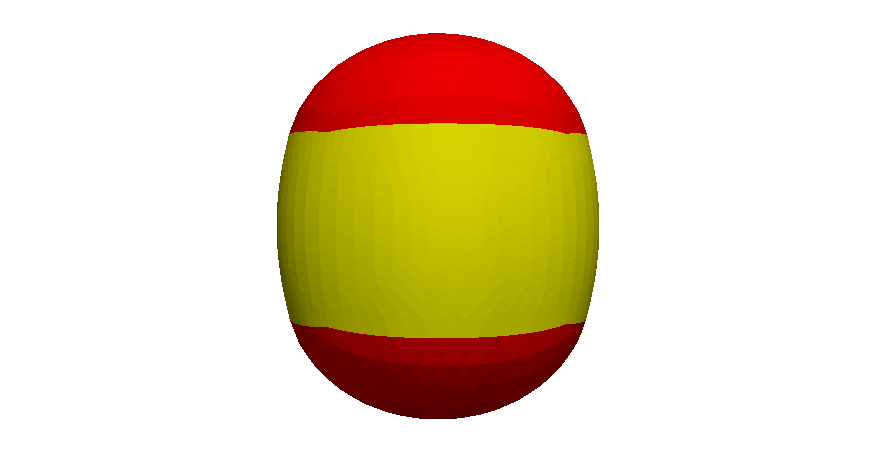}
\includegraphics[angle=-0,width=0.24\textwidth]{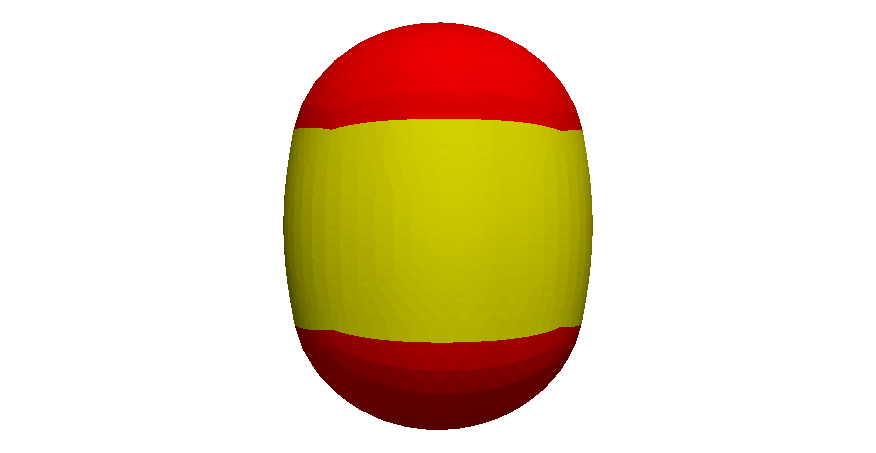} 
\includegraphics[angle=-0,width=0.24\textwidth]{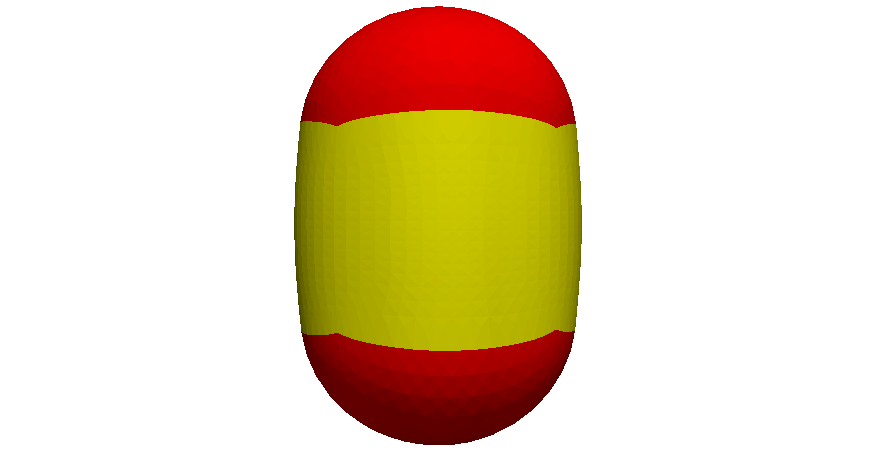} \\
\includegraphics[angle=-90,width=0.32\textwidth]{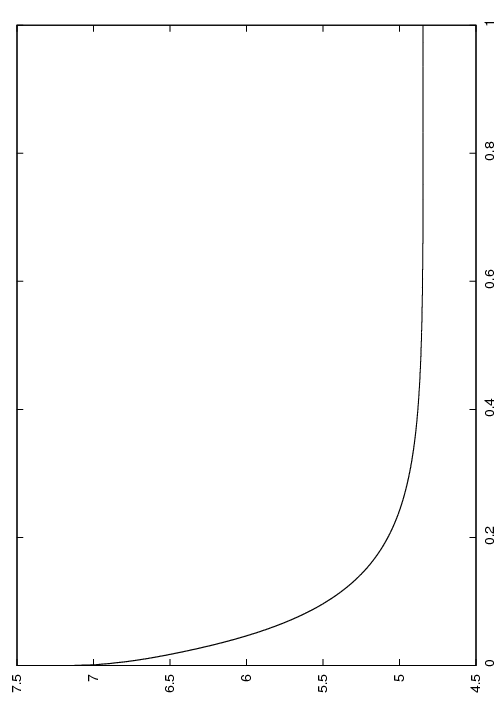}
\caption{($C^1$: $\spont_1 = -2$, $\spont_2 = -0.5$, $\varsigma = 0.1$,
$\varrho = 2$)
Surface area preserving flow.
A plot of $(\Gamma^m_i)_{i=1}^2$ at times $t=0,\ 0.1,\ 0.2,\ 1$.
Below a plot of the discrete energy $E^{m+1}((\Gamma^m)_{i=1}^2)$.
}
\label{fig:c1ball1}
\end{figure}%
The volume preserving variant is shown in Figure~\ref{fig:c1ball2},
where in order to dampen the tangential motion we choose $\theta = 0.05$.
Here the observed relative volume loss is $-0.12\%$.
\begin{figure}
\center
\includegraphics[angle=-0,width=0.24\textwidth]{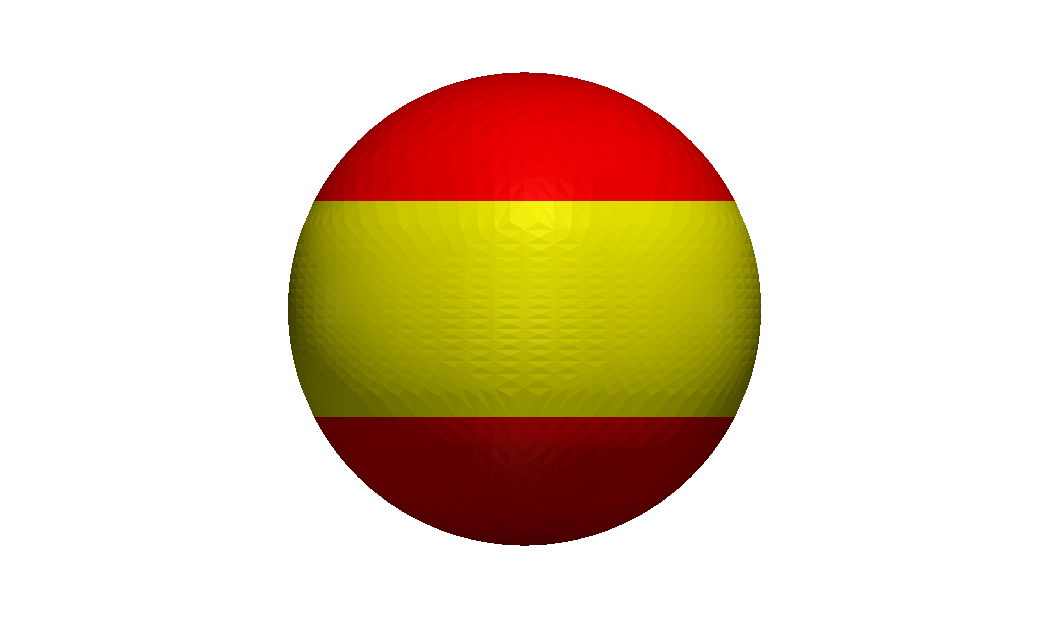}
\includegraphics[angle=-0,width=0.24\textwidth]{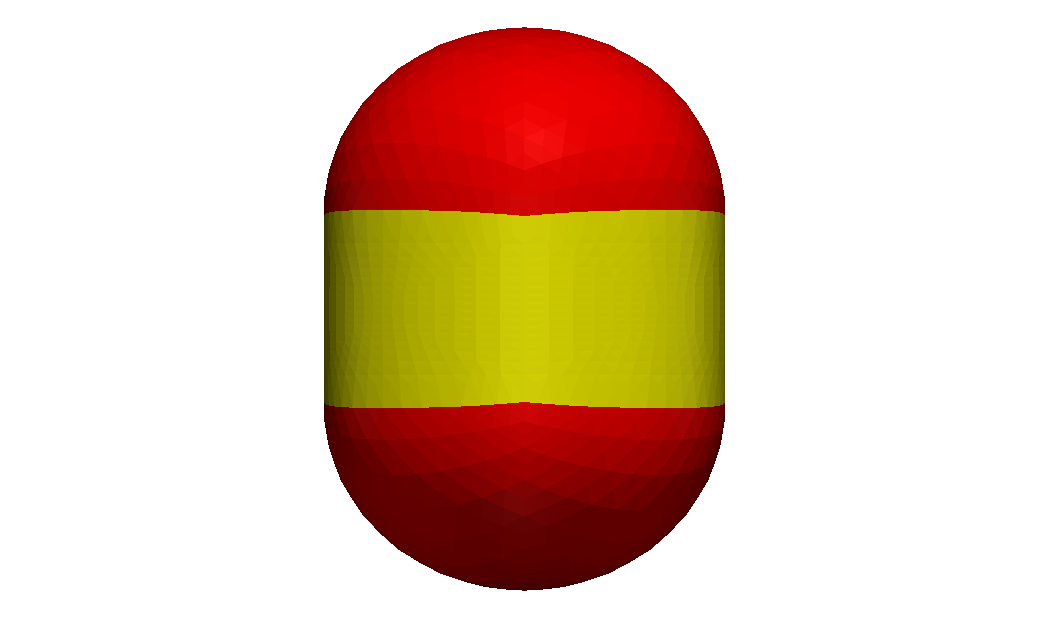}
\includegraphics[angle=-0,width=0.24\textwidth]{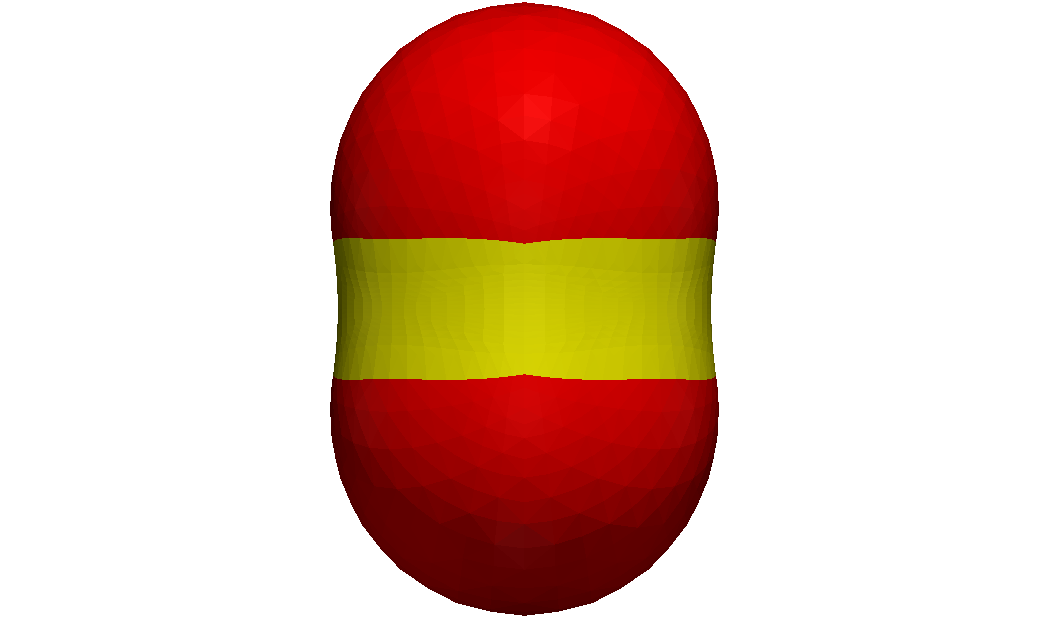} 
\includegraphics[angle=-0,width=0.24\textwidth]{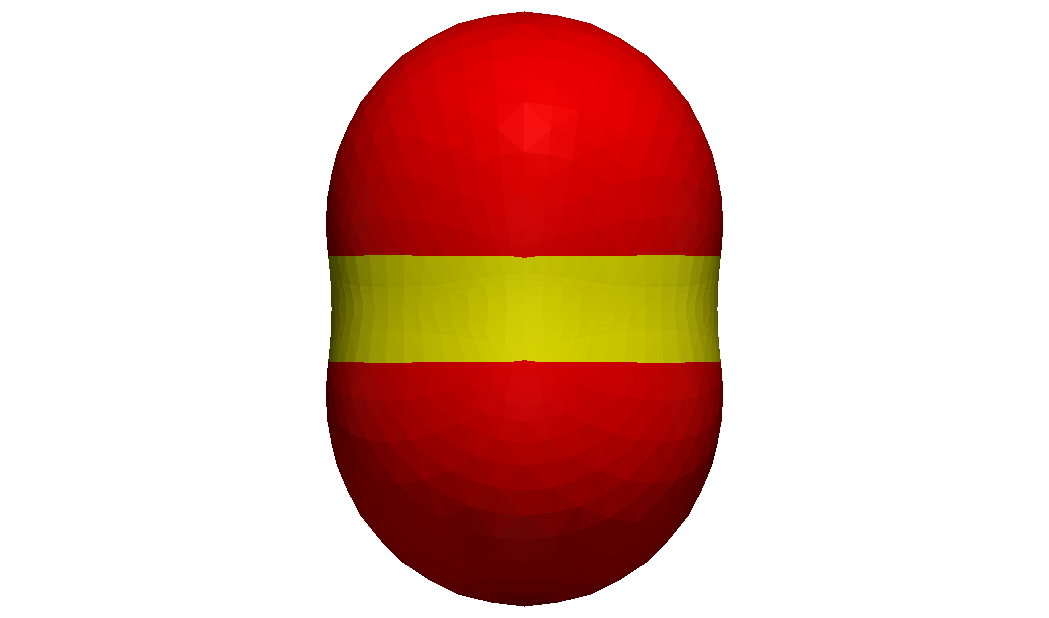} \\
\includegraphics[angle=-90,width=0.32\textwidth]{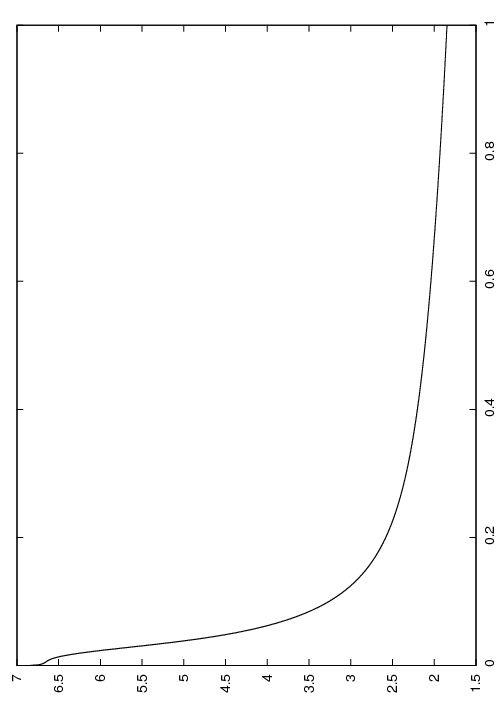}
\caption{($C^1$: $\spont_1 = -2$, $\spont_2 = -0.5$, $\varsigma = 0.1$,
$\varrho = 2$, $\theta = 0.05$)
Volume preserving flow.
A plot of $(\Gamma^m_i)_{i=1}^2$ at times $t=0,\ 0.1,\ 0.5,\  1$.
Below a plot of the discrete energy $E^{m+1}((\Gamma^m)_{i=1}^2)$.
}
\label{fig:c1ball2}
\end{figure}%

A repeat of the simulation in Figure~\ref{fig:c0budding0}, 
now in the context of a $C^1$--condition on $\gamma$, 
is shown in Figure~\ref{fig:c1budding0}. Once again, we observe that the
$C^1$--condition has a dramatic effect on the evolution.
\begin{figure}
\center
\includegraphics[angle=-0,width=0.24\textwidth]{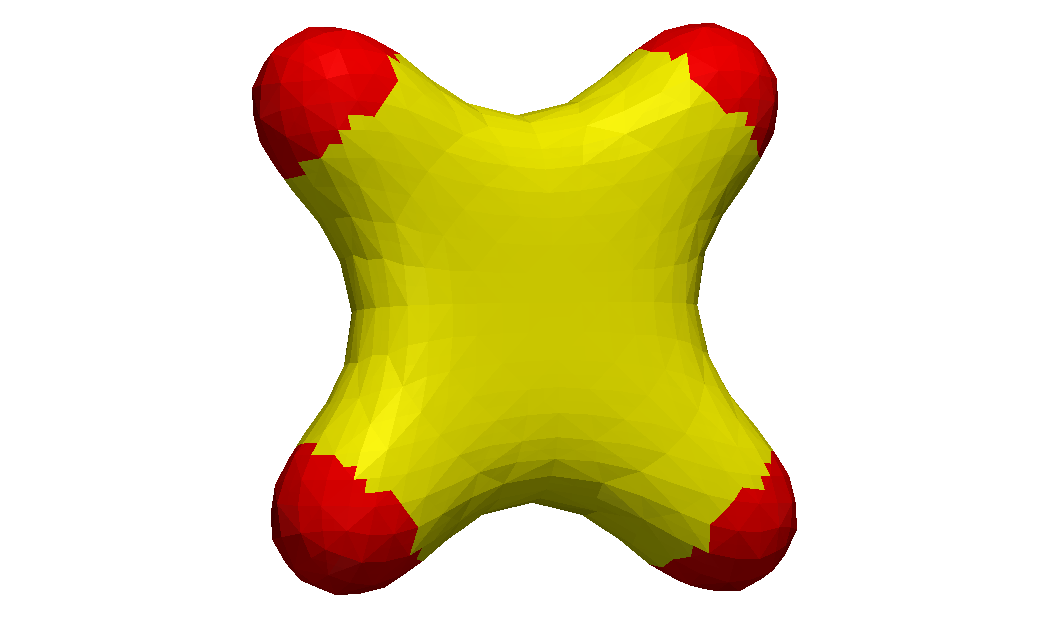}
\includegraphics[angle=-0,width=0.24\textwidth]{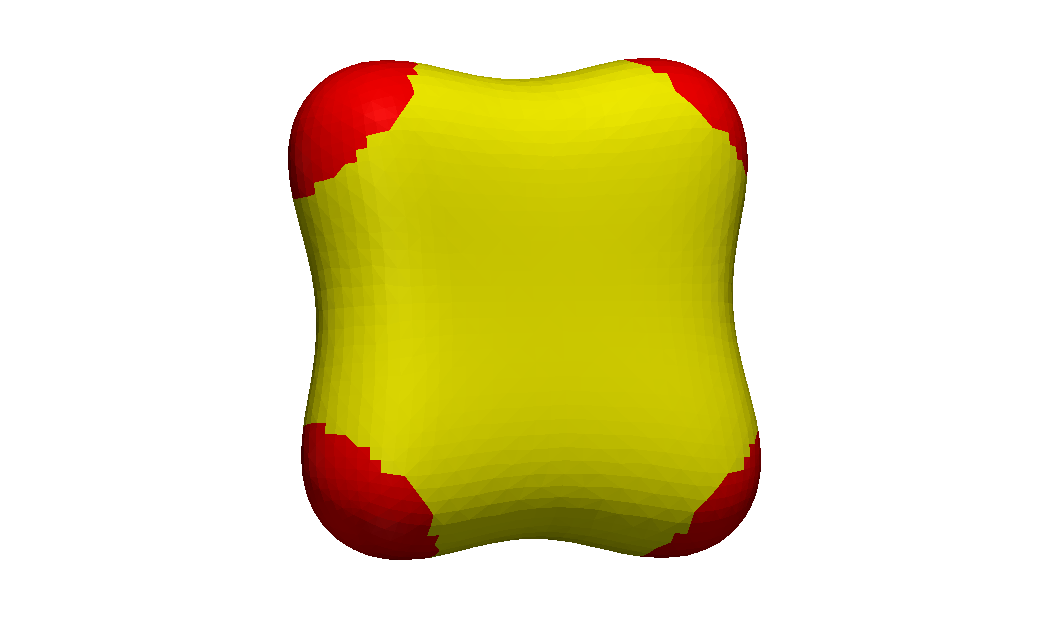}
\includegraphics[angle=-0,width=0.24\textwidth]{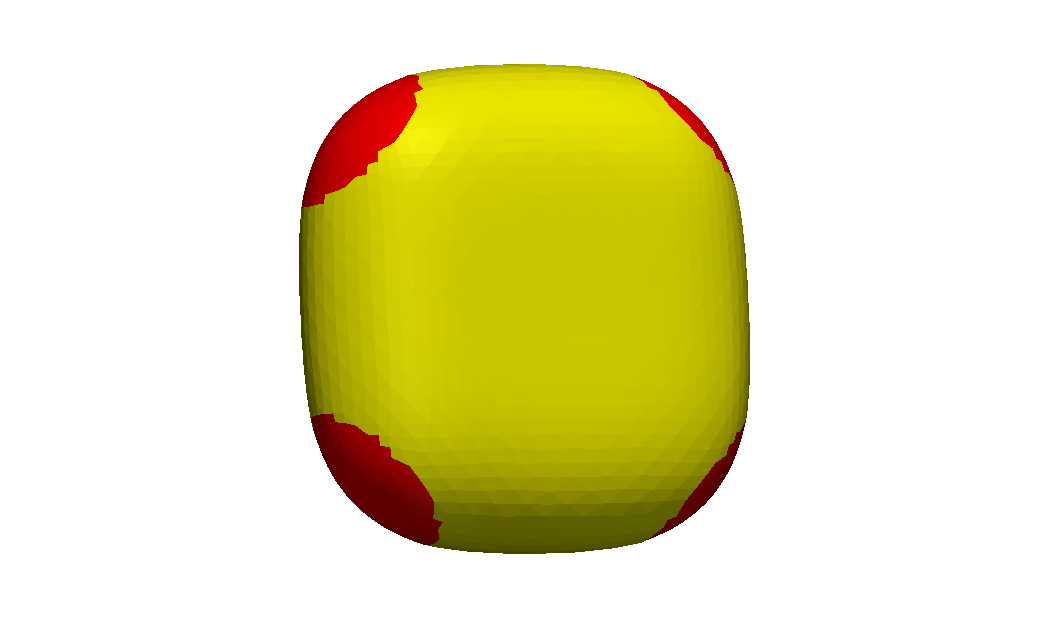}
\includegraphics[angle=-0,width=0.24\textwidth]{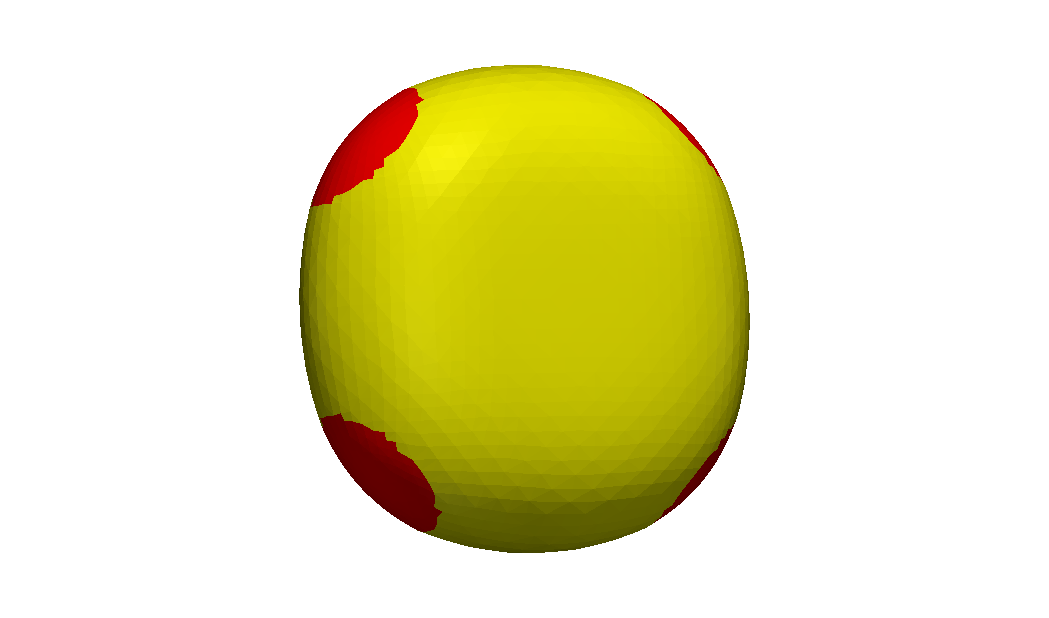}
\includegraphics[angle=-90,width=0.32\textwidth]{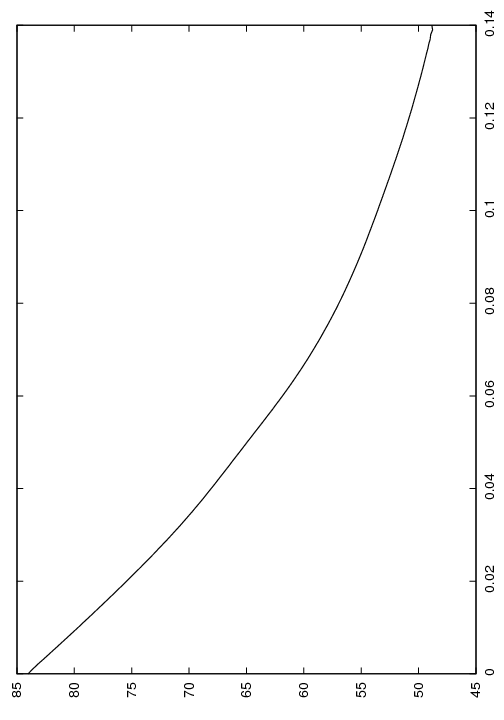}
\caption{($C^1$: $\spont_1 = \spont_2 = 0$, $\varsigma = 1$)
A plot of $(\Gamma^m_i)_{i=1}^2$ at times $t=0,\ 0.05,\ 0.1,\ 0.14$.
Below a plot of the discrete energy $E^{m+1}((\Gamma^m)_{i=1}^2)$.
}
\label{fig:c1budding0}
\end{figure}%

In the 
next experiments
we investigate the possible influence of the Gaussian
curvature energy. If we choose the initial surface as in
Figure~\ref{fig:c0gauss1}, and running with 
$\spont_1 = \spont_2 = -0.5$ and $\varsigma = \alpha^G_1=\alpha^G_2 = 0$, 
then we obtain an expanding sphere, with symmetric phases $1$ and $2$, which
approximates the solution to the nonlinear ODE \cite[(5.1)]{pwfade}.
On the continuous level, thanks to the Gauss--Bonnet theorem, the same solution
is obtained when choosing $\alpha^G_1=\alpha^G_2=0.5$, and this is also
replicated by our numerical approximation.
Choosing $\alpha^G_1=0.5$ and $\alpha^G_2=1$, on the other hand, leads to
phase $1$ growing on the expanding surface. 
Here we remark that (\ref{eq:alphaGbound}) clearly holds, and that reducing the
relative size of phase $2$ is energetically favourable.
See Figure~\ref{fig:c1gauss} for the evolution.
\begin{figure}
\center
\includegraphics[angle=-0,width=0.24\textwidth]{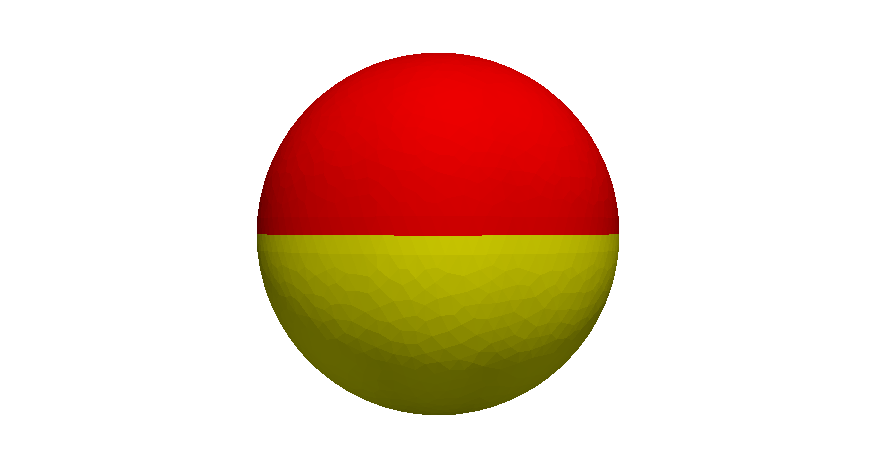}
\includegraphics[angle=-0,width=0.24\textwidth]{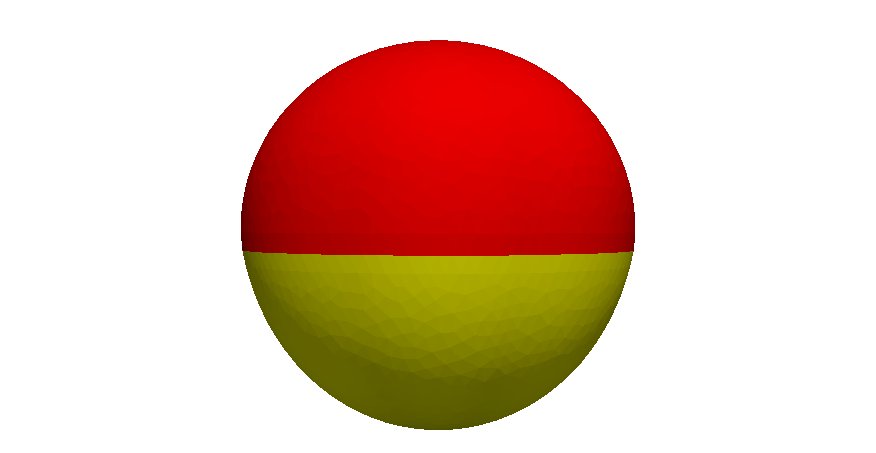}
\includegraphics[angle=-0,width=0.24\textwidth]{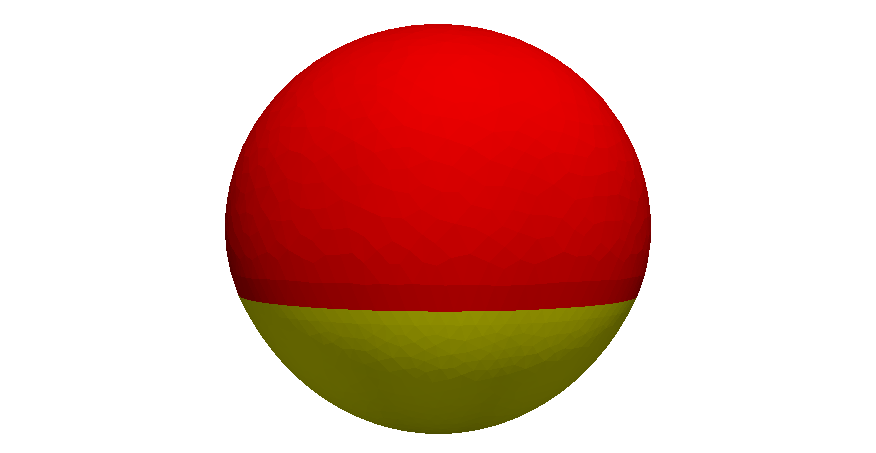}
\includegraphics[angle=-0,width=0.24\textwidth]{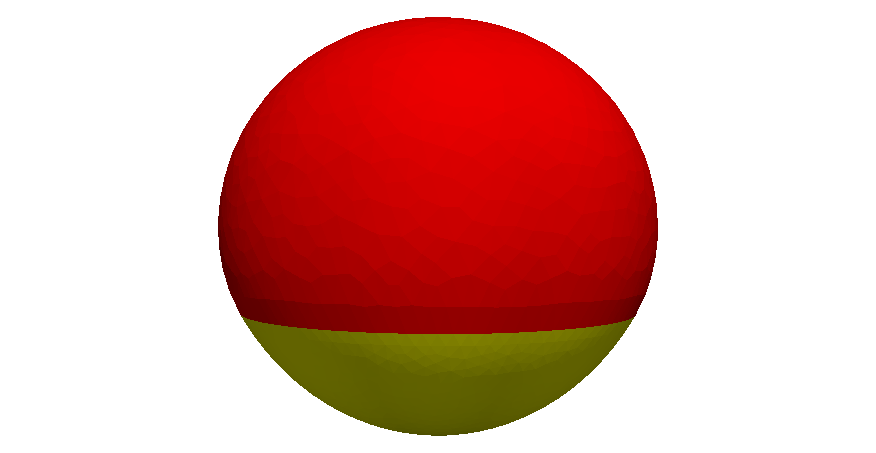}
\includegraphics[angle=-90,width=0.32\textwidth]{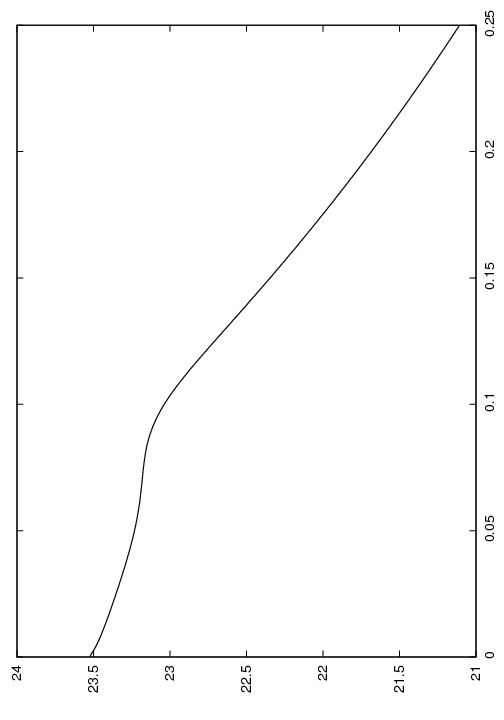}
\caption{($C^1$: $\spont_1 = \spont_2 = -0.5$, $\varsigma = 0$,
$\alpha^G_1=0.5$, $\alpha^G_2=1$)
A plot of $(\Gamma^m_i)_{i=1}^2$ at times $t=0,\ 0.1,\ 0.2,\ 0.25$.
Below a plot of the discrete energy $E^{m+1}((\Gamma^m)_{i=1}^2)$.
}
\label{fig:c1gauss}
\end{figure}%

A well known phenomenon is the moving of the phase boundary in relation to the
neck of a dumbbell for different values of the Gaussian bending rigidities,
see e.g.\ \cite[\S4.3]{ElliottS13}. We now demonstrate this behaviour in the
sharp interface context. To this end, we choose as initial data a membrane with
a neck, and then start an evolution of volume and surface area preserving flow
with $\alpha^G_1 \in \{-2, 0, 2\}$, while $\alpha^G_2 = 0$ and $\varsigma=9$.
For these experiments we choose $\ttau = 10^{-4}$ and let $\varrho = 1$.
See Figure~\ref{fig:C1gaussneck} for the different evolution. Here we can
clearly see that for $\alpha^G_1 = 2$ the interface moves down
relative to the neck of the dumbbell, while for $\alpha^G_1 = -2$ it moves up.
Of course, this is due to the neck having negative Gaussian curvature.
\begin{figure}
\center
\includegraphics[angle=-0,width=0.24\textwidth]{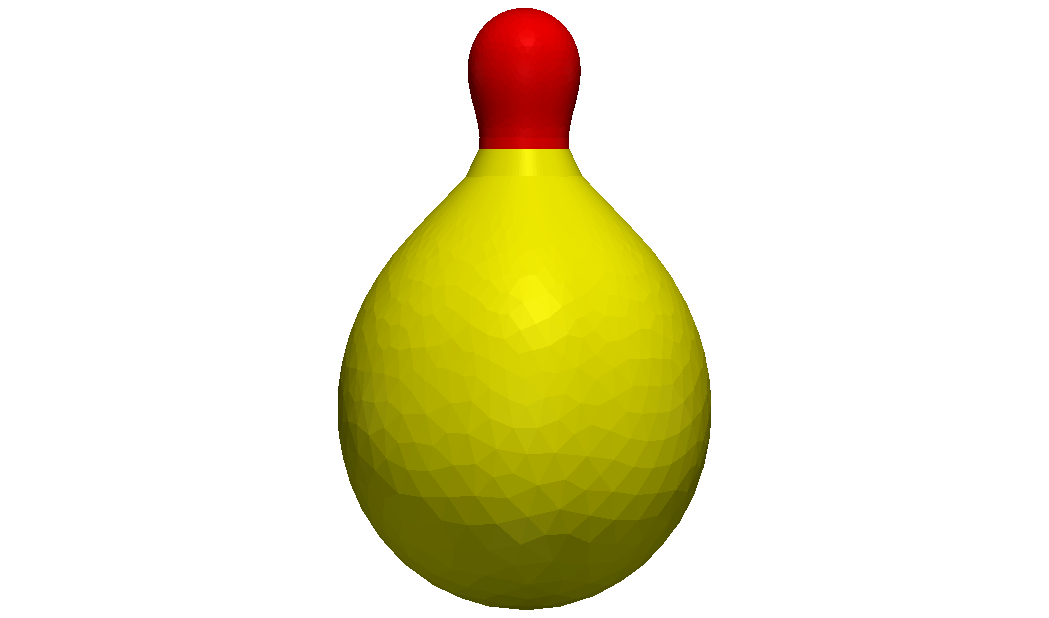}
\includegraphics[angle=-0,width=0.24\textwidth]{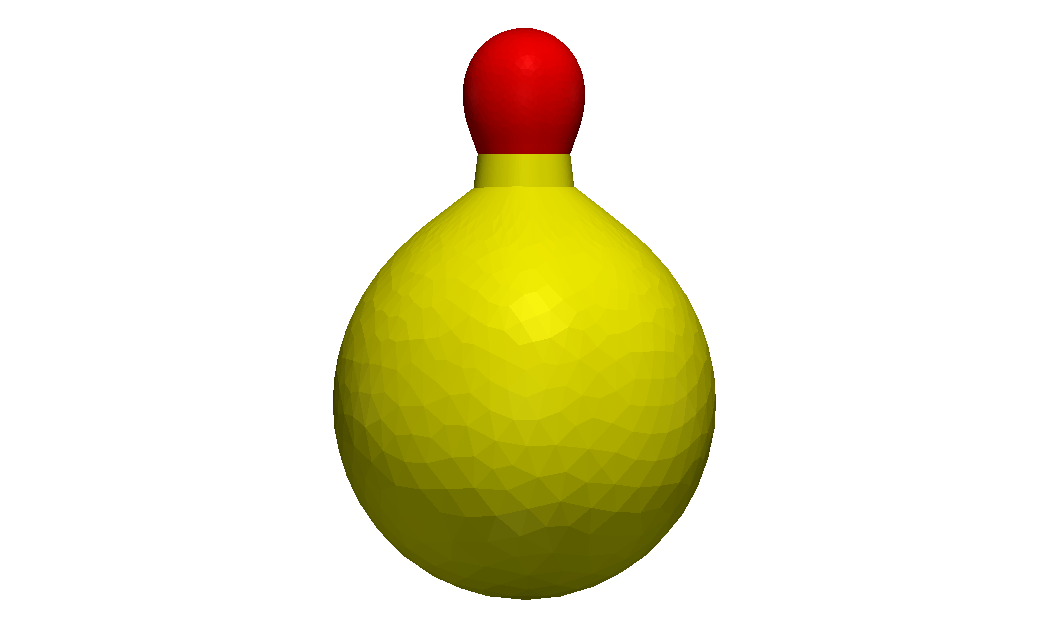}
\includegraphics[angle=-0,width=0.24\textwidth]{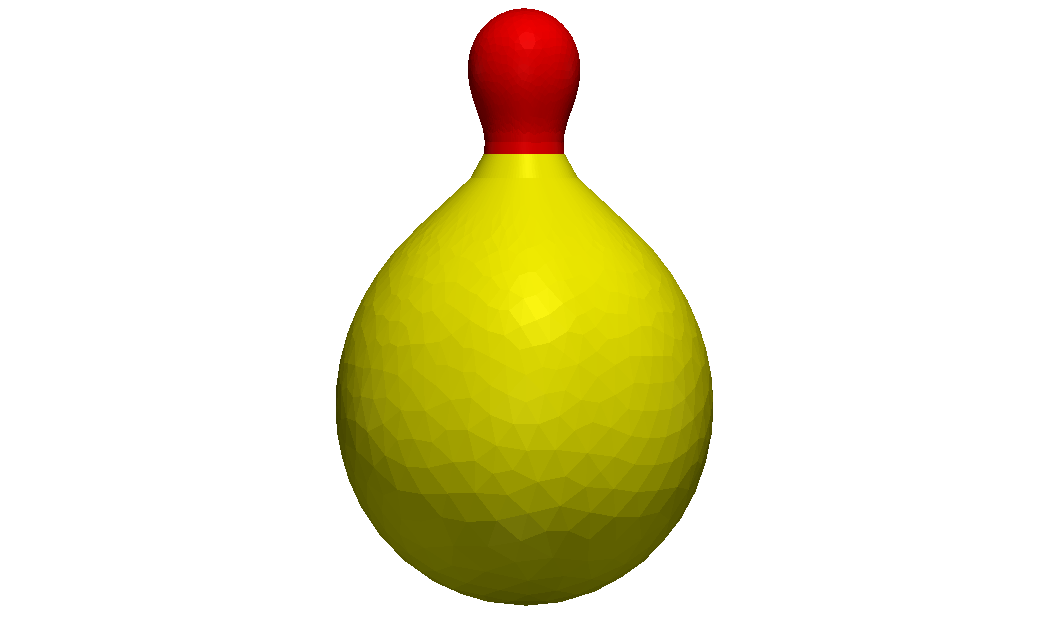}
\includegraphics[angle=-0,width=0.24\textwidth]{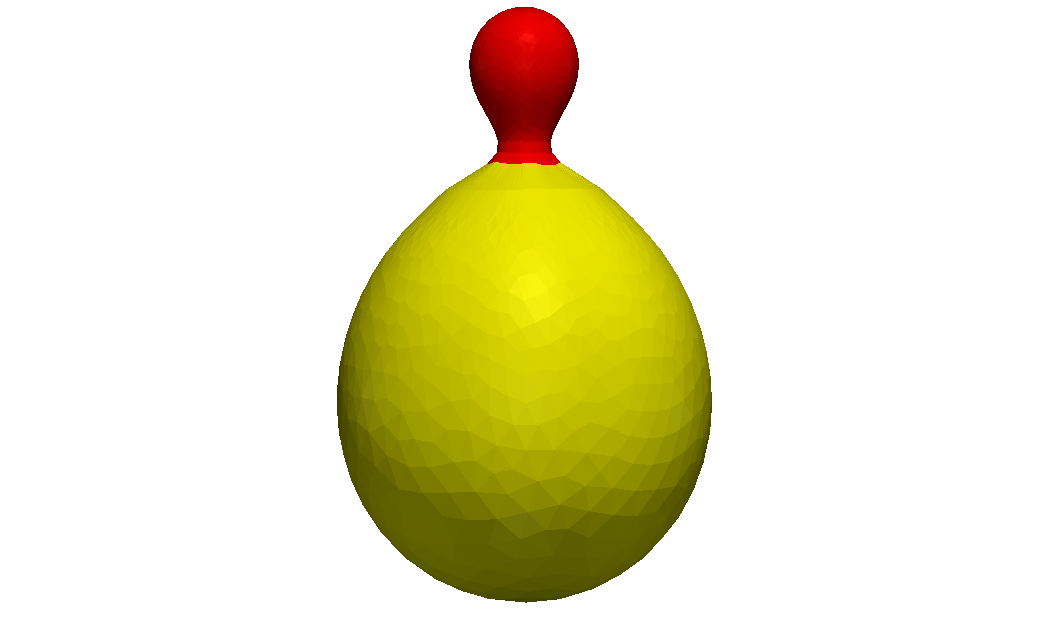}
\caption{($C^1$: $\spont_1 = \spont_2 = 0$, $\varsigma = 9$, $\alpha^G_2 = 0$,
$\varrho=4$)
The initial shape on the left, and $(\Gamma^m_i)_{i=1}^2$ at time $t=0.01$ 
for $\alpha^G_1 = -2$, $0$ and $2$, respectively.
}
\label{fig:C1gaussneck}
\end{figure}%

In the final experiments we approximate well-known equilibrium shapes from
\cite[Fig.\ 8]{JulicherL96}. 
To this end, we consider volume and surface area conserving flow for 
initial surfaces with reduced volumes $v_r \in \{0.95,\, 0.91,\, 0.9\}$,
where
\[
v_r = \frac{3\,\mathcal{L}^3(\Omega^0)}
{4\,\pi\,(\frac{\mathcal{H}^2(\Gamma^0)}{4\,\pi})^\frac32}
= \frac{6\,\pi^\frac12\,\mathcal{L}^3(\Omega^0)}
{({\mathcal{H}^2(\Gamma^0))^\frac32}}
\,,
\]
with $\Omega^0$ denoting the interior of $\Gamma^0$. In addition,
the two phases are chosen such that they have a surface area ratio of $0.1$.
See Figure~\ref{fig:c1ESinit} for the initial shapes,
where in each case we have that the initial discrete surface $\Gamma^0$ 
satisfies $(J_1,J_2) = (2274,2274)$ and $(K_1,K_2) = (1188,1188)$ and
$\mathcal{H}^{2}(\Gamma^0) = 4\,\pi$.
For these experiments we set $\varsigma = 9$ and $\varrho = 4$.
Choosing a time step size of $\ttau = 10^{-4}$, we integrate the volume and
surface area conserving flow to a final time of $t=0.25$ and report on the
obtained shapes in Figure~\ref{fig:c1ESt025}. These configurations appear to
agree well with the computed shapes in \cite[Fig.\ 8]{JulicherL96}.
\begin{figure}
\includegraphics[angle=-0,width=0.24\textwidth]{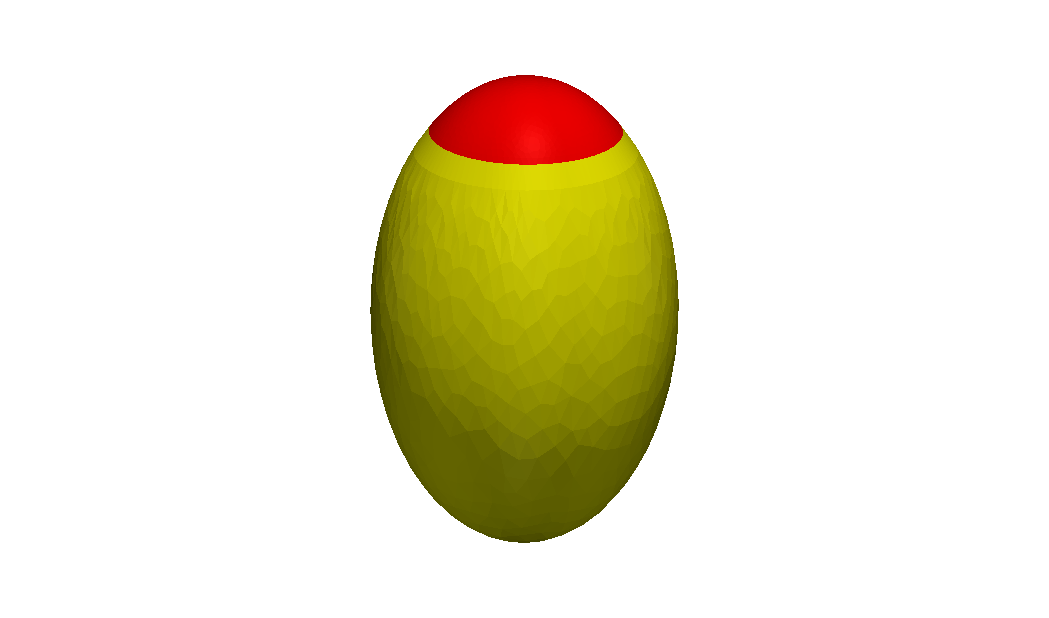}
\includegraphics[angle=-0,width=0.24\textwidth]{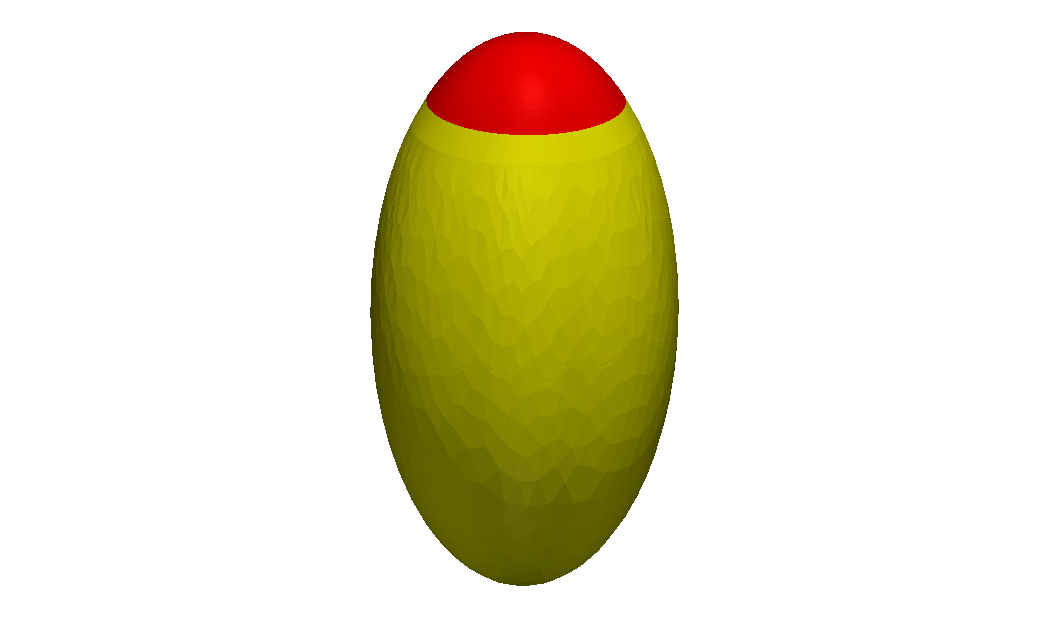}
\includegraphics[angle=-0,width=0.24\textwidth]{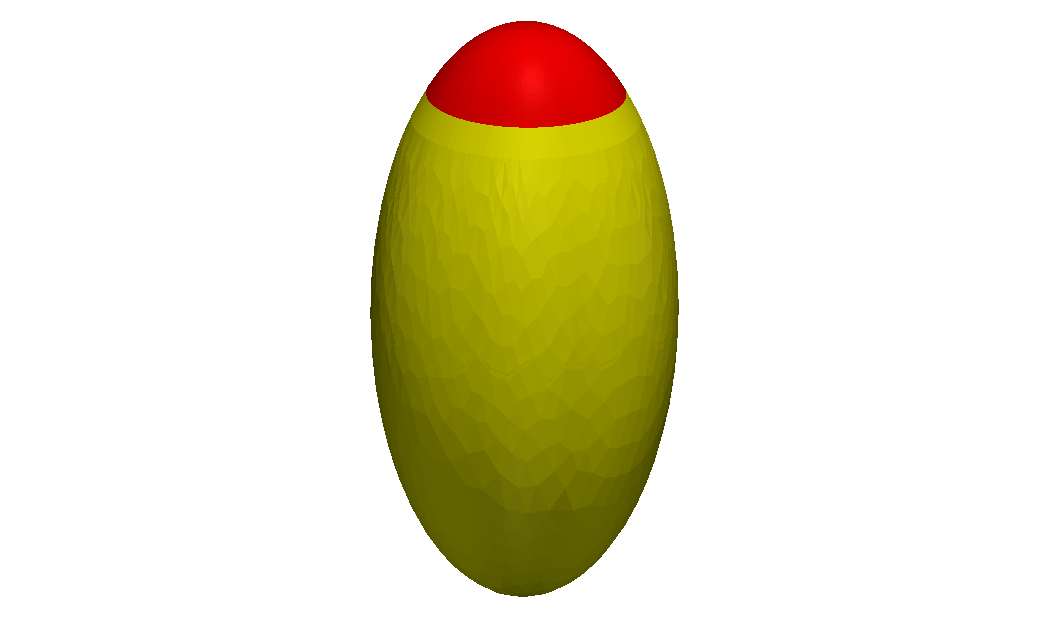}
\caption{The initial shapes for $v_r = 0.95$, $0.91$ and $0.9$, respectively.
}
\label{fig:c1ESinit}
\end{figure}%
\begin{figure}
\center
\includegraphics[angle=-0,width=0.24\textwidth]{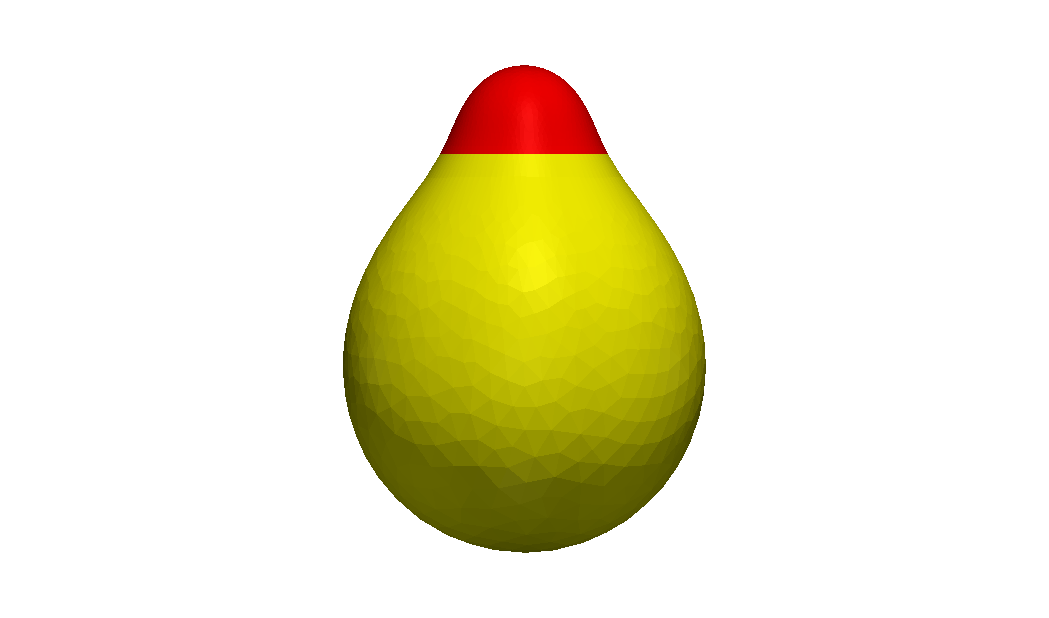}
\includegraphics[angle=-0,width=0.24\textwidth]{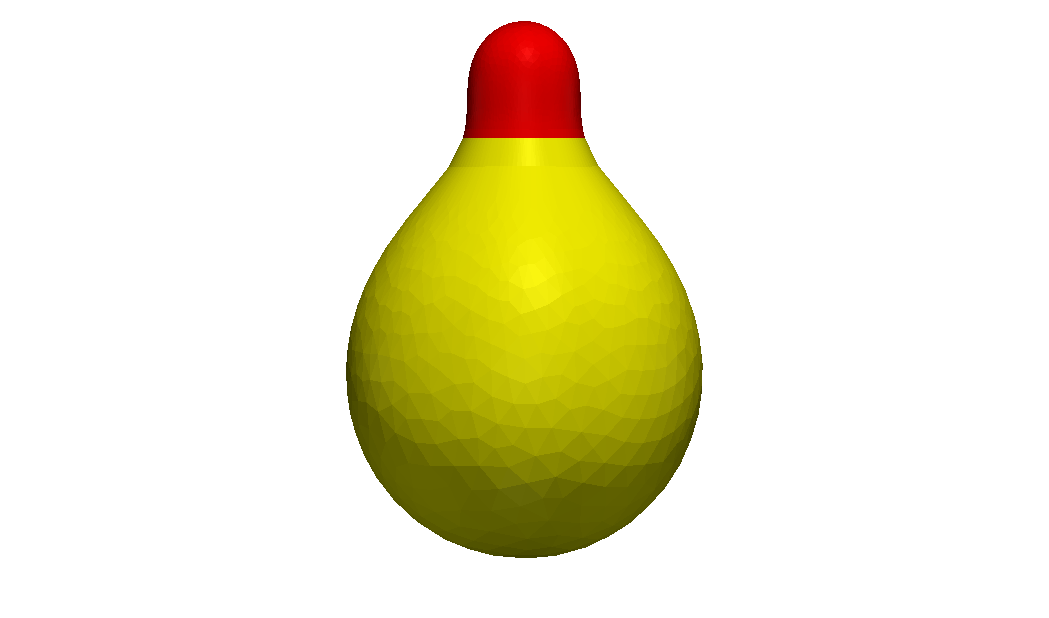}
\includegraphics[angle=-0,width=0.24\textwidth]{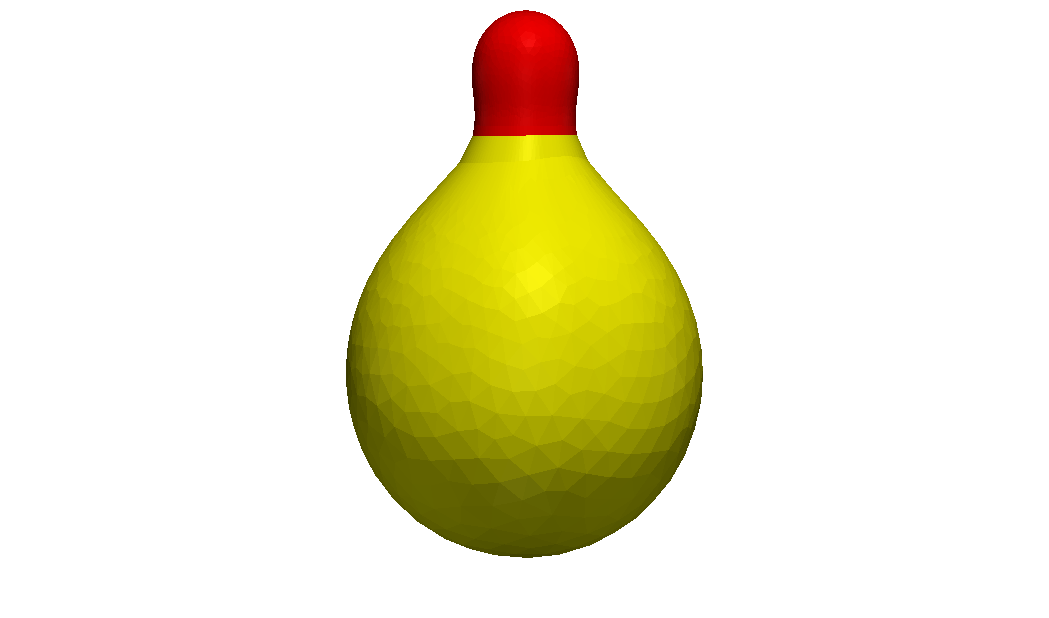}
\caption{($C^1$: $\spont_1 = \spont_2 = 0$, $\varsigma = 9$,
$\varrho=4$)
A plot of $(\Gamma^m_i)_{i=1}^2$ at time $t=0.25$ for the 
reduced volumes $v_r = 0.95$, $0.91$ and $0.9$, respectively.
}
\label{fig:c1ESt025}
\end{figure}%

\appendix

\section{
Derivation of strong formulation and boundary conditions} 

We recall from Section~\ref{sec:2} that our numerical method is based on the
weak formulation 
(\ref{eq:weakGD3a}) and (\ref{eq:LM4a}--f)
of the generalized $L^2$--gradient flow of the energy
$E((\Gamma_i(t))_{i=1}^2)$, see (\ref{eq:E2}). 
It follows from (\ref{eq:LBop}), (\ref{eq:LM4b})
and (\ref{eq:normvar}) that
\begin{align*}
& \vec\varkappa_i\,.\,\partial_\epsilon^0\,\vec\nu_i = 0\,,\qquad
\partial_\epsilon^0\,(\mat Q_{i,\theta}\,\vec\varkappa_i) = - (1-\theta)\,
\varkappa_i\,[\nabs\,\vec\chi]^T\,\vec\nu_i\,,\nonumber \\
& \tfrac12\left[ \alpha_i\,
|\vec\varkappa_i - \spont_i\,\vec\nu_i|^2 - 
2\, \mat Q_{i,\theta}\,\vec y_i\,.\,\vec\varkappa_i
\right] = -\tfrac12\,\alpha_i\,(\varkappa_i^2 - \spont_i^2)
\quad\text{on }\ \Gamma_i(t),\ i=1,2\,.
\end{align*}
We recall that on the continuous level $\vec{\rm m}_i = \vec\mu_i$
and that $\theta \in [0,1]$ is a fixed parameter. Here we need to choose
$\theta = 0$, 
as otherwise the two conditions in (\ref{eq:LM3b},c) are 
incompatible in general.
Then this weak formulation can be formulated as follows.
Given $\Gamma_i(0)$, for all $t\in(0,T]$ find
$\Gamma_i(t)$ 
and $\vec y_i(t) \in \Vti$ such that
\begin{align}
& \left\langle \vec{\mathcal{V}} , \vec\chi \right\rangle_{\Gamma(t)} \!
+ \varrho \left\langle \vec{\mathcal{V}} , \vec\chi \right\rangle_{\gamma(t)}
 =
\sum_{i=1}^2 \left[
\left\langle \nabs\,\vec y_i, \nabs\,\vec\chi \right\rangle_{\Gamma_i(t)} 
+ \left\langle \nabs\,.\,\vec y_i, \nabs\,.\,\vec\chi 
\right\rangle_{\Gamma_i(t)}
- \left\langle (\nabs\,\vec y_i)^T, \mat D(\vec\chi)\,(\nabs\,\vec\id)^T
\right\rangle_{\Gamma_i(t)}
\right. \nonumber \\ & \qquad \left.
+ \tfrac12\,\alpha_i\left\langle (\varkappa_i^2 - \spont_i^2)
\,\nabs\,\vec\id,\nabs\,\vec\chi \right\rangle_{\Gamma_i(t)}
- (1-\theta) \left\langle \varkappa_i\,\vec y_i, 
[\nabs\,\vec\chi]^T\,\vec\nu_i \right\rangle_{\Gamma_i(t)} \right]
- \varsigma \left\langle \vec\id_s, \vec\chi_s \right\rangle_{\gamma(t)}
\nonumber \\ & \qquad 
+ \sum_{i=1}^2
\alpha^G_i \left[ \left\langle \vec\varkappa_\gamma\,.\,\vec\mu_i, 
\vec\id_s\,.\,\vec\chi_s \right\rangle_{\gamma(t)}
+ \left\langle \mat{\mathcal{P}}_\gamma\,(\vec\mu_i)_s, \vec\chi_s
\right\rangle_{\gamma(t)} \right]
\quad \forall\ \vec\chi \in \xspacec \,, \label{eq:weakopena}
\end{align}
with $\gamma(t) = \partial\Gamma_1(t) = \partial\Gamma_2(t)$,
\begin{equation} \label{eq:yiui}
\vec y_i = y_i\,\vec\nu_i + \vec u_i\,, \quad\text{where}\quad
y_i = \alpha_i\,(\varkappa_i - \spont_i)
\ \text{ and }\
\vec u_i \,.\,\vec\nu_i = 0\,,\quad \text{on } \Gamma_i(t)\,,\ i=1,2\,.
\end{equation}
Of course, (\ref{eq:LM4b}) implies that 
$\vec u_i = \vec0$ if $\theta \in (0,1]$.
Hence, as
$\vec\nu_i\,.\,[\nabs\,\vec\chi_i]^T\,\vec\nu_i = 
([\nabs\,\vec\chi_i]\,\vec\nu_i)\,.\,\vec\nu_i = \vec 0\,.\,\vec\nu_i = 0$, 
it holds that
\begin{align} 
(1-\theta) \left\langle \varkappa_i\,\vec y_i, 
[\nabs\,\vec\chi_i]^T\,\vec\nu_i \right\rangle_{\Gamma_i(t)}
& =
(1-\theta) \left\langle \varkappa_i\,y_i\,\vec\nu_i,
[\nabs\,\vec\chi_i]^T\,\vec\nu_i \right\rangle_{\Gamma_i(t)}
+ (1-\theta)\left\langle \varkappa_i\,\vec u_i, 
[\nabs\,\vec\chi_i]^T\,\vec\nu_i \right\rangle_{\Gamma_i(t)} 
\nonumber \\ & 
= (1-\theta)\left\langle \varkappa_i\,\vec u_i, 
[\nabs\,\vec\chi_i]^T\,\vec\nu_i \right\rangle_{\Gamma_i(t)}
= \left\langle \varkappa_i\,\vec u_i, 
[\nabs\,\vec\chi_i]^T\,\vec\nu_i \right\rangle_{\Gamma_i(t)} \label{eq:T59}
\end{align}
for all $\theta\in[0,1]$.

In (\ref{eq:weakopena}) 
the mean curvatures $\varkappa_i$ are defined by (\ref{eq:LBop}),
the curve curvature vector $\vec\varkappa_\gamma$ is given by
(\ref{eq:ident}), and the conormals $\vec\mu_i(t)$ are 
defined by (\ref{eq:mu}) and satisfy $C_1\,(\vec\mu_1 + \vec\mu_2)=\vec0$. 
In addition, we have from (\ref{eq:LM4c}) that
\begin{equation} \label{eq:appybound}
\vec y_i =  - \alpha^G_i\,\vec\varkappa_\gamma - C_1\,\vec\phi
\quad \text{on }\ \gamma(t)\,,\quad i=1,2\,.
\end{equation}

Starting from the weak formulation (\ref{eq:weakopena}), 
we will now recover the
corresponding strong formulation together with the boundary conditions that are
enforced by it.
It follows from (\ref{eq:weakopena}) and (\ref{eq:T59}) that
\begin{align}
& \left\langle \vec{\mathcal{V}} , \vec\chi \right\rangle_{\Gamma(t)} 
+ \varrho \left\langle \vec{\mathcal{V}} , \vec\chi \right\rangle_{\gamma(t)}
=
\sum_{i=1}^2 \left[
\left\langle \nabs\,(y_i\,\vec\nu_i), \nabs\,\vec\chi \right\rangle_{\Gamma_i(t)} 
+ 
\left\langle\nabs\,.\,(y_i\,\vec\nu_i), \nabs\,.\,\vec\chi \right\rangle_{\Gamma_i(t)}
\right. \nonumber \\ & \qquad \left.
- \left\langle [\nabs\,(y_i\,\vec\nu_i)]^T, \mat D(\vec\chi)\,(\nabs\,\vec\id)^T
\right\rangle_{\Gamma_i(t)} 
+ \tfrac12\,\alpha_i\left\langle (\varkappa_i^2 - \spont_i^2),
\nabs\,.\,\vec\chi \right\rangle_{\Gamma_i(t)}
+ \left\langle \nabs\,\vec u_i, \nabs\,\vec\chi \right\rangle_{\Gamma_i(t)} 
\right. \nonumber \\ & \qquad \left.
+ 
\left\langle \nabs\,.\,\vec u_i, \nabs\,.\,\vec\chi \right\rangle_{\Gamma_i(t)}
- \left\langle (\nabs\,\vec u_i)^T, \mat D(\vec\chi)\,(\nabs\,\vec\id)^T
\right\rangle_{\Gamma_i(t)}
- \left\langle \varkappa_i\,\vec u_i, 
[\nabs\,\vec\chi]^T\,\vec\nu_i \right\rangle_{\Gamma_i(t)} \right]
\nonumber \\ & \qquad 
+ \varsigma \left\langle \vec\varkappa_\gamma, \vec\chi 
\right\rangle_{\gamma(t)}
+ \sum_{i=1}^2 \alpha^G_i 
\left[ \left\langle \vec\varkappa_\gamma\,.\,\vec\mu_i, 
\vec\id_s\,.\,\vec\chi_s \right\rangle_{\gamma(t)}
+ \left\langle \mat{\mathcal{P}}_\gamma\,[\vec\mu_i]_s, \vec\chi_s
\right\rangle_{\gamma(t)} \right]
\nonumber \\ & \qquad 
=: \sum_{i=1}^2 \sum_{\ell=1}^8 T_\ell^{(i)}
+ \varsigma \left\langle \vec\varkappa_\gamma, \vec\chi 
\right\rangle_{\gamma(t)}
+ \sum_{i=1}^2 \alpha^G_i 
\left[ \left\langle \vec\varkappa_\gamma\,.\,\vec\mu_i, 
\vec\id_s\,.\,\vec\chi_s \right\rangle_{\gamma(t)}
+ \left\langle \mat{\mathcal{P}}_\gamma\,[\vec\mu_i]_s, \vec\chi_s
\right\rangle_{\gamma(t)} \right]
\nonumber \\ & \quad \hspace{10cm}
\quad \forall\ \vec\chi \in \xspacec\,.
\label{eq:appdagger}
\end{align}
In order to identify the first term on the right hand side in
(\ref{eq:appdagger}), we now recall (A.21) and (A.30) in \cite{pwfopen},
where we note that in our situation $\beta = 0$, and that the results there are
for $d=3$, but are also true for $d=2$, where we always assume that
$\varsigma = \alpha^G_1 = \alpha^G_2 = 0$.
Hence we have that
\begin{align}
\sum_{\ell=1}^8 T_\ell^{(i)} & =
\left\langle - \alpha_i\,\Delta_s\, \varkappa_i 
+ \tfrac12\,\alpha_i\,(\varkappa_i - \spont_i)^2 \,\varkappa_i
- \alpha_i\,(\varkappa_i - \spont_i)\,|\nabs\,\vec\nu_i|^2,
\vec\chi \,.\,\vec\nu_i \right\rangle_{\Gamma_i(t)}
+ \alpha_i\left\langle (\nabs\, \varkappa_i)\,.\,\vec\mu_i , \vec\chi\,.\,\vec\nu_i
\right\rangle_{\gamma(t)} 
\nonumber \\ & \quad
- \alpha_i \left\langle (\varkappa_i - \spont_i) \,(\nabs\,\vec\nu_i)\,\vec\mu_i, \vec\chi \right\rangle_{\gamma(t)} 
- \tfrac12\,\alpha_i \left\langle (\varkappa_i - \spont_i)^2 ,
\vec\chi\,.\,\vec\mu_i \right\rangle_{\gamma(t)} + B^{(i)}
\qquad\forall\ \vec\chi \in \xspacec\,, 
\label{eq:strongopena}
\end{align}
where
\begin{align}
B^{(i)} & = \left\langle \nabs\,.\,\vec u_i , \vec\chi\,.\,\vec\mu_i
\right\rangle_{\gamma(t)}
- \left\langle \nabs\,\vec u_i , \vec\mu_i\otimes\vec\chi
\right\rangle_{\gamma(t)}
- \left\langle \varkappa_i\,\vec u_i\,.\,\vec\mu_i, \vec\chi\,.\,\vec\nu_i
\right\rangle_{\gamma(t)} 
+ \left\langle (\nabs\,\vec u_i)\,\vec\mu_i,
(\vec\chi\,.\,\vec\nu_i)\,\vec\nu_i \right\rangle_{\gamma(t)} 
\nonumber \\ & 
=: \sum_{\ell=1}^4 D_\ell^{(i)}\,.
\label{eq:B}
\end{align}
It immediately follows from (\ref{eq:strongopena}) that the strong formulation
of the flow equation is
\begin{equation} 
\vec{\mathcal{V}} = \left[ - \alpha_i\,\Delta_s\,\varkappa_i + 
\tfrac12\,\alpha_i\,(\varkappa_i - \spont_i)^2 \,\varkappa_i -
\alpha_i\,(\varkappa_i - \spont_i) \,|\nabs\,\vec\nu_i|^2 \right] \vec\nu_i
\quad \text{on }\ \Gamma_i(t)\,.
\label{eq:strongopen}
\end{equation}
Collecting the boundary terms arising in (\ref{eq:appdagger}) and
(\ref{eq:strongopena}), similarly to \cite[(A.32)]{pwfopen}, gives:
\begin{align}
\varrho \left\langle \vec{\mathcal{V}} , \vec\chi \right\rangle_{\gamma(t)}
& = \varsigma \left\langle \vec\varkappa_\gamma, \vec\chi 
\right\rangle_{\gamma(t)} +
\sum_{i=1}^2 \sum_{\ell = 1}^6 B_\ell^{(i)}  \nonumber \\ & =
\varsigma \left\langle \vec\varkappa_\gamma, \vec\chi 
\right\rangle_{\gamma(t)}
+  \sum_{i=1}^2 \left[
 \left\langle \alpha_i\,(\nabs\, \varkappa_i)\,.\,\vec\mu_i
, \vec\chi\,.\,\vec\nu \right\rangle_{\gamma(t)} 
- \left\langle \tfrac12\,\alpha_i\,(\varkappa_i - \spont_i)^2, 
\vec\chi\,.\,\vec\mu_i \right\rangle_{\gamma(t)} \right.
\nonumber \\ & \qquad \qquad \qquad \qquad \left.
- \left\langle \alpha_i\,(\varkappa_i - \spont_i)\,(\nabs\,\vec\nu_i)\,
\vec\mu_i, \vec\chi \right\rangle_{\gamma(t)} 
+ \alpha^G_i \left\langle \vec\varkappa_\gamma\,.\,\vec\mu_i, 
\vec\id_s\,.\,\vec\chi_s \right\rangle_{\gamma(t)} \right.
\nonumber \\ & \qquad \qquad \qquad \qquad \left.
+ \alpha^G_i \left\langle \mat{\mathcal{P}}_\gamma\,\vec\chi_s, [\vec\mu_i]_s
\right\rangle_{\gamma(t)} + B^{(i)} \right] 
\qquad\forall\ \vec\chi \in \xspacec\,.
\label{eq:BBB}
\end{align}

We now investigate the boundary conditions arising from (\ref{eq:BBB}) 
in the case $C_1 = 0$. To this end, we recall from (\ref{eq:yiui}),
(\ref{eq:appybound}) and (\ref{eq:idss}) that 
\begin{equation} \label{eq:vecu}
\alpha_i\,(\varkappa_i - \spont_i)
+\alpha^G_i\,\vec\varkappa_\gamma\,.\,\vec\nu_i 
= 0 \quad\text{and}\quad
\vec u_i = -\alpha^G_i\,(\vec\varkappa_\gamma\,.\,\vec\mu_i)\,\vec\mu_i 
\quad\text{on }\ \gamma(t)\,,\ i=1,2\,.
\end{equation}
Using the simplifications in (A.34)--(A.41) in 
\cite{pwfopen}, as well as (\ref{eq:vecu}), the right hand side of 
(\ref{eq:BBB}) can be simplified to obtain
\begin{align}
\varrho \left\langle \vec{\mathcal{V}} , \vec\chi \right\rangle_{\gamma(t)}
 & =\varsigma \left\langle \vec\varkappa_\gamma, \vec\chi 
\right\rangle_{\gamma(t)} +
\sum_{i=1}^2 \sum_{\ell = 1}^6 B_\ell^{(i)} \nonumber \\ & =
\varsigma \left\langle \vec\varkappa_\gamma, \vec\chi 
\right\rangle_{\gamma(t)}
+ \sum_{i=1}^2 \left\langle
(-\tfrac12\,\alpha_i\,(\varkappa_i - \spont_i)^2 
- \alpha^G_i\,\Gauss_i)\,\vec\mu_i 
+ ( (\alpha_i\,(\nabs\,\varkappa_i)\,.\,\vec\mu_i - 
 \alpha^G_i\,(\tau_i)_s)\,\vec\nu_i, \vec\chi \right\rangle_{\gamma(t)}
\nonumber \\ & \hspace{9cm}
\qquad\forall\ \vec\chi \in \xspacec\,.
\label{eq:BBB2}
\end{align}
It follows from (\ref{eq:vecu}) and (\ref{eq:BBB2}) that the necessary
boundary conditions are
\begin{subequations}
\begin{align}
&
\alpha_i\,(\varkappa_i - \spont_i) + \alpha^G_i\,
\vec\varkappa_\gamma\,.\,\vec\nu_i = 0 
\quad \text{on }\ \gamma(t)\,,\quad i = 1,2\,, \label{eq:strongC0bc1} \\
& \varsigma\, \vec\varkappa_\gamma+ 
\sum_{i=1}^2
(-\tfrac12\,\alpha_i\,(\varkappa_i - \spont_i)^2 
- \alpha^G_i\,\Gauss_i)\,\vec\mu_i 
+ ( (\alpha_i\,(\nabs\,\varkappa_i)\,.\,\vec\mu_i - 
 \alpha^G_i\,(\tau_i)_s)\,\vec\nu_i
 = \varrho\,\vec{\mathcal{V}} \quad
\text{on }\ \gamma(t)\,.
\label{eq:strongC0bc2} \end{align}
\end{subequations}
In the case of surface area preservation, there is an extra term 
\[
-\sum_{i=1}^2 
\lambda^A_i\left\langle \nabs\,\vec\id,\nabs\,\vec\chi 
\right\rangle_{\Gamma_i(t)}
 = -\sum_{i=1}^2 \lambda^A_i \left\langle \nabs\,.\,\vec\chi,1 
\right\rangle_{\Gamma_i(t)} 
= \sum_{i=1}^2 \lambda^A_i\left[\left\langle \varkappa_i\,\vec\nu_i,\vec\chi 
\right\rangle_{\Gamma_i(t)}
-\left\langle 1, \vec\chi\,.\,\vec\mu_i\right\rangle_{\gamma(t)}
\right] 
\]
on the right hand side of (\ref{eq:weakopena}), on
recalling (\ref{eq:areaE}), (\ref{eq:vardet}), (\ref{eq:varEkb}) and
(\ref{eq:DEthm2.10}). 
Similarly, in the case of volume conservation, there is
an extra term $-\lambda^V\,\sum_{i=1}^2 \langle \vec\nu_i,\vec\chi
\rangle_{\Gamma_i(t)}$ on the right hand side of (\ref{eq:weakopena}),
on recalling (\ref{eq:areaE}), a variational variant of (\ref{eq:dtvol}) and
(\ref{eq:DEthm2.10}). 
Hence overall we obtain
\begin{equation} \label{eq:stronggradflowlambda}
\vec{\mathcal{V}}\,.\,\vec\nu_i = - \alpha_i\,\Delta_s\,\varkappa_i +
\tfrac12\,\alpha_i\,(\varkappa_i - \spont_i)^2 \,\varkappa_i -
\alpha_i\,(\varkappa_i - \spont_i)\,|\nabs\,\vec\nu_i|^2 +
\lambda^A_i\,\varkappa_i - \lambda^V
\quad\text{on } \Gamma_i(t)\,,
\end{equation}
in place of (\ref{eq:strongopen}), as well as 
\begin{equation} \label{eq:strong2C0bc2}
\sum_{i=1}^2 \left[
 ( (\alpha_i\,(\nabs\,\varkappa_i)\,.\,\vec\mu_i - 
 \alpha^G_i\,(\tau_i)_s)\,\vec\nu_i
-(\tfrac12\,\alpha_i\,(\varkappa_i - \spont_i)^2 
+ \alpha^G_i\,\Gauss_i + \lambda^A_i)\,\vec\mu_i \right] +
 \varsigma\, \vec\varkappa_\gamma
 = \varrho\,\vec{\mathcal{V}} 
\quad \text{on }\ \gamma(t)\,,
\end{equation}
in place of (\ref{eq:strongC0bc2}).

We now investigate the boundary conditions arising from (\ref{eq:BBB}) 
in the case $C_1 = 1$, where we recall that in this situation 
$\vec\nu = \vec\nu_1 = \vec\nu_2$ and $\vec\mu = \vec\mu_2 = - \vec\mu_1$
on $\gamma(t)$.
To this end, we obtain from (\ref{eq:yiui}), (\ref{eq:appybound}) 
and (\ref{eq:idss}) that 
$\vec y_2 - \vec y_1 = -(\alpha^G_2 - \alpha^G_1)\,\vec\varkappa_\gamma$
on $\gamma(t)$, and hence
\begin{subequations}
\begin{align} 
\alpha_1\,(\varkappa_1 - \spont_1) +
\alpha^G_1\,\vec\varkappa_\gamma\,.\,\vec\nu & =
\alpha_2\,(\varkappa_2 - \spont_2) +
\alpha^G_2\,\vec\varkappa_\gamma\,.\,\vec\nu \quad\text{on }\ \gamma(t)\,,
\label{eq:akkakn} \\
\vec u_2 - \vec u_1 & = -(\alpha^G_2 - \alpha^G_1)\,
(\vec\varkappa_\gamma\,.\,\vec\mu)\,\vec\mu
\quad\text{on }\ \gamma(t)\,.\label{eq:vecuC1}
\end{align}
\end{subequations}
We now rewrite some of the terms in (\ref{eq:BBB}). To this end,
we first note that it follows from (\ref{eq:idss}), 
(\ref{eq:secondform}) and (\ref{eq:II}) that
\begin{equation} \label{eq:kappanu}
\vec\varkappa_\gamma\,.\,\vec\nu_i = 
\vec\id_{ss}\,.\,\vec\nu_i = - \vec\id_s\,.\,[\vec\nu_i]_s
= \II_i(\vec\id_s,\vec\id_s) \quad\text{and}\quad
\varkappa_i = \II_i(\vec\id_s,\vec\id_s) + \II_i(\vec\mu_i,\vec\mu_i)
\quad\text{on }\ \gamma(t)\,.
\end{equation}
Moreover, it follows from (\ref{eq:torsion2}) and (\ref{eq:mu}) that
\begin{equation} \label{eq:torsion}
[\vec\nu_i]_s \times \vec\id_s = - \tau_i\,\vec\mu_i\times\vec\id_s =
(-1)^{i}\,\tau_i\,\vec\nu_i \qquad \text{on } \gamma(t)\,,
\end{equation}
where we have observed that $[\vec\nu_i]_s$ is perpendicular to $\vec\nu_i$.
We also note from (\ref{eq:torsion}) and (\ref{eq:mu}) that
\begin{equation} \label{eq:mus}
[\vec\mu_i]_s = 
(-1)^{i}\left(
[\vec\nu_i]_s\times \vec\id_s + \vec\nu_i\times\vec\id_{ss} \right)
= \tau_i\,\vec\nu_i + (-1)^{i}\,(\vec\varkappa_\gamma\,.\,\vec\mu_i)\,
\vec\nu_i\times \vec\mu_i 
= \tau_i\,\vec\nu_i - (\vec\varkappa_\gamma\,.\,\vec\mu_i)\,\vec\id_s
\quad \text{on } \gamma(t)\,.
\end{equation}
Now it follows from (\ref{eq:II}) that
\begin{equation*} 
(\nabs\,\vec\nu_i)\,\vec\mu_i\,.\,\vec\chi
= -\II_i(\vec\mu_i,\vec\id_s)\,\vec\chi\,.\,\vec\id_s -
\II_i(\vec\mu_i,\vec\mu_i)\,\vec\chi\,.\,\vec\mu_i\,,
\end{equation*}
and so
\begin{align}
B_3^{(i)} & = 
\alpha_i \left\langle (\varkappa_i - \spont_i), 
\II_i(\vec\mu_i, \vec\id_s)\,\vec\chi\,.\,\vec\id_s 
+ \II_i(\vec\mu_i,\vec\mu_i)\,\vec\chi\,.\,\vec\mu_i 
\right\rangle_{\gamma(t)} \nonumber \\ &
= \alpha_i \left\langle (\varkappa_i - \spont_i), 
\tau_i\,\vec\chi\,.\,\vec\id_s + (\varkappa_i -
\vec\varkappa_\gamma\,.\,\vec\nu_i)\,
\vec\chi\,.\,\vec\mu_i \right\rangle_{\gamma(t)}
\,,
\label{eq:B3}
\end{align}
where we have noted (\ref{eq:torsion2}) and (\ref{eq:kappanu}). 
It follows from (A.36) and (A.37) in \cite{pwfopen} that
\begin{align}
B_4^{(i)} + B_5^{(i)} & = 
\alpha^G_i \left\langle \left[ 
\vec\varkappa_\gamma\,.\,\vec\nu_i\,\tau_i 
- (\vec\varkappa_\gamma\,.\,\vec\mu_i)_s \right] \vec\id_s
+ \left[ \tau_i^2 - (\vec\varkappa_\gamma\,.\,\vec\mu_i)^2 \right]\vec\mu_i 
,\vec\chi \right\rangle_{\gamma(t)} \nonumber \\ & \qquad
- \alpha^G_i \left\langle 
 \left[ (\tau_i)_s + (\vec\varkappa_\gamma\,.\,\vec\mu_i)\,
  \vec\varkappa_\gamma\,.\,\vec\nu_i \right] \vec\nu_i
,\vec\chi \right\rangle_{\gamma(t)}\,.
\label{eq:B45}
\end{align}
We have from (\ref{eq:B}) above and (A.39) in \cite{pwfopen} that
\begin{equation} \label{eq:D12}
D_1^{(i)} + D_2^{(i)} =
\left\langle (\vec u_i)_s, (\vec\chi\,.\,\vec\mu_i)\,
\vec\id_s - (\vec\chi\,.\,\vec\id_s)\,\vec\mu_i \right\rangle_{\gamma(t)} .
\end{equation}
As $\vec u_i\,.\,\vec\nu_i=0$, we have from (\ref{eq:B}), (\ref{eq:II}) and
(\ref{eq:kappanu}) that
\begin{align}
D_3^{(i)} + D_4^{(i)} & =
- \left\langle \varkappa_i\,\vec\mu_i + (\nabs\,\vec\nu_i)\,\vec\mu_i,
(\vec\chi\,.\,\vec\nu_i)\,\vec u_i \right\rangle_{\gamma(t)} 
\nonumber \\ &
= \left\langle \II_i(\vec\mu_i,\vec\id_s), (\vec
u_i\,.\,\vec\id_s)\,\vec\chi\,.\,\vec\nu_i \right\rangle_{\gamma(t)} 
+ \left\langle \II_i(\vec\mu_i,\vec\mu_i) - \varkappa_i, (\vec
u_i\,.\,\vec\mu_i)\,\vec\chi\,.\,\vec\nu_i \right\rangle_{\gamma(t)} 
\nonumber \\ &
= \left\langle \tau_i\, \vec u_i\,.\,\vec\id_s -
(\vec\varkappa_\gamma\,.\,\vec\nu_i)\,\vec u_i\,.\,\vec\mu_i,
\vec\chi\,.\,\vec\nu_i \right\rangle_{\gamma(t)} .
\label{eq:D34}
\end{align}
Therefore (\ref{eq:B}), (\ref{eq:D12}) and (\ref{eq:D34}) yield that
\begin{align*}
\sum_{i=1}^2 
B_6^{(i)} & = 
\left\langle (\vec u_2 - \vec u_1)_s, (\vec\chi\,.\,\vec\mu)\,
\vec\id_s - (\vec\chi\,.\,\vec\id_s)\,\vec\mu \right\rangle_{\gamma(t)} 
+ \left\langle (\vec u_2 - \vec u_1)\,.\,\vec\id_s, 
\tau\,(\vec\chi\,.\,\vec\nu) \right\rangle_{\gamma(t)} \nonumber \\ & \qquad
- \left\langle (\vec u_2 - \vec u_1)\,.\,\vec\mu
(\vec\varkappa_\gamma\,.\,\vec\nu)\,\vec\chi\,.\,\vec\nu 
\right\rangle_{\gamma(t)} , 
\end{align*}
where $\tau = \tau_2 = -\tau_1$. Hence we obtain from (\ref{eq:vecuC1})
and (\ref{eq:mus}) that
\begin{align}
\sum_{i=1}^2 
B_6^{(i)} & = [\alpha^G_i]_1^2
\left\langle (\vec\varkappa_\gamma\,.\,\vec\mu)_s\,\vec\id_s 
+ (\vec\varkappa_\gamma\,.\,\vec\mu)^2\,\vec\mu
+ (\vec\varkappa_\gamma\,.\,\vec\mu)\,(\vec\varkappa_\gamma\,.\,\vec\nu)
\,\vec\nu , \vec\chi \right\rangle_{\gamma(t)} .
\label{eq:B6new}
\end{align}
Combining (\ref{eq:BBB}), (\ref{eq:B3}), (\ref{eq:B45}) and (\ref{eq:B6new}) 
yields, on noting (\ref{eq:akkakn}) and (\ref{eq:idss}), that
\begin{align*}
&\varrho \left\langle \vec{\mathcal{V}} , \vec\chi \right\rangle_{\gamma(t)}
 =\varsigma \left\langle \vec\varkappa_\gamma, \vec\chi 
\right\rangle_{\gamma(t)} +
\sum_{i=1}^2 \sum_{\ell = 1}^6 B_\ell^{(i)} 
=
 \left\langle [\alpha_i\,(\nabs\, \varkappa_i)]_1^2\,.\,\vec\mu
- [\alpha^G_i]_1^2\,\tau_s + \varsigma\,\vec\varkappa_\gamma\,.\,\vec\nu
, \vec\chi\,.\,\vec\nu \right\rangle_{\gamma(t)} 
\nonumber \\ & \quad\quad
+ \left\langle -\tfrac12\,[\alpha_i\,(\varkappa_i - \spont_i)^2]_1^2
+ [\alpha_i\,(\varkappa_i - \spont_i)\,(\varkappa_i -
  \vec\varkappa_\gamma\,.\,\vec\nu)]_1^2
+ [\alpha^G_i]_1^2\,\tau^2 + \varsigma\,\vec\varkappa_\gamma\,.\,\vec\mu,
 \vec\chi\,.\,\vec\mu \right\rangle_{\gamma(t)} 
\nonumber \\ & \hspace{11cm}
\qquad\forall\ \vec\chi \in \xspacec\,. 
\end{align*}
This yields the boundary conditions
\begin{subequations}
\begin{align}
& [ \alpha_i\,(\nabs\,\varkappa_i)]_1^2\,.\,\vec\mu 
- [\alpha^G_i]_1^2\,\tau_s +
\varsigma\,\vec\varkappa_\gamma\,.\,\vec\nu
 = \varrho\,\vec{\mathcal{V}}\,.\,\vec\nu 
\quad\text{on }\ \gamma(t) \,, \label{eq:C1norm} \\
& -\tfrac12\, [ \alpha_i\,(\varkappa_i - \spont_i)^2]_1^2
+ [\alpha_i\,(\varkappa_i - \spont_i)\,(\varkappa_i -
\vec\varkappa_\gamma\,.\,\vec\nu)]_1^2
+ [\alpha^G_i]_1^2\,\tau^2 
 + \varsigma\,\vec\varkappa_\gamma\,.\,\vec\mu = 
\varrho\,\vec{\mathcal{V}}\,.\,\vec\mu 
\quad\text{on }\ \gamma(t)\,,
\label{eq:C1conorm}
\end{align}
\end{subequations}
as well as $\varrho\,\vec{\mathcal{V}}\,.\,\vec\id_s = 0$ on $\gamma(t)$,
which has no effect on the evolution of $(\Gamma_i(t))_{i=1}^2$.
Clearly, (\ref{eq:akkakn}) and (\ref{eq:C1norm},b) yield
the conditions (\ref{eq:C1bc1}--c), on accounting for the surface area
constraints analogously to the case $C_1=0$.

\section*{Acknowledgements}
The authors gratefully acknowledge the support 
of the Regensburger Universit\"atsstiftung Hans Vielberth.

\providecommand\noopsort[1]{}\def\soft#1{\leavevmode\setbox0=\hbox{h}\dimen7=\ht0\advance
  \dimen7 by-1ex\relax\if t#1\relax\rlap{\raise.6\dimen7
  \hbox{\kern.3ex\char'47}}#1\relax\else\if T#1\relax
  \rlap{\raise.5\dimen7\hbox{\kern1.3ex\char'47}}#1\relax \else\if
  d#1\relax\rlap{\raise.5\dimen7\hbox{\kern.9ex \char'47}}#1\relax\else\if
  D#1\relax\rlap{\raise.5\dimen7 \hbox{\kern1.4ex\char'47}}#1\relax\else\if
  l#1\relax \rlap{\raise.5\dimen7\hbox{\kern.4ex\char'47}}#1\relax \else\if
  L#1\relax\rlap{\raise.5\dimen7\hbox{\kern.7ex
  \char'47}}#1\relax\else\message{accent \string\soft \space #1 not
  defined!}#1\relax\fi\fi\fi\fi\fi\fi}


\begin{thebibliography}{10}

\bibitem{AbelsGM15}
H.~Abels\noopsort{GarckeMuller}, H.~Garcke, and L.~M{\"u}ller.
\newblock Local well-posedness for volume-preserving mean curvature and
  {W}illmore flows with line tension.
\newblock \emph{Math. Nachr.}, 289:136--174, 2016.

\bibitem{triplej}
J.~W. Barrett\noopsort{GarckeNurnberg}, H.~Garcke, and R.~N\"urnberg.
\newblock A parametric finite element method for fourth order geometric
  evolution equations.
\newblock \emph{J. Comput. Phys.}, 222:441--462, 2007.

\bibitem{gflows3d}
J.~W. Barrett\noopsort{GarckeNurnberg}, H.~Garcke, and R.~N\"urnberg.
\newblock On the parametric finite element approximation of evolving
  hypersurfaces in {${\mathbb R}^3$}.
\newblock \emph{J. Comput. Phys.}, 227:4281--4307, 2008.

\bibitem{clust3d}
J.~W. Barrett\noopsort{GarckeNurnberg}, H.~Garcke, and R.~N\"urnberg.
\newblock Parametric approximation of surface clusters driven by isotropic and
  anisotropic surface energies.
\newblock \emph{Interfaces Free Bound.}, 12:187--234, 2010.

\bibitem{fdfi}
J.~W. Barrett\noopsort{GarckeNurnberg}, H.~Garcke, and R.~N\"urnberg.
\newblock The approximation of planar curve evolutions by stable fully implicit
  finite element schemes that equidistribute.
\newblock \emph{Numer. Methods Partial Differential Equations}, 27:1--30, 2011.

\bibitem{pwftj}
J.~W. Barrett\noopsort{GarckeNurnberg}, H.~Garcke, and R.~N\"urnberg.
\newblock Elastic flow with junctions: Variational approximation and
  applications to nonlinear splines.
\newblock \emph{Math. Models Methods Appl. Sci.}, 22:1250037, 2012.

\bibitem{tpfs}
J.~W. Barrett\noopsort{GarckeNurnberg}, H.~Garcke, and R.~N\"urnberg.
\newblock On the stable numerical approximation of two-phase flow with
  insoluble surfactant.
\newblock \emph{M2AN Math. Model. Numer. Anal.}, 49:421--458, 2015.

\bibitem{pwfade}
J.~W. Barrett\noopsort{GarckeNurnberg}, H.~Garcke, and R.~N\"urnberg.
\newblock Computational parametric {W}illmore flow with spontaneous curvature
  and area difference elasticity effects.
\newblock \emph{SIAM J. Numer. Anal.}, 54:1732--1762, 2016.

\bibitem{nsns2phase}
J.~W. Barrett\noopsort{GarckeNurnberg}, H.~Garcke, and R.~N\"urnberg.
\newblock Finite element approximation for the dynamics of fluidic two-phase
  biomembranes.
\newblock \emph{M2AN Math. Model. Numer. Anal.}, 51:2319--2366, 2017.

\bibitem{pwfopen}
J.~W. Barrett\noopsort{GarckeNurnberg}, H.~Garcke, and R.~N\"urnberg.
\newblock Stable variational approximations of boundary value problems for
  {W}illmore flow with {G}aussian curvature.
\newblock \emph{IMA J. Numer. Anal.}, 37:1657--1709, 2017.

\bibitem{BaumgartDWJ05}
T.~Baumgart\noopsort{DasWebbJenkins}, S.~Das, W.~W. Webb, and J.~T. Jenkins.
\newblock Membrane elasticity in giant vesicles with fluid phase coexistence.
\newblock \emph{Biophys. J.}, 89:1067--1080, 2005.

\bibitem{BaumgartHW03}
T.~Baumgart\noopsort{HessWebb}, S.~T. Hess, and W.~W. Webb.
\newblock Imaging coexisting fluid domains in biomembrane models coupling
  curvature and line tension.
\newblock \emph{Nature}, 425:821--824, 2003.

\bibitem{ChoksiMV13}
R.~Choksi\noopsort{MorandottiVeneroni}, M.~Morandotti, and M.~Veneroni.
\newblock Global minimizers for axisymmetric multiphase membranes.
\newblock \emph{ESAIM Control Optim. Calc. Var.}, 19:1014--1029, 2013.

\bibitem{CoxL15}
G.~Cox and J.~Lowengrub.
\newblock The effect of spontaneous curvature on a two-phase vesicle.
\newblock \emph{Nonlinearity}, 28:773--793, 2015.

\bibitem{DasJB09}
S.~L. Das, J.~T. Jenkins, and T.~Baumgart.
\newblock Neck geometry and shape transitions in vesicles with co-existing
  fluid phases: {R}ole of {G}aussian curvature stiffness vs. spontaneous
  curvature.
\newblock \emph{Europhys. Lett.}, 86:48003, 2009.

\bibitem{Davis04}
T.~A. Davis.
\newblock Algorithm 832: {UMFPACK} {V}4.3---an unsymmetric-pattern multifrontal
  method.
\newblock \emph{ACM Trans. Math. Software}, 30:196--199, 2004.

\bibitem{DeckelnickDE05}
K.~Deckelnick\noopsort{DziukElliott}, G.~Dziuk, and C.~M. Elliott.
\newblock Computation of geometric partial differential equations and mean
  curvature flow.
\newblock \emph{Acta Numer.}, 14:139--232, 2005.

\bibitem{DeckelnickGR17}
K.~Deckelnick\noopsort{GrunauRoger}, H.-C. Grunau, and M.~R{\"o}ger.
\newblock Minimising a relaxed {W}illmore functional for graphs subject to
  boundary conditions.
\newblock \emph{Interfaces Free Bound.}, 19:109--140, 2017.

\bibitem{Dziuk91}
G.~Dziuk.
\newblock An algorithm for evolutionary surfaces.
\newblock \emph{Numer. Math.}, 58:603--611, 1991.

\bibitem{Dziuk08}
G.~Dziuk.
\newblock Computational parametric {W}illmore flow.
\newblock \emph{Numer. Math.}, 111:55--80, 2008.

\bibitem{DziukE13}
G.~Dziuk and C.~M. Elliott.
\newblock Finite element methods for surface {PDE}s.
\newblock \emph{Acta Numer.}, 22:289--396, 2013.

\bibitem{ElliottS10}
C.~M. Elliott\noopsort{Stinner} and B.~Stinner.
\newblock Modeling and computation of two phase geometric biomembranes using
  surface finite elements.
\newblock \emph{J. Comput. Phys.}, 229:6585--6612, 2010.

\bibitem{ElliottS10a}
C.~M. Elliott\noopsort{Stinner} and B.~Stinner.
\newblock A surface phase field model for two-phase biological membranes.
\newblock \emph{SIAM J. Appl. Math.}, 70:2904--2928, 2010.

\bibitem{ElliottS13}
C.~M. Elliott\noopsort{Stinner} and B.~Stinner.
\newblock Computation of two-phase biomembranes with phase dependent material
  parameters using surface finite elements.
\newblock \emph{Commun. Comput. Phys.}, 13:325--360, 2013.

\bibitem{Helmers11}
M.~Helmers.
\newblock Snapping elastic curves as a one-dimensional analogue of
  two-component lipid bilayers.
\newblock \emph{Math. Models Methods Appl. Sci.}, 21:1027--1042, 2011.

\bibitem{Helmers13}
M.~Helmers.
\newblock Kinks in two-phase lipid bilayer membranes.
\newblock \emph{Calc. Var. Partial Differential Equations}, 48:211--242, 2013.

\bibitem{Helmers15}
M.~Helmers.
\newblock Convergence of an approximation for rotationally symmetric two-phase
  lipid bilayer membranes.
\newblock \emph{Q. J. Math.}, 66:143--170, 2015.

\bibitem{JulicherL93}
F.~J{\"u}licher and R.~Lipowsky.
\newblock Domain-induced budding of vesicles.
\newblock \emph{Phys. Rev. Lett.}, 70:2964--2967, 1993.

\bibitem{JulicherL96}
F.~J{\"u}licher and R.~Lipowsky.
\newblock Shape transformations of vesicles with intramembrane domains.
\newblock \emph{Phys. Rev. E}, 53:2670--2683, 1996.

\bibitem{LowengrubRV09}
J.~S. Lowengrub\noopsort{RatzVoigt}, A.~R{\"a}tz, and A.~Voigt.
\newblock Phase-field modeling of the dynamics of multicomponent vesicles:
  {S}pinodal decomposition, coarsening, budding, and fission.
\newblock \emph{Phys. Rev. E}, 79:0311926, 2009.

\bibitem{MerckerM-C15}
M.~Mercker\noopsort{Marciniak-Czochra} and A.~Marciniak-Czochra.
\newblock Bud-neck scaffolding as a possible driving force in {ESCRT}-induced
  membrane budding.
\newblock \emph{Biophys. J.}, 108:833--843, 2015.

\bibitem{MerckerM-CRH13}
M.~Mercker\noopsort{Marciniak-CzochraRichterHartmann}, A.~Marciniak-Czochra,
  T.~Richter, and D.~Hartmann.
\newblock Modeling and computing of deformation dynamics of inhomogeneous
  biological surfaces.
\newblock \emph{SIAM J. Appl. Math.}, 73:1768--1792, 2013.

\bibitem{Nitsche93}
J.~C.~C. Nitsche.
\newblock Boundary value problems for variational integrals involving surface
  curvatures.
\newblock \emph{Quart. Appl. Math.}, 51:363--387, 1993.

\bibitem{Alberta}
A.~Schmidt and K.~G. Siebert.
\newblock \emph{Design of Adaptive Finite Element Software: The Finite Element
  Toolbox {ALBERTA}}, volume~42 of \emph{Lecture Notes in Computational Science
  and Engineering}.
\newblock Springer-Verlag, Berlin, 2005.

\bibitem{SchmidtS10}
S.~Schmidt and V.~Schulz.
\newblock Shape derivatives for general objective functions and the
  incompressible {N}avier--{S}tokes equations.
\newblock \emph{Control Cybernet.}, 39:677--713, 2010.

\bibitem{Taylor11I}
M.~E. Taylor.
\newblock \emph{Partial differential equations {I}. {B}asic theory}, volume 115
  of \emph{Applied Mathematical Sciences}.
\newblock Springer, New York, 2011.

\bibitem{Troltzsch10}
F.~Tr{\"o}ltzsch.
\newblock \emph{Optimal Control of Partial Differential Equations: Theory,
  Methods and Applications}, volume 112 of \emph{Graduate Studies in
  Mathematics}.
\newblock American Mathematical Society, Providence, RI, 2010.

\bibitem{Tu13}
Z.-C. Tu.
\newblock Challenges in theoretical investigations of configurations of lipid
  membranes.
\newblock \emph{Chin. Phys. B}, 22:28701, 2013.

\bibitem{TuO-Y04}
Z.~C. Tu and Z.~C. Ou-Yang.
\newblock A geometric theory on the elasticity of bio-membranes.
\newblock \emph{J. Phys. A}, 37:11407--11429, 2004.

\bibitem{WangD08}
X.~Wang and Q.~Du.
\newblock Modelling and simulations of multi-component lipid membranes and open
  membranes via diffuse interface approaches.
\newblock \emph{J. Math. Biol.}, 56:347--371, 2008.

\bibitem{Wutz10}
C.~Wutz.
\newblock \emph{{V}ariationsprobleme f{\"u}r elastische {B}iomembranen unter
  {B}er{\"u}cksichtigung von {L}inienenergien}.
\newblock Diploma thesis, University Regensburg, 2010.
\newblock (appeared as Preprint 03/2017, University Regensburg, Germany).

\end{thebibliography}
\end{document}